\documentclass{elsarticle}
\usepackage{graphicx}
\usepackage{multirow}
\usepackage{amsmath,amssymb,amsfonts}
\usepackage{amsthm}
\usepackage{mathrsfs}
\usepackage[title]{appendix}
\usepackage{xcolor}
\usepackage{textcomp}
\usepackage{booktabs}
\usepackage{algorithm}
\usepackage{algorithmicx}
\usepackage{algpseudocode}
\usepackage{listings}
\usepackage{geometry}
\usepackage{indentfirst}
\usepackage{cases}
\usepackage{tikz}
\usepackage{dsfont}
\usepackage{float}
\usepackage{lmodern}
\newtheorem{thm}{Theorem}[section]
\newtheorem{lem}[thm]{Lemma}
\newtheorem{cor}[thm]{Corollary}
\newtheorem{defi}[thm]{Definition}
\newtheorem{prop}[thm]{Proposition}
\newtheorem{rem}[thm]{Remark}
\numberwithin{equation}{section}
\begin{document}
\begin{frontmatter}
\title{Almost global existence and nonlinear asymptotic stability for bubble dynamics in inviscid compressible liquid}
\author[author1]{Liangchen Zou}
\address[author1]{School of Mathematical Sciences, University of Science and Technology of China, Hefei, Anhui, 230026, PR China, zlc0601@mail.ustc.edu.cn}
\date{}
\begin{abstract}
The present paper considers the full nonlinear dynamics of a homogeneous bubble inside an unbounded isentropic compressible inviscid liquid. This model is described by a free-boundary problem of compressible Euler equations with nonlinear boundary conditions. The liquid is governed by the compressible Euler equation, while the bubble surface is determined by the kinematic and dynamic boundary conditions on the bubble-liquid interface. This classical model is of great concern in physics due to its wide applications. \newline
We begin by proving the local existence and uniqueness using energy methods under an iteration scheme. For long-time behavior, we developed a generalized weighted space-time estimate, which extends the Keel-Smith-Sogge estimate to nonlinear wave equations regardless of the boundary conditions, at the cost of the appearance of a boundary term with only lowest-order derivatives. This term is handled by using characteristics to track the backward pressure wave. Then the almost global existence and nonlinear radiative decay are proved through a bootstrap argument, which encompasses the energy estimate, the generalized Keel-Smith-Sogge estimate, and the analysis of backward pressure waves. \newline
The analysis of the backward pressure wave by characteristics involves a loss of derivative due to the quasilinear nature of the system. This is overcome by the above generalized weighted spacetime estimate with the lowest-order boundary term, which actually provides a mechanism to gain derivatives back. The coupling of these two methods is the novelty of the present paper and can not only be used for the current question but is expected to be applied to other questions regarding nonlinear wave equations with complicated boundary conditions.
\end{abstract}
\end{frontmatter}
\section{Introduction}
\subsection{Equations governing the liquid-bubble system}
We consider a spherical gas bubble immersed in a compressible inviscid liquid. This physically relevant model plays an important role in a wide variety of fields, including underwater explosion \cite{1956Damping}, ultrasound imaging \cite{FRINKING2020892}, medicine delivery \cite{KOOIMAN201428}, acoustic blocking \cite{Acousticblocking}, acoustic communication across the water-air interface \cite{Acousticcommunication}, in which cases acoustic waves generated by compressibility and bubble oscillation are essential. For other applications of bubble dynamics, one can refer to \cite{FBPGB,  FSTS, RDBO}, and the references therein. \\
\indent  The liquid under consideration is assumed to be inviscid, compressible and spherically symmetric relative to the bubble. Meanwhile, due to the conservation laws of mass and momentum, by passing to the center of mass coordinate, we can assume without loss of generality that the center of mass of the bubble is located at the origin, see \cite[Appendix C]{RDBO}. Therefore, the dynamics of liquid is governed by the compressible Euler equations, which, written in radial coordinates, are given by 
\begin{equation}
\begin{cases}
{} \partial_t\rho+r^{-2}\partial_r(r^2\rho u)=0, &r>R(t),\;t>0,\\
\rho\partial_t u+\rho u\partial_r u+\partial_r p=0, &r>R(t),\;t>0.\label{1.1}
\end{cases} 
\end{equation}
Here $\rho$, $u$, $p$ are the density, velocity and pressure of liquid, respectively, and $R(t)$ is the bubble radius. On the other hand, the bubble surface is determined by the kinematic condition and the stress balance on the interface:  
\begin{equation}
\begin{cases}
{} \frac{dR}{dt}=u|_{r=R(t)}, &t>0,\\
p|_{r=R(t)}=p_b-2\sigma R^{-1}, &t>0,
\end{cases} \label{1.2}
\end{equation}
where $p_b$ is the pressure at the inner bubble surface, and $\sigma$ denotes the surface tension.  \\
\indent We further make the assumptions that the liquid is isentropic, so $p=C_0\rho^\gamma$ for some positive constants $C_0$ and $\gamma>1$, and that the bubble is homogeneous and satisfies a polytropic gas law, which means that the bubble pressure is uniform in space and is proportional to a power of bubble volume, say $p_b(t)=C_1|R(t)|^{-3\gamma_0}$ for positive constants $C_1$ and $\gamma_0>1$.
By nondimensionalization \cite[Appendix C]{RDBO}, we can assume without loss of generality that the liquid pressure $p=\frac{Ca}{2}\rho^{\gamma}$, the bubble pressure $p_b=\left(\frac{Ca}{2}+\frac{2}{We}\right)R^{-3\gamma_0}$, and the surface tension $\sigma=\frac{1}{We}$,  where $Ca$ and $We$ are dimensionless constants, namely the Cavitation number and the Weber number, respectively. Therefore, the system (\ref{1.1})(\ref{1.2}) admits an equilibrium state $(u,\,\rho,\,R)=(0,1,0)$. We also assume that  the pressure at infinity is equal to the pressure at this equilibrium state, or equivalently, $\rho=1$ at infinity. \\
\indent To deal with this free boundary problem, a natural way is to introduce the Lagrangian coordinate $x:=\int_{R(t)}^r\rho(t,s)s^2dx$, which corresponds to the mass of the liquid within a spherical domain of radius $r$ and exterior to the bubble. Using the continuity equation and that $\frac{dR}{dt}=u|_{r=R(t)}$, there is $\partial_t x(r,t)=-(r^2\rho u)(r,t)$, therefore, the Jacobian $\frac{\partial(x,t)}{\partial(r,t)}$ as well as its inverse are given by
\begin{equation}
\left[\begin{matrix}
\frac{\partial x}{\partial r} & \frac{\partial x}{\partial t}\\
\frac{\partial t}{\partial r} & \frac{\partial t}{\partial t}
\end{matrix}\right]
=\left[\begin{matrix}
\rho r^2 & -\rho r^2 u\\
0 & 1
\end{matrix}\right],\;
\left[\begin{matrix}
\frac{\partial r}{\partial x} & \frac{\partial r}{\partial t}\\
\frac{\partial t}{\partial x} & \frac{\partial t}{\partial t}
\end{matrix}\right]
=\left[\begin{matrix}
(\rho r^2)^{-1} & u\\
0 & 1
\end{matrix}\right]\label{Jacob}.
\end{equation}
By a change of variable using (\ref{Jacob}), equations (\ref{1.1})(\ref{1.2}) are transformed into 
\begin{numcases}
{} \partial_t\rho+\rho^2\partial_x(r^2u)=0, &$x>0,\;t>0$,\label{cnte0}\\
\partial_t u+\frac{Ca}{2}r^2\partial_x(\rho^\gamma) =0, &$x>0,\;t>0$,\label{mmte0}\\
\frac{dR}{dt}=u|_{x=0}, &$t>0$,\label{kbcd0}\\
\frac{Ca}{2}\rho^\gamma|_{x=0}=\left(\frac{Ca}{2}+\frac{2}{We}\right)R^{-3\gamma_0}-\frac{2}{We}R^{-1}=:\tilde{f}(R), &$t>0$,\label{dbcb0}\\
r=\left(R(t)^3+3\int_0^x\rho^{-1}(y,t)dy\right)^{\frac{1}{3}}=r(x,0)+\int_0^t u(y,\tau)d\tau, &$x>0,\;t>0$\label{radius0}.
\end{numcases}
Next, we define $q=\rho^{-1}-1$, and denote the specific sound speed by $c=\left(\frac{Ca\gamma}{2}\rho^{\gamma+1}\right)^{\frac{1}{2}}$ and the value of $c$ at the equilibrium state by $c_0:=\left(\frac{Ca\gamma}{2}\right)^\frac{1}{2}$. Then the fourth equation implies 
$$q|_{x=0}=\left[\left(1+\frac{2}{Ca}\frac{2}{We}\right)R^{-3\gamma_0}-\frac{2}{Ca}\frac{2}{We}R^{-1}\right]^{-\frac{1}{\gamma}}-1=:f(R).$$
$f(R)$ is smooth and strictly increasing for $R\in(0,\overline{R})$ with $\overline{R}=\left(1+\frac{Ca}{2}\frac{We}{2}\right)^{\frac{1}{3\gamma_0-1}}$, and a simple computation gives
$c_0^2f^\prime(1)=-\tilde{f}^\prime(1)$. Hence $f$ admits a smooth inverse function $f^{-1}$ defined on the range $f((0,\overline{R}))$. Replacing $\rho$ by $(1+q)^{-1}$ in the above equations yields
\begin{numcases}
{} \partial_t q=\partial_x(r^2u), &$x>0,\;t>0,$\label{cnte}\\
\partial_t u=c^2r^2\partial_x q, &$x>0,\;t>0,$\label{mmte}\\
\frac{dR}{dt}=u|_{x=0}, &$t>0$,\label{kbce}\\
q|_{x=0}=f(R), &$t>0,$\label{dbce}\\
r=\left(R(t)^3+3x+3\int_0^xq(y,t)dy\right)^{\frac{1}{3}}=r(x,0)+\int_0^t u(y,\tau)d\tau, &$x>0,\;t>0.$\label{radius}
\end{numcases}
We now impose the initial value at $t=0$:
\begin{equation}
(u,\,q,\, R)|_{t=0}=(u_{in},\, q_{in},\, R_{in}),\label{initial}
\end{equation}
and compatibly $r(x,0)=\left(R_{in}^3+3x+\int_0^xq_{in}(y)dy\right)^{\frac{1}{3}}$. Moreover, we assume that the initial value satisfies the compatible conditions up to first order. The $0$-th order compatible condition is obtained by passing $t$ to $0$ in (\ref{dbce}):
\begin{equation}
q_{in}(0)=f(R_{in}).
\label{CC0}\end{equation}
To obtain the first-order compatible condition, we use (\ref{kbce}) and take the time derivative in (\ref{dbce}) to get
$\partial_tq|_{x=0}=f^\prime(R)u|_{x=0}$. Then passing $t$ to $0$ and using (\ref{cnte}) yields
\begin{equation}
    f^\prime(R_{in})u_{in}(0)=\partial_x(r_{in}^2u_{in})(0)=R^2_{in}\partial_xu_{in}(0)+2R_{in}^{-1}(0)(1+q_{in}(0))u_{in}(0).
\label{CC1}\end{equation}
The bubble-liquid system is then determined by equations (\ref{cnte}-\ref{radius}) with the initial value (\ref{initial}) satisfying the compatible conditions (\ref{CC0})(\ref{CC1}).\\
\indent In addition to the above equations of $(u,q,R)$, for later use, we hereby derive the equation governing the velocity potential $\varphi$, which is defined by $\rho r^2\partial_x\varphi=u$ and $\varphi=0$ at infinity. Substituting $u$ by $\rho r^2\partial_x\varphi$ in (\ref{mmte0}) and using (\ref{Jacob})(\ref{cnte0}), there is
$$\begin{aligned}
&-\frac{Ca\gamma}{2(\gamma-1)}\rho r^2\partial_x\left(\rho^{\gamma-1}\right)=\partial_tu\\
=&\rho r^2\partial_x\partial_t\varphi+\partial_t(\rho r^2)\partial_x\varphi=\rho r^2\partial_x\partial_t\varphi-\rho^2\partial_x(r^2u)r^2\partial_x\varphi+2r\rho u\partial_x\varphi\\
=&\rho r^2\partial_x\partial_t\varphi-\rho^2r^4\partial_x u\partial_x\varphi=\rho r^2\partial_x\left(\partial_t\varphi-\frac{u^2}{2}\right).
\end{aligned}$$
Hence by canceling $\rho r^2\partial_x$ on each side, we obtain the relation between $\rho$ and $\varphi$, namely the Bernoulli equation: 
\begin{equation}
-\frac{Ca}{2}\frac{\gamma}{\gamma-1}(\rho^{\gamma-1}-1)=\partial_t\varphi-\frac{u^2}{2}.\label{rho-phi}
\end{equation}
Multiplying (\ref{cnte0}) by $\rho^{\gamma-2}$ and replacing $\rho$ using (\ref{rho-phi}) yields
$$
0=\partial_t(\partial_t\varphi-\frac{u^2}{2})-\frac{Ca\gamma}{2}\rho^\gamma\partial_x(\rho r^4\partial_x\varphi).
$$
Since $u\cdot\partial_tu=-\rho r^2\partial_x\varphi\cdot \frac{Ca\gamma}{2}\rho^{\gamma-1}r^2\partial_x\rho$, $\varphi$ satisfies the following  quasilinear wave equation on $(x,t)\in\mathbb{R}_+\times\mathbb{R}_+$, 
\begin{equation}
\partial_t^2\varphi-c^2\partial_x(r^4\partial_x\varphi)=0. \label{phiwave}
\end{equation}
\subsection{Observation from linear approximation}
To illustrate the damping mechanism of bubble radius, we look at the linearization of (\ref{phiwave}) at the equilibrium state. Denote by $\xi(x)=(1+3x)^{\frac{1}{3}}$ the radius function and by $c_0=\left(\frac{Ca\gamma}{2}\right)^{\frac{1}2}$ the sound speed at equilibrium state, then the linearized equation reads
$$\partial_t^2\varphi-c_0^2(\partial_\xi^2+\frac{2}{\xi}\partial_\xi)\varphi=0.$$
In view of (\ref{rho-phi}), the boundary conditions (\ref{kbce})(\ref{dbce}) under linearization become $\frac{dR}{dt}=\partial_\xi\varphi|_{\xi=1}$, and $\partial_t\varphi|_{\xi=1}=c_0^2f'(1)(R-1)$. Eliminating $R$ by taking $t$-derivative in the latter yields $$\partial_t^2\varphi|_{\xi=1}-c_0^2f'(1)\partial_{\xi}\varphi|_{\xi=1}=0.$$
Let $\varphi=\frac{\psi}{\xi}$, then $\psi$ satisfy the one-dimensional wave equation $\partial_t^2\psi-c_0^2\partial_\xi^2\psi=0$, while the above equation of boundary value becomes 
$$
\partial_t^2\psi|_{\xi=1}+c_0f'(1)\partial_t\psi|_{\xi=1}+c_0^2f'(1)\psi|_{\xi=1}=c_0f'(1)(\partial_t\psi+c_0\partial_\xi\psi)|_{\xi=1}.
$$
This is a second-order inhomogeneous ODE of $\psi|_{\xi=1}$ with the source term $c_0f'(1)(\partial_t\psi+c_0\partial_\xi\psi)|_{\xi=1}$, which corresponds to the backward pressure wave. Since $\psi$ satisfies the one-dimensional wave equation, we have
$$\left.\left(\partial_t\psi+c_0\partial_\xi\psi\right)\right|_{\xi=1}(t)=\left(\partial_t\psi+c_0\partial_\xi\psi\right)(1+c_0t,0).$$
The two eigenvalues, denoted by $\Lambda_1$, $\Lambda_2$, have negative real parts since $f$ is strictly increasing. Moreover, the two eigenvalues coincide if and only if $f^\prime(1)=4$. In this paper, we will focus on the nonlinear stability, so for simplicity we exclude the special case $f^\prime(1)=4$.  Therefore, $\Lambda_1\neq\Lambda_2$, and $(\psi|_{\xi=1},\partial_t\psi|_{\xi=1})$ as well as $R-1$ enjoy an exponential decay in time if we put the source term aside. In fact, solving the above ODE gives
$$\begin{aligned}
\partial_t\psi|_{\xi=1}(t)
=&\frac{\Lambda_1}{\Lambda_1-\Lambda_2}e^{\Lambda_1t}\left.\left(\partial_t\psi-\Lambda_2\psi\right)\right|_{\xi=1}(0)-\frac{\Lambda_2}{\Lambda_1-\Lambda_2}e^{\Lambda_2t}\left.\left(\partial_t\psi-\Lambda_1\psi\right)\right|_{\xi=1}(0)\\
&+c_0f^\prime(1)\int_0^t\left(\frac{\Lambda_1}{\Lambda_1-\Lambda_2}e^{\Lambda_1(t-s)}-\frac{\Lambda_2}{\Lambda_1-\Lambda_2}e^{\Lambda_2(t-s)}\right)(\partial_t\psi+c_0\partial_\xi\psi)(1+c_0s,0)ds.
\end{aligned}$$
Heuristically, for the nonlinear equations, $(\partial_t\psi+c_0\partial_\xi\psi)$ will be replaced by  $$w_B:=\partial_t\psi+cr^2\partial_x\psi$$ since it stands for the backward pressure wave, and $(1+c_0s,0)$ should be replaced by $(\xi_0(s),0)$, where $\xi_0(s)$ is defined such that the characteristic starting from $(\xi_0(s),0)$ intersects $\xi=1$ at $(1,s)$. This leads to defining
\begin{equation}\begin{aligned}
\mathcal{R}(t):=&\frac{\Lambda_1}{\Lambda_1-\Lambda_2}e^{\Lambda_1t}\left.\left(\partial_t\psi-\Lambda_2\psi\right)\right|_{\left\{x=0,t=0\right\}}-\frac{\Lambda_2}{\Lambda_1-\Lambda_2}e^{\Lambda_2t}\left.\left(\partial_t\psi-\Lambda_1\psi\right)\right|_{\left\{x=0,t=0\right\}}\\
&+c_0f^\prime(1)\int_0^t\left(\frac{\Lambda_1}{\Lambda_1-\Lambda_2}e^{\Lambda_1(t-s)}-\frac{\Lambda_2}{\Lambda_1-\Lambda_2}e^{\Lambda_2(t-s)}\right)w_B(\xi_0(s),0)ds,
\end{aligned}\label{calR}\end{equation}
where we also modify the definition of $\psi$ to $\psi=r\varphi.$ In the course of later proof, $\mathcal{R}(t)$ will be regarded as the principle part of $\partial_t\psi|_{\xi=1}(t)$.  Meanwhile, we denote the forward pressure wave by $w_F:=\partial_t\psi-cr^2\partial_x\psi.$ 
We will derive the precise nonlinear boundary ODE in Section \ref{sec5} and specify the source term thereby, while the control of this source term is given in Section \ref{sec6}. 
\subsection{Review on previous results regarding bubble dynamics and free-boundary Euler equations}
The bubble dynamics was first studied by L. Rayleigh \cite{RAY} in 1917, which describes how the liquid pressure develops in the process of collapse of a vacuum cavity in water. This first attempt considered the most simplified model. The cavity and liquid are assumed to be spherical, and other factors including compressibility, viscosity, and thermal effects are omitted, while the pressure at infinity is assumed to be constant. Plesset (1949) \cite{10.1115/1.4009975} extended Rayleigh's work to the case where the pressure at infinity is time-dependent, which also established the well-known Plesset-Rayleigh equation. Plesset-Rayleigh equations have been widely applied in numerical simulations of bubble-related models in which compressibility is less important. \\
\indent To describe the pressure wave transmitted from an underwater explosion into the surrounding liquid and to explain the damping of the bubble oscillation which the Plesset-Rayleigh model failed to predict, Keller (1956) \cite{1956Damping} modified Plesset's theory by taking compressibility into consideration and assumed the velocity potential satisfies a wave equation, which, as explained in a later work (1980) \cite{Ke-Mi}, is derived from the linear approximation of Euler equations. It was also in \cite{Ke-Mi} that Keller and Miksis established another frequently used ordinary differential equation (Keller-Miksis equation) regarding bubble dynamics when pressure wave emission has to be taken into account. Numerical simulations using the Keller-Miksis model indicate radiation decay due to acoustic wave emission by the bubble oscillation \cite{smith_wang_2018, wang_2016}, which also agrees well with the experimental results.\\
\indent Prosperetti and Lezzi (1987) \cite{lezzi_prosperetti_1987, prosperetti_lezzi_1986} developed a second-order theory (in the sense of asymptotic expansion) in terms of the bubble-wall Mach number, which significantly increased the determinacy compared with the previous first-order models. There are also many other generalizations of the Plesset-Rayleigh equation regarding additional physical factors, including viscosity \cite{1952Gil}, the loss of spherical symmetry \cite{RDBO, KLASEBOER200659}, the presence of mass exchange between bubble and liquid \cite{Prosperettigeneralization}, thermal effects \cite{Franc2007}. For a history of the understanding of bubbles, the reader can refer to the review by Prosperetti in the book \cite[Page 735]{PrnstCp}. Furthermore, numerical simulations have been carried out to identify the mechanisms and features of non-spherical dynamics of acoustic cavitation bubbles \cite{wang_blake_2010, wang_blake_2011}.\\
\indent There are also many results on the bubble dynamics in the recent years. When thermal effects are considered and the outside liquid is assumed to be unbounded and incompressible, Lai and Weinstein \cite{FBPGB, LAI2024113397} proved exponential asymptotic stability for a family of spherically symmetric equilibrium states with respect to spherically symmetric perturbations. When the outside liquid is bounded, Hao-Luo-Yang \cite{HLY2024expstab} proved the nonlinear exponential asymptotic stability of equilibrium for a gas bubble. In linear approximation, Weinstein and his collaborators in \cite{doi:10.1137/120892659, RDBO} proved exponential stability for a non-spherical bubble in compressible liquid under the assumption that the velocity is $0$ and the liquid pressure is a constant at the initial time. As to nonlinear questions regarding the bubble in a compressible liquid, our previous work \cite{2022arXiv221200299Z} studied the model of a spherical gas bubble in a compressible viscous fluid and proved the nonlinear viscous damping for the bubble oscillation. In this work, we study the full nonlinear free-boundary question for the bubble dynamics in a compressible inviscid liquid. Moreover, we admit non-constant initial values of velocity and density.\\
\indent As presented in (\ref{1.1}-\ref{1.2}), the bubble dynamics is described by compressible Euler equations on a free-boundary domain. There have been numerous studies on well-posedness and singularity formation for free-boundary problems of Euler equations. In the incompressible setting, Coutand-Shkoller \cite{https://doi.org/10.1090/S0894-0347-07-00556-5, 10.3934/dcdss.2010.3.429} proved the local well-posedness for inviscid fluids on bounded free-boundary domains with or without surface tension. On singularity formation, Castro-C\'{o}rdoba-Fefferman-Gancedo-G\'{o}mez-Serrano \cite{https://doi.org/10.4007/annals.2013.178.3.6} constructed finite-time splash singularity solutions for inviscid fluid on free-boundary domains. Meanwhile, the study of free-boundary compressible Euler equations has been developing rapidly in the past decade. Local well-posedness has been established for gas in a physical vacuum via Lagrangian coordinates by many authors \cite{10.1007/s00220-010-1028-5, 10.1007/s00205-012-0536-1, 10.3934/dcdsb.2015.20.2885, https://doi.org/10.1002/cpa.21517} and via Eulerian coordinates by Ifrim-Tataru \cite{10.4171/AIHPC/91}. For the liquid case, an a priori estimate is established by Lindblad-Luo \cite{10.1002/cpa.21734}. Local well-posedness is also obtained for other Euler-based free-boundary problems. Luo-Xin-Zeng \cite{10.1007/s00205-014-0742-0} proved local well-posedness for self-gravitating gas. Lindblad-Zhang \cite{10.1007/s00205-023-01917-1} proved local well-posedness for free-boundary magnetohydrodynamics; see also \cite{10.1007/s00526-023-02462-1}.
\subsection{Statement of the main result}
The goal of the present paper is to rigorously study the full nonlinear dynamics of the acoustic bubble in a compressible, inviscid liquid without any approximation or asymptotic expansion. From (\ref{cnte}-\ref{initial}) and (\ref{phiwave}), we see that this question corresponds to a quasilinear wave equation on a free-boundary exterior domain and nonlinear boundary conditions. The following is the statement of results in this work. Since we will work with the $(x,t)$ variables instead of the usual Euclidean coordinates, we first introduce the functional spaces that we will work in.
\begin{defi}
Write $\xi(x)=(1+3x)^{\frac{1}{3}}$, and thus $x=\frac{\xi^3-1}{3}$, $\partial_\xi=\xi^2\partial_x$. For $j\geq 0$, define $H^j:=\{v\in L^2(0,+\infty): (\xi^2\partial_x)^iv\in L^2(0,+\infty) \text{ for } i\in [0,j]\cap\mathbb{Z} \}$. For each $M\in(1,+\infty]$, denote by $L^2_\xi(1,M):=\{\tilde{v}\in L^1_{loc}(1,M):\int_1^M|\tilde{v}(\xi)|^2\xi^2d\xi<+\infty\}$ the space of 3-dimensional radial square integrable functions on $B(0,M)\setminus\overline{B(0,1)}$, and $H_\xi^j(1,M):=\{\tilde{v}\in L^2_\xi(1,M):\partial_\xi^i\tilde{v}\in L_\xi^2(1,M)\text{ for } i\in[0,j]\cap\mathbb{Z}\}$ the corresponding Sobolev spaces. 
\end{defi}
From the definition, we immediately see that if $v\in L^2(0,+\infty)$, then $\tilde{v}(\xi):=v(\frac{\xi^3-1}{3})\in L^2_\xi(1,+\infty)$, and $v\in H^j$ is equivalent to $\tilde{v}\in H_\xi^j(1,+\infty)$. We also define the operators $L_0^j$, $\tilde{L}_0^j$, $j=1,2$ by
$$L_0^1q:=c_0^2\partial_\xi q,\quad L_0^2q:=c_0^2\left(\partial_\xi^2+\frac{2}{\xi}\partial_\xi\right)q,$$
$$\tilde{L}_0^1u=\left(\partial_\xi+\frac{2}{\xi}\partial_\xi\right)u,\quad \tilde{L}_0^2u=c_0^2\left(\partial_\xi^2+\frac{2}{\xi}\partial_\xi-\frac{2}{\xi^2}\right)u.$$
The main results are the local existence for general regular data and almost global existence for small data. 
\begin{thm}[Local existence and uniqueness]
Let the initial value $(u_{in},\,q_{in},\,R_{in})\in H^2\times H^2\times\mathbb{R}_+$ be such that $-1<\underline{q}\leq q_{in}\leq\overline{q}<+\infty$ for all $x>0$ and constants $\underline{q}$, $\overline{q}$, and that $0<R_{in}<\overline{R}$. Moreover, we suppose that the initial value satisfies the compatible conditions (\ref{CC0})(\ref{CC1}). Then there exists a time $T>0$ and a unique solution
$$(u,\,q,\,R)\in L^\infty([0,T];H^2)\cap C([0,T];H^1)\times L^\infty([0,T];H^2)\cap C([0,T];H^1)\times C^1[0,T]$$
to (\ref{cnte}-\ref{initial}) with initial value $(u_{in},\,q_{in},\,R_{in})$. Moreover, a lower bound of $T$ can be given in terms of $\|u_{in}\|_{H^2}$, $\|q_{in}\|_{H^2}$, $\underline{q}$, $\overline{q}$ and $|f(R_{in})|$.
\label{thm 1.2}\end{thm}
\begin{thm}[Almost global existence]
Suppose $f^\prime(1)\neq 4$, and $(u_{in},q_{in},R_{in})$ satisfies the compatible condition (\ref{CC0})(\ref{CC1}). Define velocity potential $\varphi$ by $u=\rho r^2\partial_x\varphi,$ and let $\psi:=r\varphi$. \\
(\romannumeral1) There exist small constants $\epsilon_0,\tilde{\epsilon}_0,\kappa_0>0$ such that for initial value  $(u_{in},\,q_{in},\,R_{in})\in H^2\times H^2\times \mathbb{R}_+$ with the the smallness
$$\epsilon^2:=\sum_{j=0}^2\|L_0^jq_{in}\|_{L^2}^2+\sum_{j=0}^2\|\tilde{L}_0^ju_{in}\|_{L^2}^2+|f(R_{in})|^2\leq\epsilon_0^2,$$
\begin{equation}
\tilde{\epsilon}:=\|\xi^2u_{in}\|_{L^2}+\|\xi^2q_{in}\|_{L^2}\leq\tilde{\epsilon}_0,
\label{1.18}\end{equation}
the system (\ref{cnte}-\ref{initial}) has a unique solution $(u,\,q,\,R)$ on $[0,T_\epsilon]$ satisfying
$$(u,\,q)\in L^\infty\left([0,T_\epsilon]; H^2\right)\cap C\left([0,T_\epsilon];H^1\right)\times L^\infty\left([0,T_\epsilon]; H^2\right)\cap C\left([0,T_\epsilon];H^1\right),\;R\in C^1[0, T_\epsilon].$$
Moreover, the sound speed can be bounded from below and above by $\underline{c}$ and $\overline{c}$, which are small perturbations of $c_0$:
$$-\epsilon\lesssim\underline{c}-c_0\leq c-c_0\leq\overline{c}-c_0\lesssim\epsilon,$$
and the lifespan can be bounded below by 
$T_\epsilon\geq T_0:=\overline{c}^{-1}\left(\exp\left(\frac{\kappa_0}{\epsilon}\right)-1\right).$\\
(\romannumeral2) 
For $t\in[0,T_0]$, the bubble radius $R$ satisfies the decay estimate 
$$\left|F(R)-\mathcal{R}(t)\right|\lesssim\frac{\log(1+\overline{c}t)}{1+\underline{c}t}\epsilon^\frac{3}{2}\tilde{\epsilon}^\frac{1}{2}+\frac{\epsilon^\frac{3}{2}}{1+\underline{c}t}\left(\epsilon^\frac{1}{2}+\tilde{\epsilon}^\frac{1}{2}\right)+\frac{\epsilon}{(1+\underline{c}t)^2}(\epsilon+\tilde{\epsilon}),$$
where $F$ is an explicit smooth function defined around $R=1$, cf. (\ref{7.11})(\ref{7.13}), and $\mathcal{R}(t)$ is as in (\ref{calR}).
If we assume additionally $w_B|_{t=0}=0$, where $w_B:=\partial_t\psi+cr^2\partial_x\psi$, then the above estimate can be improved to
$$\left|F(R)-\mathcal{R}(t)\right|\lesssim\frac{\epsilon^\frac{3}{2}}{1+\underline{c}t}\left(\epsilon^\frac{1}{2}+\tilde{\epsilon}^\frac{1}{2}\right).$$
\label{thm 1.3}\end{thm}
If we assume that the initial value is compactly supported in $[0, x_b)$, so that the entire backward pressure wave at $t=0$ arrives at the bubble surface in finite time $T_b:=\underline{c}^{-1}(\xi_b-1)$, where $\xi_b=(1+3x_b)^\frac{1}{3}$, then we can also derive a pointwise decay estimate for the velocity and density in a set slightly smaller than the backward acoustic cone determined by the points $(1, T_b)$ and $(1,T_0)$, which is characterized by
\begin{equation}
D_b:=\left\{(x,t)\mid\xi-1\leq\underline{c}\min\left\{t-T_b,T_0-t\right\},\,\xi=(1+3x)^\frac{1}{3},\,T_b\leq t\leq T_0\right\}.\label{1.19}\end{equation}
\begin{cor}[Pointwise decay near the surface]
Let $(u_{in},\,q_{in},\,R_{in})$ be as in Theorem \ref{thm 1.3} with $\epsilon\ll\kappa_0$. Suppose $u_{in}$, $q_{in}$ are compactly supported in $[0,x_b)$ with $\xi_b\ll\exp\left(\frac{\kappa_0}{\epsilon}\right)$. Then for each point $(x,t)\in D_b$, the solution $(u,\,q,\,R)$ given by Theorem \ref{thm 1.3} with the initial value $(u_{in},\,q_{in},\,R_{in})$ satisfies the following pointwise decay for $t\gtrsim1$:
$$\begin{aligned}
|u(x,t)|+|q(x,t)|\lesssim&\xi^{-1}e^{\Lambda(t-T_b)}\left(|Y(0)|+\frac{|\tilde{f}^\prime(1)|}{c_0}\int_0^{T_b}e^{-\Lambda s}|w_B(\xi_0(s),0)|ds\right)\\
&+\xi^{-1}(2+\underline{c}t-\xi)^{-1}\log\left(1+\overline{c}\left(t+\frac{\xi-1}{\underline{c}}\right)\right)\epsilon^2\xi_b,
\end{aligned}$$
where $\Lambda:=\max\left\{\text{Re}\Lambda_1,\text{Re}\Lambda_2\right\}$. If assume additionally $w_B|_{t=0}=0$, then the above bound can be improved to
$$\begin{aligned}
|u(x,t)|+|q(x,t)|\lesssim&\xi^{-1}e^{\Lambda(t-T_b)}\left(|Y(0)|+\frac{|\tilde{f}^\prime(1)|}{c_0}\int_0^{T_b}e^{-\Lambda s}|w_B(\xi_0(s),0)|ds\right)\\
&+\xi^{-1}(2+\underline{c}t-\xi)^{-1}\epsilon^2\xi_b.
\end{aligned}$$
\label{cor 1.4}\end{cor}
\subsection{Comments on the results}
\noindent 1. The maximal lifespan. The bound on the maximal lifespan given in Theorem \ref{thm 1.3} is sharp in three dimensions, as shown by finite propagation speed and the counterexample of Sideris \cite{Sideris1985475} on singularities formation of boundaryless compressible fluids, which shows that a general small data may lead to blow up after a time of size $\exp\frac{\kappa}{\epsilon}$ . Meanwhile, as shown in (\ref{phiwave}), the equations (\ref{cnte}-\ref{initial}) can be written as a quasilinear wave equation, and the bound on the maximal lifespan in Theorem \ref{thm 1.3} also agrees with the results of Keel, Smith, and Sogge's studies on nonlinear wave equations in exterior domains \cite{KSS2001, MR2015331, MR2217314}. Counterexamples of Sideris \cite{doi:10.1080/03605308308820304} and John \cite{John1985} on nonlinear wave equations also indicate this optimal lifespan.  \vspace{0.3cm}\\
2. The KSS type estimate (\ref{4.1}). Our techniques in Section \ref{sec4} are in some way an extension to the KSS (Keel-Smith-Sogge) estimate developed in \cite{MR2015331} for wave equations in exterior domains with Dirichlet conditions and can be applied to general quasilinear wave equations in exterior domains regardless of boundary conditions. In fact, our calculation in Section \ref{sec4} does not rely on any boundary conditions over the velocity potential $\varphi$, which satisfies the quasilinear wave equation (\ref{phiwave}) in the exterior domain. The main feature of this extension can be illustrated by (\ref{4.22}). We have the boundary values of the higher order space-time derivatives on the left-hand side of (\ref{4.22}) and the boundary value of a lower order time derivative on the right. Hence (\ref{4.22}) actually provides an approach to controlling higher regularities in terms of the lower ones on the boundary, which is unusual in any analysis. Moreover, by taking a delicate linear combination, all the boundary values of derivatives of $\varphi$ can be controlled by barely the boundary value of $\varphi$ on the $L^2_t$ level, cf. (\ref{4.1}). We also note that the spherical symmetry assumption is not necessary for this extension. In fact, the only assumption regarding the boundary we need is a uniform distance from the origin, which we used in (\ref{4.6}) and (\ref{4.7}). \vspace{0.3cm}\\
3. The method of characteristics. The estimate of $w_B|_{\xi=1}$ in Proposition \ref{prop 6.0} is obtained by solving the hyperbolic system (\ref{5.1})(\ref{5.2}) using characteristics. We note that the treatment in Section \ref{sec6} can be extended to hyperbolic systems with variant coefficients on $(\xi,t)\in[1,+\infty)\times[0,+\infty)$ of the form 
$$\begin{aligned}
(\partial_t-c(x,t)\partial_\xi)w_B+\epsilon f_1(\xi,t)w_F+\epsilon g_1(\xi,t)w_B=0,
\end{aligned}$$
$$\begin{aligned}
(\partial_t+c(x,t)\partial_\xi)w_F+\epsilon f_2(\xi,t)w_B+\epsilon g_2(\xi,t)w_F=0,
\end{aligned}$$ 
where $|c-c_0|\lesssim\epsilon\xi^{-1}$, $|f_i(\xi,t)|\lesssim\xi^{-1}$, $|g_i(\xi,t)|\leq g(\xi)$ for a given $g\in L^1(1,+\infty)$, and $\epsilon$ is sufficiently small. This will result in a Volterra inequality that holds for $t\lesssim\exp\left(\frac{\kappa}{\epsilon}\right)$ and $\kappa\lesssim1$:
$$\begin{aligned}
v_B(\xi_0(t),0)\leq&|w_B(\xi_0(t),0)|+C\epsilon\mathcal{V}(\xi_0(t))+C\epsilon\int_0^{t}\left(1+\frac{\underline{c}}{2}(t-s)\right)^{-1}|w_F(1,t)|dt\\&+C\epsilon\kappa\int_1^{\xi_0(t)}\left(2\mu+\xi_0(t)-\eta\right)^{-1}v_B(\eta,0)d\eta,
\end{aligned}$$
$$\mathcal{V}(\xi)=|w_B(\xi,0)|+\xi^{-1}\int_1^\xi |w_F(\eta,0)|d\eta.$$
Here $\xi_0(t)$ is the $\xi$ coordinate of the intersection of $t=0$ with the backward characteristic starting from $(1,t)$, and $v_B(\xi_0(t))$ is the maximum of $|w_B|$ along this characteristic. This Volterra inequality shall give a bound on $|w_B|$, and thus can play a role in decoupling the above hyperbolic system. 
\vspace{0.3cm}\\
4. The decay rate in second-order terms. The difference between $F(R)$ and its linear part $\mathcal{R}(t)$ is second-order in terms of the size of the initial value and decays at a rate of $(1+\underline{c}t)^{-1}\log(1+\overline{c}t)$ or $(1+\underline{c}t)^{-1}$ depending on whether the backward pressure wave $w_B$ exists at the initial time. This second-order error comes from the reflection of pressure waves due to the variation of sound speed, or by formula from the terms $$\left(\frac{\partial_tc+cr^2\partial_xc}{2c}\right)\frac{w_B}{c^\frac{1}{2}}, \quad \left(\frac{\partial_tc-cr^2\partial_xc}{2c}\right)\frac{w_F}{c^\frac{1}{2}}$$ 
in (\ref{5.1})(\ref{5.2}). As proved in Section \ref{sec6}, these two terms will result in the error estimate (\ref{6.41}). We note that the decay rate $(1+\underline{c}t)^{-1}$ in (\ref{6.41}) of second-order terms is optimal, which can be deduced from the Volterra inequality (\ref{6.24}). In fact, if we regard $w_B|_{t=0}\neq0$ as first order, then $\epsilon\mathcal{V}(\xi_0(t_*))$ and $\epsilon\Psi(t_*)$ should be viewed as second order perturbations, so we may extract the main part of (\ref{6.24}) as
$$v_B(\xi_0(t),0)\leq\left|\frac{w_B}{c^\frac{1}{2}}(\xi_0(t),0)\right|+C\epsilon\int_1^{\xi_0(t)}\left(2\mu+\xi_0(t)-\eta\right)^{-1}v_B(\eta,0)d\eta.$$
Even if we assume $\left(c^{-\frac{1}{2}}w_B\right)(\xi,0)=\mathds{1}_{[1,\xi_b]}(\xi)$ for a constant $\xi_b\simeq 1$ so that it is compactly supported, solving the above Volterra equation in the standard way gives that
$$\begin{aligned}
v_B(\xi,0)=&\mathds{1}_{[1,\xi_b]}(\xi)+\sum_{k=1}^\infty(C\epsilon)^k\int_1^\xi\cdots\int_1^\xi\prod_{i=1}^k\mathds{1}_{\left\{\xi_{i-1}>\xi_i\right\}}(2\mu+\xi_{i-1}-\xi_i)^{-1}\mathds{1}_{[1,\xi_b]}(\xi_n)d\xi_n\cdots d\xi_1\\
\geq&\mathds{1}_{[1,\xi_b]}(\xi)+C\epsilon\log\left(1+\frac{\xi_b-1}{2\mu+\xi-\xi_b}\right)
\end{aligned}$$
with $\xi_0:=\xi$. Hence for large $\xi$ we see that $v_B(\xi,0)\geq C\epsilon\frac{\xi_b-1}{2\mu+\xi-\xi_b}$, and one cannot expect a faster decay rate than $(1+\underline{c}t)^{-1}$ in the second-order part of $v_B(\xi_0(t),0)$. \vspace{0.3cm}\\
5. The condition (\ref{1.18}). Generally, the norms of the form in (\ref{1.18}) will grow quadratically in time even for linear wave equations, so it may seem unnatural to impose (\ref{1.18}) as a condition on the initial value. However, we will only use this condition at $t=0$.  In fact, (\ref{1.18}) is only used in Lemma \ref{lem 7.1} to deduce the controls of $w_B|_{t=0}$ and $w_F|_{t=0}$ at the initial time. While the norms in (\ref{1.18}) may increase quadratic in time, the controls in Lemma \ref{lem 7.1} can propagate uniformly to $t>0$, although we do not need this throughout our proofs. In other words, condition (\ref{1.18}) is not a necessary but convenient condition to deduce the bounds in Lemma \ref{lem 7.1}. \vspace{0.3cm}\\
6. Expected results in higher dimensions or with null conditions. On the one hand, the three-dimensional isentropic compressible Euler equation is a quasilinear wave equation without any null condition. On the other hand, it is known that the solutions to quasilinear wave equations in exterior domains are global for small initial values in higher dimensions ($n\geq 4$)  \cite{MR2244378, MR2679723}, or with additional null conditions in three dimensions \cite{MR2198183, MR2299569}. In particular, if we consider a Chaplygin fluid (in which case the pressure $p$ is linear in $\rho^{-1}$, or linear in $q$ in our notation) outside the bubble instead of an isentropic fluid, our techniques can be easily modified to yield a global solution. In that case, the specific sound speed $c$ is a constant, and our analysis would be greatly simplified. \vspace{0.3cm}
\subsection{Strategy of the proof}
 We provide in this section an overview of our proof.\\
\noindent 1.Local existence. We construct an approximate solution sequence by iteration. In each step of this iteration, we define the approximate solution $(u_k,\,q_k,\,f(R_k))$ by solving the linear system (\ref{icnte}-\ref{iinitial}) in which the parameters $r_{k-1}$, $c_{k-1}$, $f^\prime(R_{k-1})$ are determined by $(u_{k-1},\,q_{k-1},\,f(R_{k-1}))$ from the previous step. It is worth noting that we choose $f(R_k)$ as the unknown in the linear equations (\ref{idbce}) instead of naively linearizing $f(R)$ at $R=1$ and setting $q_k|_{x=0}=f^\prime(R_{k-1})(R_k-1)$. The latter choice will result in unclosure in the energy identity due to the appearance of the terms involving $$\partial_t^{j+1}R_{k}\cdot\partial_t^i R_{k-1}.$$ This can also be observed from Lemma \ref{Lem 2.2}, where we do not allow an inhomogeneity in $q|_{x=0}=f$. We prove by induction local-in-time uniform $L^2$ bounds on all the space-time derivatives of order no larger than two of each term in the approximate solution sequence, cf. Proposition \ref{prop2.3}. Due to the presence of the boundary, the higher-order energy functionals $e_k^j$ are defined by taking time derivatives, which preserve the boundary. The method is to prove the local-in-time bounds on $e_k^j$, then convert these bounds on time derivatives to all space-time derivatives using equations (\ref{icnte})(\ref{immte}), cf. Lemma \ref{Lem 2.6}. Using the uniform bounds obtained in Proposition \ref{prop2.3}, the Aubin-Lions compactness theorem is employed to yield a convergent subsequence and a limit, which is verified to be the solution of (\ref{cnte}-\ref{initial}).\vspace{0.3cm}\\
2. Bootstrap scheme. The almost global existence is proved under a bootstrap setting. We consider an interval $[0,T]$ on which the energy functionals satisfy the bound (\ref{3.1}). It also follows that the parameters $r$, $c$ and the solution itself satisfy the pointwise control (\ref{3.6}-\ref{3.10}). We aim at strictly improving (\ref{3.1}) on the interval $[0,T]$, so the solution can be continued beyond $[0,T]$ until $T$ reaches the bound (\ref{3.11}). \vspace{0.3cm}\\
3. Control of nonlinearities via KSS estimate. To improve the bound (\ref{3.1}), the first step is to control the nonlinearities in the energy identity (\ref{kee}) (removing the iteration subscripts $k$) or (\ref{7.1}). In view of the pointwise controls (\ref{3.6}-\ref{3.10}), the integrands in these nonlinearities are of the size $\epsilon\xi^{-1}\left((\partial_t^ju)^2+(\partial_t^jq)^2\right)$. Since the equations (\ref{cnte}-\ref{initial}) can be transformed to the quasilinear wave equation (\ref{phiwave}), the estimate of these nonlinearities can hopefully be achieved by the said KSS (Keel-Smith-Sogge) estimate, which has been applied to control such nonlinearities for nonlinear wave equations in exterior domains with Dirichlet condition in \cite{KSS2001, MR2015331}. The difference is that we need additional control over $\int_0^T\varphi^2|_{x=0}dt$ due to the absence of a vanishing boundary condition.\vspace{0.3cm}\\
4. Hyperbolic system (\ref{5.1})(\ref{5.2}) coupling boundary ODE (\ref{5.3}). To control $\int_0^T\varphi^2|_{x=0}dt$, we begin by looking at how $\varphi|_{x=0}$ is determined. By taking time derivative in (\ref{dbcb0}) and substituting $\rho$ using (\ref{rho-phi}), we arrive at the ODE (\ref{5.3}) of $\psi|_{x=0}$ with $\psi=r\varphi$. On one hand, the backward pressure wave $w_B|_{x=0}$ acts as an inhomogeneity in the ODE (\ref{5.3}), so we have to carry out the characteristic method to estimate $w_B|_{x=0}$, which is initialized in (\ref{5.1})(\ref{5.2}), or equivalently (\ref{6.1})(\ref{6.2}). On the other hand, to solve the hyperbolic system (\ref{5.1})(\ref{5.2}) one needs the data of $w_F$ on the boundary $\xi=1$, which is in fact given by the ODE (\ref{5.3}) and $w_F=2\partial_t\psi-w_B$.\vspace{0.3cm}\\ 
5. Decoupling (\ref{5.1})(\ref{5.2})(\ref{5.3}). We first integrate (\ref{5.1})(\ref{5.2}) on each characteristic to obtain inequalities (\ref{6.5})(\ref{6.6})(\ref{6.7}) of $v_B$ and $v_F$, which are the maximum of $|c^{-\frac{1}{2}}w_B|$ along the backward characteristic and the maximum of $|c^{-\frac{1}{2}}w_F|$ along the forward characteristic. To combine these inequalities, we study the regularity of the coordinate transforms between the backward and forward characteristics in different regions through a geometric approach and simplify these inequalities to (\ref{6.16})(\ref{6.17})(\ref{6.18}). Then from (\ref{6.16})(\ref{6.17})(\ref{6.18}), we derive a Volterra inequality (\ref{6.24}) on $v_B$ which then deduces an estimate (\ref{6.30}) of $w_B|_{\xi=1}$. With the estimate (\ref{6.30}) we are able to decouple (\ref{5.3}) from (\ref{5.1})(\ref{5.2}) and obtain the estimate (\ref{6.41}) for $\psi|_{x=0}$ and $\partial_t\psi|_{x=0}$.  \vspace{0.3cm}\\ 
6. Closing the bootstrap. Combining the KSS estimate (\ref{4.1}) and (\ref{6.41}), the bound (\ref{3.1}) is then strictly improved on $[0,T]$ for all $T$ satisfying (\ref{3.11}), which finishes the bootstrap argument. A by-product of (\ref{6.41}) is an error estimate of the bubble radius $R$ from the solution of the linearized equation as shown in Theorem \ref{thm 1.3}. Moreover, a pointwise decay-in-time estimate of $(u,q)$ as shown in Corollary \ref{cor 1.4} is recovered from the estimates of $w_F$ and $w_B$ obtained in Section \ref{sec6}. 
\subsection{Organization of the sections}
In Section \ref{sec2}, we prove the local existence and uniqueness under the iteration scheme. In Section \ref{sec3}, we state the bootstrap assumptions and deduce the resulting pointwise bounds on $c$, $r$ and the solution itself. In Section \ref{sec4}, we derive the generalized KSS estimate without boundary conditions. In Section \ref{sec5}, we derive the boundary ODE determining $\psi|_{x=0}$ and the hyperbolic system of the forward and backward pressure waves $w_F$, $w_B$. In Section \ref{sec6}, we first study the regularity of the coordinate transforms between the forward and backward characteristics. Using these coordinate transforms, we prove the Volterra inequality (\ref{6.24}), and use (\ref{6.24}) to obtain the control on $\psi|_{x=0}$. In Section \ref{sec7}, we close the bootstrap argument by combining the KSS estimate and the control on $\psi|_{x=0}$, and prove the  pointwise estimate of $(u,\,q,\,R)$. \vspace{3mm}\\
\textbf{Acknowledgements.} To do this study has been caused by helpful discussions with Associate Professor Lifeng Zhao (University of Science and Technology of China). This work was supported by National Natural Science Foundation of China Grant No. 12271497.
\section{Energy estimate and Local existence}\label{sec2}
In this section, we shall prove Theorem \ref{thm 1.2} by means of an iterative scheme to solve (\ref{cnte}-\ref{initial}) and to show a uniform bound for each iterative, then proceed by compactness argument.
Before the analysis of iteration, we look at an elliptic estimate and a model question, which will provide the energy identity. 
\begin{lem}[Weighted elliptic estimate]
Let $v\in H^1$ and $r\in AC[0,+\infty)$ be such that $r(0)>0$ and $+\infty>\overline{\theta}\geq r^2\partial_xr\geq\underline{\theta}>0$ for some constants $\underline{\theta}$ and $\overline{\theta}$. Then we have the identity
\begin{equation}\begin{aligned}
\int_0^\infty(r^2\partial_xr)^{-1}(\partial_x(r^2v))^2dx+2rv^2|_{x=0}
=\int_0^\infty(r^2\partial_xr)^{-1}(r^2\partial_xv)^2dx+2\int_0^\infty(r^2\partial_xr)r^{-2}v^2dx,
\end{aligned}\label{div-grd}\end{equation}
and the pointwise control
\begin{equation}\begin{aligned}
&|v(x)|^2\leq r^{-2}(x)(\|v\|_{L^2}^2+\|r^2\partial_x v\|_{L^2}^2)\\
&|v(x)|^2\leq r^{-2}(x)((8r(0)^{-2}\overline{\theta}^2+1)\|v\|_{L^2}^2+2\|\partial_x(r^2v)\|_{L^2}^2)
\label{ptwc}
\end{aligned}\end{equation}
\label{Lem 2.1}\end{lem}
\begin{proof}
The assumption of $r$ gives $r(0)^3+3\overline{\theta}x\geq r(x)^3\geq r(0)^3+3\underline{\theta}x$, so $r^2\partial_xv\in L^2$.
Note that $(\partial_x(r^2v))^2=(r^2\partial_xv)^2+2(r^2\partial_xr)^2r^{-2}v^2+2(r^2\partial_xr)\partial_x(rv^2)$, therefore, (\ref{div-grd}) follows from integration by parts. For any $y>x>0$, there is $v^2(x)-v^2(y)=\int_x^y2v(z)\partial_xv(z)dz$.
Integrating in $y$ on $I:=[x,\tilde{x}]$ yields $|I|v^2(x)=\int_Iv^2(y)dy+2\int_I\int_x^yv(z)\partial_xv(z)dzdy$. Then by Cauchy-Schwartz
$$
v^2(x)\leq(\tilde{x}-x)^{-1}\int_Iv^2(y)dy+r^{-2}(x)\int_Iv^2(z)dz+r^{-2}(x)\int_I(r^2\partial_xv(z))^2dz
$$
The first inequality of (\ref{ptwc}) follows by passing $\tilde{x}$ to infinity. Noticing that $|r^2\partial_xv|\leq|\partial_x(r^2v)|+2|r\partial_xr||v|$ and $|r\partial_xr|\leq r(0)^{-1}\overline{\theta}$, the second line in (\ref{ptwc}) follows from using Cauchy-Schwartz.
\end{proof}
\begin{lem}[Model question]
Let $(u,\,q,\,f)$ satisfy the following equations
$$\begin{cases}
{} \partial_t u-c^2r^2\partial_xq=F_1,&x>0,\;t>0,\\ 
\partial_t q-\partial_x(r^2u)=F_2,&x>0,\;t>0,\\
u|_{x=0}=g\frac{d}{dt}f+F_3,&t>0,\\
q|_{x=0}=f,&t>0.
\end{cases}$$
Then the energy identity reads
\begin{equation}\begin{aligned}
&\frac{1}{2}\frac{d}{dt}\left[\int_0^\infty\left(u^2+c^2q^2\right)dx+(c^2r^2)|_{x=0}gf^2\right]\\
=&\int_0^\infty\left(\frac{1}{2}(\partial_tc^2)q^2-(r^2\partial_xc^2)uq+F_1 u+c^2F_2q\right)dx
+\frac{1}{2}\partial_t[(c^2r^2)|_{x=0}g]f^2-(c^2r^2)|_{x=0}F_3f.
\end{aligned}\label{mdq}\end{equation}
\label{Lem 2.2}\end{lem}
\noindent The proof is straightforward by multiplying $u$ on the first equation and integrating by parts.\\
\indent We now begin with the construction of the iteration sequence. Define 
$$c_0^2=\frac{Ca\gamma}{2},\quad r_0=(1+3x)^\frac{1}{3},$$
$$c_1^2=\frac{Ca\gamma}{2}(1+q_{in})^{-1-\gamma},\quad r_1(x)=r_{in}(x)=\left(R_{in}^3+3x+3\int_0^xq_{in}(y)dy\right)^{\frac{1}{3}},\quad R_1=R_{in}.$$ 
Since $\underline{q}\leq q_{in}\leq \overline{q}$, it follows
\begin{equation}\min\left\{R_{in}^3, 1+\underline{q}\right\}\leq\frac{r_{in}^3}{r_0^3}\leq\max\left\{R_{in}^3,1+\overline{q}\right\},\quad\partial_xr_{in}=r_{in}^{-2}(1+q_{in})\simeq r_0^{-2}.\label{rinr0}\end{equation}
Then define $(u_k,\,q_k,\,R_k)$ , $k\geq 2$ inductively by solving the linear equations of $(u_k,\, q_k,\, f(R_k))$
\begin{numcases}
{} \partial_t q_k=\partial_x(r_{k-1}^2 u_k), &$x>0,\;t>0$,\label{icnte}\\
\partial_t u_k=c_{k-1}^2r_{k-1}^2\partial_x q_k, &$x>0,\;t>0$, \label{immte}\\
(f'(R_{k-1}))^{-1}\frac{d}{dt}f(R_k)=u_k|_{x=0}, &$t>0$,\label{ikbce}\\
q_k|_{x=0}=f(R_k), &$t>0$,\label{idbce}\\
(u_k,\,q_k,\, f(R_k))|_{t=0}=(u_{in},\,q_{in},\,f(R_{in})).\label{iinitial}
\end{numcases}
Here the parameters $r_{k-1}$ and $c_{k-1}$, $k\geq 3$ are determined by the previous step:
$$\partial_tr_{k-1}=u_{k-1},\quad r_{k-1}(x,0)=r_{in}(x),\quad c_{k-1}^2=\frac{Ca\gamma}{2}(1+q_{k-1})^{-1-\gamma}.$$
For convenience, denote $q_1:=q_{in}$, $u_1:=0$, so the above identities also hold for $k=2$.
We note that the compatible conditions up to first order of the initial-boundary value problem (\ref{icnte}-\ref{iinitial}) are
$$q_{in}(0)=f(R_{in}),\quad f^\prime(R_{k-1}(0))u_{in}(0)=\partial_x\left((r_{k-1}
|_{t=0})^2u_{in}\right)(0),$$
which are fulfilled due to (\ref{CC0})(\ref{CC1}) and $r_{k-1}|_{t=0}=r_{in}$, $R_{k-1}|_{t=0}=R_{in}$. We introduce the time-dependent operators 
$$\begin{aligned}
&L_{k-1}^{2j}:=\left[\partial_x(r_{k-1}^4c_{k-1}^2\partial_x)\right]^j,\quad L_{k-1}^{2j+1}:=c_{k-1}^2r_{k-1}^2\partial_xL_{k-1}^{2j},\\
&\tilde{L}_{k-1}^{2j}:=\left[c_{k-1}^2r_{k-1}^2\partial_x^2r_{k-1}^2\right]^j,\quad\tilde{L}_{k-1}^{2j+1}:=\partial_x(r_{k-1}^2\tilde{L}_{k-1}^{2j}).
\end{aligned}$$ 
Equations (\ref{icnte})(\ref{immte}) can then be written as $\partial_tq_k=\tilde{L}_{k-1}^1u_k$ and $\partial_tu_k=L_{k-1}^1q_k$. Moreover, for each $j$, it holds 
$$\tilde{L}_{k-1}^{j+1}=L_{k-1}^{j}\tilde{L}_{k-1}^1,\quad L_{k-1}^{j+1}=\tilde{L}_{k-1}^jL_{k-1}^1.
$$
The key point in order to construct the local solution is to prove a uniform bound for each $(u_k,\,q_k,\,f(R_k))$ on a common interval $[0,T]$. To begin with, we define for $j=0,1,2$ and $k\geq 2$ that
$$\begin{aligned}
e_k^j(t):=\frac{1}{2}\int_0^\infty\left((\partial_t^ju_k(t))^2+c_{k-1}^2(t)(\partial_t^jq_k(t))^2\right)dx
+\frac{1}{2}c_{k-1}(t)|_{x=0}^2f'(R_{k-1}(t))^{-1}\left(\partial_t^jf(R_k)(t)\right)^2,
\end{aligned}$$
$$E_k(t):=\sum_{j=0}^2|\partial_t^jf(R_k)(t)|^2+\sum_{j_1+j_2\leq2}\left(\|L_{k-1}^{j_2}\partial_t^{j_1}q_k(t)\|_{L^2}^2+\|\tilde{L}_{k-1}^{j_2}\partial_t^{j_1}u_k(t)\|_{L^2}^2\right),$$
$$\epsilon^2:=\sum_{j=0}^2\|L_0^jq_{in}\|_{L^2}^2+\sum_{j=0}^2\|\tilde{L}_0^ju_{in}\|_{L^2}^2+|f(R_{in})|^2.
$$
In order to obtain the uniform bound in $H^2$, we will apply an induction argument. One of the steps in this induction argument is to prove the lemma below.
\begin{lem}[Uniform bound for mixed derivatives]
There exists $T>0$ and a constant $A>0$ such that for each $k\geq2$,
\begin{equation}\begin{aligned}
&\sum_{j=0}^2|\partial_t^jf(R_k)|^2+\sum_{j_1+j_2\leq2}\left(\|L_{k-1}^{j_2}\partial_t^{j_1}q_k\|_{L^2}^2+\|\tilde{L}_{k-1}^{j_2}\partial_t^{j_1}u_k\|_{L^2}^2\right)\\
\leq &A\left(\sum_{j=0}^2\|L_0^jq_{in}\|_{L^2}^2+\sum_{j=0}^2\|\tilde{L}_0^ju_{in}\|_{L^2}^2+|f(R_{in})|^2\right),
\end{aligned}\label{kme}\end{equation}
holds on $t\in[0,T]$.
\label{prop 2.4}\end{lem}
Now we state the uniform bound on the approximate solutions. 
\begin{prop}[Uniform bounds in $H^2$]
There exists $T>0$ such that for all $k\geq2$ and $t\in[0,T]$, $j=0,1,2$ that 
$$\|\partial_t^ju_k(t)\|_{H^{2-j}}+\|\partial_t^jq_k(t)\|_{H^{2-j}}\lesssim A^\frac{1}{2}(1+A^\frac{1}{2}\epsilon)\epsilon,$$
and $\tilde{u}_k(\xi,t):=u_k\left(\frac{\xi^3-1}{3},t\right)$, $\tilde{q}_k(\xi,t):=q_k\left(\frac{\xi^3-1}{3},t\right)$ satisfy 
$$\|\partial_t^j\tilde{u}_k(t)\|_{H^{2-j}_{\xi}(1,+\infty)}+\|\partial_t^j\tilde{q}_k(t)\|_{H^{2-j}_{\xi}(1,+\infty)}\lesssim A^\frac{1}{2}(1+A^\frac{1}{2}\epsilon)\epsilon.$$
\label{prop2.3}\end{prop}
The proof will be completed by induction in $k$ with the assistance of the following lemmas.
\begin{lem}[Control of parameters]
There exists $T>0$ and constant $C>0$ dependent on $\underline{q}$, $\overline{q}$ and $R_{in}$ such that for each $k\geq2$ and $t\in[0,T]$
\begin{equation}
\frac{3}{2}\partial_xr_{in}\geq\partial_xr_{k-1}\geq\frac{1}{2}\partial_xr_{in},\quad \frac{3}{2}r_{in}(x)\geq r_{k-1}(x,t)\geq\frac{1}{2}r_{in}(x),
\label{kqb}\end{equation}
\begin{equation}
\frac{Ca\gamma}{2}\left(\frac{3}{2}+\frac{3}{2}\overline{q}\right)^{-\gamma-1}\leq c_{k-1}^2(x,t)\leq\frac{Ca\gamma}{2}\left(\frac{\underline{q}+1}{2}\right)^{-\gamma-1},
\label{kcb}\end{equation}
\begin{equation}
\frac{1}{2}(\underline{q}-1)\leq f(R_{k-1})\leq\frac{3\overline{q}+1}{2},
\label{kRb}\end{equation}
\begin{equation}
\left\|\frac{\partial_tr_{k-1}}{r_{k-1}}\right\|_{L^\infty}+\left\|r_{k-1}^2\partial_x\left(\frac{\partial_tr_{k-1}}{r_{k-1}}\right)\right\|_{L^\infty}+\left\|\frac{\partial_tc_{k-1}}{c_{k-1}}\right\|_{L^\infty}\leq CA^\frac{1}{2}\epsilon,
\label{kfb}\end{equation}
\begin{equation}
\left\|r_0^2\partial_x\left(\frac{r_{k-1}^2}{r_0^2}\right)\right\|_{L^\infty}\leq C,\quad \left\|r_0^2\partial_x(c_{k-1}^2)\right\|_{L^\infty}\leq C\|L_0^1q_{k-1}\|_{L^\infty},
\label{krb}\end{equation}
\begin{equation}
\left\|r_0^2r_{k-1}^2\partial_x^2\left(\frac{r_{k-1}^2}{r_0^2}\right)\right\|_{L^2}\leq C\left(1+\|L_0^1q_{in}\|_{L^2}
\right).\label{embed}\end{equation}
\label{Lem 2.5}\end{lem}
\begin{lem}
There exist constants $B>0$ and $\tilde{B}>0$ such that for each $k\geq 2$ 
\begin{equation}
E_k(t)\leq B\sum_{j=0}^2e_k^j(t),\quad t\in[0,T],\label{Etoe}
\end{equation}
\begin{equation}
\sum_{j=0}^2e_k^j(0)\leq\tilde{B}\epsilon^2.
\label{etoep}
\end{equation}
\label{Lem 2.6}\end{lem}
By Lemma \ref{Lem 2.6}, we see that if $\sum_{j=0}^2e_k^j(t)\leq2\sum_{j=0}^2e_k^j(0)$  holds on $t\in [0,T]$ for some $T$ independent on $k$, then we obtain $E_k(t)\leq 2B\tilde{B}\epsilon^2$, i.e. (\ref{kme}) holds for $A\geq2B\tilde{B}$. \\
\indent Now we suppose that Lemma \ref{prop 2.4}, Propostion \ref{prop2.3}, Lemma \ref{Lem 2.5} and Lemma \ref{Lem 2.6} hold for all $k'\leq k-1$ and Lemma \ref{Lem 2.5} also holds for $k$. Note that Lemma \ref{Lem 2.5} holds for $k=2$ since we have selected $r_1=r_{in}$, $c_{1}^2=c_0^2(1+q_{in})^{-1-\gamma}$, $R_1=R_{in}$ while $q_{in}$ satisfies the regularity condition as stated in Theorem \ref{thm 1.2}. The first step is to show Lemma \ref{Lem 2.6} holds for $k$. \vspace{0.3cm}\\
\textit{Proof of Lemma \ref{Lem 2.6}.}
Equations (\ref{icnte})(\ref{immte}) are equivalent to $\partial_tu_k=L_{k-1}^1q_k$ and $\partial_tq_k=\tilde{L}_{k-1}^1u_k$, and also yield that 
\begin{equation}
L_{k-1}^2q_k=\tilde{L}_{k-1}^{1}\partial_tu_k,\quad \tilde{L}_{k-1}^2u_k=L_{k-1}^1\partial_tq_k.
\label{L2-L1t}\end{equation}
According to the construction of $E_k$, in order to show (\ref{Etoe}) it suffices to control $\|L_{k-1}^1\partial_tq_k\|_{L^2}$ and $\|\tilde{L}_{k-1}^1\partial_t u_k\|_{L^2}$. By differentiating (\ref{icnte})(\ref{immte}) we obtain for $t\in[0,T]$
\begin{equation}\begin{aligned}
\partial_t^2q_{k}-\tilde{L}_{k_1}^1\partial_tu_{k}=&2\partial_x(r_{k-1}\partial_tr_{k-1}u_{k})
=2\frac{\partial_tr_{k-1}}{r_{k-1}}\tilde{L}^1_{k-1}u_k+2r_{k-1}^2\partial_x\left(\frac{\partial_tr_{k-1}}{r_{k-1}}\right)u_k,
\end{aligned}\label{exchq}\end{equation}
\begin{equation}\begin{aligned}
\partial_t^2u_k-L_{k-1}^1\partial_tq_k=&\frac{\partial_t(c_{k-1}^2r_{k-1}^2)}{c_{k-1}^2r_{k-1}^2}L^1_{k-1}q_k.
\label{exchu}\end{aligned}\end{equation}
Therefore, by Lemma \ref{Lem 2.5}, for $t\in[0,T]$ there are
\begin{equation}\begin{aligned}
&\|\partial_t^2q_k-\tilde{L}_{k-1}^1\partial_tu_k\|_{L^2}\\
\leq &C\left(\left\|\frac{\partial_tr_{k-1}}{r_{k-1}}\right\|_{L^\infty}+\left\|r_{k-1}^2\partial_x\left(\frac{\partial_tr_{k-1}}{r_{k-1}}\right)\right\|_{L^\infty}\right)\left(\|\partial_tq_{k}\|_{L^2}+\|u_k\|_{L^2}\right)\\
\leq &CA^{\frac{1}{2}}\epsilon\left(\|\partial_tq_k\|_{L^2}+\|u_k\|_{L^2}\right)\\
=& CA^{\frac{1}{2}}\epsilon\left(\|\tilde{L}_{k-1}^1u_k\|_{L^2}+\|u_k\|_{L^2}\right),
\label{erexchq}\end{aligned}\end{equation}
\begin{equation}\begin{aligned}
\|\partial_t^2u_k-L_{k-1}^1\partial_tq_k\|_{L^2}
\leq &C\left\|\frac{\partial_t(c_{k-1}^2r_{k-1}^2)}{c_{k-1}^2r_{k-1}^2}\right\|_{L^\infty}\|\partial_tu_k\|_{L^2}
\leq CA^\frac{1}{2}\epsilon\|\partial_tu_k\|_{L^2},
\end{aligned}\label{erexchu}\end{equation}
which complete the proof of (\ref{Etoe}) for $B=C(1+A^\frac{1}{2}\epsilon)$.\\
To show (\ref{etoep}), restricting $t$ to $0$ in (\ref{exchq})(\ref{exchu}) and using (\ref{L2-L1t}) gives 
\begin{equation}\begin{aligned}
\partial_t^2q_k|_{t=0}-L_{k-1}^2q_{in}=2\left.\frac{\partial_tr_{k-1}}{r_{k-1}}\right|_{t=0}\tilde{L}^1_{k-1}u_{in}+2\left[\left.r_{k-1}^2\partial_x\left(\frac{\partial_tr_{k-1}}{r_{k-1}}\right)\right]\right|_{t=0}u_{in},
\end{aligned}\label{2.24}\end{equation}
\begin{equation}\begin{aligned}
\partial_t^2u_k|_{t=0}-\tilde{L}_{k-1}^2u_{in}=\left.\left[\frac{\partial_t(c_{k-1}^2r_{k-1}^2)}{c_{k-1}^2r_{k-1}^2}\right]\right|_{t=0}L_{k-1}^1q_{in}.
\end{aligned}\label{2.25}\end{equation}
If $k=2$, then all the terms on the right-hand side of (\ref{2.24})(\ref{2.25}) vanish. If $k\geq3$, we have
$$\begin{aligned}
\left\|\left.\frac{\partial_tr_{k-1}}{r_{k-1}}\right|_{t=0}\right\|_{L^\infty}\leq C\|u_{in}\|_{L^\infty},\quad\left\|\left.\frac{\partial_tc_{k-1}}{c_{k-1}}\right|_{t=0}\right\|_{L^\infty}\leq C\|\partial_x(r_{in}^2u_{in})\|_{L^\infty},
\end{aligned}$$
$$\begin{aligned}
\left\|\left.r_{k-1}^2\partial_x\left(\frac{\partial_tr_{k-1}}{r_{k-1}}\right)\right|_{t=0}\right\|_{L^2}=\|r_{in}^{-1}\tilde{L}_{k-1}^1u_{in}-3u_{in}\partial_xr_{in}\|_{L^2}\leq C\left(\|\tilde{L}_{k-1}^1u_{in}\|_{L^2}+\|u_{in}\|_{L^2}\right),
\end{aligned}$$
where we used the fact $r_{in}(x)\geq r_{in}(0)>0.$. Moreover, we can re-express $L_{k-1}^j$ and $\tilde{L}_{k-1}^j$ by $L_0^j$ and $\tilde{L}_0^j$ as follows. 
\begin{equation}\begin{aligned}
\tilde{L}_{k-1}^1u_k=\frac{r_{k-1}^2}{r_0^2}\tilde{L}_0^1u_k+r_0^2\partial_x\left(\frac{r_{k-1}^2}{r_0^2}\right)u_k,
\end{aligned}\label{LktoL0u1}\end{equation}
\begin{equation}\begin{aligned}
L_{k-1}^1q_k=\frac{c_{k-1}^2}{c_0^2}\frac{r_{k-1}^2}{r_0^2}L_0^1q_k,
\end{aligned}\label{LktoL0q1}\end{equation}
\begin{equation}\begin{aligned}
\tilde{L}_{k-1}^2u_k=\frac{c_{k-1}^2r_{k-1}^4}{c_0^2r_{0}^4}\tilde{L}_0^2u_k+2c_{k-1}^2r_{k-1}^2\partial_x\left(\frac{r_{k-1}^2}{r_0^2}\right)\tilde{L}_0^1u_{k}+c_{k-1}^2r_{k-1}^2r_0^2\partial_x^2\left(\frac{r_{k-1}^2}{r_0^2}\right)u_k,
\end{aligned}\label{LktoL0u2}\end{equation}
\begin{equation}\begin{aligned}
L_{k-1} ^2q_k=&\frac{c_{k-1}^2}{c_0^2}\frac{r_{k-1}^4}{r_0^4}L_0^2q_k+r_0^2\partial_x\left(\frac{r_{k-1}^4}{r_0^4}\frac{c_{k-1}^2}{c_0^2}\right)L_0^1q_k.
\end{aligned}\label{LktoL0q2}\end{equation}
Hence by (\ref{LktoL0u1}-\ref{LktoL0q2}) and (\ref{kqb}-\ref{embed}) , it holds
\begin{equation}
C^{-1}\|L_0^1q_k\|_{L^2}\leq\|L_{k-1}^1q_k\|_{L^2}\leq C\|L_0^1q_k\|_{L^2},
\label{2.30}\end{equation}
\begin{equation}
\|\tilde{L}_{k-1}^1u_k-r_0^{-2}r_{k-1}^2\tilde{L}_0^1u_k\|_{L^2}\leq C\|u_k\|_{L^2},
\label{2.31}\end{equation}
\begin{equation}
\left\|L_{k-1}^2q_k-\frac{c_{k-1}^2}{c_0^2}\frac{r_{k-1}^4}{r_0^4}L_0^2q_k\right\|_{L^2}\leq C\left(1+\|L_0^1q_{k-1}\|_{L^\infty}\right)\|L_0^1q_k\|_{L^2},
   \label{2.32}\end{equation}
\begin{equation}
\left\|\tilde{L}_{k-1}^2u_k-\frac{c_{k-1}^2}{c_0^2}\frac{r_{k-1}^4}{r_0^4}\tilde{L}_0^2u_k
\right\|_{L^2}\leq C\left(1+\|L_0^1q_{in}\|_{L^2}\right)\left(\|u_k\|_{L^2}+\|\tilde{L}_0^1u_k\|_{L^2}\right).
\label{2.33}\end{equation}
Then using (\ref{2.30}-\ref{2.33}) in (\ref{2.24})(\ref{2.25}) yields
$$\begin{aligned}
&\|\partial_t^2q_k|_{t=0}\|_{L^2}+\|\partial_t^2u_k|_{t=0}\|_{L^2}\\
\leq&\|L_{k-1}^2q_{in}\|_{L^2}+\|\tilde{L}_{k-1}^2u_{in}\|_{L^2}\\
&+C\left(\|u_{in}\|_{L^\infty}+\|\partial_x(r_{in}^2u_{in})\|_{L^\infty}\right)\left(\|\tilde{L}_{k-1}^1u_{in}\|_{L^2}+\|u_{in}\|_{L^2}+\|L_{k-1}^1q_{in}\|_{L^2}\right)
\\
\leq&C(1+\|u_{in}\|_{L^\infty}+\|\partial_x(r_{in}^2u_{in})\|_{L^\infty}+\|L_0^1q_{in}\|_{L^\infty\cap L^2})\sum_{j=0}^2\left(\|L_0^jq_{in}\|_{L^2}+\|\tilde{L}_0^ju_{in}\|_{L^2}\right).
\end{aligned}$$
By Lemma \ref{Lem 2.1}, we have
$$\|u_{in}\|_{L^\infty}+\|\partial_x(r_{in}^2u_{in})\|_{L^\infty}+\|L_0^1q_{in}\|_{L^\infty\cap L^2}\lesssim\sum_{j=0}^2\left(\|\tilde{L}_0^ju_{in}\|_{L^2}+\|L_0^jq_{in}\|_{L^2}\right)\leq\epsilon.$$
It follow that 
$$\|\partial_t^2q_k|_{t=0}\|_{L^2}+\|\partial_t^2u_k|_{t=0}\|_{L^2}\leq C\left(1+\epsilon\right)\sum_{j=0}^2\left(\|\tilde{L}_0^ju_{in}\|_{L^2}+\|L_0^jq_{in}\|_{L^2}\right).$$
Inequalities (\ref{2.30})(\ref{2.31}) also show that the first-order $t$-derivatives can be controlled by the above bound. It remains to bound $\partial_t^jf(R_{k})|_{t=0}$, for which we have
$$\begin{aligned}
&\left|\partial_tf(R_k)|_{t=0}\right|=f'(R_{k-1})|_{t=0}\left|u_k|_{(x=0,t=0)}\right|\\
=&f'(R_{in})\left|u_{in}|_{x=0}\right|\leq C\left(\|\tilde{L}_0^1u_{in}\|_{L^2}+\|u_{in}\|_{L^2}\right),
\end{aligned}$$
$$\begin{aligned}
&\left|\partial_t^2f(R_k)|_{t=0}\right|\\
=&\left|f'(R_{k-1})\partial_tu_k|_{(x=0,t=0)}+\partial_t(f'(R_{k-1}))u_k|_{(x=0,t=0)}\right| \\
=&\left|f'(R_{k-1})L_{k-1}^1q_k|_{(x=0,t=0)}+\left(f^\prime\circ f^{-1}\right)'(f(R_{k-1}))\partial_tf(R_{k-1})u_k|_{(x=0,t=0)}\right|\\
=&\left|f'(R_{in})\frac{c_{in}^2}{c_0^2}\frac{r_{in}^2}{r_0^2}L_0^1q_{in}+\left(f^\prime\circ f^{-1}\right)'(f(R_{in}))f'(R_{in})u_{in}^2|_{x=0}\right|\\
\leq & C\left(\|L_0^2q_{in}\|_{L^2}+\|L_0^1q_{in}\|_{L^2}\right)+C\|u_{in}\|_{L^\infty}\left(\|u_{in}\|_{L^2}+\|\tilde{L}_0^1u_{in}\|\right).
\end{aligned}$$
Collecting the above bounds gives
$$\begin{aligned}
\sum_{j=0}^2e_k^j(0)\leq &C\sum_{j=0}^2\left(\|\partial_t^jq_k|_{t=0}\|_{L^2}^2+\|\partial_t^ju_k|_{t=0}\|_{L^2}^2+|\partial_t^jf(R_{k})|_{t=0}|^2\right)\\
\leq &C(1+\epsilon)\left[\sum_{j=0}^2\left(\|L_0^jq_k\|_{L^2}^2+\|\tilde{L}_0^ju_k\|_{L^2}^2\right)+|f(R_{in})|^2\right].
\end{aligned}$$
Therefore, by choosing $\tilde{B}=C(1+\epsilon)$, we arrive at (\ref{etoep}).$\hfill\qedsymbol$\vspace{0.3cm}\\
The second step is to prove Lemma \ref{prop 2.4} for $k$, which is achieved by propagating $e_k^j$ through an energy method and applying Lemma \ref{Lem 2.6}.\vspace{0.3cm}\\
\textit{Proof of Lemma \ref{prop 2.4}.} Applying $\partial_t^j$ to (\ref{icnte}-\ref{idbce}) yields
$$\begin{cases}
{} \partial_t\partial_t^ju_k-c_{k-1}^2r_{k-1}^2\partial_x\partial_t^jq_{k}=[\partial_t^j,c_{k-1}^2r_{k-1}^2\partial_x]q_k, &x>0,\;t>0,\\
\partial_t\partial_t^jq_k-\partial_x(r_{k-1}^2\partial_t^ju_k)=[\partial_t^j,\partial_xr_{k-1}^2]u_k, &x>0,\;t>0,\\
\partial_t^jq_{k}|_{x=0}=\partial_t^jf(R_k),&t>0,\\
\partial_t^ju_k|_{x=0}=f'(R_{k-1})^{-1}\partial_t^{j+1}f(R_k)+[\partial_t^j,f'(R_{k-1})^{-1}]\partial_tf(R_k),&t>0.
\end{cases}$$
(\ref{mdq}) gives
\begin{equation}\begin{aligned}
&\frac{1}{2}\frac{d}{dt}\left[\int_0^\infty((\partial_t^ju_k)^2+c_{k-1}^2(\partial_t^jq_k)^2)dx+(c_{k-1}^2r_{k-1}^2)|_{x=0}f'(R_{k-1})^{-1}|\partial_t^jf(R_k)|^2\right]\\
=&\int_0^\infty\left(\frac{1}{2}\partial_tc_{k-1}^2(\partial_t^jq_k)^2-r_{k-1}^2\partial_xc_{k-1}^2\partial_t^ju_k\partial_t^jq_k\right)dx\\
&+\frac{1}{2}\partial_t\left[(c_{k-1}^2r_{k-1}^2)|_{x=0}f'(R_{k-1})^{-1}\right]|\partial_t^jf(R_k)|^2\\
&+\int_0^\infty\left([\partial_t^j,c_{k-1}^2r_{k-1}^2\partial_x]q_k\partial_t^ju_k+c_{k-1}^2[\partial_t^j,\partial_xr_{k-1}^2]u_k\partial_t^jq_k\right)dx\\
&-(c_{k-1}^2r_{k-1}^2)|_{x=0}[\partial_t^j,f'(R_{k-1})^{-1}]\partial_tf(R_k)\partial_t^jf(R_k).
\end{aligned}\label{kee}\end{equation}
For the first term on the right-hand side, Lemma \ref{Lem 2.5} gives
\begin{equation}\left\|\frac{\partial_tc_{k-1}^2}{c_{k-1}^2}\right\|_{L^\infty}+\|c_{k-1}^{-1}r_{k-1}^2\partial_xc_{k-1}^2\|_{L^\infty}\leq CA^\frac{1}{2}\epsilon.\label{NL0k1}\end{equation}
For the second term, it holds 
\begin{equation}\begin{aligned}
\left|\partial_t[(c_{k-1}^2r_{k-1}^2)|_{x=0}f'(R_{k-1})^{-1}]\right|
=&c_0^2\left|\partial_t[r_{k-1}^2|_{x=0}(f(R_{k-1})+1)^{-\gamma-1}f'(R_{k-1})^{-1}]\right|\\
\leq &C\left(|\partial_tf(R_{k-1})|+|u_{k-1}|_{x=0}|\right)\\
\leq &CA^\frac{1}{2}\epsilon. 
\end{aligned}\label{NL0k3}\end{equation}
Therefore, the first and second terms are bounded by $CA^\frac{1}{2}\epsilon e_k^j(t)$.
For the commutators, we introduce the notations $NL_{\alpha,k}^j$ ,$\alpha=1,2,3,4,5$ as follows:
$$\begin{aligned}
\relax &[\partial_t^j, c_{k-1}^2r_{k-1}^2\partial_x]q_{k}\partial_t^ju_k\\
=&\sum_{i=1}^j\binom{j}{i}\partial_t^i(r_{k-1}^2c_{k-1}^2)\partial_t^{j-i}\partial_xq_k\partial_t^ju_k\\
=&\sum_{i=1}^j\binom{j}{i}\partial_t^i(r_{k-1}^2c_{k-1}^2)\partial_t^{j-i}(c_{k-1}^{-2}r_{k-1}^{-2}\partial_tu_k)\partial_t^ju_k\\
=&\sum_{i=1}^j\sum_{l=0}^{j-i}\binom{j}{i}\binom{j-i}{l}\partial_t^i(r_{k-1}^2c_{k-1}^2)\partial_t^l(c_{k-1}^{-2}r_{k-1}^{-2})\partial_t^{j+1-i-l}u_k\partial_t^ju_k\\
=:&NL_{1,k}^j,
\end{aligned}$$
$$\begin{aligned}
&c_{k-1}^2[\partial_t^j,\partial_xr_{k-1}^2]u_k\partial_t^jq_k\\
=&c_{k-1}^2\sum_{i=1}^j\binom{j}{i}\left(\partial_x\partial_t^i r_{k-1}^2\partial_t^{j-i}u_k\partial_t^jq_k+\partial_t^ir_{k-1}^2\partial_t^{j-i}\partial_xu_{k}\partial_t^jq_k\right)\\
=&c_{k-1}^2\sum_{i=1}^j\binom{j}{i}\partial_x\partial_t^ir_{k-1}^2\partial_t^{j-i}u_k\partial_t^jq_k\\
&+c_{k-1}^2\sum_{i=1}^j\binom{j}{i}\partial_t^ir_{k-1}^2\partial_t^{j-i}\left(r_{k-1}^{-2}\partial_tq_k
+\partial_x(r_{k-1}^{-2})r_{k-1}^2u_k\right)\partial_t^jq_k\\
=&c_{k-1}^2\sum_{i=1}^j\binom{j}{i}\partial_x\partial_t^ir_{k-1}^2\partial_t^{j-i}u_k\partial_t^jq_{k}\\
&+c_{k-1}^2\sum_{i=1}^j\binom{j}{i}\partial_t^ir_{k-1}^2\sum_{l=0}^{j-i}\binom{j-i}{l}\partial_t^l(r_{k-1}^{-2})\partial_t^{j+1-i-l}q_k\partial_t^jq_k\\
&+c_{k-1}^2\sum_{i=1}^j\binom{j}{i}\partial_t^ir_{k-1}^2\sum_{l=0}^{j-i}\binom{j-i}{l}\partial_t^l\left(r_{k-1}^2\partial_x(r_{k-1}^{-2})\right)\partial_t^{j-i-l}u_k\partial_t^jq_k\\
=:& NL_{2,k}^j+NL_{3,k}^j+NL_{4,k}^j,
\end{aligned}$$
$$\begin{aligned}
NL_{5,k}^j:=&\left(c_{k-1}^2 r_{k-1}^2\right)|_{x=0}[\partial_t^j, f'(R_{k-1})^{-1}]\partial_tf(R_k)\partial_t^jf(R_k)\\
=&\sum_{i=1}^j\binom{j}{i}(c^2_{k-1}r_{k-1}^2)|_{x=0}\partial_t^i\left(f'(R_{k-1})^{-1}\right)\partial_t^{j+1-i}f(R_{k})\partial_t^jf(R_k).
\end{aligned}$$
$\bullet$\textbf{ $NL_{1,k}^j$ terms.} We expand the coefficients as 
$$\partial_t(r_{k-1}^2c_{k-1}^2)=2c^2_{k-1}r_{k-1}u_{k-1}-(\gamma+1)c_{k-1}^2(1+q_{k-1})^{-1}r_{k-1}^2\partial_tq_{k-1},$$
$$\begin{aligned}
&\partial_t^2(r_{k-1}^2c_{k-1}^2)\\
=&2c_{k-1}^2r_{k-1}\partial_tu_{k-1}+2c_{k-1}^2u_{k-1}^2+(\gamma+1)(\gamma+2)c_{k-1}^2(1+q_{k-1})^{-2}r_{k-1}^2(\partial_tq_{k-1})^2\\
&-(\gamma+1)c_{k-1}^2(1+q_{k-1})^{-1}r_{k-1}^2\partial_t^2q_{k-1}-4(\gamma+1)c_{k-1}^2(1+q_{k-1})^{-1}r_{k-1}u_{k-1}\partial_tq_{k-1},
\end{aligned}$$
$$\begin{aligned}
\partial_t(c_{k-1}^{-2}r_{k-1}^{-2})=-2c_{k-1}^{-2}r_{k-1}^{-3}u_{k-1}+(\gamma+1)c_{k-1}^{-2}(1+q_{k-1})^{-1}r_{k-1}^{-2}\partial_tq_{k-1}. 
\end{aligned}$$
Hence except for the case $(i,l)=(2,0)$, it holds
$$\begin{aligned}
\left|\partial_t^i(r_{k-1}^2c_{k-1}^2)\partial_t^l(c_{k-1}^{-2}r_{k-1}^{-2})\right|\leq &C\left(1+|u_{k-1}|+|
\partial_tq_{k-1}|\right)\left(|u_{k-1}|+|\partial_t q_{k-1}|\right)\\
\leq &CA^\frac{1}{2}\epsilon(1+A^\frac{1}{2}\epsilon),
\end{aligned}$$
while for $(i,l)=(2,0)$ we have
$$\begin{aligned}
&\left|c_{k-1}^{-2}r_{k-1}^{-2}\left(\partial_t^2(r_{k-1}^2c_{k-1}^2)+(\gamma+1)c_{k-1}^2(1+q_{k-1})^{-1}r_{k-1}^2\partial_t^2q_{k-1}\right)\right|\\
\leq& C\left(|\partial_tu_{k-1}|+|u_{k-1}|^2+|\partial_tq_{k-1}|^2\right)\\
\leq &CA^{\frac{1}{2}}\epsilon(1+A^\frac{1}{2}\epsilon),
\end{aligned}$$
and 
$$\begin{aligned}
&\left|\int_0^\infty(\gamma+1)(1+q_{k-1})^{-1}\partial_t^2q_{k-1}\partial_tu_k\partial_t^2u_kdx\right|\\
\leq & C\|\partial_t^2q_{k-1}\|_{L^2}\|\partial_t^2u_k\|_{L^2}\|\partial_tu_k\|_{L^\infty}\\
\leq & CA^\frac{1}{2}\epsilon\|\partial_t^2u_k\|_{L^2}\left(\|\partial_tu_k\|_{L^2}+\|\tilde{L}_0^1\partial_tu_k\|_{L^2}\right)\\
\leq & CA^\frac{1}{2}\epsilon\|\partial_t^2u_k\|_{L^2}\left(\|\partial_tu_k\|_{L^2}+\|\partial_t^2q_k\|_{L^2}+CA^\frac{1}{2}\epsilon\left(\|\partial_tq_k\|_{L^2}+\|u_k\|_{L^2}\right)\right),
\end{aligned}$$
where in the last step we used (\ref{erexchq}).
Collecting the above bounds for all possible pairs $(i,l)$, we obtain 
\begin{equation}\int_0^\infty |NL_{1,k}^1|dx\leq CA^\frac{1}{2}\epsilon\|\partial_t u_k\|_{L^2}^2,\label{NL1k1}\end{equation}
\begin{equation}\begin{aligned}
\int_0^\infty |NL_{1,k}^2|dx\leq &CA^\frac{1}{2}\epsilon(1+A^\frac{1}{2}\epsilon)\|\partial_t^2u_k\|_{L^2}\left(\|\partial_tu_k\|_{L^2}+\|\partial_t^2u_k\|_{L^2}+\|\partial_t^2q_k\|_{L^2}\right)\\
&+CA\epsilon^2\|\partial_t^2u_k\|_{L^2}\left(\|\partial_tq_k\|_{L^2}+\|u_k\|_{L^2}\right).
\end{aligned}\label{NL1k2}\end{equation}
$\bullet$ $NL_{2,k}^j$ \textbf{terms.} We expand the coefficients $\partial_x\partial_t^ir_{k-1}^2$ as follows: 
$$\begin{aligned}
\partial_x\partial_tr_{k-1}^2=&2r_{k-1}\partial_xu_{k-1}+2\partial_xr_{k-1}u_{k-1},\end{aligned}$$
$$\begin{aligned}
\partial_x\partial_t^2r_{k-1}^2=&\partial_x(2u_{k-1}^2+2r_{k-1}\partial_tu_{k-1})
=4u_{k-1}\partial_xu_{k-1}+2\partial_xr_{k-1}\partial_tu_{k-1}+2r_{k-1}\partial_t\partial_xu_{k-1}.
\end{aligned}$$
With the help of the bounds in Lemma \ref{Lem 2.5} and Proposition \ref{prop2.3} for $k-1$, applying Lemma \ref{Lem 2.1} upon $u_{k-1}$ and $r_0^2\partial_xu_{k-1} $ gives
$$\|r_{k-1}\partial_xu_{k-1}\|_{L^\infty}+\|\partial_xr_{k-1}u_{k-1}\|_{L^\infty}\lesssim\|r_0^2\partial_xu_{k-1}\|_{L^\infty}+\|u_{k-1}\|_{L^\infty}\lesssim\|u_{k-1}\|_{H^2}\lesssim A^\frac{1}{2}\epsilon(1+A^\frac{1}{2}\epsilon).$$
It follows that
\begin{equation}\begin{aligned}
\int_0^\infty|NL_{2,k}^1|dx\leq &C\left(\|r_{k-1}\partial_xu_{k-1}\|_{L^\infty}+\|\partial_xr_{k-1}u_{k-1}\|_{L^\infty}\right)\|u_k\|_{L^2}\|\partial_tq_k\|_{L^2}\\
\leq &C(1+A^\frac{1}{2}\epsilon)A^\frac{1}{2}\epsilon\|u_k\|_{L^2}\|\partial_tq_k\|_{L^2},
\end{aligned}\label{NL2k1}\end{equation}
\begin{equation}\begin{aligned}
\int_0^\infty|NL_{2,k}^2|dx
\leq& C\left(\|r_{k-1}\partial_xu_{k-1}\|_{L^\infty}+\|\partial_xr_{k-1}u_{k-1}\|_{L^\infty}\right)\|\partial_tu_k\|_{L^2}\|\partial_t^2q_k\|_{L^2}\\
&+C\|r_{k-1}\partial_t\partial_xu_{k-1}\|_{L^2}\|u_k\|_{L^\infty}\|\partial_t^2q_{k}\|_{L^2}\\
&+C\left(\|u_{k-1}\|_{L^\infty}\|\partial_xu_{k-1}\|_{L^\infty}+\|\partial_tu_{k-1}\|_{L^\infty}\right)\|u_k\|_{L^2}\|\partial_t^2q_k\|_{L^2}\\
\leq&C(1+A^\frac{1}{2}\epsilon)^2A^\frac{1}{2}\epsilon\left(\|u_k\|_{L^2}+\|\partial_tu_{k}\|_{L^2}+\|\partial_tq_{k}\|_{L^2}\right)\|\partial_t^2q_k\|_{L^2}.
\end{aligned}\label{NL2k2}\end{equation}
Here we used 
$$\|r_{k-1}\partial_t\partial_xu_{k-1}\|_{L^2}\lesssim\|r_0^2\partial_x\partial_tu_{k-1}\|_{L^2}\lesssim\|\partial_tu_{k-1}\|_{H^1}\lesssim A^\frac{1}{2}(1+A^\frac{1}{2}\epsilon)\epsilon,$$
$$\|u_k\|_{L^\infty}\lesssim\|u_k\|_{L^2}+\|\tilde{L}_{k-1}^1u_k\|_{L^2}=\|u_k\|_{L^2}+\|\partial_tq_k\|_{L^2}.$$
$\bullet$ $NL_{3,k}^j$\textbf{ terms.} Note that
$$\left|\partial_t(r_{k-1}^{-2})\right|=\left|-2r_{k-1}^{-3}u_{k-1}\right|\leq CA^\frac{1}{2}\epsilon r_{k-1}^{-3},$$
$$\left|\partial_t(r_{k-1}^2)\right|=2\left|r_{k-1}u_{k-1}\right|\leq CA^\frac{1}{2}\epsilon r_{k-1},$$
$$\left|\partial_t^2(r_{k-1}^2)\right|=\left|2u_{k-1}^2+2r_{k-1}\partial_t u_{k-1}\right|\leq CA^\frac{1}{2}\epsilon(1+A^\frac{1}{2}\epsilon)r_{k-1}.$$
Collecting the bounds of $\left|\partial_t^i(r_{k-1}^2)\partial_t^l(r_{k-1}^{-2})\right|$ for each possible pair $(i,l)$ gives
\begin{equation}
\int_0^\infty|NL_{3,k}^1|dx\leq CA^\frac{1}{2}\epsilon\|\partial_tq_k\|_{L^2}^2,
\label{NL3k1}\end{equation}
\begin{equation}
\int_0^\infty|NL_{3,k}^2|dx\leq CA^\frac{1}{2}\epsilon(1+A^\frac{1}{2}\epsilon)\left(\|\partial_tq_k\|_{L^2}+\|\partial_t^2q_k\|_{L^2}\right)\|\partial_t^2q_k\|_{L^2}.    
\label{NL3k2}\end{equation}
$\bullet$ $NL_{4,k}^j$ \textbf{ terms.} The controls of $|\partial_t^ir_{k-1}^2|$ have been obtained in the previous step. To bound $\partial_t^l\left(r_{k-1}^2\partial_x(r_{k-1}^{-2})\right)$, we use (\ref{rinr0})(\ref{kqb}) to obtain
$$\left|r_{k-1}^2\partial_x(r_{k-1})^{-2}\right|=\left|2r_{k-1}^{-1}\partial_xr_{k-1}\right|\leq Cr_{in}^{-1}\partial_xr_{in}\leq Cr_0^{-3},$$
$$\begin{aligned}
\left|\partial_t\left(r_{k-1}^2\partial_x(r_{k-1}^{-2})\right)\right|=&\left|2\partial_t\left(r_{k-1}^{-1}\partial_xr_{k-1}\right)\right|=2\left|r_{k-1}^{-1}\partial_xu_{k-1}-r_{k-1}^{-2}\partial_xr_{k-1}u_{k-1}\right|\leq CA^\frac{1}{2}\epsilon r_{0}^{-3}.
\end{aligned}$$
Hence 
\begin{equation}
\int_0^\infty|NL_{4,k}^1|dx\leq CA^\frac{1}{2}\epsilon\|u_k\|_{L^2}\|\partial_tq_k\|_{L^2},
\label{NL4k1}\end{equation}
\begin{equation}
\int_0^\infty|NL_{4,k}^2|dx\leq CA^\frac{1}{2}\epsilon(1+A^\frac{1}{2}\epsilon)\left(\|u_k\|_{L^2}+\|\partial_tu_k\|_{L^2}\right)\|\partial_t^2q_k\|_{L^2}.
\label{NL4k2}\end{equation}
$\bullet$ $NL_{5,k}^j$ \textbf{ terms.}
Write $\left(f'(R_{k-1})\right)^{-1}$ as $\left(\frac{1}{f'}\circ f^{-1}\right)\circ f(R_{k-1})$, then
$$\partial_t\left(f'(R_{k-1})^{-1}\right)=\left(\frac{1}{f'}\circ f^{-1}\right)'\left(f(R_{k-1})\right)\partial_tf(R_{k-1}),$$
$$\begin{aligned}
\partial_t^2\left(f'(R_{k-1})^{-1}\right)=&\left(\frac{1}{f'}\circ f^{-1}\right)^{\prime\prime}(f(R_{k-1}))(\partial_tf(R_{k-1}))^2\\
&+\left(\frac{1}{f'}\circ f^{-1}\right)'(f(R_{k-1}))\partial_t^2f(R_{k-1}).
\end{aligned}$$
Since $\frac{1}{f^\prime}\circ f^{-1}$ is smooth in the range of $f$, and $f(R_{k-1})$ satisfies the bound (\ref{kRb}), $NL_{5,k}^j$ can be controlled as follows:
\begin{equation}\begin{aligned}
|NL_{5,k}^1|\leq CA^\frac{1}{2}\epsilon|\partial_tf(R_{k})|^2,
\end{aligned}\label{NL5k1}\end{equation}
\begin{equation}\begin{aligned}
|NL_{5,k}^2|\leq CA^\frac{1}{2}\epsilon(1+A^\frac{1}{2}\epsilon)\left(|\partial_tf(R_k)|+|\partial_t^2f(R_k)|\right)|\partial_t^2f(R_k)|.
\end{aligned}\label{NL5k2}\end{equation}
$\bullet$ \textbf{ Final estimate of $e_k^j$.} Collecting the bounds (\ref{NL0k1}-\ref{NL5k2}) for each term on the right-hand side of (\ref{kee}), we obtain
$$\frac{d}{dt}\sum_{j=0}^2e_k^j(t)\leq CA^\frac{1}{2}\epsilon(1+A^\frac{1}{2}\epsilon)^2\sum_{j=0}^2e_k^j(t).$$
Hence $\sum_{j=0}^2e_k^j(t)\leq\sum_{j=0}^2e_k^j(0)\exp\left\{CA^\frac{1}{2}\epsilon(1+A^\frac{1}{2}\epsilon)^2t\right\}$, and by Lemma \ref{Lem 2.6}
$$E_k(t)\leq B\tilde{B}\exp\left\{CA^\frac{1}{2}\epsilon(1+A^\frac{1}{2}\epsilon)^2t\right\}\epsilon^2=C(1+\epsilon)(1+A^\frac{1}{2}\epsilon)\exp\left\{CA^\frac{1}{2}\epsilon(1+A^\frac{1}{2}\epsilon)^2t\right\}\epsilon^2.$$
Now we choose $A$ large enough such that 
$$A\geq 2C(1+\epsilon)(1+A^\frac{1}{2}\epsilon),$$
and choose a small $T$ satisfying
$$\exp\left\{CA^\frac{1}{2}\epsilon(1+A^\frac{1}{2}\epsilon)^2T\right\}\leq 2,$$
then for all $t\in[0,T]$ we have $E_k(t)\leq A\epsilon^2$, which is (\ref{kme}).$\hfill\qedsymbol$\vspace{0.3cm}\\
\indent The next step to close the induction is to prove Lemma \ref{Lem 2.5} for $k$. \vspace{0.3cm}\\
\textit{Proof of Lemma \ref{Lem 2.5}.}
We first show (\ref{kqb}). The case $k=2$ is trivial since $r_{1}=r_{in}$. For $k\geq 3$, by the construction of $r_{k-1}$we have
$$\begin{aligned}
\partial_xr_{k-1}(x,t)=&\partial_xr_{in}(x)+\int_0^t\partial_xu_{k-1}(x,s)ds.
\end{aligned}$$
By Proposition \ref{prop2.3} and Sobolev embedding, $$\left|\int_0^t\partial_xu_{k-1}ds\right|\leq r_0^{-2}(x)\int_0^t\|r_0^2\partial_xu_{k-1}(t)\|_{L^\infty}ds\leq CA^\frac{1}{2}\left(1+A^\frac{1}{2}\epsilon\right)\epsilon t\cdot r_0^{-2}(x).$$
Note that $r_{in}^3\leq R_{in}^3+3(1+\overline{q})x\leq\max\left\{R_{in}^3,1+\overline{q}\right\}r_0^3$ and $\max\left\{R_{in}^3,1+\overline{q}\right\}^{\frac{2}{3}}r_{in}^{-2}\geq r_0^{-2}.$
Hence for $T$ small enough such that $$C\max\left\{R_{in}^3,1+\overline{q}\right\}^\frac{2}{3}A^\frac{1}{2}\left(1+A^\frac{1}{2}\epsilon\right)\epsilon T\leq \frac{1}{2}(1+\underline{q}),$$
it holds for $t\in[0,T]$ that $\left|\int_0^t\partial_xu_{k-1}ds\right|\leq\frac{1}{2}r_{in}^{-2}(1+\underline{q})\leq\frac{1}{2}\partial_xr_{in}$ and thus
$$\frac{3}{2}\partial_xr_{in}(x)\geq\partial_xr_{k-1}(x,t)\geq\frac{1}{2}\partial_xr_{in}(x).$$
Meanwhile, since $$r_{k-1}(0,t)=r_{in}(0)+\int_0^tu_{k-1}(0,s)ds, \quad \left|\int_0^tu_{k-1}(0,s)ds\right|\leq CA^\frac{1}{2}(1+A^\frac{1}{2}\epsilon)t,$$
for $T$ small enough such that $CA^\frac{1}{2}(1+A^\frac{1}{2}\epsilon)\epsilon T\leq\frac{1}{2}r_{in}(0)$, we have $\frac{3}{2}r_{in}(0)\geq r_{k-1}(0,t)\geq\frac{1}{2}r_{in}(0)$ for $t\in[0,T]$. Together with $\frac{3}{2}\partial_xr_{in}\geq\partial_xr_{k-1}\geq\frac{1}{2}\partial_xr_{in}$, we obtain $\frac{3}{2}r_{in}\geq r_{k-1}\geq\frac{1}{2}r_{in}$.\\
\indent Next, (\ref{kcb})(\ref{kRb}) is trivial for $k=2$. If $k\geq3$, for $x>0$ we write $$|q_{k-1}(x,t)-q_{in}(x)|\leq\int_0^t\|\partial_tq_{k-1}(s)\|_{L^\infty}ds.$$
Using Lemma \ref{Lem 2.1} with $r=r_{k-2}$ yields $\|\partial_tq_{k-1}\|_{L^\infty}\leq C\left(\|\partial_tq_{k-1}\|_{L^2}+\|L_{k-2}^1\partial_tq_{k-1}\|_{L^2}\right)\leq CA^\frac{1}{2}\epsilon.$ Therefore, for $T$ small enough such that $CA^\frac{1}{2}\epsilon T\leq\frac{1}{2}(\underline{q}+1)$, we obtain for $t\in[0,T]$ that
$\int_0^t\|\partial_tq_{k-1}(s)\|_{L^\infty}ds\leq\frac{1}{2}(\underline{q}+1)$ and thus
$\frac{\underline{q}-1}{2}\leq q_{k-1}(x,t)\leq\frac{1}{2}(3\overline{q}+1).$
The bound (\ref{kcb}) follows immediately by the construction of $c_{k-1}$. To show (\ref{kRb}), we simply use $|u_{k-1}|_{x=0}|\leq C\left(\|u_{k-1}\|_{L^2}+\|\tilde{L}_{k-2}^1u_{k-1}\|_{L^2}\right)$ to get 
$$\left|\frac{d}{dt}f(R_{k-1})\right|\leq C|f^\prime(R_{k-2})|u_{k-1}|_{x=0}|\leq CA^\frac{1}{2}\epsilon.$$
Hence for $T$ small enough such that $CA^\frac{1}{2}\epsilon T\leq
\frac{1}{2}f(R_{in})$, it follows $\frac{1}{2}f(R_{in})\leq f(R_{k-1})\leq\frac{3}{2}f(R_{in})$, $t\in[0,T]$.\\
\indent To show (\ref{kfb}), for $k\geq3$, we write
$$\begin{aligned}
r_{k-1}^2\partial_x\left(\frac{\partial_tr_{k-1}}{r_{k-1}}\right)=&r_{k-1}^2\partial_x\left(\frac{1}{r_{k-1}r_{k-2}^2}r_{k-2}^2u_{k-1}\right)\\
=&\frac{r_{k-1}}{r_{k-2}^2}\tilde{L}^1_{k-2}u_{k-1}-\left(\frac{2r_{k-1}}{r_{k-2}}\partial_xr_{k-2}+\partial_xr_{k-1}\right)u_{k-1}.
\end{aligned}$$
By (\ref{kqb}) for $r_{k-1}$ and $r_{k-2}$, we have 
$\frac{r_{k-1}}{r_{k-2}^2}\leq Cr_{in}(0)^{-1}$ and
$$\left|\frac{2r_{k-1}}{r_{k-2}}\partial_xr_{k-2}+\partial_xr_{k-1}\right|\leq C\partial_xr_{in}\leq Cr_{in}(0)^{-2}(1+\overline{q}).$$
Meanwhile, applying Lemma \ref{Lem 2.1} upon $\tilde{L}_{k-2}^1u_{k-1}$ and $u_{k-1}$ with $r=r_{k-2}$ yields $\|\tilde{L}_{k-2}^1u_{k-1}\|_{L^\infty}+\|u_{k-1}\|_{L^\infty}\leq CA^\frac{1}{2}\epsilon$.
Therefore, we obtain $\left\|r_{k-1}^2\partial_x\left(\frac{\partial_tr_{k-1}}{r_{k-1}}\right)\right\|_{L^\infty}\leq CA^\frac{1}{2}\epsilon$ and $\left\|\frac{\partial_tr_{k-1}}{r_{k-1}}\right\|_{L^\infty}\leq Cr_{in}(0)^{-1}\|u_{k-1}\|_{L^\infty}\leq CA^\frac{1}{2}\epsilon.$ For the third term in (\ref{kfb}), notice that $$\left|\frac{\partial_tc_{k-1}}{c_{k-1}}\right|=(\gamma+1)(1+q_{k-1})^{-1}|\partial_tq_{k-1}|.$$ 
By Lemma \ref{Lem 2.1}
$\|\partial_tq_{k-1}\|_{L^\infty}\leq C\left(\|\partial_tq_{k-1}\|_{L^2}+\|L_{k-2}^1\partial_tq_{k-1}\|_{L^2}\right)\leq CA^\frac{1}{2}\epsilon.$
Hence we also have $\left\|\frac{\partial_tc_{k-1}}{c_{k-1}}\right\|_{L^\infty}\leq CA^\frac{1}{2}\epsilon.$\\
\indent By (\ref{kqb}), it holds $$\left|r_0^2\partial_x\left(\frac{r_{k-1}^2}{r_0^2}\right)\right|=\left|2r_{k-1}\partial_xr_{k-1}-2r_0^{-1}r_{k-1}^2\partial_xr_0\right|\leq C\left(r_{in}\partial_xr_{in}+r_0^{-3}r_{in}^2\right)\leq C.$$ 
Since $$r_0^2\partial_x(c_{k-1}^2)=(\gamma+1)r_0^2c_{k-1}^2(1+q_{k-1})^{-1}\partial_xq_{k-1}=(\gamma+1)\frac{c_{k-1}^2}{c_0^2}(1+q_{k-1})^{-1}L_0^1q_{k-1},$$
we have $\|r_0^2\partial_x(c_{k-1}^2)\|_{L^\infty}\leq C\|L_0^1q_{k-1}\|_{L^\infty}$ by (\ref{kcb}), and we get (\ref{krb}).\\
\indent To show (\ref{embed}), recalling that $\partial_xr_0=r_0^{-2}$, a direct computation gives
$$r_{k-1}^2r_0^2\partial_x^2\left(\frac{r_{k-1}^2}{r_0^2}\right)=2r_{k-1}^3\partial_x^2r_{k-1}+2r_{k-1}^2(\partial_xr_{k-1})^2-8r_{k-1}^3r_0^{-3}\partial_xr_{k-1}+10r_0^{-6}r_{k-1}^4.$$
From (\ref{kqb}) and $1+\underline{q}\leq r_{in}^2\partial_xr_{in}\leq1+\overline{q}$, it follows 
$$\left\|r_{k-1}^2r_0^2\partial_x^2\left(\frac{r_{k-1}^2}{r_0^2}\right)\right\|_{L^2}\leq C(1+\|r_{k-1}^3\partial_x^2r_{k-1}\|_{L^2}).$$
If $k=2$, we have $\|r_{k-1}^3\partial_x^2r_{k-1}\|_{L^2}=\|r_{k-1}^3\partial_x^2r_{in}\|_{L^2}$. Using (\ref{kqb}), it holds 
$$\|r_{k-1}^3\partial_x^2r_{in}\|_{L^2}=\left\|-2r_{in}^{-5}r_{k-1}^3(1+q_{in})^2+r_{in}^{-2}r_{k-1}^3\partial_xq_{in}\right\|_{L^2}\leq C(1+\|L_0^1q_{in}\|_{L^2}).$$
For $k\geq3$, we write $\partial_x^2r_{k-1}(x,t)=\partial_x^2r_{in}(x)+\int_0^t\partial_x^2u_{k-1}(x,s)ds$.
By (\ref{kqb}) again, we also have
$$\begin{aligned}
\left\|r_{k-1}^3\int_0^t\partial_x^2u_{k-1}(\cdot,s)ds\right\|_{L^2}\leq C\int_0^t\|r_0^3\partial_x^2u_{k-1}(\cdot,s)\|_{L^2}ds\leq C\int_0^t\sum_{j=0}^2\|\tilde{L}_0^ju_{k-1}(s)\|_{L^2}ds.
\end{aligned}$$
Hence for $T$ small enough such that $CA^\frac{1}{2}(1+A^\frac{1}{2}\epsilon)\epsilon T\leq1$, it follows
$\left\|r_{k-1}^3\int_0^t\partial_x^2u_{k-1}(\cdot,s)ds\right\|_{L^2}\leq1$, and we arrive at (\ref{embed}). 
$\hfill\qedsymbol$\vspace{0.3cm}\\
\indent Finally, we prove Proposition \ref{prop2.3} for $k$.\\
\textit{Proof of Proposition \ref{prop2.3}.}
The case $j=2$ is obvious. Using inequalities (\ref{2.30}-\ref{2.33}), (\ref{kme}) and the bounds in Lemma \ref{Lem 2.5} for $c_k$ and $r_k$, we obtain
$$\|q_k(t)\|_{H^2}+\|u_k(t)\|_{H^2}\lesssim\sum_{j=0}^2\left(\|L_0^jq_k(t)\|_{L^2}+\|\tilde{L}_0^ju_k(t)\|_{L^2}\right)\lesssim A^\frac{1}{2}(1+A^\frac{1}{2}\epsilon)\epsilon.$$
Since $L_0^1\partial_tq_k=\frac{c_0^2}{c_{k-1}^2}\frac{r_0^2}{r_{k-1}^2}L_{k-1}^1\partial_tq_k$, $\frac{r_{k-1}^2}{r_0^2}\tilde{L}_0^1\partial_tu_k=\tilde{L}_{k-1}^1\partial_tu_k-2(r_{k-1}\partial_xr_{k-1}-r_0^{-3} r_{k-1}^2)\partial_tu_k$ and $|r_{k-1}\partial_xr_{k-1}|+|r_0^{-3}r_{k-1}^2|\leq C$ which can be deduced from (\ref{rinr0})(\ref{kqb}), there is
$$\|\partial_tq_k(t)\|_{H^1}+\|\partial_tu_k(t)\|_{H^1}\lesssim\sum_{j=0}^1\left(\|L_0^j\partial_tq_k(t)\|_{L^2}+\|\tilde{L}_0^j\partial_tu_k(t)\|_{L^2}\right)\lesssim A^\frac{1}{2}\epsilon.$$
The bounds on $\tilde{u}_k$ and $\tilde{q}_k$ follow by a change of variable.$\hfill\qedsymbol$\vspace{0.3cm}
\begin{rem}
From the proof above, we see that the restriction on $T$ is
$$A^\frac{1}{2}\epsilon(1+A^\frac{1}{2}\epsilon)^2T\leq c$$ for a small constant $c$ dependent on the initial value. Hence the maximal lifespan admits a lower bound in the form of $c\left[A^\frac{1}{2}\epsilon(1+A^\frac{1}{2}\epsilon)^2\right]^{-1}.$
\end{rem}
\indent At this point, the $k$ -th step of induction in Lemma \ref{prop 2.4}, Proposition \ref{prop2.3}, Lemma \ref{Lem 2.5} and Lemma \ref{Lem 2.6} has been complete and therefore Lemma \ref{prop 2.4}, Proposition \ref{prop2.3}, Lemma \ref{Lem 2.5} and Lemma \ref{Lem 2.6} hold for all $k\geq2$. \\
\indent With the uniform bound in Proposition \ref{prop2.3}, we are able to construct the solution to (\ref{cnte}-\ref{initial}). For arbitrary $M>1$, since we have the compact embedding $H_{\xi}^2(1,M)\subset\subset H_{\xi}^1(1,M)\subset\subset L^2_{\xi}(1,M)$, applying the Aubin-Lions compactness theorem twice yields a subsequence, still denoted by $\{\tilde{u}_k\}$, $\{\tilde{q_k}\}$, and limit points $\tilde{u}$, $\tilde{q}\in L^\infty\left([0,T],H^2_{\xi}(1,M)\right)$ such that 
\begin{equation}\begin{aligned}
\tilde{u}_k\rightharpoonup\tilde{u} \text{ weakly* in }L^\infty\left([0,T]; H^2_{\xi}(1,M)\right)&,\\
\tilde{u}_k\rightarrow\tilde{u} \text{ strongly in }C\left([0,T];H^1_{\xi}(1,M)\right),&
\end{aligned}\label{2.48}\end{equation}
\begin{equation}\begin{aligned}
\tilde{q}_k\rightharpoonup\tilde{q} \text{ weakly* in }L^\infty\left([0,T]; H^2_{\xi}(1,M)\right)&,\\
\tilde{q}_k\rightarrow\tilde{q} \text{  strongly in }C\left([0,T];H^1_{\xi}(1,M)\right),&
\end{aligned}\label{2.49}\end{equation}
\begin{equation}\begin{aligned}
\partial_t\tilde{u}_k\rightharpoonup\partial_t\tilde{u}\text{ weakly* in }L^\infty\left([0,T];H^1_{\xi}(1,M)\right)&,\\ \partial_t\tilde{u}_k\rightarrow\partial_t\tilde{u} \text{ strongly in }C\left([0,T];L^2_\xi(1,M)\right),&
\end{aligned}\label{2.50}\end{equation}
\begin{equation}\begin{aligned}
\partial_t\tilde{q}_k\rightharpoonup\partial_t\tilde{q}\text{ weakly* in }L^\infty\left([0,T];H^1_{\xi}(1,M)\right)&,\\ \partial_t\tilde{q}_k\rightarrow\partial_t\tilde{q} \text{ strongly in }C\left([0,T];L^2_\xi(1,M)\right).&
\end{aligned}\label{2.51}\end{equation}
Moreover, using a diagonal argument and further passing to a subsequence, we can assume that 
(\ref{2.48}-\ref{2.51}) hold for each $M>1$.
Meanwhile, applying Azela-Ascoli theorem on $\{f(R_k)\}$ and $\{\partial_tf(R_k)\}$ yields the existence of a subsequence of $\{f(R_k)\}$, still denoted by $\{f(R_k)\}$, and a limit point $\tilde{f}$ such that
$$f(R_k)\rightarrow \tilde{f}\text{ in $C^1[0,T]$}.$$
Then we construct the solution $(u,\,q,\,R)$ as follows:
$$R(t):=f^{-1}\tilde{f}(t),\quad r(x,t):=r_{in}(x)+\int_0^t u(x,s)ds,$$ $$u(x,t):=\tilde{u}((1+3x)^{\frac{1}{2}},t),\quad q(x,t):=\tilde{q}((1+3x)^\frac{1}{3},t).$$
To verify $(u,\,q,\,R)$ is a solution to (\ref{cnte}-\ref{initial}), we denote $\tilde{r}_k(\xi,t):=r_k\left(\frac{\xi^3-1}{3},t\right)$, and $\tilde{r}(\xi,t):=r\left(\frac{\xi^3-1}{3},t\right)$.
From the inequalities
$$|\tilde{r}_k-\tilde{r}|(\xi,t)\leq\int_0^t|\tilde{u}_k-\tilde{u}|(\xi,s)ds\leq\int_0^t\|\tilde{u}_k(s)-\tilde{u}(s)\|_{L^\infty_\xi}ds,$$
$$\|\partial_\xi\tilde{r}_k-\partial_\xi\tilde{r}\|_{L^2}\leq\int_0^t\|\partial_\xi\tilde{u}_k(s)-\partial_\xi\tilde{u}(s)\|_{L^2_\xi}ds,$$
it follows that $\tilde{r}_k$ converges point-wise to $\tilde{r}$, and $\partial_\xi\tilde{r}_{k}-\partial_\xi\tilde{r}$ converges to $0$ in $L^2_\xi$ due to (\ref{2.48}) and the embedding $\|\cdot\|_{L^\infty}\lesssim\|\cdot\|_{H^1}$. Meanwhile, (\ref{2.49}) implies that $(1+\tilde{q}_k)^{-\gamma-1}$ converges to $(1+\tilde{q})^{-\gamma-1}$ in $L^\infty_\xi$. By equation (\ref{immte}), we see that $\tilde{u}_k$, $\tilde{q}_k$ satisfy $$\partial_t\tilde{u}_k=c_0^2(1+\tilde{q}_{k-1})^{-\gamma-1}\frac{\tilde{r}_{k-1}^2}{\xi^2}\partial_\xi\tilde{q}_k.$$
Noting that $\tilde{u}_k$ converges to $\tilde{u}$ in $L^\infty$, passing $k$ to infinity and using (\ref{2.49})(\ref{2.50}) yields (\ref{mmte}). Next, (\ref{icnte}) is equivalent to $$\partial_t\tilde{q}_k=\xi^{-2}\partial_\xi(\tilde{r}_{k-1}^2\tilde{u}_k)=\xi^{-2}\tilde{r}_{k-1}^2\partial_\xi\tilde{u}_k+2\xi^{-2}\tilde{r}_{k-1}\partial_\xi\tilde{r}_{k-1}\tilde{u}_k.$$
Passing $k$ to infinity and using (\ref{2.48})(\ref{2.51}) yields (\ref{cnte}). For boundary conditions, passing $k$ to infinity in the equation $$(\tilde{u}_k,\tilde{q}_k)|_{\xi=1}=(u_k,q_k)|_{x=0}=(f'(R_{k-1})\frac{d}{dt}f(R_k),f(R_k))$$
and applying the trace theorem yields $(u,q)|_{x=0}=(\tilde{u},\tilde{q})|_{\xi=1}=(\frac{dR}{dt}, f(R))$, which is (\ref{kbce})(\ref{dbce}). Now that $\frac{dR}{dt}=u|_{x=0}=\partial_tr|_{x=0}$, we recover $r|_{x=0}=R$. Moreover, (\ref{cnte})(\ref{mmte}) gives $\partial_t\partial_xr^3=3\partial_x(r^2u)=3\partial_tq$ and thus 
$$\partial_xr^3-3q=\partial_xr_{in}^3-3q_{in}=3x.$$
Together with $r|_{x=0}=R$, this implies (\ref{radius}). At last, (\ref{initial}) holds since $\tilde{u}_k,\,\tilde{q}_k,\,f(R_k)$ converge to
$\tilde{u},\,\tilde{q},\,f(R)$ in $C\left([0,T];H^1_\xi(1,M)\right)$ or $C^1$ as $k\rightarrow+\infty$ and $(\tilde{u}_k,\,\tilde{q}_k,\,f(R_k))|_{t=0}=(\tilde{u}_{in},\,\tilde{q}_{in},\,f(R_{in}))$. Therefore, we conclude that $(u,\,q,\,R)$ is a solution to (\ref{cnte}-\ref{initial}) with the desired regularity of Theorem 1.2. In routine, the uniqueness can be shown in a standard way by taking the difference.
\begin{prop}[Uniqueness]
Let $(u_i,\,q_,\,R_i)$, $i=1,2$ be two solutions to (\ref{cnte}-\ref{initial}) with $u_i,q_i\in L^\infty\left([0,T];H^2\right)\cap C\left([0,T];H^1\right)$ , $R_i\in C^1([0,T])$ and $-1<\underline{q}\leq q_i\leq\overline{q}<+\infty$, $0<R_i<\overline{R}$, then $(u_1,\,q_1,\,R_1)=(u_2,\,q_2,\,R_2)$.
\end{prop}
\begin{proof}
By taking difference, we obtain
$$\partial_t(u_1-u_2)=c_0^2(1+q_1)^{-\gamma-1}r_1^2\partial_x(q_1-q_2)+c_0^2[(1+q_{1})^{-\gamma-1}r_1^2-(1+q_2)^{-\gamma-1}r_2^2]\partial_xq_2,$$
$$\partial_t(q_1-q_2)=\partial_x[r_1^2(u_1-u_2)+(r_1^2-r_2^2)u_2],$$
$$\frac{d}{dt}(f(R_1)-f(R_2))=f'(R_1)(u_1-u_2)|_{x=0}+(f'(R_1)-f'(R_2))u_2|_{x=0},$$
$$(q_1-q_2)|_{x=0}=f(R_1)-f(R_2).$$
Multiplying the first two equations by $(u_1-u_2)$ and $(q_1-q_2)$ respectively and integrating by parts yields
\begin{equation}\begin{aligned}
&\frac{1}{2}\frac{d}{dt}\left[\int_0^\infty\left((u_1-u_2)^2+c_0^2(1+q_1)^{-\gamma-1}(q_1-q_2)^2\right)dx\right]\\
&+\frac{1}{2}\frac{d}{dt}\left[c_0^2R_1^2(1+q_1)^{-\gamma-1}|_{x=0}(f'(R_1))^{-1}(f(R_1)-f(R_2))^2\right]\\
=&\int_0^\infty\left[\frac{c_0^2}{2}\partial_t(1+q_1)^{-\gamma-1}(q_1-q_2)^2-c_0^2r_1^2\partial_x(1+q_1)^{-\gamma-1}(u_1-u_2)(q_1-q_2)\right]dx\\
&+\int_0^\infty c_0^2[(1+q_1)^{-\gamma-1}r_1^2-(1+q_2)^{-\gamma-1}r_2^2]\partial_xq_2(u_1-u_2)dx\\
&+\int_0^\infty c_0^2(1+q_1)^{-\gamma-1}\partial_x((r_1^2-r_2^2)u_2)(q_1-q_2)dx\\
&+\frac{c_0^2}{2}\partial_t\left[(1+q_1)^{-\gamma-1}|_{x=0}R_1^2(f'(R_1)^{-1})\right](f(R_1)-f(R_2))^2\\
&+c_0^2(1+q_1)^{-\gamma-1}|_{x=0}R_1^2(f'(R_1))^{-1}\left(f'(R_1)-f'(R_2)\right)u_2|_{x=0}(f(R_1)-f(R_2)).
\end{aligned}\label{difference}\end{equation}
To control the right-hand side, write
$$\begin{aligned}
\frac{r_2}{r_1}-1=&\frac{r_2^3-r_1^3}{r_1(r_2^2+r_1r_2+r_1^2)}\\
=&\left[r_1(r_2^2+r_2r_1+r_1^2\right]^{-1}\left[R_2^3-R_1^3+\int_0^x(q_2-q_1)(y,t)dy\right].
\end{aligned}$$
From the $L^2-L^2$ boundedness of the maximal functional, it follows
$$\begin{aligned}
\left\|\frac{r_2}{r_1}-1\right\|_{L^2}\leq &C|f(R_2)-f(R_1)|+\left\|\frac{x}{r_1(r_2^2+r_2r_1+r_2^2)}\right\|_{L^\infty}\left\|\frac{1}{x}\int_0^x(q_2-q_1)(y,t)dy\right\|_{L^2}\\
\leq &C\left(|f(R_2)-f(R_1)|+\|q_2-q_1\|_{L^2}\right),
\end{aligned}$$
and similarly $$\left\|\frac{r_1}{r_2}-1\right\|_{L^2}\leq C\left(|f(R_2)-f(R_1)|+\|q_1-q_2\|_{L^2}\right).$$
The first term is bounded by multiple of $\|q_1-q_2\|_{L^2}+\|u_1-u_2\|_{L^2}$ since $\partial_tq_1=\partial_x(r_1^2u_1)$ and $\|\partial_x(r_1^2u_1)\|_{L^\infty}+\|r_1^2\partial_xq_1\|_{L^\infty}\leq\|u_1\|_{H^2}+\|q_1\|_{H^2}\leq C.$ 
For the second term on the right, we compute
$$\begin{aligned}
&(1+q_1)^{-\gamma-1}\frac{r_1^2}{r_2^2}-(1+q_2)^{-\gamma-1}\\
=&(q_1-q_2)\frac{(1+q_1)^{-\gamma-1}-(1+q_2)^{-\gamma-1}}{q_1-q_2}+(1+q_1)^{-\gamma-1}(\frac{r_1}{r_2}-1)(\frac{r_1}{r_2}+1),
\end{aligned}$$
and thus by Sobolev embedding
$$\begin{aligned}
\left\|(1+q_1)^{-\gamma-1}\frac{r_1^2}{r_2^2}-(1+q_2)^{-\gamma-1}\right\|_{L^2}\leq C\left(|(f(R_1)-f(R_2))|+\|q_1-q_2\|_{L^2}\right),\end{aligned}$$
$$\begin{aligned}
&\left|\int_0^\infty c_0^2[(1+q_1)^{-\gamma-1}r_1^2-(1+q_2)^{-\gamma-1}r_2^2]\partial_xq_2(u_1-u_2)dx\right|\\
\leq &C\|r_2^2\partial_xq_2\|_{L^\infty}\left(|f(R_1)-f(R_2)|+\|q_1-q_2\|_{L^2}\right)\|u_1-u_2\|_{L^2}\\
\leq &C\left(|f(R_1)-f(R_2)|+\|q_1-q_2\|_{L^2}\right)\|u_1-u_2\|_{L^2}.
\end{aligned}$$
For the third term, it holds
$$\begin{aligned}
\partial_x[(r_1^2-r_2^2)u_2]=2r_1^{-1}(q_1-q_2)u_2+2(r_1^{-1}-r_2^{-1})(1+q_2)u_2+\left(1-\frac{r_1^2}{r_2^2}\right)r_2^2\partial_xu_2,
\end{aligned}$$
and thus
$$\begin{aligned}
&\left\|\partial_x[(r_1^2-r_2^2)u_2]\right\|_{L^2}\\
\leq &C\left(\|\frac{u_2}{r_1}\|_{L^\infty}(1+\|q_2\|_{L^\infty})+\left(1+\|r_2^{-1}r_1\|_{L^\infty}\right)\|r_2^2\partial_xu_2\|_{L\infty}\right)\\
&\cdot\left(|f(R_1)-f(R_2)|+\|q_1-q_2\|_{L^2}\right)\\
\leq &C\left(|f(R_1)-f(R_2)|+\|q_1-q_2\|_{L^2}\right).
\end{aligned}$$
Since $q_1|_{x=0}=f(R_1)$ is a $C^1$ function and $f$ is invertible, the last two boundary terms can be controlled by multiples of $|f(R_1)-f(R_2)|^2$. Therefore, we conclude from (\ref{difference}) that 
$$\frac{d}{dt}e\leq C\left(\|u_1-u_2\|_{L^2}^2+\|q_1-q_2\|_{L^2}^2+|f(R_1)-f(R_2)|^2\right)\leq Ce$$
where we denote 
$$\begin{aligned}
e:=&\int_0^\infty\left((u_1-u_2)^2+c_0^2(1+q_1)^{-\gamma-1}(q_1-q_2)^2\right)dx
\\&+c_0^2R_1^2(1+q_1)^{-\gamma-1}|_{x=0}(f'(R_1))^{-1}(f(R_1)-f(R_2))^2.
\end{aligned}$$
Gronwall's inequality then gives $e=0$ on $[0,T]$, which shows $(u_1,\,q_1,\,R_1)=(u_2,\,q_2,\,R_2)$.
\end{proof}
\section{Setup of Bootstrap}\label{sec3}
The following several sections are devoted to the proof of Theorem \ref{thm 1.3}. The strategy is to use a bootstrap argument encompassing the norms in energy spaces and parameters. In this section, we detail the set-up of the bootstrap. \\
We inherit the notations from Section \ref{sec2} and drop the subscript $k$: 
$$\begin{aligned}
e^j(t):=&\frac{1}{2}\int_0^\infty\left(\left(\partial_t^ju(t)\right)^2+c^2(t)\left(\partial_t^jq(t)\right)^2\right)dx\\
&+\frac{1}{2}c(t)|_{x=0}^2f'(R(t))^{-1}\left(\partial_t^jf(R)(t)\right)^2,
\end{aligned}$$
$$E(t):=\sum_{j=0}^2|\partial_t^jf(R)(t)|^2+\sum_{j_1+j_2\leq2}\left(\|L^{j_2}\partial_t^{j_1}q(t)\|_{L^2}^2+\|\tilde{L}^{j_2}\partial_t^{j_1}u(t)\|_{L^2}^2\right),$$
$$L^1q=c^2r^2\partial_xq,\,L^2q=\partial_x(c^2r^4\partial_xq),\,\tilde{L}^1u=\partial_x(r^2u),\,\tilde{L}^2u=c^2r^2\partial_x^2(r^2u).$$
Assume the initial value $(u_{in},\,q_{in},\,R_{in})$ satisfies the assumption of Theorem \ref{thm 1.3}:
$$\epsilon^2:=\sum_{j=0}^2\|L_0^jq_{in}\|_{L^2}^2+\sum_{j=0}^2\|\tilde{L}_0^ju_{in}\|_{L^2}^2+|f(R_{in})|^2\leq\epsilon_0^2,$$
where the constant $\epsilon_0^2$ will be determined later. Taking $r=\xi(x)=(1+3x)^{\frac{1}{3}}$ in Lemma \ref{Lem 2.1} yields $\|q_{in}\|_{L^\infty}\leq C\epsilon$, and by choosing $\epsilon_0$ small, it holds $$-\frac{1}{4}<-C\epsilon\leq q_{in}\leq C\epsilon\leq\frac{1}{4}.$$
Meanwhile, since $f$ is smooth and strictly increasing near $R=1$, we have $|R_{in}-1|\leq C\epsilon$ and thus $0<1-C\epsilon\leq R_{in}\leq 1+C\epsilon<\overline{R}$. Then Theorem \ref{thm 1.2} gives the existence of a local solution. Note that the inequality (\ref{etoep}) is essentially independent of $k$, therefore $\sum_{j=0}^2e^j(0)\leq\tilde{B}\epsilon^2$ for some constant $\tilde{B}$. We now consider the time interval $[0,T]$ such that the following bounds hold on $[0,T]$:
\begin{equation}
\sum_{j=0}^2e^j(t)\leq 4\sum_{j=0}^2e^j(0)\leq 4\tilde{B}\epsilon^2,
\label{3.1}\end{equation}
\begin{equation}
-\frac{1}{2}\leq q(x,t)\leq\frac{1}{2},\;x>0,
\label{3.2}\end{equation}
and consequently for $x>0$, $t\in[0,T]$ 
$$-\frac{1}{2}\leq f(R)\leq\frac{1}{2},$$
$$\frac{Ca\gamma}{2}\left(\frac{2}{3}\right)^{\gamma+1}\leq c^2(x,t)=\frac{Ca\gamma}{2}(1+q(x,t))^{-\gamma-1}\leq\frac{Ca\gamma}{2}2^{\gamma+1},$$
$$R^3(t)+\frac{3}{2}x\leq r^3(x,t)=R^3(t)+3\int_0^x(1+q(y,t))dy\leq R^3(t)+\frac{9}{2}x.$$
Such an interval exists due to continuity and by choosing $\epsilon_0$ further small. 
By dropping the subscript in (\ref{exchq})(\ref{exchu}), we obtain
$$\partial_t^2q-\tilde{L}^1\partial_tu=2r^{-1}u\partial_tq+2r^2\partial_x(r^{-1}u)u=4r^{-1}u\partial_tq-6r^{-2}(1+q)u^2,$$
$$\partial_t^2u-L^1\partial_tq=\left(-(\gamma+1)(1+q)^{-1}\partial_tq+2r^{-1}u\right)\partial_tu.$$
Hence by Lemma \ref{Lem 2.1} the following bounds hold on $[0,T]$:
\begin{equation}\begin{aligned}
\|L^2q\|_{L^2}=\|\tilde{L}^1\partial_tu\|_{L^2}\leq&\|\partial_t^2q\|_{L^2}+C\left(\|u\|_{L^2}+\|\tilde{L}^1u\|_{L^2}\right)\left(\|\partial_tq\|_{L^2}+\|u\|_{L^2}\right)\\
=&\|\partial_t^2q\|_{L^2}+C\left(\|u\|_{L^2}+\|\partial_tq\|_{L^2}\right)^2\\
\leq& C\sum_{j=0}^2e^j(t)^\frac{1}{2},
\end{aligned}\label{3.3}\end{equation}
\begin{equation}\begin{aligned}
\|\tilde{L}^2u\|_{L^2}=\|L^1\partial_tq\|_{L^2}\leq&\|\partial_t^2u\|_{L^2}+C\left(\|\partial_tu\|_{L^2}+\|\tilde{L}^1\partial_tu\|_{L^2}\right)\left(\|u\|_{L^2}+\|\partial_tq\|_{L^2}\right)\\
\leq&\|\partial_t^2q\|_{L^2}+C\epsilon\left(\|u\|_{L^2}+\|\partial_tq\|_{L^2}\right)\\
\leq& C\sum_{j=0}^2e^j(t)^\frac{1}{2},
\end{aligned}\label{3.4}\end{equation}
where we have used that $\sum_{j=0}^2e^j(t)\leq 4\tilde{B}\epsilon^2$ on $[0,T]$ and $\epsilon<\epsilon_0\lesssim1$. Therefore, there exists a constant $A>0$ such that on $[0,T]$
\begin{equation}
E(t)\leq A\sum_{j=0}^2e^j(t).
\label{3.5}\end{equation}
Moreover, applying Lemma \ref{Lem 2.1} again gives that for $x>0$, $t\in[0,T]$
\begin{equation}\begin{aligned}
&|q(x,t)|+|\partial_\xi q(x,t)|+|\partial_t q(x,t)|\lesssim\epsilon\xi^{-1},\\
&|u(x,t)|+|\partial_\xi u(x,t)|+|\partial_t u(x,t)|\lesssim\epsilon\xi^{-1},
\end{aligned}\label{3.6}\end{equation}
where we used that $\xi^{-1}\lesssim r^{-1}\lesssim\xi^{-1}$ since $R(t)$ is bounded from below and above, and consequently
\begin{equation}\begin{aligned}
|c(x,t)-c_0|+|\partial_\xi c(x,t&)|+|\partial_t c(x,t)|\lesssim \epsilon\xi^{-1},\\
|R-&1|\lesssim\epsilon.
\end{aligned}\label{3.7}\end{equation}
By H{\"o}lder and (\ref{radius}), we also compute
\begin{equation}\begin{aligned}
\left|\frac{r(x,t)^2}{\xi^2}-1\right|=&\frac{\xi r(x,t)+\xi^2}{r(x,t)^2+\xi r(x,t)+\xi^2}\xi^{-3}\left|R(t)^3-1+3\int_0^xq(y,t)dy\right|\\
\lesssim& \xi^{-3}\left(|R(t)^3-1|+x^\frac{1}{2}\|q\|_{L^2}\right)\\
\lesssim& \epsilon\xi^{-\frac{3}{2}}.
\end{aligned}\label{3.8}\end{equation}
Combining (\ref{3.7})(\ref{3.8}) yields
\begin{equation}
\left|cr^2\xi^{-2}-c_0\right|\lesssim\epsilon\xi^{-1},
\label{3.9}\end{equation}
and thus there exist constants $\overline{c}$ and $\underline{c}$ such that
\begin{equation}
-\epsilon\lesssim\underline{c}-c_0\leq cr^2\xi^{-2}-c_0\leq\overline{c}-c_0\lesssim\epsilon.
\label{3.10}\end{equation}
If $T$ is large enough so that $1+\overline{c}T\geq\exp\left(\frac{\kappa_0}{\epsilon}\right)$, then we have obtained the almost global existence. Hence we assume without loss of generality that
\begin{equation}
1+\overline{c}T\leq\exp\left(\frac{\kappa}{\epsilon}\right)
\label{3.11}\end{equation}
holds for a constant $\kappa\leq\kappa_0$ with $\kappa_0$ to be determined. 
The heart of the proof to the almost global existence is the following:
\begin{prop}
Assume that (\ref{3.1})(\ref{3.2}) and consequently (\ref{3.3}-\ref{3.10}) hold on $[0,T]$ with $T$ satisfying (\ref{3.11}). Then the following holds:\\
1. The bounds (\ref{3.1})(\ref{3.2}) can be strictly improved on $[0,T]$. To be specific, we will prove
\begin{equation}\sum_{j=0}^2e^j(t)\leq 3\tilde{B}\epsilon^2,\; t\in[0,T],\label{3.12}\end{equation}
and thus $|q(x,t)|\lesssim\epsilon$ by (\ref{3.6}), which improves (\ref{3.2}) in turn for suitable $\epsilon_0$.\\
2. (Control of error). The following estimate holds on $[0,T]$:
$$\left|F(R)-\mathcal{R}(t)\right|\lesssim\frac{\log(1+\overline{c}t)}{1+\underline{c}t}\epsilon^\frac{3}{2}\tilde{\epsilon}^\frac{1}{2}+\frac{\epsilon^\frac{3}{2}}{1+\underline{c}t}\left(\epsilon^\frac{1}{2}+\tilde{\epsilon}^\frac{1}{2}\right)+\frac{\epsilon}{(1+\underline{c}t)^2}(\epsilon+\tilde{\epsilon}).$$
\end{prop}
\begin{rem}
In fact, system (\ref{cnte0}-\ref{radius0}) admits a conserved energy. We have by (\ref{cnte0}) and (\ref{mmte0}) that
$$\begin{aligned}
0=&u\left(\partial_tu+\frac{Ca}{2}r^2\partial_x(\rho^\gamma)\right)+\frac{Ca}{2}\rho^{-2}\left(\rho^\gamma-1\right)\left(\partial_t\rho+\rho^2\partial_x(r^2u)\right)\\
=&\partial_t\left(\frac{1}{2}u^2+\frac{Ca}{2}\left(\frac{1}{\gamma-1}\rho^{\gamma-1}-\frac{\gamma}{\gamma-1}+\rho^{-1}\right)\right)+\frac{Ca}{2}\partial_x\left(r^2u\left(\rho^\gamma-1\right)\right).
\end{aligned}$$
Integrating the above identity on the strip $\mathbb{R}_+\times [0,T]$ and using the boundary conditions gives
$$\begin{aligned}
0=&\int_0^\infty\left(\frac{1}{2}u^2+H(\rho)\right)(T)dx-\int_0^\infty\left(\frac{1}{2}u^2+H(\rho)\right)(0)dx\\
&-\int_0^T\left(R^2\left(\left(\frac{Ca}{2}+\frac{2}{We}\right)R^{-3\gamma_0}-\frac{2}{We}R^{-1}-\frac{Ca}{2}\right)\frac{dR}{dt}\right)dt,
\end{aligned}$$
where $H(\rho):=\frac{Ca}{2(\gamma-1)}\left(\rho^{\gamma-1}-\gamma+(\gamma-1)\rho^{-1}\right)$ denotes the entropy. Note that 
$$\begin{aligned}
&-R^2\left(\left(\frac{Ca}{2}+\frac{2}{We}\right)R^{-3\gamma_0}-\frac{2}{We}R^{-1}-\frac{Ca}{2}\right)\frac{dR}{dt}\\
=&\frac{d}{dt}\left\{\frac{1}{3\gamma_0-3}\left(\frac{Ca}{2}+\frac{2}{We}\right)\left(R^{-3\gamma_0+3}-1\right)+\frac{1}{We}\left(R^2-1\right)+\frac{Ca}{6}\left(R^3-1\right)\right)\\
=:&\frac{d}{dt}W(R).
\end{aligned}$$
$W(R)$ is actually the work done by the surface tension and the pressure at both the bubble surface and the infinity. Since $W(R)$ is convex in $R$, we obtain the conserved energy:
$$\int_0^\infty\left(\frac{1}{2}u^2+H(\rho)\right)dx+W(R).$$
This conserved energy can give a uniform in time the upper and lower bound of $R$ as well as the control of $L^2$ norms of $u$ and $q$. However, since $L^2$ is subcritical compared with the existence space, such controls are of less importance, and we will not use this in the proceeding proof.
\end{rem}
\section{KSS type estimate with boundary terms}\label{sec4}
In this section, we establish a KSS (Keel-Smith-Sogge) type estimate regarding the quasilinear wave equation (\ref{phiwave}). This will provide a long-time control over the weighted norms of the solution, which will be used later to improve (\ref{3.1}). However, due to the non-vanishing boundary value, we will obtain boundary terms in this KSS type estimate. Nevertheless, these boundary terms have positive signs in higher derivatives and negative signs in lower derivatives. Therefore, by taking a delicate linear combination, the only left bad term is the one with the lowest regularity, whose treatment will be given in the next two sections. To be specific, the goal of this section is to prove the estimate below.
\begin{prop} Let $\varphi$ be a solution to (\ref{phiwave}) on $[0,T]$, and assume that (\ref{3.1})(\ref{3.2}) hold on $[0,T]$, then the following bound holds:
\begin{equation}\begin{aligned}
&\sum_{j=0}^2\left\{\int_0^T\int_0^\infty\xi^{-1}\left(\left(\partial_t^{1+j}\varphi\right)^2+\left(cr^2\partial_x\partial_t^j\varphi\right)^2\right)dxdt+\int_0^T\int_0^\infty\xi^{-3}\left(\partial_t^j\varphi\right)^2dxdt\right.\\
&\left.+\int_0^T\left.\left(\partial_t^{1+j}\varphi\right)^2\right|_{x=0}dt+\int_0^T\left.\left(cr^2\partial_x\partial_t^j\varphi+\frac{c}{\rho r}\partial_t^j\varphi\right)^2\right|_{x=0}dt\right\}\\
\lesssim&\max\left\{\log_2 T,1\right\}\sum_{j=0}^2\max_{0\leq t\leq T}\int_0^\infty\left(\left(\partial_t^{1+j}\varphi\right)^2+\left(cr^2\partial_x\partial_t^j\varphi\right)^2\right)dx+\int_0^T\varphi^2|_{x=0}dt.
\end{aligned}\label{4.1}\end{equation}
\end{prop}
\begin{proof}
\textbf{Step 1.} We begin with defining a sequence of modified momentum density for $j=0,1,2$ and $k\in\mathbb{N}\cup\{0\}$:
$$P_{0,k}^j:=M_k\partial_t^{1+j}\varphi\cdot cr^2\partial_x\partial_t^j\varphi+\frac{1}{2}N_k\partial_t^j\varphi\partial_t^{1+j}\varphi,$$
$$P_{1,k}^j:=\frac{1}{2}M_k\left[\left(\partial_t^{1+j}\varphi\right)^2+\left(cr^2\partial_x\partial_t^j\varphi\right)^2\right]+\frac{1}{2}N_k\partial_t^j\varphi\cdot cr^2\partial_x\partial_t^j\varphi-\frac{1}{4}cr^2(\partial_xN_k)(\partial_t^j\varphi)^2,$$
where $M_k$, $N_k$ are undetermined weight functions. 
By taking time derivative in (\ref{phiwave}), the following holds:
\begin{equation}\partial_t^2\partial_t^j\varphi-c^2\partial_x(r^4\partial_x\partial_t^j\varphi)=[\partial_t^j,c^2\partial_xr^4\partial_x]\varphi.
\label{4.2}\end{equation}
We compute by brute force and (\ref{4.2}) that
\begin{equation}\begin{aligned}
&\int_0^T\int_0^\infty\left(-\partial_tP_{0,k}^j+\partial_x(cr^2P_{1,k}^j)\right)dxdt\\
=&\int_0^T\int_0^\infty\left(\frac{r^2\partial_x(cM_k)}{2}+\frac{c\partial_xr^2M_k}{2}-\frac{N_k}{2}\right)\left(\partial_t^{1+j}\varphi\right)^2dxdt\\
&+\int_0^T\int_0^\infty\left(\frac{r^2\partial_x(cM_k)}{2}+\frac{N_k}{2}-\frac{c\partial_xr^2M_k}{2}\right)\left(cr^2\partial_x\partial_t^j\varphi\right)^2dxdt\\
&+\int_0^T\int_0^\infty-\frac{1}{4}\partial_x(c^2r^4\partial_xN_k)\left(\partial_t^j\varphi\right)^2dxdt\\
&+\left(\text{nonlinear terms}\right),
\label{4.3}\end{aligned}\end{equation}
with the nonlinearities given by
$$\begin{aligned}
&\left(\text{nonlinear terms}\right)\\
=&\int_0^T\int_0^\infty M_k(r^2\partial_xc)\left(cr^2\partial_x\partial_t^j\varphi\right)^2dxdt\\
&-\int_0^T\int_0^\infty M_k\frac{\partial_t(cr^2)}{cr^2}\partial_t^{1+j}\varphi\cdot cr^2\partial_x\partial_t^j\varphi dxdt\\
&+\int_0^T\int_0^\infty N_k(r^2\partial_xc)\partial_t^j\varphi\cdot cr^2\partial_x\partial_t^j\varphi dxdt\\
&-\int_0^T\int_0^\infty\left(M_kcr^2\partial_x\partial_t^j\varphi+\frac{1}{2}N_k\partial_t^j\varphi\right)[\partial_t^j,c^2\partial_xr^4\partial_x]\varphi dxdt\\
&-\int_0^T\int_0^\infty\partial_tM_k\partial_t^{1+j}\varphi\cdot cr^2\partial_x\partial_t^j\varphi dxdt-\int_0^T\int_0^\infty\frac{\partial_tN_k}{2}\partial_t^j\varphi\partial_t^{1+j}\varphi dxdt.
\end{aligned}$$
The construction of $P_0$'s and $P_1$'s as well as the calculation is an analogue to Section 5 of \cite{MR2217314}, which also considers a more general case without symmetry; see also \cite{MR2015331}, but for integrity, we write down the computation process of the present context in Appendix A.  
Now select $M_k:=\frac{1}{c}\frac{\xi}{2^k+\xi}$, $N_k:=c\partial_xr^2\cdot M_k=\frac{2}{\rho r}\frac{\xi}{2^k+\xi}$, and thus the coefficients of the leading-order quadratic terms are
$$\frac{r^2\partial_x(cM_k)}{2}+\frac{c\partial_xr^2 M_k}{2}-\frac{N_k}{2}=\frac{r^2\partial_x(cM_k)}{2}-\frac{c\partial_xr^2 M_k}{2}+\frac{N_k}{2}=\frac{r^2}{2\xi^2}\frac{2^k}{(2^k+\xi)^2}.$$
To compute the coefficient of the third quadratic term, we have
$$-\partial_x\left(c^2r^4\partial_x\left(\frac{1}{\rho r}\frac{\xi}{2^k+\xi}\right)\right)
=-\partial_x\left(\frac{c^2r^3}{\rho \xi^2}\frac{2^k}{(2^k+\xi)^2}\right)+\partial_x\left(\frac{c^2}{\rho^2}\frac{\xi}{2^k+\xi}\right)+\partial_x\left(\frac{c^2r^3\partial_x\rho}{\rho^2}\frac{\xi}{2^k+\xi}\right),
$$
$$-\partial_x\left(\frac{c^2r^3}{\rho \xi^2}\frac{2^k}{(2^k+\xi)^2}\right)=-\partial_x\left(\frac{c^2r^3}{\rho\xi^3}\right)\frac{2^k\xi}{(2^k+\xi)^2}-\frac{c^2r^3}{\rho\xi^5}\frac{2^k(2^k-\xi)}{(2^k+\xi)^3},$$
$$
\partial_x\left(\frac{c^2}{\rho^2}\frac{\xi}{2^k+\xi}\right)=\partial_x\left(\frac{c^2}{\rho^2}\right)\frac{\xi}{2^k+\xi}+\frac{c^2}{\rho^2\xi^2}\frac{2^k}{(2^k+\xi)^2},
$$
$$
\partial_x\left(\frac{c^2r^3\partial_x\rho}{\rho^2}\frac{\xi}{2^k+\xi}\right)=\partial_x\left(\frac{c^2r^3\partial_x\rho}{\rho^2}\right)\frac{\xi}{2^k+\xi}+\frac{c^2r^3\partial_x\rho}{\rho^2\xi^2}\frac{2^k}{(2^k+\xi)^2}.
$$
Hence the coefficient of the third term is
$$\begin{aligned}
&-\frac{1}{4}\partial_x\left(c^2r^4\partial_x\left(\frac{2}{\rho r}\frac{\xi}{2^k+\xi}\right)\right)\\
=&\frac{1}{2}\frac{c^2}{\rho^2\xi^2}\frac{2^k}{(2^k+\xi)^2}\left(1+\frac{\rho r^3}{\xi^3}\frac{\xi-2^k}{2^k+\xi}\right)-\frac{1}{2}\partial_x\left(\frac{c^2r^3}{\rho\xi^3}\right)\frac{2^k\xi}{(2^k+\xi)^2}\\
&+\frac{1}{2}\partial_x\left(\frac{c^2}{\rho^2}\right)\frac{\xi}{2^k+\xi}+\frac{1}{2}\partial_x\left(\frac{c^2r^3\partial_x\rho}{\rho^2}\right)\frac{\xi}{2^k+\xi}+\frac{1}{2}\frac{c^2r^3\partial_x\rho}{\rho^2\xi^2}\frac{2^k}{(2^k+\xi)^2}\\
=&\frac{1}{2}\frac{c^2}{\rho^2}\left(1+\frac{\rho r^3}{\xi^3}\right)\frac{2^k}{\xi(2^k+\xi)^3}+\frac{1}{2}\frac{c^2}{\rho^2}\left(1-\frac{\rho r^3}{\xi^3}\right)\frac{2^{2k}}{\xi^2(2^k+\xi)^3}-\frac{1}{2}\partial_x\left(\frac{c^2r^3}{\rho\xi^3}\right)\frac{2^k\xi}{(2^k+\xi)^2}\\
&+\frac{1}{2}\partial_x\left(\frac{c^2}{\rho^2}\right)\frac{\xi}{2^k+\xi}+\frac{1}{2}\partial_x\left(\frac{c^2r^3\partial_x\rho}{\rho^2}\right)\frac{\xi}{2^k+\xi}+\frac{1}{2}\frac{c^2r^3\partial_x\rho}{\rho^2\xi^2}\frac{2^k}{(2^k+\xi)^2}.
\end{aligned}$$
The principal part is $\frac{1}{2}\frac{c^2}{\rho^2}\left(1+\frac{\rho r^3}{\xi^3}\right)\frac{2^k}{\xi(2^k+\xi)^3}$, while the terms left will be regarded as nonlinearities. 
Next, by divergence theorem, there is
\begin{equation}\begin{aligned}
&\int_0^\infty P_{0,k}^j(x,0)dx-\int_0^\infty P_{0,k}^j(x,T)dx-\int_0^T\left(cr^2P_{1,k}^j\right)(0,t)dt\\
=&\int_0^T\int_0^\infty\left(-\partial_t P_{0,k}^j+\partial_x(cr^2 P_{1,k}^j)\right)dxdt.
\end{aligned}\label{4.4}\end{equation}
Moreover, we expand the boundary term $P_{1,k}^j(0,t)$ as
$$\begin{aligned}
&P_{1,k}^j(0,t)\\
=&\left.\left[\frac{1}{2}M_k\left(\left(\partial_t^{1+j}\varphi\right)^2+\left(cr^2\partial_x\partial_t^j\varphi\right)^2\right)+\frac{1}{2}N_k\partial_t^j\cdot cr^2\partial_x\partial_t^j\varphi\right]\right|_{x=0}\\
&-\left.\left[\frac{1}{4}cr^2\partial_xN_k\left(\partial_t^j\varphi\right)^2\right]\right|_{x=0}\\
=&\left.\left[\frac{1}{2}M_k\left(\partial_t^{1+j}\varphi\right)^2+\frac{1}{2}M_k\left(cr^2\partial_x\partial_t^j\varphi+\frac{N_k}{2M_k}\partial_t^j\varphi\right)^2\right]\right|_{x=0}\\
&-\left.\left[\left(\frac{N_k^2}{8M_k}+\frac{1}{4}cr^2\partial_xN_k\right)\left(\partial_t^j\varphi\right)^2\right]\right|_{x=0},
\end{aligned}$$
and compute the coefficients explicitly:
$$\frac{1}{2}M_k(0,t)=\frac{1}{2c(0,t)}\frac{1}{2^k+1},\quad\frac{N_k(0,t)}{2M_k(0,t)}=\left(\frac{c}{\rho r}\right)(0,t),$$
$$\frac{N_k(0,t)^2}{8M_k(0,t)}=\frac{1}{2}\left(\frac{c}{\rho^2 r^2}(0,t)\right)\frac{1}{2^k+1},$$
$$\begin{aligned}
&\frac{1}{4}\left(cr^2\partial_x N_k\right)(0,t)\\
=&\left.\frac{1}{4}\left[cr^2\partial_x\left(\frac{2}{\rho r}\frac{\xi}{2^k+\xi}\right)\right]\right|_{x=0}\\
=&\left.\frac{1}{2}\left[\frac{cr}{\rho}\frac{2^k}{\xi^2(2^k+\xi)^2}\right]\right|_{x=0}-\left.\frac{1}{2}\left[\frac{c}{\rho^2}\left(r\partial_x\rho+r^{-2}\right)\frac{\xi}{2^k+\xi}\right]\right|_{x=0}\\
=&\frac{1}{2}\left.\left(\frac{cr}{\rho}\right)\right|_{x=0}\frac{2^k}{(2^k+1)^2}-\frac{1}{2}\left.\left(\frac{c}{\rho^2 r^2}\right)\right|_{x=0}\frac{1}{2^k+1}-\frac{1}{2}\left.\left(\frac{c}{\rho^2}r\partial_x\rho\right)\right|_{x=0}\frac{1}{2^k+1},
\end{aligned}$$
$$\frac{N_k(0,t)^2}{8M_k(0,t)}+\frac{1}{4}\left(cr^2\partial_xN_k\right)(0,t)=\frac{1}{2}\left(\left.\frac{cr}{\rho}\right)\right|_{x=0}\frac{2^k}{(2^k+1)^2}-\frac{1}{2}\left.\left(\frac{c}{\rho^2}r\partial_x\rho\right)\right|_{x=0}\frac{1}{2^k+1}.$$
Combining (\ref{4.3})(\ref{4.4}) yields the identity
\begin{equation}\begin{aligned}
&\int_0^T\int_0^\infty\frac{r^2}{2\xi^2}\frac{2^k}{(2^k+\xi)^2}\left(\left(\partial_t^{1+j}\varphi\right)^2+\left(cr^2\partial_x\partial_t^j\varphi\right)^2\right)dxdt\\
&+\int_0^T\int_0^\infty\frac{c^2}{2\rho^2}\left(1+\frac{\rho r^3}{\xi^3}\right)\frac{2^k}{\xi(2^k+\xi)^3}\left(\partial_t^j\varphi\right)^2dxdt\\
&+\int_0^T\frac{R^2}{2}\frac{1}{2^k+1}\left.\left(\partial_t^{1+j}\varphi\right)^2\right|_{x=0}dt+\int_0^T\frac{R^2}{2}\frac{1}{2^k+1}\left.\left(cr^2\partial_x\partial_t^j\varphi+\frac{c}{\rho r}\partial_t^j\varphi\right)^2\right|_{x=0}dt\\
=&\int_0^\infty P_{0,k}^j(x,0)dx-\int_0^\infty P_{0,k}^j(x,T)dx+\int_0^T\left.\frac{1}{2}\left(\frac{c^2r^3}{\rho}\right)\right|_{x=0}\frac{2^k}{(2^k+1)^2}\left.\left(\partial_t^j\varphi\right)^2\right|_{x=0}dt\\
&+\sum_{l=1}^8 NL_{k,l}^j,
\end{aligned}\label{4.5}\end{equation}
with the nonlinearities given by
$$\begin{aligned}
NL_{k,1}^j:=&-\int_0^T\int_0^\infty\left[\frac{1}{2}\frac{c^2}{\rho^2}\left(1-\frac{\rho r^3}{\xi^3}\right)\frac{2^{2k}}{\xi^2(2^k+\xi)^3}-\frac{1}{2}\partial_x\left(\frac{c^2r^3}{\rho \xi^3}\right)\frac{2^k\xi}{(2^k+\xi)^2}\right.\\&
\left.+\frac{1}{2}\partial_x\left(\frac{c^2}{\rho^2}\right)\frac{\xi}{2^k+\xi}+\frac{1}{2}\partial_x\left(\frac{c^2r^3\partial_x\rho}{\rho^2}\right)\frac{\xi}{2^k+\xi}\right.\\
&\left.+\frac{1}{2}\frac{c^2r^3\partial_x\rho}{\rho^2\xi^2}\frac{2^k}{(2^k+\xi)^2}\right]\left(\partial_t^j\varphi\right)^2dxdt,\\
NL_{k,2}^j:=&-\int_0^T\int_0^\infty (r^2\partial_xc)M_k\left(cr^2\partial_x\partial_t^j\varphi\right)^2dxdt,\\
NL_{k,3}^j:=&-\int_0^T\int_0^\infty M_k\frac{\partial_t(cr^2)}{cr^2}\partial_t^{1+j}\varphi\cdot cr^2\partial_x\partial_t^j\varphi dxdt,\\
NL_{k,4}^j:=&-\int_0^T\int_0^\infty N_k (r^2\partial_xc )\partial_t^j\varphi\cdot cr^2\partial_x\partial_t^j\varphi dxdt,
\end{aligned}$$
$$\begin{aligned}
NL_{k,5}^j:=&\int_0^T\int_0^\infty\left(M_kcr^2\partial_x\partial_t^j\varphi+\frac{1}{2}N_k\partial_t^j\varphi\right)[\partial_t^j,c^2\partial_x r^4\partial_x]\varphi dxdt,\\
NL_{k,6}^j:=&\int_0^T\int_0^\infty\partial_tM_{k}\partial_t^{1+j}\varphi\cdot cr^2\partial_x\partial_t^j\varphi dxdt,\\
NL_{k,7}^j:=&\int_0^T\int_0^\infty\frac{\partial_tN_k}{2}\partial_t^j\varphi\partial_t^{1+j}\varphi dxdt,\\
NL_{k,8}^j:=&-\int_0^T\frac{1}{2}\left.\left(\frac{c^2r^3}{\rho^2}\partial_x\rho\right)\right|_{x=0}\frac{1}{2^k+1}\left.\left(\partial_t^j\varphi\right)^2\right|_{x=0}dt.
\end{aligned}$$
Note that $\frac{2^k}{(2^k+\xi)^2}\simeq\frac{1}{\xi}$ and $\frac{2^k}{\xi(2^k+\xi)^3}\simeq\frac{1}{\xi^3}$ for $\xi\in[2^k,2^{k+1}]$. Summing in $k$ with $2^k\leq T$, or equivalently, for $k\leq K$ with $K:=\max\{-1,\lfloor\log_2 T\rfloor\}$ gives 
\begin{equation}\begin{aligned}
&\int_0^T\int_0^\infty\frac{1}{\xi}\left(\left(\partial_t^{1+j}\varphi\right)^2+\left(cr^2\partial_x\partial_t^j\varphi\right)^2\right)dxdt\\
\leq& \sum_{k=0}^K
\int_0^T\int_{2^k\leq\xi<2^{k+1}}\frac{1}{\xi}\left(\left(\partial_t^{1+j}\varphi\right)^2+\left(cr^2\partial_x\partial_t^j\varphi\right)^2\right)dxdt\\
&+\int_0^T\int_{\xi\geq 2^{K+1}}\frac{1}{\xi}\left(\left(\partial_t^{1+j}\varphi\right)^2+\left(cr^2\partial_x\partial_t^j\varphi\right)^2\right)dxdt\\
\lesssim&\sum_{k=0}^K\int_0^T\int_0^\infty\frac{r^2}{2\xi^2}\frac{2^k}{(2^k+\xi)^2}\left(\left(\partial_t^{1+j}\varphi\right)^2+\left(cr^2\partial_x\partial_t^j\varphi\right)^2\right)dxdt\\
&+\max_{0\leq t\leq T}\int_0^\infty\left(\left(\partial_t^{1+j}\varphi\right)^2+\left(cr^2\partial_x\partial_t^j\varphi\right)^2\right)dx,
\end{aligned}\label{4.6}\end{equation}
where we used $2^{-(K+1)}\leq T^{-1}$ and (\ref{3.2}), and similarly, 
\begin{equation}\begin{aligned}
&\int_0^T\int_0^\infty\frac{1}{\xi^3}\left(\partial_t^j\varphi\right)^2dxdt\\
\leq&\sum_{k=0}^K\int_0^T\int_{2^k\leq\xi<2^{k+1}}\frac{1}{\xi^3}\left(\partial_t^j\varphi\right)^2dxdt+\int_0^T\int_{\xi\geq 2^{K+1}}^\infty\frac{1}{\xi^3}\left(\partial_t^j\varphi\right)^2dxdt\\
\lesssim &\sum_{k=0}^K\int_0^T\int_0^\infty\frac{1}{2}\frac{c^2}{\rho^2}\left(1+\frac{\rho r^3}{\xi^3}\right)\frac{2^k}{\xi(2^k+\xi)^3}\left(\partial_t^j\varphi\right)^2dxdt+\max_{0\leq t\leq T}\int_0^\infty\frac{1}{\xi^2}\left(\partial_t^j\varphi\right)^2dx.
\end{aligned}\label{4.7}\end{equation}
\textbf{Step 2.} In this part, we shall bound the nonlinearities by the left hand of (\ref{4.6}) and (\ref{4.7}) with the help of (\ref{3.1}-\ref{3.10}). Since we are using the new unknown $\varphi$, it is necessary to build a bridge between $\varphi$ and $u$, $\rho$. Recall $\frac{c}{\rho}u=c r^2\partial_x\varphi$ and (\ref{rho-phi}) that $\partial_t\varphi+\frac{Ca\gamma}{2(\gamma-1)}(\rho^{\gamma-1}-1)=\frac{1}{2}u^2$. By taking time derivatives, the following equations hold:
$$\partial_t^2\varphi-\frac{c^2}{\rho}\partial_tq=u\partial_tu,$$
$$\partial_t^3\varphi-\frac{c^2}{\rho}\partial_t^2 q=\partial_t\left(\frac{c^2}{\rho}\right)\partial_tq+\left(\partial_tu\right)^2+u\partial_t^2u.$$
Similarly, taking time derivatives on $\partial_x\varphi=\frac{u}{\rho r^2}$ yields
$$cr^2\partial_x\partial_t\varphi=\frac{c}{\rho}\partial_t u+cu\left(\partial_tq-\frac{2u}{\rho r}\right),$$
$$cr^2\partial_x\partial_t^2\varphi=\frac{c}{\rho}\partial_t^2u+2c\left(\partial_tq-\frac{2u}{\rho r}\right)\partial_tu+cu\left(-\frac{2}{\rho}\frac{\partial_tu}{r}+\frac{6u^2}{\rho r^2}-\frac{4u\partial_tq}{r}+\partial_t^2q\right).$$
Hence by (\ref{3.6})(\ref{3.7})
\begin{equation}\begin{aligned}
\left|\partial_t\varphi-\frac{c^2}{\rho}\left(\rho^{\gamma-1}-1\right)\right|\lesssim\epsilon\xi^{-1}|u|\simeq&\epsilon\xi^{-1}|cr^2\partial_x\varphi|,\\
\left|\partial_t^2\varphi-\frac{c^2}{\rho}\partial_tq\right|\lesssim\epsilon\xi^{-1}|u|\simeq\epsilon\xi^{-1}&|cr^2\partial_x\varphi|,\\
\left|\partial_t^3\varphi-\frac{c^2}{\rho}\partial_t^2q\right|\lesssim\epsilon\xi^{-1}(|\partial_tq|+|\partial_t&u|+|\partial_t^2u|),
\end{aligned}\label{4.8}\end{equation}
\begin{equation}\begin{aligned}
|cr^2\partial_x\partial_t\varphi-\frac{c}{\rho}\partial_tu|\lesssim\epsilon\xi^{-1}|u|\simeq\epsilon\xi^{-1}|cr^2\partial_x\varphi|&,\\
|cr^2\partial_x\partial_t^2\varphi-\frac{c}{\rho}\partial_t^2u|\lesssim\epsilon\xi^{-1}(|u|+|\partial_tu|+|\partial_t^2q|&).
\end{aligned}\label{4.9}\end{equation}
Moreover, for sufficiently small $\epsilon_0$ and $\epsilon<\epsilon_0$, (\ref{4.8})(\ref{4.9}) further implies that
\begin{equation}\begin{aligned}
|\partial_t^2u|+|\partial_t^2q|\lesssim&|cr^2\partial_x\partial_t^2\varphi|+|\partial_t^3\varphi|+\epsilon\xi^{-1}(|u|+|\partial_tq|+|\partial_tu|)\\
\lesssim&|cr^2\partial_x\partial_t^2\varphi|+|\partial_t^3\varphi|+\epsilon\xi^{-1}(|cr^2\partial_x\varphi|+|cr^2\partial_x\partial_t\varphi|+|\partial_t^2\varphi|),
\end{aligned}\label{4.10}\end{equation}
and consequently
\begin{equation}
|\partial_t^2u-\rho r^2\partial_x\partial_t^2\varphi|\lesssim\epsilon\xi^{-1}\sum_{j=0}^2\left(|cr^2\partial_x\partial_t^j\varphi|+|\partial_t^{j+1}\varphi|\right),
\label{4.11}\end{equation}
\begin{equation}
|\partial_t^2q-\frac{\rho}{c^2}\partial_t^3\varphi|\lesssim\epsilon\xi^{-1}\sum_{j=0}^2\left(|cr^2\partial_x\partial_t^j\varphi|+|\partial_t^{j+1}\varphi|\right).
\label{4.12}\end{equation}
$\bullet$ \textbf{$NL_{1,k}^j$ terms.}
We use the bootstrap bounds (\ref{3.6})(\ref{3.8}) to compute
$$\left|\frac{c^2}{\rho^2}\left(1-\frac{\rho r^3}{\xi^3}\right)\frac{2^{2k}}{\xi^2(2^k+\xi)^3}\right|\lesssim\left(|\rho-1|+\rho\left|\frac{r^3}{\xi^3}-1\right|\right)\frac{2^{2k}}{\xi^2(2^k+\xi)^3}\lesssim\epsilon\frac{2^{2k}}{\xi^3(2^k+\xi)^3},$$
$$\left|\partial_x\left(\frac{c^2r^3}{\rho\xi^3}\right)\frac{2^k\xi}{(2^k+\xi)^2}\right|\lesssim\left|\partial_x\left(\rho^\gamma\frac{r^3}{\xi^3}\right)\right|\frac{2^k\xi}{(2^k+\xi)^2}\lesssim\epsilon\frac{2^k}{\xi^2(2^k+\xi)^2},$$
$$\left|\partial_x\left(\frac{c^2}{\rho^2}\right)\frac{\xi}{2^k+\xi}\right|\lesssim|\partial_x(\rho^{\gamma-1})|\frac{\xi}{2^k+\xi}\lesssim\epsilon\frac{1}{\xi^2(2^k+\xi)},$$
$$\left|\frac{c^2r^3\partial_x\rho}{\rho^2\xi^2}\frac{2^k}{(2^k+\xi)^2}\right|=\left|c^2r^3\partial_xq\right|\frac{2^k}{\xi^2(2^k+\xi)^2}\lesssim\epsilon\frac{2^k}{\xi^2(2^k+\xi)^2}.$$
Noting that $L^2q=\partial_x\left(c^2r^4\partial_xq\right)$, it follows that
$$\begin{aligned}\left|\partial_x\left(\frac{c^2r^3\partial_x\rho}{\rho^2}\right)\frac{\xi}{2^k+\xi}\right|=&\left|\partial_x(c^2r^3\partial_xq)\right|\frac{\xi}{2^k+\xi}
\leq\left|\frac{c^2}{\rho}\partial_xq\right|\frac{\xi}{2^k+\xi}+\frac{1}{r}|L^2q|\frac{\xi}{2^k+\xi}\\
\lesssim&\frac{\epsilon}{\xi^2(2^k+\xi)}+|L^2q|\frac{1}{2^k+\xi}.
\end{aligned}$$
Using the above bounds, we estimate $NL_{1,k}^j$ except the term involving $\partial_x\left(\frac{c^2r^3\partial_x\rho}{\rho^2}\right)$:
$$\begin{aligned}
&\sum_{k=0}^K\left|\int_0^T\int_0^\infty\left[\frac{1}{2}\frac{c^2}{\rho^2}\left(1-\frac{\rho r^3}{\xi^3}\right)\frac{2^{2k}}{\xi^2(2^k+\xi)^3}-\frac{1}{2}\partial_x\left(\frac{c^2r^3}{\rho \xi^3}\right)\frac{2^k\xi}{(2^k+\xi)^2}\right.\right.\\
&\left.\left.+\frac{1}{2}\partial_x\left(\frac{c^2}{\rho^2}\right)\frac{\xi}{2^k+\xi}+\frac{1}{2}\frac{c^2r^3\partial_x\rho}{\rho^2\xi^2}\frac{2^k}{(2^k+\xi)^2}\right]\left(\partial_t^j\varphi\right)^2dxdt\right|\\
\lesssim&\epsilon\sum_{k=0}^K\int_0^T\int_0^\infty\left(\frac{2^{2k}}{\xi^3(2^k+\xi)^3}+\frac{2^k}{\xi^2(2^k+\xi)^2}+\frac{1}{\xi^2(2^k+\xi)}\right)\left|\partial_t^j\varphi\right|^2dxdt\\
\lesssim&\epsilon(1+K)\int_0^T\int_0^\infty\xi^{-3}\left|\partial_t^j\varphi\right|^2dxdt,\\
\lesssim&\kappa\int_0^T\int_0^\infty\xi^{-3}\left|\partial_t^j\varphi\right|^2dxdt.
\end{aligned}$$
The last inequality holds due to the assumption (\ref{3.11}) and that
\begin{equation}
\epsilon(1+K)\leq\epsilon\max\{0,1+\lfloor\log_2 T\rfloor\}\lesssim\kappa.
\label{4.13}\end{equation}
Next, we integrate by parts and obtain
$$\begin{aligned}
&\left|\int_0^T\int_0^\infty\frac{1}{2}\partial_x\left(\frac{c^2r^3\partial_x\rho}{\rho^2}\right)\frac{\xi}{2^k+\xi}\left(\partial_t^j\varphi\right)^2dxdt\right|\\
\leq&\left|\int_0^T\frac{1}{2}\left(c^2r^3\partial_xq\frac{\xi}{2^k+\xi}\left(\partial_t^j\varphi\right)^2\right)|_{x=0}dt\right|\\
&+\left|\int_0^T\int_0^\infty\frac{1}{2}cr^2r^3\partial_xq\left(\frac{2^k}{\xi^2(2^k+\xi)^2}\left(\partial_t^j\varphi\right)^2+\frac{2\xi}{2^k+\xi}\partial_t^j\varphi\partial_x\partial_t^j\varphi\right)dxdt\right|\\
\lesssim&\epsilon\left(\int_0^T\int_0^\infty\frac{2^k}{\xi^2(2^k+\xi)^2}\left(\partial_t^j\varphi\right)^2dxdt+\int_0^T\int_0^\infty\frac{1}{2^k+\xi}\left(cr^2\partial_x\partial_t^j\varphi\right)^2dxdt\right)\\
&+\epsilon\frac{1}{2^k+1}\int_0^T\left.\left(\partial_t^j\varphi\right)^2\right|_{x=0}dt.
\end{aligned}$$
Then sum in $k$ and use (\ref{4.13}) again:
$$\begin{aligned}
&\sum_{k=0}^K\left|\int_0^T\int_0^\infty\frac{1}{2}\partial_x\left(\frac{c^2r^3\partial_x\rho}{\rho^2}\right)\frac{\xi}{2^k+\xi}\left(\partial_t^j\varphi\right)^2dxdt\right|\\
\lesssim&\kappa\int_0^T\int_0^\infty\left[\xi^{-3}\left(\partial_t^j\varphi\right)^2+\xi^{-1}\left(cr^2\partial_x\partial_t^j\varphi\right)^2\right]dxdt+\epsilon\int_0^T\left.\left(\partial_t^j\varphi\right)^2\right|_{x=0}dt.
\end{aligned}$$
Therefore, we conclude with the control of $NL_{1,k}^j$:
\begin{equation}\begin{aligned}
&\sum_{k=0}^K\left|NL_{1,k}^j\right|\\
\lesssim&\kappa\int_0^T\int_0^\infty\left[\xi^{-3}\left(\partial_t^j\varphi\right)^2+\xi^{-1}\left(cr^2\partial_x\partial_t^j\varphi\right)^2\right]dxdt+\epsilon\int_0^T\left.\left(\partial_t^j\varphi\right)^2\right|_{x=0}dt.
\end{aligned}\label{4.14}\end{equation}
$\bullet$\textbf{ $NL_{2,k}^j$, $NL_{3,k}^j$, $NL_{4,k}^j$, $NL_{8,k}^j$ terms}. These terms can be bound easily by using (\ref{3.6})(\ref{3.7}) and (\ref{4.13}):
\begin{equation}\begin{aligned}
\sum_{k=0}^K\left|NL_{2,k}^j\right|\leq&\sum_{k=0}^K\int_0^T\int_0^\infty\left|\frac{r^2\partial_xc}{c}\right|\frac{\xi}{2^k+\xi}\left(cr^2\partial_x\partial_t^j\varphi\right)^2dxdt\\
\lesssim&\epsilon\sum_{k=0}^K\int_0^T\int_0^\infty\frac{1}{2^k+\xi}\left(cr^2\partial_x\partial_t^j\varphi\right)^2dxdt\\
\lesssim&\kappa\int_0^T\int_0^\infty\xi^{-1}\left(cr^2\partial_x\partial_t^j\varphi\right)^2dxdt,
\end{aligned}\label{4.15}\end{equation}
\begin{equation}\begin{aligned}
\sum_{k=0}^K\left|NL_{3,k}^j\right|\lesssim&\sum_{k=0}^K\int_0^T\int_0^\infty\left|\frac{\partial_t(cr^2)}{c^2r^2}\right|\frac{\xi}{2^k+\xi}\left(\left(\partial_t^{1+j}\varphi\right)^2+\left(cr^2\partial_x\partial_t^j\varphi\right)^2\right)dxdt\\
\lesssim&\epsilon\sum_{k=0}^K\int_0^T\int_0^\infty\frac{1}{2^k+\xi}\left(\left(\partial_t^{1+j}\varphi\right)^2+\left(cr^2\partial_x\partial_t^j\varphi\right)^2\right)dxdt\\
\lesssim&\kappa\left(\int_0^T\int_0^\infty\xi^{-1}\left(\partial_t^{1+j}\varphi\right)^2dxdt+\int_0^T\int_0^\infty\xi^{-1}\left(cr^2\partial_x\partial_t^j\varphi\right)^2dxdt\right),
\end{aligned}\label{4.16}\end{equation}
\begin{equation}\begin{aligned}
\sum_{k=0}^K\left|NL_{4,k}^j\right|\lesssim&\sum_{k=0}^K\int_0^T\int_0^\infty\frac{2}{\rho r}\frac{1}{2^k+\xi}|r^2\partial_xc|\left(\xi^{-2}\left(\partial_t^j\varphi\right)^2+\left(cr^2\partial_x\partial_t^j\varphi\right)^2\right)dxdt,\\
\lesssim&\epsilon\sum_{k=0}^K\int_0^T\int_0^\infty\frac{1}{2^k+\xi}\left(\xi^{-2}\left(\partial_t^j\varphi\right)^2+\left(cr^2\partial_x\partial_t^j\varphi\right)^2\right)dxdt\\
\lesssim&\kappa\left(\int_0^T\int_0^\infty\xi^{-3}\left(\partial_t^j\varphi\right)^2dxdt+\int_0^T\int_0^\infty\xi^{-1}\left(cr^2\partial_x\partial_t^j\varphi\right)^2dxdt\right),
\end{aligned}\label{4.17}\end{equation}
\begin{equation}\begin{aligned}
\sum_{k=0}^K\left|NL_{8,k}^j\right|\leq&\sum_{k=0}^K\int_0^T\frac{1}{2}\left.\left|c^2r^3\partial_xq\right|_{x=0}\right|\frac{1}{2^k+1}\left.\left(\partial_t^j\varphi\right)^2\right|_{x=0}dt\\
\lesssim&\epsilon\int_0^T\left.\left(\partial_t^j\varphi\right)^2\right|_{x=0}dt.
\end{aligned}\label{4.18}\end{equation}
$\bullet$\textbf{ $NL_{5,k}^j$ term.} We expand the commutators and use (\ref{phiwave}) to obtain:
$$\begin{aligned}
\relax&[\partial_t^j,c^2\partial_xr^4\partial_x]\varphi\\
=&[\partial_t^j,c^2]\left(\partial_x(r^4\partial_x\varphi)\right)+c^2\partial_x\left([\partial_t^j,r^4]\partial_x\varphi\right)\\
=&[\partial_t^j,c^2]\left(c^{-2}\partial_t^2\varphi\right)+c^2\partial_x\left([\partial_t^j,r^4]r^{-4}r^4\partial_x\varphi\right)\\
=&[\partial_t^j,c^2]\left(c^{-2}\partial_t^2\varphi\right)+c^2[\partial_t^j,r^4]r^{-4}\left(c^{-2}\partial_t^2\varphi\right)+c^2\left[\partial_x,[\partial_t^j,r^4]r^{-4}\right]r^4\partial_x\varphi.
\end{aligned}$$
For $j=1$, we compute:
$$\begin{aligned}
\relax[\partial_t,c^2\partial_xr^4\partial_x]\varphi=\left(\frac{\partial_tc^2}{c^2}+\frac{\partial_tr^4}{r^4}\right)\partial_t^2\varphi+c^2\partial_x\left(\frac{\partial_tr^4}{r^4}\right)r^4\partial_x\varphi
\end{aligned}$$
and using (\ref{3.6}):
$$\left|[\partial_t,c^2\partial_xr^4\partial_x]\varphi\right|\lesssim\epsilon\xi^{-1}\left|\partial_t^2\varphi\right|+\epsilon\xi^{-2}\left|cr^2\partial_x\varphi\right|.$$
For $j=2$, the commutator reads
$$\begin{aligned}
\relax&[\partial_t^2,c^2\partial_xr^4\partial_x]\varphi\\
=&\left(\partial_t^2c^2+2\partial_tc^2\partial_t\right)\left(c^{-2}\partial_t^2\varphi\right)+c^2\left(\frac{\partial_t^2r^4}{r^4}+2\frac{\partial_tr^4}{r^4}\partial_t+2\partial_tr^4\partial_tr^{-4}\right)\left(c^{-2}\partial_t^2\varphi\right)\\
&+c^2\left[\partial_x,\frac{\partial_t^2r^4}{r^4}+2\frac{\partial_tr^4}{r^4}\partial_t+2\partial_tr^4\partial_tr^{-4}\right]r^4\partial_x\varphi\\
=&\left(2\frac{\partial_tc^2}{c^2}+2\frac{\partial_tr^4}{r^4}\right)\partial_t^3\varphi\\
&+\left(\frac{\partial_t^2c^2}{c^2}+2\partial_tc^2\partial_tc^{-2}+\frac{\partial_t^2r^4}{r^4}+2\partial_tr^4\partial_tr^{-4}+2\frac{\partial_tr^4}{r^4}c^2\partial_tc^{-2}\right)\partial_t^2\varphi\\
&+c^2r^2\left(\partial_x\left(\frac{\partial_t^2r^4}{r^4}\right)+2\partial_x\left(\partial_tr^4\partial_tr^{-4}\right)\right)r^2\partial_x\varphi\\&
+2c^2r^2\partial_x\left(\frac{\partial_tr^4}{r^4}\right)\left(r^2\partial_x\partial_t\varphi+\frac{\partial_tr^4}{r^4}r^2\partial_x\varphi\right).
\end{aligned}$$
We also compute:
$$\frac{\partial_t^2c^2}{c^2}=-\left(\gamma+1\right)\frac{\partial_t(c^2\rho\partial_tq)}{c^2}=-\left(\gamma+1\right)\rho\partial_t^2q+\left(\gamma+1\right)\left(\gamma+2\right)c^2\rho^2(\partial_tq)^2,$$
$$\frac{\partial_t^2r^4}{r^4}=4\frac{\partial_tu}{r}+12\frac{u^2}{r^2},$$
$$\partial_x\left(\frac{\partial_tr^4}{r^4}\right)=4\partial_x\left(\frac{u}{r}\right)=4r^{-3}\left(\partial_tq-3\rho^{-1}u\right),$$
$$\begin{aligned}
\partial_x\left(\frac{\partial_t^2r^4}{r^4}\right)=&4\partial_x\left(\frac{\partial_t(r^2u)}{r^3}+\frac{u^2}{r^2}\right)\\
=&4r^{-3}\left(\partial_t^2q-\frac{3}{\rho r^3}\left(2ru^2+r^2\partial_tu\right)-\frac{6u^2}{\rho r^2}+\frac{2u}{r}\partial_tq\right),
\end{aligned}$$
where we used that $\partial_t\partial_x(r^2u)=\partial_t^2q$ to avoid the presence of mixed derivative.
Therefore, it follows from (\ref{3.6})(\ref{3.7}) that
$$\begin{aligned}
\left|[\partial_t^2,c^2\partial_xr^4\partial_x]\varphi\right|\lesssim&\epsilon\xi^{-1}|\partial_t^3\varphi|+\epsilon\xi^{-2}\left(|\partial_t^2\varphi|+|cr^2\partial_x\partial_t\varphi|+\xi^{-1}|cr^2\partial_x\varphi|\right)\\
&+|\partial_t^2q|\left(|\partial_t^2\varphi|+\xi^{-1}|cr^2\partial_x\varphi|\right).
\end{aligned}$$
Then we conclude in view of (\ref{4.8})(\ref{4.10}) that 
\begin{equation}\begin{aligned}
\sum_{k=0}^K\left(\left|NL_{5,k}^1\right|+\left|NL_{5,k}^2\right|\right)
\lesssim&\epsilon\sum_{k=0}^K\int_0^T\int_0^\infty\frac{1}{2^k+\xi}\sum_{i=1}^2\left(|cr^2\partial_x\partial_t^i\varphi|+\xi^{-1}|\partial_t^i\varphi|\right)\\
&\cdot\left(|\partial_t^3\varphi|+|\partial_t^2\varphi|+|cr^2\partial_x\partial_t\varphi|+|cr^2\partial_x\varphi|\right)dxdt\\
\lesssim&\kappa\sum_{i=0}^2\int_0^T\int_0^\infty\xi^{-1}\left(\left(\partial_t^{1+j}\varphi\right)^2+\left(cr^2\partial_x\partial_t^j\varphi\right)^2\right)dxdt.
\end{aligned}\label{4.19}\end{equation}
$\bullet$\textbf{ $NL_{6,k}^j$, $NL_{7,k}^j$ terms.}
We estimate by (\ref{3.6})(\ref{3.7}) that
\begin{equation}\begin{aligned}
\sum_{k=0}^K\left|NL_{6,k}^j\right|\leq&\sum_{k=0}^K\int_0^T\int_0^\infty\left|\partial_tc^{-1}\right|\frac{\xi}{2^k+\xi}|\partial_t^{1+j}\varphi||cr^2\partial_x\partial_t^j\varphi|dxdt\\
\lesssim&\epsilon\sum_{k=0}^K\int_0^T\int_0^\infty\frac{1}{2^k+\xi}|\partial_t^{1+j}\varphi||cr^2\partial_x\partial_t^j\varphi|dxdt\\
\lesssim&\kappa\left(\int_0^T\int_0^\infty\xi^{-1}\left(\partial_t^{1+j}\varphi\right)^2dxdt+\int_0^T\int_0^\infty\xi^{-1}\left(cr^2\partial_x\partial_t^j\varphi\right)^2\right)dxdt,
\end{aligned}\label{4.20}\end{equation}
\begin{equation}\begin{aligned}
\sum_{k=0}^K\left|NL_{7,k}^j\right|\leq&\sum_{k=0}^K\int_0^T\int_0^\infty\left|\partial_t\left(\frac{1}{\rho r}\right)\right|\frac{\xi}{2^k+\xi}|\partial_t^j\varphi||\partial_t^{1+j}\varphi|dxdt\\
\lesssim&\epsilon\sum_{k=0}^K\int_0^T\int_0^\infty\xi^{-1}\frac{1}{2^k+\xi}|\partial_t^j\varphi||\partial_t^{1+j}\varphi|dxdt\\
\lesssim&\kappa\left(\int_0^T\int_0^\infty\xi^{-3}\left(\partial_t^j\varphi\right)^2dxdt+\int_0^T\int_0^\infty\xi^{-1}\left(\partial_t^{1+j}\varphi\right)^2dxdt\right).
\end{aligned}\label{4.21}\end{equation}
\textbf{Step 3.} Now we combine (\ref{4.5})(\ref{4.6})(\ref{4.7}) to obtain:
\begin{equation}\begin{aligned}
&\int_0^T\int_0^\infty\xi^{-1}\left(\left(\partial_t^{1+j}\varphi\right)^2+\left(cr^2\partial_x\partial_t^j\varphi\right)^2\right)dxdt+\int_0^T\int_0^\infty\xi^{-3}\left(\partial_t^j\varphi\right)^2dxdt\\
&+\int_0^T\left.\left(\partial_t^{1+j}\varphi\right)^2\right|_{x=0}dt+\int_0^T\left.\left(cr^2\partial_x\partial_t^j\varphi+\frac{c}{\rho r}\partial_t^j\varphi\right)^2\right|_{x=0}dt\\
\lesssim&\sum_{k=0}^K\left\{\left|\int_0^\infty P_{0,k}^j(x,0)dx\right|+\left|\int_0^\infty P_{0,k}^j(x,T)dx\right|+\sum_{l=1}^8\left|NL_{l,k}^j\right|\right\}\\
&+\int_0^T\left.\left(\partial_t^j\varphi\right)^2\right|_{x=0}dt+\max_{0\leq t\leq T}\int_0^\infty\left(\left(\partial_t^{1+j}\varphi\right)^2+\left(cr^2\partial_x\partial_t^j\varphi\right)^2+\xi^{-2}\left(\partial_t^j\varphi\right)^2\right)dx.
\end{aligned}\label{4.22}\end{equation}
By the explicit expression of $P_{0,k}^j$ and Hardy's inequality, it follows that
$$\begin{aligned}
\sum_{k=0}^K\left|\int_0^\infty P_{0,k}^j(x,t)dx\right|\lesssim& (1+K)\int_0^\infty\left(\left(\partial_t^{1+j}\varphi\right)^2+\left(cr^2\partial_x\partial_t^j\varphi\right)^2+\xi^{-2}\left(\partial_t^j\varphi\right)^2\right)dx\\
\lesssim&\max\left\{\log_2T,1\right\}\int_0^\infty\left(\left(\partial_t^{1+j}\varphi\right)^2+\left(cr^2\partial_x\partial_t^j\varphi\right)^2\right)dx.
\end{aligned}$$
Moreover, by taking a suitable linear combination of (\ref{4.22}) for $j=0,1,2$, we can cancel the boundary term on the right-hand side except the one with the lowest order. Then we conclude by using the bounds (\ref{4.14}-\ref{4.21}) of nonlinearities:
$$\begin{aligned}
&\sum_{j=0}^2\left\{\int_0^T\int_0^\infty\xi^{-1}\left(\left(\partial_t^{1+j}\varphi\right)^2+\left(cr^2\partial_x\partial_t^j\varphi\right)^2\right)dxdt+\int_0^T\int_0^\infty\xi^{-3}\left(\partial_t^j\varphi\right)^2dxdt\right.\\
&+\left.\int_0^T\left.\left(\partial_t^{1+j}\varphi\right)^2\right|_{x=0}dt+\int_0^T\left.\left(cr^2\partial_x\partial_t^j\varphi+\frac{c}{\rho r}\partial_t^j\varphi\right)^2\right|_{x=0}dt\right\}\\
\lesssim&\max\left\{\log_2T,1\right\}\sum_{j=0}^2\max_{0\leq t\leq T}\int_0^\infty\left(\left(\partial_t^{1+j}\varphi\right)^2+\left(cr^2\partial_x\partial_t^j\varphi\right)^2\right)dx+\int_0^T\varphi^2|_{x=0}dt\\
&+\sum_{j=0}^2\sum_{k=0}^K\sum_{l=1}^8\left|NL_{l,k}^j\right|\\
\lesssim&\max\left\{\log_2T,1\right\}\sum_{j=0}^2\max_{0\leq t\leq T}\int_0^\infty\left(\left(\partial_t^{1+j}\varphi\right)^2+\left(cr^2\partial_x\partial_t^j\varphi\right)^2\right)dx+\int_0^T\varphi^2|_{x=0}dt\\
&+\kappa\sum_{j=0}^2\left\{\int_0^T\int_0^\infty\left[\xi^{-1}\left(\left(\partial_t^{1+j}\varphi\right)^2+\left(cr^2\partial_x\partial_t^j\varphi\right)^2\right)+\xi^{-3}\left(\partial_t^j\varphi\right)^2\right]dxdt\right\}\\
&+\epsilon\sum_{j=0}^2\int_0^T\left.\left(\partial_t^j\varphi\right)^2\right|_{x=0}dt.
\end{aligned}$$
Hence for sufficiently small $\kappa_0$ and $\epsilon_0$ we arrive at (\ref{4.1}).
\end{proof}
\section{Derivation of precise boundary ODE}\label{sec5}
In order to handle the boundary value of $\varphi$ present on the right-hand side of (\ref{4.1}), we derive in this section the equation satisfied by $\varphi|_{x=0}$. Introduce $\psi:=r\varphi$, and substitute $\varphi$ by $\frac{\psi}{r}$ in (\ref{phiwave}) yields
$$\begin{aligned}
0=&r^{-1}\partial_t^2\psi+2\partial_tr^{-1}\partial_t\psi+\partial_t^2r^{-1}\psi-c^2\partial_x\left(r^3\partial_x\psi+r^4\partial_xr^{-1}\psi\right)\\
=&r^{-1}\left(\partial_t^2\psi+2r\partial_tr^{-1}\partial_t\psi-c^2r^2\partial_x\left(r^2\partial_x\psi\right)\right)+\left(\partial_t^2r^{-1}-c^2\partial_x\left(r^4\partial_xr^{-1}\right)\right)\psi.
\end{aligned}$$
On the other hand, from (\ref{mmte0}) it follows
$$\begin{aligned}
0=\partial_tu+\frac{Ca}{2}r^2\partial_x(\rho^\gamma)
=\partial_t^2r-\frac{Ca\gamma}{2}\rho^{\gamma+1}r^2\partial_x\rho^{-1}=\partial_tr^2-c^2r^2\partial_x(r^2\partial_xr),
\end{aligned}$$
and thus
$$\partial_t^2r^{-1}-c^2\partial_x\left(r^4\partial_xr^{-1}\right)=\partial_t^2r^{-1}+c^2\partial_x(r^2\partial_xr)=\partial_t^2r^{-1}+r^{-2}\partial_t^2r=2r^{-3}(\partial_tr)^2.$$
Henceforth $$0=\partial_t^2\psi-c^2r^2\partial_x\left(r^2\partial_x\psi\right)-2\frac{u}{r}\partial_t\psi+2\frac{u^2}{r^2}\psi.$$
We rewrite the above equation to a hyperbolic system of the form
$$\begin{aligned}
0=&\left(\partial_t+cr^2\partial_x\right)\left(\partial_t-cr^2\partial_x\right)\psi+\left(\frac{\partial_tc+cr^2\partial_xc}{c}+\frac{2u}{r}\right)cr^2\partial_x\psi-\frac{2u}{r}\partial_t\psi+\frac{2u^2}{r^2}\psi\\
=&c^\frac{1}{2}\left(\partial_t+cr^2\partial_x\right)\left(c^{-\frac{1}{2}}\left(\partial_t-cr^2\partial_x\right)\psi\right)+\left(\frac{\partial_tc+cr^2\partial_xc}{2c}\right)\left(\partial_t+cr^2\partial_x\right)\psi\\
&-\frac{2u}{r}\left(\partial_t-cr^2\partial_x\right)\psi+\frac{2u^2}{r^2}\psi,
\end{aligned}$$
and similarly,
$$\begin{aligned}
0=&c^\frac{1}{2}\left(\partial_t-cr^2\partial_x\right)\left(c^{-\frac{1}{2}}\left(\partial_t+cr^2\partial_x\right)\psi\right)+\left(\frac{\partial_tc-cr^2\partial_xc}{2c}\right)\left(\partial_t-cr^2\partial_x\right)\psi\\
&-\frac{2u}{r}\left(\partial_t+cr^2\partial_x\right)\psi+\frac{2u^2}{r^2}\psi.
\end{aligned}$$
By introducing $w_F=\left(\partial_t-cr^2\partial_x\right)\psi$ and $w_B=\left(\partial_t+cr^2\partial_x\right)\psi$, which in fact stand for the forward and backward pressure wave, we see that $w_F$, $w_B$ satisfy the transport equations for $(x,t)\in\mathbb{R}_+\times [0,T]$:
\begin{equation}
0=\left(\partial_t+cr^2\partial_x\right)\frac{w_F}{c^\frac{1}{2}}+\left(\frac{\partial_tc+cr^2\partial_xc}{2c}\right)\frac{w_B}{c^\frac{1}{2}}-\frac{2u}{r}\frac{w_F}{c^\frac{1}{2}}+\frac{2u^2}{r^2}\frac{\psi}{c^\frac{1}{2}},
\label{5.1}\end{equation}
\begin{equation}
0=\left(\partial_t-cr^2\partial_x\right)\frac{w_B}{c^\frac{1}{2}}+\left(\frac{\partial_tc-cr^2\partial_xc}{2c}\right)\frac{w_F}{c^\frac{1}{2}}-\frac{2u}{r}\frac{w_B}{c^\frac{1}{2}}+\frac{2u^2}{r^2}\frac{\psi}{c^\frac{1}{2}}.
\label{5.2}\end{equation}
One can check that (\ref{5.1})(\ref{5.2}) is genuinely nonlinear in $w_F$, $w_B$, and this is actually a main feature of the isentropic compressible Euler equations. In view of Majda's conjecture \cite[Page 89]{a4a7096a663d4b1b9d1e78bc9abf6205}, it is reasonable not to expect a global solution. To solve (\ref{5.1})(\ref{5.2}) as a hyperbolic system on the half line, we shall need the boundary data of $w_F$. Recalling (\ref{dbcb0}) and taking time derivative yields
$$\left.\frac{Ca\gamma}{2}\rho^{\gamma-1}\partial_t\rho\right|_{x=0}=\tilde{f}^\prime(R)\left.\left(\rho r^2\partial_x\varphi\right)\right|_{x=0}=\tilde{f}^\prime(R)\left.\left(\rho r\partial_x\psi-r^{-2}\psi\right)\right|_{x=0}.$$
Then substituting $\rho^{\gamma}-1$ by (\ref{rho-phi}) gives
$$-\left.\partial_t\left(\partial_t\varphi-\frac{1}{2}u^2\right)\right|_{x=0}=\tilde{f}^\prime(R)\left.\left(r\partial_x\psi\right)\right|_{x=0}-\tilde{f}^\prime(R)\left.\left(\rho^{-1}r^{-2}\psi\right)\right|_{x=0}.$$
In order to obtain an equation of $\psi|_{x=0}$, we compute
$$\partial_t^2\varphi=r^{-1}\partial_t^2\psi+2\partial_tr^{-1}\partial_t\psi+\partial_t^2r^{-1}\psi=r^{-1}\partial_t^2\psi-2\frac{u}{r^2}\partial_t\psi-\left(\frac{\partial_tu}{r^2}-2\frac{u^2}{r^3}\right)\psi,$$
$$r^2\partial_x\psi=\frac{1}{c}\left(w_B-\partial_t\psi\right),$$
$$u=\rho r^2\partial_x\varphi=\rho r\partial_x\psi-r^{-2}\psi=\frac{\rho}{rc}\left(w_B-\partial_t\psi\right)-r^{-2}\psi.$$
Therefore, by writing $r|_{x=0}=R$ we obtain
$$\begin{aligned}
0=&\left\{\partial_t^2\psi-Ru\partial_tu-\left(\frac{\tilde{f}^\prime(R)}{c}+2R^{-1}u\right)\partial_t\psi\right.\\
&\left.\left.-\left(\tilde{f}^\prime(R)\rho^{-1}R^{-1}-2R^{-2}u^2+R^{-1}\partial_tu\right)\psi+\frac{\tilde{f}^\prime(R)}{c}w_B\right\}\right|_{x=0}\\
=&\left\{\partial_t^2\psi-\left(\frac{\tilde{f}^\prime(R)}{c}+2R^{-1}u-\frac{\rho\partial_tu}{c}\right)\partial_t\psi\right.\\
&\left.\left.-\left(\tilde{f}^\prime(R)\rho^{-1}R^{-1}-2R^{-2}u^2-R^{-1}\partial_tu\right)\psi+\left(\frac{\tilde{f}^\prime(R)}{c}-\frac{\rho\partial_tu}{c}\right)w_B\right\}\right|_{x=0}.
\end{aligned}$$
In short, we obtain
\begin{equation}\begin{aligned}
\partial_t^2\psi|_{x=0}+b_1\partial_t\psi|_{x=0}+b_0\psi|_{x=0}+aw_B|_{x=0}=0, 
\end{aligned}\label{5.3}\end{equation}
with
$$b_1(t):=-\left.\left(\frac{\tilde{f}^\prime(R)}{c}+2R^{-1}u-\frac{\rho\partial_tu}{c}\right)\right|_{x=0},$$
$$b_0(t):=\left.-\left(\tilde{f}^\prime(R)\rho^{-1}R^{-1}-2R^{-2}u^2-R^{-1}\partial_tu\right)\right|_{x=0},$$
$$a(t):=\left.\left(\frac{\tilde{f}^\prime(R)}{c}-\frac{\rho\partial_tu}{c}\right)\right|_{x=0}.$$
Moreover, we estimate these coefficients by using (\ref{3.6})(\ref{3.7}) and the smoothness of $\tilde{f}$ near $1$:
\begin{equation}\begin{aligned}
|b_1(t)+\frac{\tilde{f}^\prime(1)}{c_0}|\lesssim\epsilon,\;
|b_0(t)+\tilde{f}^\prime(1)|\lesssim\epsilon,\;
|a(t)-\frac{\tilde{f}^\prime(1)}{c_0}|\lesssim\epsilon.
\end{aligned}\label{5.4}\end{equation}
Let $\Lambda_1$, $\Lambda_2$ denote the two eigenvalues of the unperturbed equation, namely, the two roots of $\lambda^2-\frac{\tilde{f}^\prime(1)}{c_0}\lambda-\tilde{f}^\prime(1)$. Then we rewrite (\ref{5.3}) in the form of column vectors and diagonalize the resulted equation:
\begin{equation}\begin{aligned}
\partial_tY-\left[\begin{matrix}
\Lambda_1 &   \\
  & \Lambda_2
\end{matrix}\right]Y+\Delta Y+W=0,
\end{aligned}\label{5.5}\end{equation}
$$Y:=\left[\begin{matrix}
1 & -\Lambda_2 \\ 1 & -\Lambda_1
\end{matrix}\right]
\left[\begin{matrix}
\partial_t\psi|_{x=0} \\ \psi|_{x=0}
\end{matrix}\right],\;
\Delta=\frac{1}{\Lambda_1-\Lambda_2}\left[\begin{matrix}
\delta_1 & \delta_2\\
\delta_1 & \delta_2
\end{matrix}\right],\;
W=\left[\begin{matrix}
aw_B|_{x=0}\\ aw_B|_{x=0}
\end{matrix}\right],$$
$$\delta_1=\Lambda_1\left(b_1+\frac{\tilde{f}^\prime(1)}{c_0}\right)+\left(b_0+\tilde{f}^\prime(1)\right),\;
\delta_2=-\Lambda_2\left(b_1+\frac{\tilde{f}^\prime(1)}{c_0}\right)-\left(b_0+\tilde{f}^\prime(1)\right).$$
By solving (\ref{5.5}) it follows
\begin{equation}\begin{aligned}
Y(t)=&\exp\left[\begin{matrix}
\Lambda_1t & \\ & \Lambda_2t  
\end{matrix}\right]Y(0)
-\int_0^t\exp\left[\begin{matrix}
\Lambda_1(t-s) & \\ & \Lambda_2(t-s)
\end{matrix}\right]\left(\Delta Y\right)(s)ds\\
&-\int_0^t\exp\left[\begin{matrix}
\Lambda_1(t-s) & \\ & \Lambda_2(t-s)
\end{matrix}\right]W(s)ds.
\end{aligned}\label{5.6}\end{equation}
We remark here that from (\ref{dbcb0}) it holds
$\tilde{f}^\prime(1)=-3\gamma_0\left(\frac{Ca}{2}+\frac{2}{We}\right)+\frac{2}{We}<0$, which shows that $\text{Re}\Lambda_i<0$, $i=1,2$.
\section{The analysis of backward pressure wave}\label{sec6}
Recall the KSS type estimate (\ref{4.1}). To control the boundary term $\int_0^T\varphi^2|_{x=0}dt$, we have obtained in the previous section the integral equation (\ref{5.6}) satisfied by $\psi:=r^{-1}\varphi$ with the source term containing the backward pressure wave $w_B$. To close this argument, we will estimate $w_B$ in this section by using the equations (\ref{5.1})(\ref{5.2}), which form a hyperbolic system on the domain $\left\{(x,t):x\geq0, t\geq0\right\}$. Generally, in view of the signs of characteristic speeds, we will need the data of $w_B$ on $\{t=0\}$ and the data of $w_F$ on $\{t=0\}\cup\{x=0\}$ to solve the Cauchy problem associated with system (\ref{5.1})(\ref{5.2}), but in the present context, the data of $w_F$ on $\{x=0\}$ is instead given implicitly by equation (\ref{5.6}) and the relation $w_F=\left(\partial_t-cr^2\partial_x\right)\psi=2\partial_t\psi-w_B$.\\
Recall that $\xi=(1+3x)^\frac{1}{3}$ and we rewrite (\ref{5.1})(\ref{5.2}) as
\begin{equation}
0=\left(\partial_t+cr^2\xi^{-2}\partial_\xi\right)\frac{w_F}{c^\frac{1}{2}}+\left(\frac{\partial_tc+cr^2\xi^{-2}\partial_\xi c}{2c}\right)\frac{w_B}{c^\frac{1}{2}}-\frac{2u}{r}\frac{w_F}{c^\frac{1}{2}}+\frac{2u^2}{r^2}\frac{\psi}{c^\frac{1}{2}},
\label{6.1}\end{equation}
\begin{equation}
0=\left(\partial_t-cr^2\xi^{-2}\partial_\xi\right)\frac{w_B}{c^\frac{1}{2}}+\left(\frac{\partial_tc-cr^2\xi^{-2}\partial_\xi c}{2c}\right)\frac{w_F}{c^\frac{1}{2}}-\frac{2u}{r}\frac{w_B}{c^\frac{1}{2}}+\frac{2u^2}{r^2}\frac{\psi}{c^\frac{1}{2}}.
\label{6.2}\end{equation}
Now we begin the analysis by introducing the characteristics. Note that the bootstrap bounds (\ref{3.6})(\ref{3.7}) imply the Lipschitz continuity of $cr^2\xi^{-2}$ in both $\xi$ and $t$ argument. In fact, we have
\begin{equation}
\left|\partial_\xi(cr^2\xi^{-2})\right|+\left|\partial_t(cr^2\xi^{-2})\right|\lesssim\epsilon\xi^{-1}.
\label{lip}\end{equation}
For each $\xi\in[1,+\infty)$, $t\in[0,T]$, define $X_B$ as the unique solution (figure \ref{1}) to
$$
\left\{\begin{aligned}
&X_B(\xi,t;0)=(\xi,t),\\
&\frac{d}{ds}X_B(\xi,t;s)=\left(-\left(cr^2\xi^{-2}\right)(X_B(\xi,t;s)),1\right),
\end{aligned}\right.$$
and $X_F$ to be the unique solution to
$$
\left\{\begin{aligned}
&X_F(\xi,t;0)=(\xi,t),\\
&\frac{d}{ds}X_F(\xi,t;s)=\left(\left(cr^2\xi^{-2}\right)(X_F(\xi,t;s)),1\right).
\end{aligned}\right.$$
We then obtain the integral equation corresponding to (\ref{6.1})(\ref{6.2}): 
\begin{equation}\begin{aligned}
&\frac{w_F}{c^\frac{1}{2}}\left(X_F(\xi,t;s)\right)\\=&\exp\left\{\int_0^s\frac{2u}{r}\left(X_F(\xi,t;\sigma)\right)d\sigma\right\}\frac{w_F}{c^\frac{1}{2}}\left(X_F(\xi,t;0)\right)\\
&-\int_0^s\exp\left\{\int_0^{s-\sigma}\frac{2u}{r}\left(X_F(\xi,t;\tau)\right)d\tau\right\}\left(\frac{\partial_tc+cr^2\xi^{-2}c}{2c}\right)\frac{w_B}{c^\frac{1}{2}}\left(X_F(\xi,t;\sigma)\right)d\sigma\\
&-\int_0^s\exp\left\{\int_0^{s-\sigma}\frac{2u}{r}\left(X_F(\xi,t;\tau)\right)d\tau\right\}2r^{-2}u^2\frac{\psi}{c^\frac{1}{2}}\left(X_F(\xi,t;\sigma)\right)d\sigma,
\end{aligned}\label{6.3}\end{equation}
\begin{equation}\begin{aligned}
&\frac{w_B}{c^\frac{1}{2}}\left(X_B(\xi,t;s)\right)\\
=&\exp\left\{\int_0^s\frac{2u}{r}\left(X_B(\xi,t;\sigma)\right)d\sigma\right\}\frac{w_B}{c^\frac{1}{2}}\left(X_B(\xi,t;0)\right)\\
&-\int_0^s\exp\left\{\int_0^{s-\sigma}\frac{2u}{r}\left(X_B(\xi,t;\tau)\right)d\tau\right\}\left(\frac{\partial_tc+cr^2\xi^{-2}c}{2c}\right)\frac{w_F}{c^\frac{1}{2}}\left(X_B(\xi,t;\sigma)\right)d\sigma\\
&-\int_0^s\exp\left\{\int_0^{s-\sigma}\frac{2u}{r}\left(X_B(\xi,t;\tau)\right)d\tau\right\}2r^{-2}u^2\frac{\psi}{c^\frac{1}{2}}\left(X_B(\xi,t;\sigma)\right)d\sigma.
\end{aligned}\label{6.4}\end{equation}
For each $t\in[0,T]$, we denote $\xi_0(t):=X_B(1,t;-t)$, and thus $X_B(\xi_0(t),0;t)=(1,t)$. Meanwhile, for each $t_*\in[0,T]$,
denote by $\Omega_B(t_*)$ the backward acoustic cone $$\Omega_B(t_*):=\left\{(\xi,t):0\leq t\leq t_*,\;1\leq\xi\leq X_B(\xi(t_*),0;t)\right\}.$$ Introduce the new unknowns $v_F(\xi,t,t_*)$, $v_B(\xi,t,t_*)$ by
$$v_F(\xi,t;t_*):=\sup\left\{\left|\frac{w_F}{c^\frac{1}{2}}\left(X_F(\xi,t;s)\right)\right|:X_F(\xi,t;s)\in\Omega_B(t_*)\right\},$$
$$v_B(\xi,t;t_*):=\sup\left\{\left|\frac{w_B}{c^\frac{1}{2}}\left(X_B(\xi,t;s)\right)\right|:X_B(\xi,t;s)\in\Omega_B(t_*)\right\}.$$
\begin{figure}[H]
\center{\includegraphics[width=10cm]  {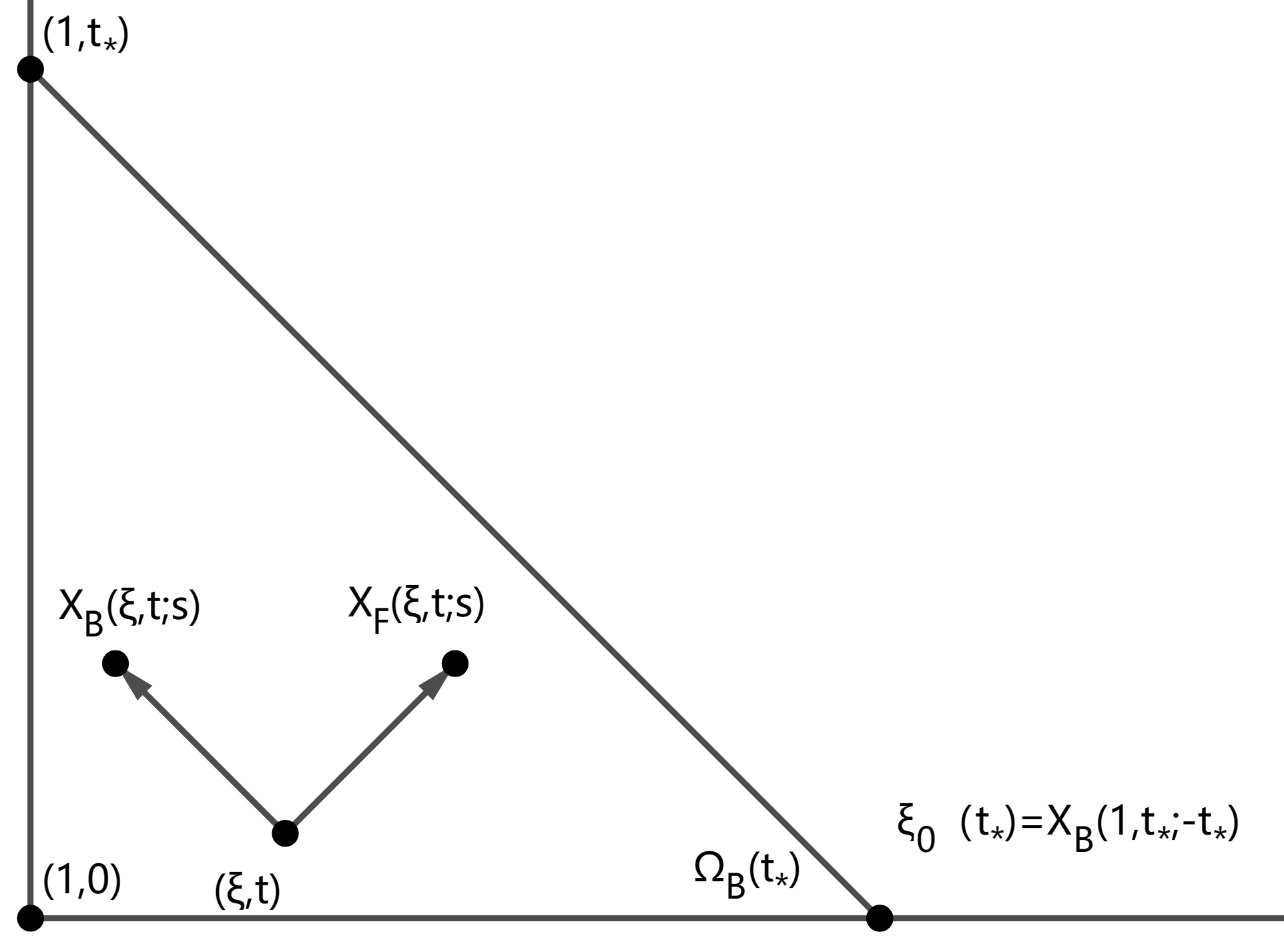}}  
\caption{The backward and forward characteristics within the domain $\Omega_B(t_*)$}\label{1} 
\end{figure}
\noindent $X_B(\xi,t;s)\in\Omega_B(t_*)$ is automatically fulfilled for each $(\xi,t)\in\Omega_B(t_*)$ due to the uniqueness of integral curves. Therefore, we abbreviate $v_B(\xi,t;t_*)$ to $v_B(\xi,t)$. Moreover, from the definition we have
$$ v_F\left(X_F(\xi,t;\tau);t_*\right)=v_F(\xi,t;t_*),\quad v_B\left(X_B(\xi,t;\tau)\right)=v_B(\xi,t).$$
It then follows from (\ref{6.3})(\ref{6.4}) that
\begin{equation}\begin{aligned}
v_F(1,t;t_*)\leq&\exp\left\{\sup_s\int_0^s\left|\frac{2u}{r}\left(X_F(1,t;\sigma)\right)\right|d\sigma\right\}\left\{\left|\frac{w_F}{c^\frac{1}{2}}(1,t)\right|\right.\\
&\left.+\sup_s\int_0^s\left|\left(\frac{\partial_tc+cr^2\xi^{-2}\partial_\xi c}{2c}\right)\left(X_F(1,t;\sigma)\right)\right|v_B\left(X_F(1,t;\sigma)\right)d\sigma\right.\\
&\left.+\sup_s\int_0^s\left|\left(2r^{-2}u^2\frac{\psi}{c^\frac{1}{2}}\right)\left(X_F(1,t;\sigma)\right)\right|d\sigma\right\},
\end{aligned}\label{6.5}\end{equation}
\begin{equation}\begin{aligned}
v_F(\xi,0;t_*)\leq&\exp\left\{\sup_s\int_0^s\left|\frac{2u}{r}\left(X_F(\xi,0;\sigma)\right)\right|d\sigma\right\}\left\{\left|\frac{w_F}{c^\frac{1}{2}}(\xi,0)\right|\right.\\
&\left.+\sup_s\int_0^s\left|\left(\frac{\partial_tc+cr^2\xi^{-2}\partial_\xi c}{2c}\right)\left(X_F(\xi,0;\sigma)\right)\right|v_B\left(X_F(\xi,0;\sigma)\right)d\sigma\right.\\
&\left.+\sup_s\int_0^s\left|\left(2r^{-2}u^2\frac{\psi}{c^\frac{1}{2}}\right)\left(X_F(\xi,0;\sigma)\right)\right|d\sigma\right\},
\end{aligned}\label{6.6}\end{equation}
\begin{equation}\begin{aligned}
v_B(\xi,0)\leq&\exp\left\{\sup_s\int_0^s\left|\frac{2u}{r}\left(X_B(\xi,0;\sigma)\right)\right|d\sigma\right\}\left\{\left|\frac{w_B}{c^\frac{1}{2}}(\xi,0)\right|\right.\\
&\left.+\sup_s\int_0^s\left|\left(\frac{\partial_tc-cr^2\xi^{-2}\partial_\xi c}{2c}\right)\left(X_B(\xi,0;\sigma)\right)\right|v_F\left(X_B(\xi,0;\sigma);t_*\right)d\sigma\right.\\
&\left.+\sup_s\int_0^s\left|\left(2r^{-2}u^2\frac{\psi}{c^\frac{1}{2}}\right)\left(X_B(\xi,0;\sigma)\right)\right|d\sigma\right\}.
\end{aligned}\label{6.7}\end{equation}
Each supremum above is taken over all $s\geq0$ such that $X_F(1,t;\sigma)$ (or $X_F(\xi,0;\sigma)$, $X_B(\xi,0;\sigma)$ respectively) remains in $\Omega_B(t_*)$ for all $\sigma\in[0,s]$.\\
The goal of this section is to obtain a decay estimate of $\psi(1,\cdot)$. To this end, we shall need an estimate of the source term $W$ in (\ref{5.6}). We begin with the estimate of the nonlinearities in (\ref{6.5}-\ref{6.7}). These estimates rely on a coordinate transform between the forward and backward characteristics, which  will be specified in the following analysis. By collecting these estimate, we will be able to measure the difference between $w_B(1,t)$ and $w_B(\xi_0(t),0)$, which further implies the estimate of $W$. Hereby we state the expected estimates of $w_B$ and $\psi$:
\begin{prop}
Introduce the notations: 
$$\mathcal{V}(\xi):=\left|w_B(\xi,0)\right|+\xi^{-1}\int_1^\xi|w_F(\eta,0)|d\eta+\epsilon\xi^{-3}|\psi(\xi,0)|+\epsilon^2\xi^{-1}\int_1^\xi\eta^{-3}|\psi(\eta,0)|d\eta,$$
$$\begin{aligned}
\Psi(t_*):=&\int_0^{t_*}\left(1+\frac{\underline{c}}{2}(t_*-t)\right)^{-1}|\partial_t\psi(1,t)|dt+\epsilon\int_0^{t_*}\left(1+\frac{\underline{c}}{2}(t_*-t)\right)^{-4}|\psi(1,t)|dt\\&+\epsilon^2\int_0^{t_*}\left(1+\frac{\underline{c}}{2}(t_*-t)\right)^{-1}|\psi(1,t)|dt,
\end{aligned}$$
then $w_B$ satisfies the following bound:
$$\begin{aligned}
&\max_{t_*\in[0,T]}\left\{\xi_0(t_*)\left(\left|\frac{w_B}{c^\frac{1}{2}}\left(1,t_*\right)-\frac{w_B}{c^\frac{1}{2}}\left(\xi_0(t_*),0\right)\right|\right)\right\}\\
\lesssim&\epsilon\max_{t_*\in[0,T]}\xi_0(t_*)\mathcal{V}(\xi_0(t_*))+\epsilon\max_{t_*\in[0,T]}\xi_0(t_*)\Psi(t_*)+\kappa\max_{t_*\in[0,T]}\left\{\xi_0(t_*)\left|\frac{w_B}{c^\frac{1}{2}}(\xi_0(t_*),0)\right|\right\}.
\end{aligned}$$
\label{prop 6.0}\end{prop}
\begin{prop}
The following estimate of $Y(t)$ holds for $t\in[0,T]$:
$$\begin{aligned}
&\left|Y(t)-Y_0(t)\right|\\
\lesssim&\frac{1}{1+\underline{c}t}\left\{\kappa\max_{s\in[0,T]}\left\{\xi_0(s)\left|\frac{w_B}{c^\frac{1}{2}}(\xi_0(s),0)\right|\right\}+\epsilon\max_{s\in[0,T]}\xi_0(s)\mathcal{V}(\xi_0(s))+\epsilon|Y(0)|\right\},
\end{aligned}$$
where
$$Y_0(t):=\exp\left[\begin{matrix}
\Lambda_1t & \\ & \Lambda_2t  
\end{matrix}\right]Y(0)-\frac{\tilde{f}^\prime(1)}{c_0}\int_0^t\exp\left[\begin{matrix}
\Lambda_1(t-s) & \\ & \Lambda_2(t-s)
\end{matrix}\right]\left[
\begin{matrix}
w_B(\xi_0(s),0) \\ w_B(\xi_0(s),0)
\end{matrix}\right]ds.$$
Since $\psi^2|_{x=0}\simeq\varphi^2|_{x=0}$, it follows
$$\begin{aligned}
\int_0^T\varphi^2|_{x=0}dt
\lesssim&\int_0^T|Y_0(t)|^2dt+\left(\kappa\max_{\eta\in[1,\xi_0(T)]}\left\{\eta\left|\frac{w_B}{c^\frac{1}{2}}(\eta,0)\right|\right\}+\epsilon\max_{\eta\in[1,\xi_0(T)]}\eta\mathcal{V}(\eta)+\epsilon|Y(0)|\right)^2.
\end{aligned}$$
\label{prop 6.01}\end{prop}
\subsection{Estimate of (\ref{6.5})}\label{subsec6.1}
Let $X_F^{(\xi)}$, $X_B^{(\xi)}$ denote the space argument of $X_F$, $X_B$ respectively. By continuity, for each $\eta\in[\xi_0(t),\xi_0(T)]$, there exists a unique $\sigma_1(\eta,t)\in[0,T-t]$ such that (figure \ref{2})
$$X_F(1,t;\sigma_1(\eta,t))=X_B(\eta,0;t+\sigma_1(\eta,t)),\quad \sigma_1(\xi_0(t),t)=0.$$
In particular, for each $t_*\in[t,T]$, $s_1(t,t_*):=\sigma_1(\xi_0(t_*),t)$ satisfies
$$X_F(1,t;s_1(t,t_*))=X_B(\xi_0(t_*),0;t+s_1(t,t_*)),$$
and thus the integrals in (\ref{6.5}) is actually taken over $\sigma\in(0,s_1(t,t_*))$. 
Using the bootstrap bounds (\ref{3.1}-\ref{3.11}), we have the following lemma for $\sigma_1$: 
\begin{lem}
For each $t\in[0,T]$, $\sigma_1(\cdot,t)$ is a strictly increasing  $C^1$ function on $(\xi_0(t),\xi_0(T))$. Moreover, for $\eta\in(\xi_0(t),\xi_0(T))$, it holds 
$$\left(2\overline{c}\right)^{-1}e^{-C\kappa}\leq\frac{d\sigma_1(\eta,t)}{d\eta}\leq\left(2\underline{c}\right)^{-1}e^{C\kappa},$$
$$1+\frac{\underline{c}e^{-C\kappa}}{\underline{c}+\overline{c}}\left(\eta-\xi_0(t_*)\right)\leq X_F^{(\xi)}(1,t;\sigma_1(\eta,t))\leq 1+\frac{\overline{c}e^{C\kappa}}{\underline{c}+\overline{c}}(\eta-\xi_0(t_*)).$$
\label{lem 6.1}\end{lem}
\begin{proof}
The $C^1$ regularity of $\sigma_1(\cdot,t)$ is given by the implicit function theorem and the $H^2$ regularity of $c$. First, we claim that $\sigma(\eta,t)$ is strictly increasing in $\eta$. If assume that $\eta_2>\eta_1$ and $\sigma_1(\eta_2,t)\leq\sigma_1(\eta_1,t)$, we find 
$$\begin{aligned}
&X_B^{(\xi)}(\eta_2,0;t+\sigma_1(\eta_2,t))=X_F^{(\xi)}(1,t;\sigma_1(\eta_2,t))\leq X_F^{(\xi)}(1,t;\sigma_1(\eta_1,t))\\=&X_B^{(\xi)}(\eta_1,0;t+\sigma_1(\eta_1,t))\leq X_B^{(\xi)}(\eta_1,0;t+\sigma_1(\eta_2,t)),
\end{aligned}$$
while $X_B^{(\xi)}(\eta_2,0;0)>X_B^{(\xi)}(\eta_1,0;0)$, which leads to contradiction with the uniqueness of integral curves. Hence  $\sigma_1(\eta_2,t)>\sigma_1(\eta_1,t)$ and for the same reason $X_B^{(\xi)}(\eta_2,0;\tau)>X_B^{(\xi)}(\eta_1,0;\tau)$ for $\tau$ in the common interval of $X_B(\eta_1,0;\cdot)$ and $X_B(\eta_2,0;\cdot)$ if $\eta_2>\eta_1$.
By mean value formula and (\ref{lip}) we have for $\eta_2,\eta_1\in[\xi_0(t),\xi_0(T)]$ with $\eta_2>\eta_1$  that
$$\begin{aligned}
\left|\frac{d}{d\tau}\left(X_B^{(\xi)}(\eta_2,0;\tau)-X_B^{(\xi)}(\eta_1,0;\tau)\right)\right|
=&\left|\left(cr^2\xi^{-2}\right)(X_B(\eta_1,0;\tau))-\left(cr^2\xi^{-2}\right)(X_B(\eta_2,0;\tau))\right|\\
\leq&C\epsilon X_B^{(\xi)}(\eta_1,0;\tau)^{-1}\left(X_B^{(\xi)}(\eta_2,0;\tau)-X_B^{(\xi)}(\eta_1,0;\tau)\right),
\end{aligned}$$
and thus by Gronwall's inequality
$$\begin{aligned}
&(\eta_2-\eta_1)\exp\left\{-\int_0^{t+\sigma_1(\eta_1,t)}C\epsilon X_B^{(\xi)}(\eta_1,0;\tau)^{-1}d\tau\right\}\\
\leq& X_B^{(\xi)}(\eta_2,0;t+\sigma_1(\eta_1,t))-X_B^{(\xi)}(\eta_1,0;t+\sigma_1(\eta_1,t))\\
\leq&(\eta_2-\eta_1)\exp\left\{\int_0^{t+\sigma_1(\eta_1,t)}C\epsilon X_B^{(\xi)}(\eta_1,0;\tau)^{-1}d\tau\right\}.
\end{aligned}$$
Using the rough bound (\ref{3.10}) $cr^2\xi^{-2}\geq\underline{c}$, we further obtain 
$$\begin{aligned}
&\int_0^{t+\sigma_1(\eta_1,t)} X_B^{(\xi)}(\eta_1,0;\tau)^{-1}d\tau\\
\leq&\int_0^{t+\sigma_1(\eta_1,t)}\left(X_B^{(\xi)}(\eta_1,0;t+\sigma_1(\eta_1,t))+\underline{c}(t+\sigma_1(\eta_1,t)-\tau)\right)^{-1}d\tau\\
\leq&\frac{1}{\underline{c}}\log\left(1+\underline{c}(t+\sigma_1(\eta_1,t))\right).
\end{aligned}$$
Hence by the assumption (\ref{3.11}) and that $\sigma_1(\eta_1,t)+t\leq T$, it follows that
$$\int_0^{t+\sigma_1(\eta_1,t)}C\epsilon X_B^{(\xi)}(\eta_1,0;\tau)^{-1}d\tau\leq C\kappa,$$
therefore
\begin{equation}
e^{-C\kappa}(\eta_2-\eta_1)\leq X_B^{(\xi)}(\eta_2,0;t+\sigma_1(\eta_1,t))-X_B^{(\xi)}(\eta_1,0;t+\sigma_1(\eta_1,t))\leq e^{C\kappa}(\eta_2-\eta_1).
\label{stab1}\end{equation}
\noindent By the rough bound (\ref{3.10}) again, it holds (figure \ref{2})
$$\begin{aligned}
&X_F^{(\xi)}(1,t;\sigma_1(\eta_1,t))+\underline{c}\left(\sigma_1(\eta_2,t)-\sigma_1(\eta_1,t)\right)\leq X_F^{(\xi)}(1,t;\sigma_1(\eta_2,t))\\
=&X_B^{(\xi)}(\eta_2,0;t+\sigma_1(\eta_2,t))\leq X_B^{(\xi)}(\eta_2,0;t+\sigma_1(\eta_1,t))-\underline{c}\left(\sigma_1(\eta_2,t)-\sigma_1(\eta_1,t)\right),
\end{aligned}$$
$$\begin{aligned}
&X_F^{(\xi)}(1,t;\sigma_1(\eta_1,t))+\overline{c}\left(\sigma_1(\eta_2,t)-\sigma_1(\eta_1,t)\right)\geq X_F^{(\xi)}(1,t;\sigma_1(\eta_2,t))\\
=&X_B^{(\xi)}(\eta_2,0;t+\sigma_1(\eta_2,t))\geq X_B^{(\xi)}(\eta_2,0;t+\sigma_1(\eta_1,t))-\overline{c}\left(\sigma_1(\eta_2,t)-\sigma_1(\eta_1,t)\right).
\end{aligned}$$
Combining the above inequalities yields
$$\begin{aligned}
2\underline{c}\left(\sigma_1(\eta_2,t)-\sigma_1(\eta_1,t)\right)
\leq& X_B^{(\xi)}(\eta_2,0;t+\sigma_1(\eta_1,t))-X_F^{(\xi)}(1,t;\sigma_1(\eta_1,t))\\
=&X_B^{(\xi)}(\eta_2,0;t+\sigma_1(\eta_1,t))-X_B^{(\xi)}(\eta_1,0;t+\sigma_1(\eta_1,t)))\\
\leq&e^{C\kappa}(\eta_2-\eta_1),
\end{aligned}$$
$$\begin{aligned}
2\overline{c}\left(\sigma_1(\eta_2,t)-\sigma_1(\eta_1,t)\right)
\geq& X_B^{(\xi)}(\eta_2,0;t+\sigma_1(\eta_1,t))-X_F^{(\xi)}(1,t;\sigma_1(\eta_1,t))\\
=&X_B^{(\xi)}(\eta_2,0;t+\sigma_1(\eta_1,t))-X_B^{(\xi)}(\eta_1,0;t+\sigma_1(\eta_1,t)))\\
\geq&e^{-C\kappa}(\eta_2-\eta_1),
\end{aligned}$$
which gives the estimate of $\frac{d\sigma_1(\eta,t)}{d\eta}$ since $\eta_2$, $\eta_1$ can be arbitrary.
Note that
$$\begin{aligned}
&\max\left\{1+\underline{c}\sigma_1(\eta,t),X_B^{(\xi)}(\eta,0;t)-\overline{c}\sigma_1(\eta,t)\right\}\\
\leq&X_F^{(\xi)}(1,t;\sigma_1(\eta,t))=X_B^{(\xi)}(\eta,0;t+\sigma_1(\eta,t))\\
\leq&\min\left\{1+\overline{c}\sigma_1(\eta,t),X_B^{(\xi)}(\eta,0;t)-\underline{c}\sigma_1(\eta,t)\right\},
\end{aligned}$$
then taking $\eta_2=\eta$, $\eta_1=\xi_0(t)$ in (\ref{stab1}) yields
$$e^{-C\kappa}(\eta-\xi_0(t))\leq X_B^{(\xi)}(\eta,0;t)-1\leq e^{C\kappa}(\eta-\xi_0(t)).$$
\begin{figure}[H]
\center{\includegraphics[width=14cm]  {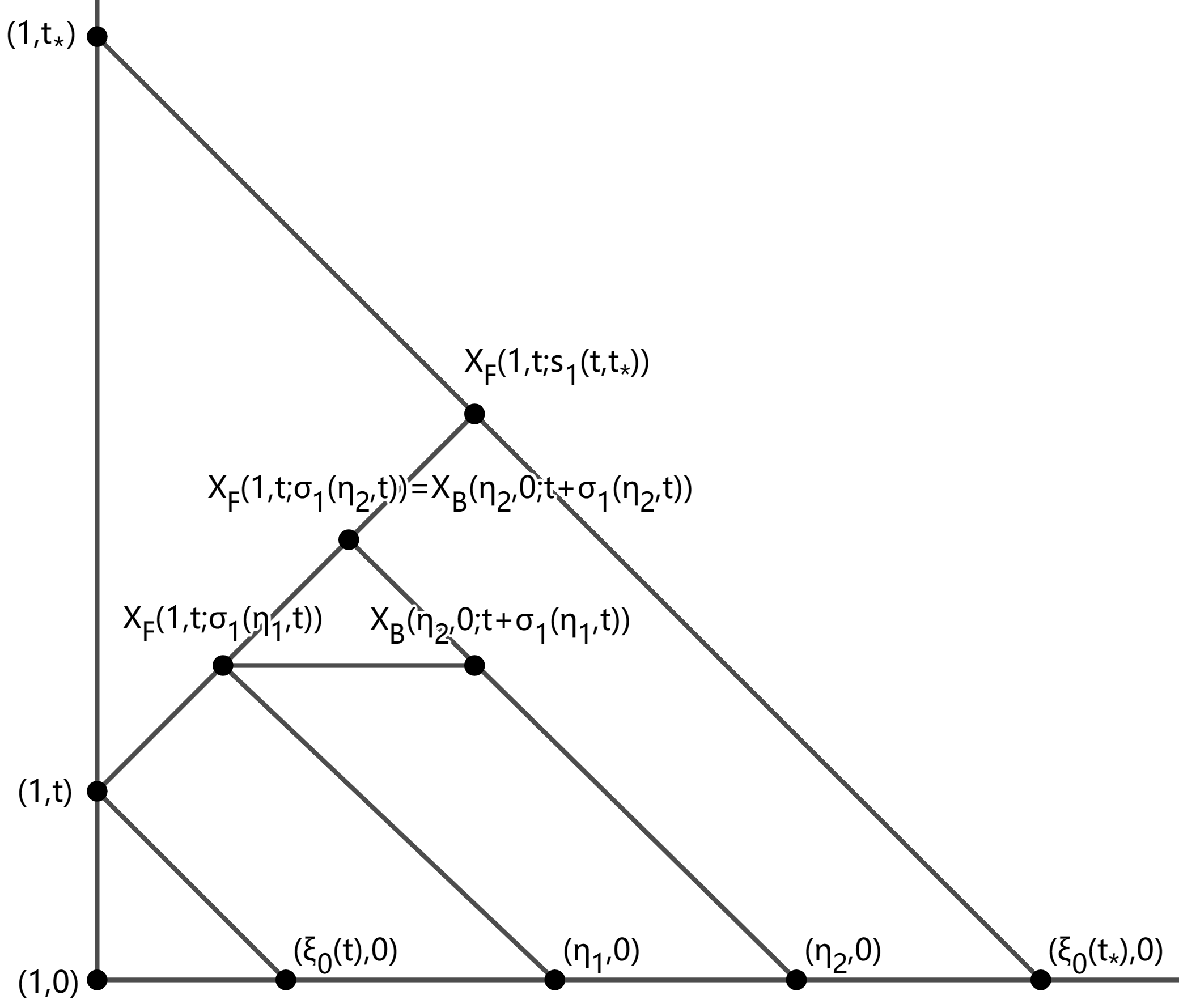}}   
\caption{Positions of the points $X_F(1,t;\sigma_1(\eta_1,t))$, $X_F(1,t;\sigma_1(\eta_2,t))$ and $X_B(\eta_2,0;t+\sigma_1(\eta_1,t))$}\label{2}  
\end{figure}
\noindent It follows
$$X_F^{(\xi)}(1,t;\sigma_1(\eta,t))\geq\left(\overline{c}+\underline{c}\right)^{-1}\left(\overline{c}+\underline{c}X_B^{(\xi)}(\eta,0;t)\right)\geq 1+\frac{\underline{c}e^{-C\kappa}}{\underline{c}+\overline{c}}(\eta-\xi_0(t_*)),$$
and
$$X_F^{(\xi)}(1,t;\sigma_1(\eta,t))\leq\left(\overline{c}+\underline{c}\right)^{-1}\left(\underline{c}+\overline{c}X_B^{(\xi)}(\eta,0;t)\right)\leq 1+\frac{\overline{c}e^{C\kappa}}{\underline{c}+\overline{c}}(\eta-\xi_0(t_*)).$$
\end{proof}
Using the bound (\ref{3.7}) and Lemma \ref{lem 6.1}, we estimate the second line in (\ref{6.5}):
\begin{equation}\begin{aligned}
&\sup_s\int_0^s\left|\left(\frac{\partial_tc+cr^2\xi^{-2}\partial_\xi c}{2c}\right)\left(X_F(1,t;\sigma)\right)\right|v_B\left(X_F(1,t;\sigma)\right)d\sigma\\
\leq&C\epsilon\int_0^{s_1(t,t_*)}X_F^{(\xi)}(1,t;\sigma)^{-1}v_B\left(X_F(1,t;\sigma)\right)d\sigma\\
\leq&\frac{Ce^{C\kappa}\epsilon}{\underline{c}}\int_{\xi_0(t)}^{\xi_0(t_*)}X_F^{(\xi)}(1,t;\sigma_1(\eta,t))^{-1}v_B\left(X_B(\eta,0;t+\sigma_1(\eta,t)\right)d\eta\\
\leq&\frac{Ce^{C\kappa}\epsilon}{\underline{c}}\int_{\xi_0(t)}^{\xi_0(t_*)}\left(1+\underline{c}\left(\overline{c}+\underline{c}\right)^{-1}e^{-C\kappa}\left(\eta-\xi_0(t)\right)\right)^{-1}v_B(\eta,0)d\eta.
\end{aligned}\label{6.8}\end{equation}
Recall that $w_B=\left(\partial_t+cr^2\partial_x\right)\psi$. We then estimate the third line in (\ref{6.5}) using (\ref{3.6}):
$$\begin{aligned}
&\sup_s\int_0^s\left|\left(2r^{-2}u^2\frac{\psi}{c^\frac{1}{2}}\right)\left(X_F(1,t;\sigma)\right)\right|d\sigma\\
\leq&\frac{C\epsilon^2}{\underline{c}^\frac{1}{2}}\int_0^{s_1(t,t_*)}X_F^{(\xi)}(1,t;\sigma)^{-4}\left|\psi\left(X_F(1,t;\sigma)\right)\right|d\sigma\\
\leq&\frac{C\epsilon^2}{\underline{c}^\frac{1}{2}}\int_0^{s_1(t,t_*)}X_F^{(\xi)}(1,t;\sigma)^{-4}\left|\psi(1,t)+\int_0^\sigma w_B\left(X_F(1,t;\tau)\right)d\tau\right|d\sigma\\
\leq&\frac{C\epsilon^2}{\underline{c}^\frac{1}{2}}\int_0^{s_1(t,t_*)}X_F^{(\xi)}(1,t;\sigma)^{-4}d\sigma\left|\psi(1,t)\right|\\
&+\frac{C\epsilon^2\overline{c}^\frac{1}{2}}{\underline{c}^\frac{1}{2}}\int_0^{s_1(t,t_*)}\int_0^\sigma X_F^{(\xi)}(1,t;\sigma)^{-4}v_B\left(X_F(1,t;\tau)\right)d\tau d\sigma,
\end{aligned}$$
and exchange the order of integration:
$$\begin{aligned}
&\int_0^{s_1(t,t_*)}\int_0^\sigma X_F^{(\xi)}(1,t;\sigma)^{-4}v_B\left(X_F(1,t;\tau)\right)d\tau d\sigma\\
=&\int_0^{s_1(t,t_*)}\left(\int_\tau^{s_1(t,t_*)}X_F^{(\xi)}(1,t;\sigma)^{-4}d\sigma\right)v_B\left(x_F(1,t;\tau)\right)d\tau.
\end{aligned}$$
Use the rough bound (\ref{3.10}) to get
$$\begin{aligned}
\int_\tau^{s_1(t,t_*)}X_F^{(\xi)}(1,t;\sigma)^{-4}d\sigma\leq&\int_\tau^{s_1(t,t_*)}\left(X_F^{(\xi)}(1,t;\tau)+\underline{c}(\sigma-\tau)\right)^{-4}d\sigma
\leq\frac{1}{3\underline{c}}X_F^{(\xi)}(1,t;\tau)^{-3}.
\end{aligned}$$
We can then proceed with Lemma \ref{lem 6.1}:
\begin{equation}\begin{aligned}
&\sup_s\int_0^s\left|\left(2r^{-2}u^2\frac{\psi}{c^\frac{1}{2}}\right)\left(X_F(1,t;\sigma)\right)\right|d\sigma\\
\leq&\frac{C\epsilon^2}{\underline{c}^\frac{3}{2}}|\psi(1,t)|+\frac{C\epsilon^2\overline{c}^\frac{1}{2}}{\underline{c}^\frac{3}{2}}\int_0^{s_1(t,t_*)}X_F^{(\xi)}(1,t;\tau)^{-3}v_B\left(X_F(1,t;\tau)\right)d\tau\\
\leq&\frac{C\epsilon^2}{\underline{c}^\frac{3}{2}}|\psi(1,t)|+\frac{C\epsilon^2\overline{c}^\frac{1}{2}}{\underline{c}^\frac{3}{2}}\frac{2\underline{c}}{e^{C\kappa}}\int_{\xi_0(t)}^{\xi_0(t_*)}X_F^{(\xi)}(1,t;\sigma_1(\eta,t))^{-3}v_B\left(X_B(\eta,0;\sigma_1(\eta,t))\right)d\eta\\
\leq&\frac{C\epsilon^2}{\underline{c}^\frac{3}{2}}|\psi(1,t)|+\frac{C\epsilon^2\overline{c}^\frac{1}{2}}{\underline{c}^\frac{5}{2}}e^{C\kappa}\int_{\xi_0(t)}^{\xi_0(t_*)}\left(1+\frac{\underline{c}}{\overline{c}+\underline{c}}e^{-C\kappa}(\eta-\xi_0(t))\right)^{-3}v_B(\eta,0)d\eta.
\end{aligned}\label{6.9}\end{equation}
\subsection{Estimate of (\ref{6.6})}
By continuity, for each $\eta\in[\xi,\xi_0(T)]$, there exists a unique $\sigma_2(\eta,\xi)\in[0,T-t]$ such that (figure \ref{3})
$$X_F(\xi,0;\sigma_2(\eta,\xi))=X_B(\eta,0;\sigma_2(\eta,\xi)),\quad \sigma_2(\xi,\xi)=0.$$
In particular, for each $t_*\leq T$ such that $\xi_0(t_*)\geq\xi$, $s_2(\xi,t_*):=\sigma_2(\xi_0(t_*),\xi)$ satisfies
$$X_F(\xi,0;s_2(\xi,t_*))=X_B(\xi_0(t_*),0;s_2(\xi,t_*)),$$
so the integrals in (\ref{6.6}) are taken over $\sigma\in(0,s_2(\xi,t_*))$.
\begin{lem}
For each $\xi\in[1,\xi_0(T)]$, $\sigma_2(\cdot,\xi)$ is a strictly increasing $C^1$ function on $(\xi,\xi_0(T))$. Moreover, for $\eta\in(\xi,\xi_0(T))$, it holds
$$\left(2\overline{c}\right)^{-1}e^{-C\kappa}\leq\frac{d\sigma_2(\eta,\xi)}{d\eta}\leq\left(2\underline{c}\right)^{-1}e^{C\kappa},$$
$$\xi+\frac{\underline{c}}{\overline{c}+\underline{c}}(\eta-\xi)\leq X_F^{(\xi)}(\xi,0;\sigma_2(\eta,\xi))\leq\xi+\frac{\overline{c}}{\overline{c}+\underline{c}}(\eta-\xi).$$
\label{lem 6.2}\end{lem}
\begin{proof}
The implicit function theorem and the $H^2$ regularity of $c$ give the $C^1$ regularity of $\sigma_2(\cdot,\xi)$. Assume otherwise $\sigma_2(\eta_2,\xi)\leq\sigma_2(\eta_1,\xi)$ for some $\eta_2>\eta_1$. Then
$$\begin{aligned}
&X_B^{(\xi)}(\eta_2,0;\sigma_2(\eta_2,\xi))=X_F^{(\xi)}(\xi,0;\sigma_2(\eta_2,\xi))\leq X_F^{(\xi)}(\xi,0;\sigma_2(\eta_1,\xi))\\
=&X_B^{(\xi)}(\eta_1,0;\sigma_2(\eta_1,\xi))\leq X_B^{(\xi)}(\eta_1,0;\sigma_2(\eta_2,\xi)),
\end{aligned}$$
while $X_B^{(\xi)}(\eta_2,0;0)>X_B^{(\xi)}(\eta_1,0;0)$, which is contradictory with the uniqueness of integral curves. Hence $\sigma_2(\eta_2,t)>\sigma_2(\eta_1,t)$ and for the same reason $X_B^{(\xi)}(\eta_2,0;\tau)>X_B^{(\xi)}(\eta_1,0;\tau)$ for $\tau$ in the common interval of $X_B(\eta_1,t;\cdot)$ and $X_B(\eta_2,t;\cdot)$. By the same argument as in Lemma \ref{lem 6.1}, for $\eta_2,\eta_1\in[\xi,\xi_0(T)]$ with $\eta_2>\eta_1$ it holds 
$$e^{-C\kappa}(\eta_2-\eta_1)\leq X_B^{(\xi)}(\eta_2,0;\sigma_2(\eta_1,t))-X_B^{(\xi)}(\eta_1,0;\sigma_2(\eta_1,t))\leq e^{C\kappa}(\eta_2-\eta_1).$$
 Use the rough bound (\ref{3.10}) (figure \ref{3}):
$$\begin{aligned}
&X_F^{(\xi)}(\xi,0;\sigma_2(\eta_1,\xi))+\underline{c}\left(\sigma_2(\eta_2,\xi)-\sigma_2(\eta_1,\xi)\right)\leq X_F^{(\xi)}(\xi,0;;\sigma_2(\eta_2,\xi))\\
=&X_B^{(\xi)}(\eta_2,0;\sigma_2(\eta_2,\xi))\leq X_B^{(\xi)}(\eta_2,0;\sigma_2(\eta_1,\xi))-\underline{c}\left(\sigma_2(\eta_2,\xi)-\sigma_2(\eta_1,\xi)\right),
\end{aligned}$$
$$\begin{aligned}
&X_F^{(\xi)}(\xi,0;\sigma_2(\eta_1,\xi))+\overline{c}\left(\sigma_2(\eta_2,\xi)-\sigma_2(\eta_1,\xi)\right)\geq X_F^{(\xi)}(\xi,0;;\sigma_2(\eta_2,\xi))\\
=&X_B^{(\xi)}(\eta_2,0;\sigma_2(\eta_2,\xi))\geq X_B^{(\xi)}(\eta_2,0;\sigma_2(\eta_1,\xi))-\overline{c}\left(\sigma_2(\eta_2,\xi)-\sigma_2(\eta_1,\xi)\right).
\end{aligned}$$
\noindent
Combining the above inequalities yields
$$\left(2\overline{c}\right)^{-1}e^{-C\kappa}(\eta_2-\eta_1)\leq\sigma_2(\eta_2,\xi)-\sigma_2(\eta_1,\xi)\leq\left(2\underline{c}\right)^{-1}e^{C\kappa}(\eta_2-\eta_1).$$
\begin{figure}[H]
\center{\includegraphics[width=14cm]  {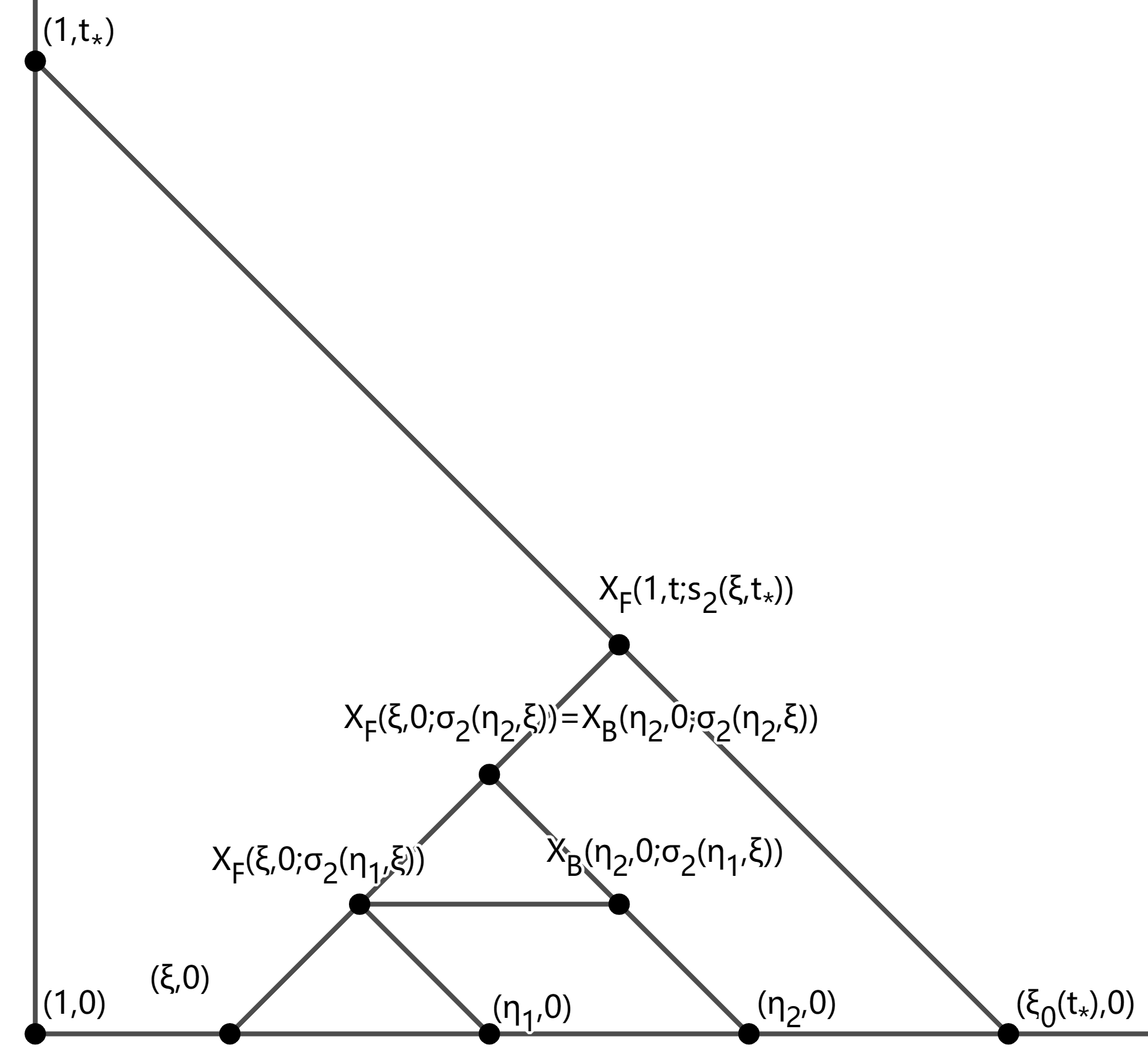}}   
\caption{Positions of the points $X_F(\xi,0;\sigma_2(\eta_1,\xi))$, $X_F(\xi,0;\sigma_2(\eta_2,\xi))$ and $X_B(\eta_2,0;\sigma_2(\eta_1,\xi))$}\label{3}   
\end{figure}
\noindent Since $\eta_2$, $\eta_1$ can be chosen arbitrarily, we obtain
$$\left(2\overline{c}\right)^{-1}e^{-C\kappa}\leq\frac{d\sigma_2(\eta,\xi)}{d\eta}\leq\left(2\underline{c}\right)^{-1}e^{C\kappa}.$$
The estimate of $X_F^{(\xi)}(\xi,0;\sigma_2(\eta,\xi))$ follows from the  following inequalities:
$$\begin{aligned}
&\max\left\{\xi+\underline{c}\sigma_2(\eta,\xi),\eta-\overline{c}\sigma_2(\eta,\xi)\right\}\\
\leq&X_F^{(\xi)}(\xi,0;\sigma_2(\eta,\xi))=X_B^{(\xi)}(\eta,0;\sigma_2(\eta,\xi))\\
\leq&\min\left\{\xi+\overline{c}\sigma_2(\eta,\xi),\eta-\underline{c}\sigma_2(\eta,\xi)\right\},
\end{aligned}$$
$$\frac{\overline{c}}{\overline{c}+\underline{c}}\left(\xi+\underline{c}\sigma_2(\eta,\xi)\right)+\frac{\underline{c}}{\overline{c}+\underline{c}}\left(\eta-\overline{c}\sigma_2(\eta,\xi)\right)=\xi+\frac{\underline{c}}{\overline{c}+\underline{c}}(\eta-\xi),$$
$$\frac{\underline{c}}{\overline{c}+\underline{c}}\left(\xi+\overline{c}\sigma_2(\eta,\xi)\right)+\frac{\overline{c}}{\overline{c}+\underline{c}}\left(\eta-\underline{c}\sigma_2(\eta,\xi)\right)=\xi+\frac{\overline{c}}{\overline{c}+\underline{c}}(\eta-\xi).$$
\end{proof}
By the bound (\ref{3.7}) and Lemma \ref{lem 6.2}, the second line in (\ref{6.6}) can be controlled by 
\begin{equation}\begin{aligned}
&\sup_s\int_0^s\left|\left(\frac{\partial_tc+cr^2\xi^{-2}\partial_\xi c}{2c}\right)\left(X_F(\xi,0;\sigma)\right)\right|v_B\left(X_F(\xi,0;\sigma)\right)d\sigma\\
\leq&C\epsilon\int_0^{s_2(\xi,t_*)}X_F^{(\xi)}(\xi,0;\sigma)^{-1}v_B\left(X_F(\xi,0;\sigma)\right)d\sigma\\
\leq&\frac{Ce^{C\kappa}\epsilon}{\underline{c}}\int_\xi^{\xi_0(T)}X_F^{(\xi)}(\xi,0;\sigma_2(\eta,\xi))^{-1}v_B\left(X_B(\eta,0;\sigma_2(\eta,\xi))\right)d\eta\\
\leq&\frac{Ce^{C\kappa}\epsilon}{\underline{c}}\int_\xi^{\xi_0(T)}\left(\xi+\frac{\underline{c}}{\overline{c}+\underline{c}}(\eta-\xi)\right)^{-1}v_B(\eta,0)d\eta.
\end{aligned}\label{6.10}\end{equation}
Use the bound (\ref{3.6}) and that $\left(\partial_t+cr^2\partial_x\right)\psi=w_B$ to estimate the third line in (\ref{6.6}):
$$\begin{aligned}
&\sup_s\int_0^s\left|\left(2r^{-2}u^2\frac{\psi}{c^\frac{1}{2}}\right)\left(X_F(\xi,0;\sigma)\right)\right|d\sigma\\
\leq&\frac{C\epsilon^2}{\underline{c}^\frac{1}{2}}\int_0^{s_2(\xi,t_*)}X_F^{(\xi)}(\xi,0;\sigma)^{-4}\left|\psi\left(X_F(\xi,0;\sigma)\right)\right|d\sigma\\
\leq&\frac{C\epsilon^2}{\underline{c}^\frac{1}{2}}\int_0^{s_2(\xi,t_*)}X_F^{(\xi)}(\xi,0;\sigma)^{-4}\left|\psi(\xi,0)+\int_0^\sigma w_B\left(X_F(\xi,0;\tau)\right)d\tau\right|d\sigma\\
\leq&\frac{C\epsilon^2}{\underline{c}^\frac{1}{2}}\int_0^{s_2(\xi,t_*)}X_F^{(\xi)}(\xi,0;\sigma)^{-4}d\sigma |\psi(\xi,0)|\\
&+\frac{C\epsilon^2\overline{c}^\frac{1}{2}}{\underline{c}^\frac{1}{2}}\int_0^{s_2(\xi,t_*)}\int_0^\sigma X_F^{(\xi)}(\xi,0;\sigma)^{-4}v_B\left(X_F(\xi,0;\tau)\right)d\tau d\sigma\\
=&\frac{C\epsilon^2}{\underline{c}^\frac{1}{2}}\int_0^{s_2(\xi,t_*)}X_F^{(\xi)}(\xi,0;\sigma)^{-4}d\sigma|\psi(\xi,0)|\\
&+\frac{C\epsilon^2\overline{c}^\frac{1}{2}}{\underline{c}^\frac{1}{2}}\int_0^{s_2(\xi,t_*)}\left(\int_\tau^{s_2(\xi,t_*)}X_F^{(\xi)}(\xi,0;\sigma)^{-4}d\sigma\right)v_B\left(X_F(\xi,0;\tau)\right)d\tau.
\end{aligned}$$
Note that the rough bound (\ref{3.10}) yields 
$$\begin{aligned}
\int_\tau^{s_2(\xi,t_*)}X_F^{(\xi)}(\xi,0;\sigma)^{-4}d\sigma\leq&\int_\tau^{s_2(\xi,t_*)}\left(X_F^{(\xi)}(\xi,0;\tau)+\underline{c}(\sigma-\tau)\right)^{-4}d\sigma\\
\leq&\frac{1}{3\underline{c}}X_F^{(\xi)}(\xi,0;\tau)^{-3}.
\end{aligned}$$
Then we proceed by Lemma \ref{lem 6.2} and obtain
\begin{equation}\begin{aligned}
&\sup_s\int_0^s\left|\left(2r^{-2}u^2\frac{\psi}{c^\frac{1}{2}}\right)\left(X_F(\xi,0;\sigma)\right)\right|d\sigma\\
\leq&\frac{C\epsilon^2}{\underline{c}^\frac{3}{2}}\xi^{-3}|\psi(\xi,0)|+\frac{C\epsilon^2\overline{c}}{\underline{c}^\frac{3}{2}}\int_0^{s_2(\xi,t_*)}X_F^{(\xi)}(\xi,0;\tau)^{-3}v_B\left(X_F(\xi,0;\tau)\right)d\tau\\
\leq&\frac{C\epsilon^2}{\underline{c}^\frac{3}{2}}\xi^{-3}|\psi(\xi,0)|+\frac{C\epsilon^2\overline{c}^\frac{1}{2}}{\underline{c}^\frac{3}{2}}\frac{2\underline{c}}{e^{C\kappa}}\int_\xi^{\xi_0(t_*)}X_F^{(\xi)}(\eta,0;\sigma_2(\eta,\xi))^{-3}v_B\left(X_B(\eta,0;\sigma_2(\eta,\xi))\right)d\eta\\
\leq&\frac{C\epsilon^2}{\underline{c}^\frac{3}{2}}\xi^{-3}|\psi(\xi,0)|+\frac{C\epsilon^2\overline{c}^\frac{1}{2}}{\underline{c}^\frac{5}{2}}e^{C\kappa}\int_\xi^{\xi_0(t_*)}\left(\xi+\underline{c}\left(\overline{c}+\underline{c}\right)^{-1}(\eta-\xi)\right)^{-3}v_B(\eta,0)d\eta.
\end{aligned}\label{6.11}\end{equation}
\subsection{Estimate of (\ref{6.7})}
In fact, we will take $\xi=\xi_0(t_*)$ in (\ref{6.7}) for the later analysis, and thus it is convenient to establish the estimate with $\xi$ replaced by $\xi_0(t_*)$ in (\ref{6.7}) with $t_*\in[0,T]$, in which case the integrals in (\ref{6.7}) is in fact taken over $\sigma\in(0,t_*)$ by the construction of $\xi_0(\cdot)$. For each $\eta\in[1,\xi_0(t_*)]$, by continuity, there exists a unique $\sigma_3(\eta,t_*)\in[0,t_*]$ such that (figure \ref{4})
$$X_F(\eta,0;\sigma_3(\eta,t_*))=X_B(\xi_0(t_*),0;\sigma_3(\eta,t_*)),\quad \sigma_3(\xi_0(t_*),t_*)=0.$$
In particular, $s_3(t_*):=\sigma_3(1,t_*)$ satisfies
$$X_F(1,0;s_3(t_*))=X_B(\xi_0(t_*),0;s_3(t_*)).$$
Analogously, for each $t\in[0,t_*]$ there exists a unique $\sigma_4(t,t_*)\in[t,t_*]$ such that (figure \ref{5})
 $$X_F(1,t;\sigma_4(t,t_*)-t)=X_B(\xi_0(t_*),0;\sigma_4(t,t_*)),\quad \sigma_4(t_*,t_*)=t_*,\quad \sigma_4(0,t_*)=s_3(t_*).$$
Recall that $X_F(1,t;\sigma_1(\eta,t))=X_B(\eta,0;t+\sigma_1(\eta,t))$ for $t\in[0,t]$ and $\eta\in[\xi_0(t),\xi_0(t_*)]$. Hence by taking $\eta=\xi_0(t_*)$, we have in particular $$\sigma_1(\xi_0(t_*),t)=\sigma_4(t,t_*)-t.$$
\begin{lem}
For each $t_*\in[0,T]$, $\sigma_3(\cdot,t_*)$ is a strictly decreasing $C^1$ function on $(1,\xi_0(t_*))$, and $\sigma_4(\cdot,t_*)$ is a strictly decreasing $C^1$ function on $(0,t_*)$. Moreover, the following estimates hold:
$$\left(2\overline{c}\right)^{-1}e^{-C\kappa}\leq-\frac{d\sigma_3(\eta,t_*)}{d\eta}\leq\left(2\underline{c}\right)^{-1}e^{C\kappa},\quad \eta\in(1,\xi_0(t_*)),$$
$$\left(2\overline{c}\right)^{-1}\underline{c}e^{-C\kappa}\leq\frac{d\sigma_4(t,t_*)}{dt}\leq(2\underline{c})^{-1}\overline{c}e^{C\kappa},\quad t\in(0,t_*),$$
$$X_B^{(\xi)}(\xi_0(t_*),0;\sigma_3(\eta,t_*))\geq\eta+\frac{\underline{c}}{\overline{c}+\underline{c}}(\xi_0(t_*)-\eta)\geq1+\frac{\underline{c}}{\overline{c}+\underline{c}}(\xi_0(t_*)-1),$$
$$X_B^{(\xi)}(\xi_0(t_*),0;\sigma_3(\eta,t_*))\leq\eta+\frac{\underline{c}}{\overline{c}+\underline{c}}(\xi_0(t_*)-\eta),$$
$$1+\frac{\underline{c}}{2}(t_*-t)\leq X_B^{(\xi)}(\xi_0(t_*),0;\sigma_4(t,t_*))\leq 1+\frac{\overline{c}}{2}(t_*-t).$$
\label{lem 6.3}\end{lem}
\begin{proof}
The $C^1$ regularity of $\sigma_3(\cdot,t_*)$ and $\sigma_4(\cdot,t_*)$ follows from the implicit function theorem and the $H^2$ regularity of $c$. Let $\eta_1,\eta_2\in[1,\xi_0(t_*)]$ with $\eta_2>\eta_1$ and $t_1,t_2\in[0,t_*]$ with $t_2>t_1$. To show the monotonicity, assume otherwise$\sigma_3(\eta_2,t_*)\geq\sigma_3(\eta_1,t_*)$. Then 
$$\begin{aligned}
&X_F^{(\xi)}(\eta_1,0;\sigma_3(\eta_1,t_*))=X_B^{(\xi)}(\xi_0(t_*),0;\sigma_3(\eta_1,t_*))\geq X_B^{(\xi)}(\xi_0(t_*),0;\sigma_3(\eta_2,t_*))\\
=&X_F^{(\xi)}(\eta_2,0;\sigma_3(\eta_2,t_*))\geq X_F^{(\xi)}(\eta_2,0;\sigma_3(\eta_1,t_*)),
\end{aligned}$$ 
while $X_F(\eta_1,0;0)<X_F(\eta_2,0;0)$, which is impossible by the uniqueness of integral curves.
Similarly, if we assume $\sigma_4(t_1,t_*)\geq\sigma_4(t_2,t_*)$, then the inequalities
$$\begin{aligned}
&X_F^{(\xi)}(1,t_2;\sigma_4(t_2,t_*)-t_2)=X_B^{(\xi)}(\xi_0(t_*),0;\sigma_4(t_2,t_*))\geq X_B^{(\xi)}(\xi_0(t_*),0;\sigma_4(t_1,t_*))\\
=&X_F^{(\xi)}(1,t_1;\sigma_4(t_1,t_*)-t_1)\geq X_F^{(\xi)}(1,t_1;\sigma_4(t_2,t_*)-t_1)\\
=&X_F^{(\xi)}(X_F(1,t_1;t_2-t_1);\sigma_4(t_2,t_*)-t_2)),
\end{aligned}$$
$$1=X_F^{(\xi)}(1,t_2;0)<X_F^{(\xi)}(1,t_1;t_2-t_1)=X_F^{(\xi)}(X_F(1,t_1;t_2-t_1);0)$$
lead to contradiction.
By recalling (\ref{3.6})(\ref{3.7}) we have
$$\begin{aligned}
\left|\frac{d}{d\tau}\left(X_F^{(\xi)}(\eta_2,0;\tau)-X_F^{(\xi)}(\eta_1,0;\tau)\right)\right|
=&\left|\left(cr^2\xi^{-2}\right)\left(X_F(\eta_2,0;\tau)\right)-\left(cr^2\xi^{-2}\right)\left(X_F(\eta_1,0;\tau)\right)\right|\\
\leq& C\epsilon X_F^{(\xi)}(\eta_1,0;\tau)^{-1}\left(X_F^{(\xi)}(\eta_2,0;\tau)-X_F^{(\xi)}(\eta_1,0;\tau)\right).
\end{aligned}$$
By integrating the above inequality and noting that (\ref{3.11}) gives
$$\begin{aligned}
C\epsilon\int_0^{\sigma_3(\eta_2,t_*)}X_F^{(\xi)}(\eta_1,0;\tau)^{-1}d\tau\leq&C\epsilon\int_0^{\sigma_3(\eta_2,t_*)}\left(X_F^{(\xi)}(\eta_1,0;0)+\underline{c}\tau\right)^{-1}d\tau\\
\leq&\frac{C\epsilon}{\underline{c}}\log\left(1+\eta_1^{-1}\underline{c}\sigma_3(\eta_2,t_*)\right)\\
\leq&C\kappa,
\end{aligned}$$
we obtain
$$e^{-C\kappa}(\eta_2-\eta_1)\leq X_F^{(\xi)}(\eta_2,0;\sigma_3(\eta_2,t_*))-X_F^{(\xi)}(\eta_1,0;\sigma_3(\eta_1,t_*))\leq e^{C\kappa}(\eta_2-\eta_1).$$
\begin{figure}[H]
\center{\includegraphics[width=14cm]  {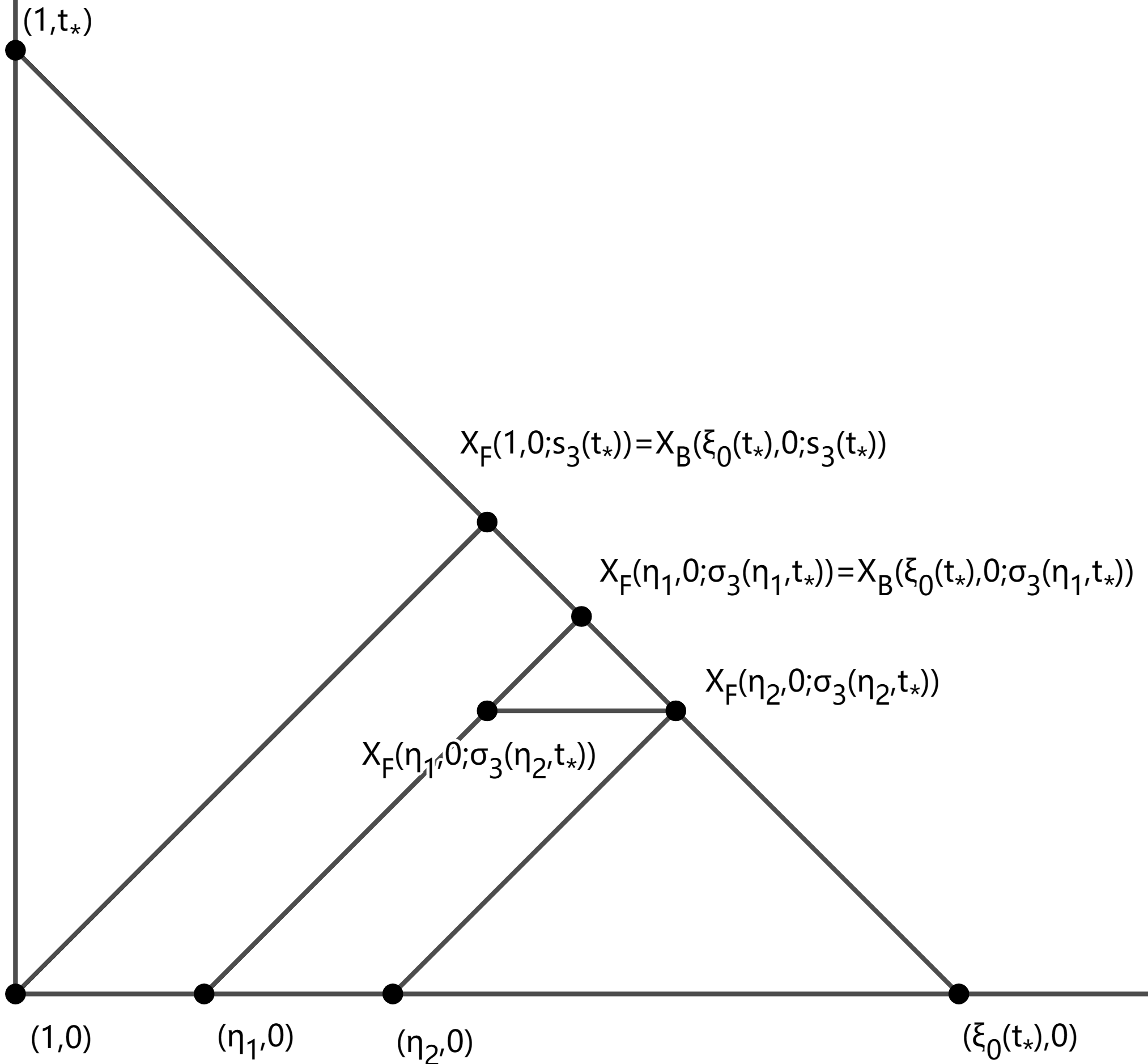}}  
\caption{Positions of the points $X_F(\eta_1,0;\sigma_3(\eta_1,t_*))$, $X_F(\eta_1,0;\sigma_3(\eta_2,t_*))$ and $X_F(\eta_2,0;\sigma_3(\eta_2,t_*))$}\label{4}   
\end{figure}
\noindent Use the rough bound (\ref{3.10}) (figure \ref{4}):
$$\begin{aligned}
&X_B^{(\xi)}(\xi_0(t_*),0;\sigma_3(\eta_2,t_*))-\overline{c}(\sigma_3(\eta_1,t_*)-\sigma_3(\eta_2,t_*))\leq X_B^{(\xi)}(\xi_0(t_*),0;\sigma_3(\eta_1,t_*))\\
=&X_F^{(\xi)}(\eta_1,0;\sigma_3(\eta_1,t_*))
\leq X_F^{(\xi)}(\eta_1,0;\sigma_3(\eta_2))+\overline{c}(\sigma_3(\eta_1,t_*)-\sigma_3(\eta_2,t_*)),
\end{aligned}$$
$$\begin{aligned}
&X_B^{(\xi)}(\xi_0(t_*),0;\sigma_3(\eta_2,t_*))-\underline{c}(\sigma_3(\eta_1,t_*)-\sigma_3(\eta_2,t_*))\geq X_B^{(\xi)}(\xi_0(t_*),0;\sigma_3(\eta_1,t_*))\\
=&X_F^{(\xi)}(\eta_1,0;\sigma_3(\eta_1,t_*))
\geq X_F^{(\xi)}(\eta_1,0;\sigma_3(\eta_2))+\underline{c}(\sigma_3(\eta_1,t_*)-\sigma_3(\eta_2,t_*)).
\end{aligned}$$
\begin{figure}[H]
\center{\includegraphics[width=14cm]  {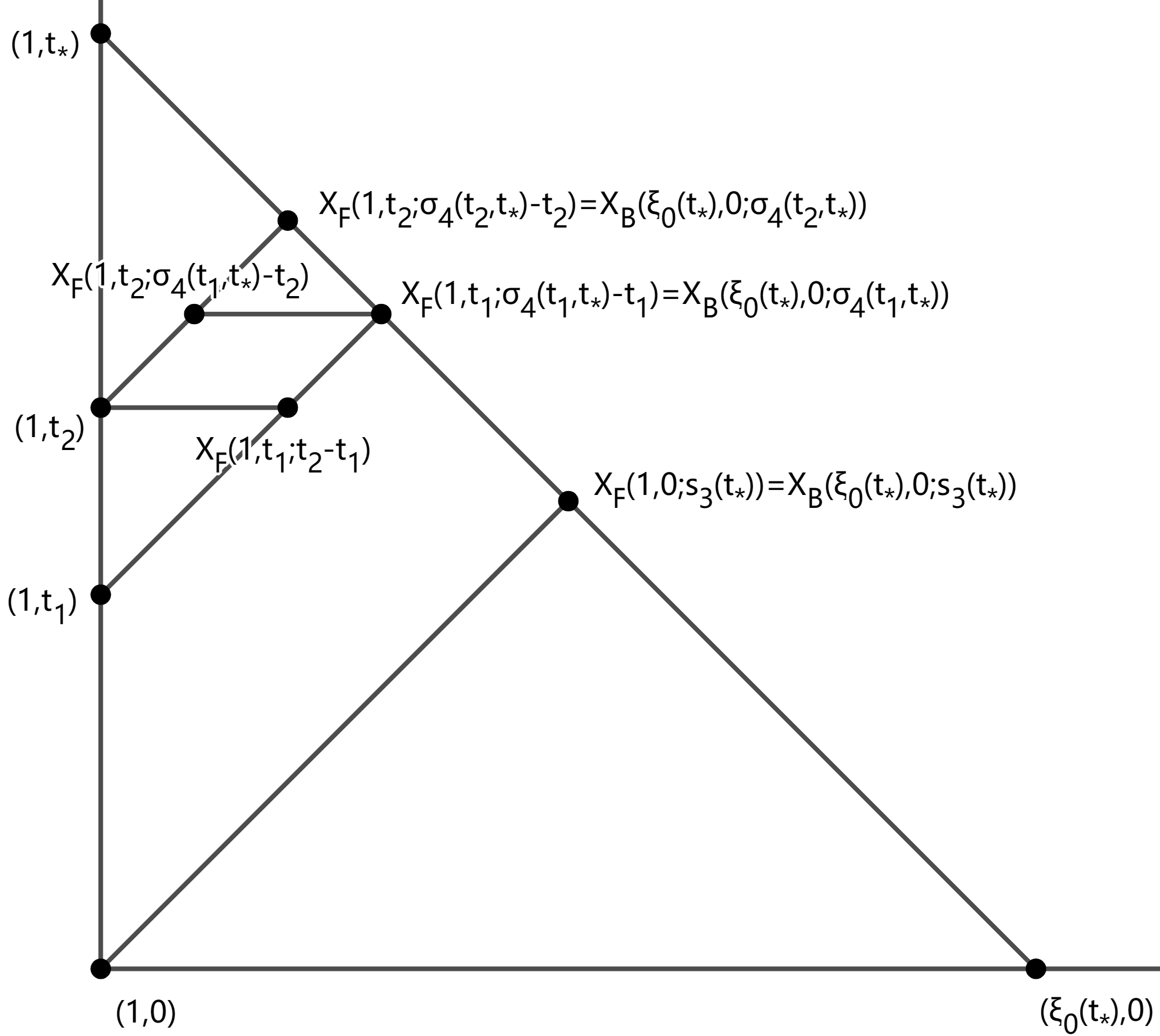}}
\caption{Positions of the points $X_F(1,t_1;t_2-t_1)$, $X_F(1,t_1;\sigma_4(t_1,t_*)-t_1)$, $X_F(1,t_2;\sigma_4(t_1,t_*)-t_2)$ and $X_F(1,t_2;\sigma_4(t_2,t_*)-t_2)$}\label{5}
\end{figure}
\noindent 
Combining the above inequalities and recalling that $X_B(\xi_0(t_*),0;\sigma_3(\eta_2,t_*))=X_F(\eta_2,0;\sigma_3(\eta_2,t_*))$, it follows
$$\left(2\overline{c}\right)^{-1}e^{-C\kappa}(\eta_2-\eta_1)\leq\sigma_3(\eta_1,t_*)-\sigma_3(\eta_2,t_*)\leq\left(2\underline{c}\right)^{-1}e^{C\kappa}(\eta_2-\eta_1),$$
which in turn gives the expected estimate of $\frac{d\sigma_3(\eta,t_*)}{d\eta}$. To obtain the estimate of $\frac{d\sigma_4(t,t_*)}{dt}$, 
using the relation
$$X_F^{(\xi)}(1,t_1;\sigma_4(t_1,t_*)-t_1)=X_F^{(\xi)}\left(X_F(1,t_1;t_2-t_1);\sigma_4(t_1,t_*)-t_2)\right),$$
it can be shown in a similar way that
$$\begin{aligned}
e^{-C\kappa}\left(X_F^{(\xi)}(1,t_1;t_2-t_1)-1\right)\leq& X_F^{(\xi)}(1,t_1;\sigma_4(t_1,t_*)-t_1)-X_F^{(\xi)}(1,t_2;\sigma_4(t_1,t_*)-t_2)\\
\leq&e^{C\kappa}\left(X_F^{(\xi)}(1,t_1;t_2-t_1)-1\right).
\end{aligned}$$
The rough bound (\ref{3.10}) yields (figure \ref{5})
$$\underline{c}(t_2-t_1)\leq X_F^{(\xi)}(1,t_1;t_2-t_1)-1\leq\overline{c}(t_2-t_1),$$
$$\begin{aligned}
&X_F^{(\xi)}(1,t_2;\sigma_4(t_1,t_*)-t_2)+\underline{c}\left(\sigma_4(t_2,t_*)-\sigma_4(t_1,t_*)\right)\leq X_F^{(\xi)}(1,t_2;\sigma_4(t_2,t_*))\\
=&X_B^{(\xi)}(\xi_0(t_*),0;\sigma_4(t_2,t_*))\leq X_B^{(\xi)}(\xi_0(t_*),0;\sigma_4(t_1,t_*))-\underline{c}\left(\sigma_4(t_2,t_*)-\sigma_4(t_1,t_*)\right),
\end{aligned}$$
$$\begin{aligned}
&X_F^{(\xi)}(1,t_2;\sigma_4(t_1,t_*)-t_2)+\overline{c}\left(\sigma_4(t_2,t_*)-\sigma_4(t_1,t_*)\right)\geq X_F^{(\xi)}(1,t_2;\sigma_4(t_2,t_*))\\
=&X_B^{(\xi)}(\xi_0(t_*),0;\sigma_4(t_2,t_*))\geq X_B^{(\xi)}(\xi_0(t_*),0;\sigma_4(t_1,t_*))-\overline{c}\left(\sigma_4(t_2,t_*)-\sigma_4(t_1,t_*)\right).
\end{aligned}$$
Combining the above inequalities and using the relation $$X_B(\xi_0(t_*),0;\sigma_4(t_1,t_*))=X_F(1,t_1;\sigma_4(t_1,t_*)-t_1),$$ we obtain
$$\left(2\overline{c}\right)^{-1}\underline{c}e^{-C\kappa}(t_2-t_1)\leq\sigma_4(t_2,t_*)-\sigma_4(t_1,t_*)\leq\left(2\underline{c}\right)^{-1}\overline{c}e^{C\kappa},$$
which gives the claimed estimate of $\frac{d\sigma_4(t,t_*)}{dt}$.
Next, we have by the rough bound (\ref{3.10})
$$\begin{aligned}
&\max\left\{\xi_0(t_*)-\overline{c}\sigma_3(\eta,t_*),\eta+\underline{c}\sigma_3(\eta,t_*)\right\}\\
\leq&X_B^{(\xi)}(\xi_0(t_*),0;\sigma_3(\eta,t_*))=X_F^{(\xi)}(\eta,0;\sigma_3(\eta,t_*))\\
\leq&\min\left\{\xi_0(t_*)-\underline{c}\sigma_3(\eta,t_*),\eta+\overline{c}\sigma_3(\eta,t_*)\right\},
\end{aligned}$$
and thus 
$$X_B^{(\xi)}(\xi_0(t_*),0;\sigma_3(\eta,t_*))\geq\eta+\frac{\underline{c}}{\overline{c}+\underline{c}}(\xi_0(t_*)-\eta)\geq1+\frac{\underline{c}}{\overline{c}+\underline{c}}(\xi_0(t_*)-1),$$
$$X_B^{(\xi)}(\xi_0(t_*),0;\sigma_3(\eta,t_*))\leq\eta+\frac{\overline{c}}{\overline{c}+\underline{c}}(\xi_0(t_*)-\eta).$$
For $X_B^{(\xi)}(\xi_0(t_*),0;\sigma_4(t,t_*)$, recall that 
$X_B(\xi_0(t_*),0;t_*)=(1,t_*)$, and thus 
$$\begin{aligned}
&1+\underline{c}(t_*-\sigma_4(t,t_*))\leq X_B(\xi_0(t_*),0;\sigma_4(t,t_*))\\=&X_B(1,t_*;\sigma_4(t,t_*)-t_*)\leq 1+\overline{c}(t_*-\sigma_4(t,t_*)).
\end{aligned}$$
It follows that
$$\begin{aligned}
&\max\left\{1+\underline{c}(t_*-\sigma_4(t,t_*)),1+\underline{c}(\sigma_4(t,t_*)-t)\right\}\\
\leq&X_F^{(\xi)}(1,t;\sigma_4(t,t_*))=X_B^{(\xi)}(\xi_0(t_*),0;\sigma_4(t,t_*))\\
\leq&\min\left\{1+\overline{c}(t_*-\sigma_4(t,t_*),1+\overline{c}(\sigma_4(t,t_*)-t)\right\}.
\end{aligned}$$
Therefore,
$$1+\frac{\underline{c}}{2}(t_*-t)\leq X_B^{(\xi)}(\xi_0(t_*),0;\sigma_4(t,t_*))\leq 1+\frac{\overline{c}}{2}(t_*-t).$$
\end{proof}
Now we estimate the second line in (\ref{6.7}) by using (\ref{3.7}), Lemma \ref{lem 6.3} and splitting the integral into two parts in which $\sigma\in(0,s_3(t_*))$ and $\sigma\in(s_3(t_*),t_*)$:
\begin{equation}\begin{aligned}
&\int_0^{s_3(t_*)}\left|\frac{\left(\partial_t-cr^2\xi^{-2}\partial_\xi\right)c}{2c}\left(X_B(\xi_0(t_*),0;\sigma\right)\right|v_F\left(X_B(\xi_0(t_*),0;\sigma);t_*\right)d\sigma\\
\leq&C\epsilon\int_0^{s_3(t_*)}X_B^{(\xi)}(\xi_0(t_*),0;\sigma)^{-1}v_F\left(X_B(\xi_0(t_*),0;\sigma);t_*\right)d\sigma\\
\leq&\frac{Ce^{C\kappa}\epsilon}{\underline{c}}\int_1^{\xi_0(t_*)}X_B^{(\xi)}(\xi_0(t_*),0;\sigma_3(\eta,t_*))^{-1}v_F\left(X_F(\eta,0;\sigma_3(\eta,t_*));t_*\right)d\eta\\
\leq&\frac{Ce^{C\kappa}\epsilon}{\underline{c}}\int_1^{\xi_0(t_*)}\left(1+\frac{\underline{c}}{\overline{c}+\underline{c}}(\xi_0(t_*)-1)\right)^{-1}v_F(\eta,0;t_*)d\eta,
\end{aligned}\label{6.12}\end{equation}
\begin{equation}\begin{aligned}
&\int_{s_3(t_*)}^{t_*}\left|\frac{\left(\partial_t-cr^2\xi^{-2}\partial_\xi\right)c}{2c}\left(X_B(\xi_0(t_*),0;\sigma\right)\right|v_F\left(X_B(\xi_0(t_*),0;\sigma);t_*\right)d\sigma\\
\leq&C\epsilon\int_{s_3(t_*)}^{t_*}X_B^{(\xi)}(\xi_0(t_*),0;\sigma)^{-1}v_F(X_B(\xi_0(t_*),0;\sigma);t_*)d\sigma\\
\leq&\frac{C\overline{c}e^{C\kappa}\epsilon}{\underline{c}}\int_0^{t_*}X_B^{(\xi)}(\xi_0(t_*),0;\sigma_4(t,t_*))^{-1}v_F\left(X_F(\eta,0;\sigma_4(t,t_*));t_*\right)dt\\
\leq&\frac{C\overline{c}e^{C\kappa}\epsilon}{\underline{c}}\int_0^{t_*}\left(1+\frac{\underline{c}}{2}(t_*-t)\right)^{-1}v_F(\eta,0;t_*)dt.
\end{aligned}\label{6.13}\end{equation}
We also split the integral in the third line of (\ref{6.7}) into two parts:
$$\begin{aligned}
&\int_0^{t_*}\left|2r^{-2}u^2\frac{\psi}{c^\frac{1}{2}}\left(X_B(\xi_0(t_*),0;\sigma)\right)\right|d\sigma\\
=&\int_0^{s_3(t_*)}\left|2r^{-2}u^2\frac{\psi}{c^\frac{1}{2}}\left(X_B(\xi_0(t_*),0;\sigma)\right)\right|d\sigma+\int_{s_3(t_*)}^{t_*}\left|2r^{-2}u^2\frac{\psi}{c^\frac{1}{2}}\left(X_B(\xi_0(t_*),0;\sigma)\right)\right|d\sigma.
\end{aligned}$$
We use $(\partial_t-cr^2\partial_x)\psi=w_F$ and (\ref{3.6})(\ref{3.10}) to estimate the first part:
$$\begin{aligned}
&\int_0^{s_3(t_*)}\left|2r^{-2}u^2\frac{\psi}{c^\frac{1}{2}}\left(X_B(\xi_0(t_*),0;\sigma)\right)\right|d\sigma\\
\leq&\frac{C\epsilon^2}{\underline{c}^\frac{1}{2}}\int_0^{s_3(t_*)}X_B^{(\xi)}(\xi_0(t_*),0;\sigma)^{-4}\left|\psi(\xi_0(t_*))+\int_0^\sigma w_F\left(X_B(\xi_0(t_*),0;\tau)\right)d\tau\right|d\sigma\\
\leq&\frac{C\epsilon^2}{\underline{c}^\frac{1}{2}}\int_0^{s_3(t_*)}X_B^{(\xi)}(\xi_0(t_*),0;\sigma)^{-4}\left|\psi(\xi_0(t_*))\right|d\sigma\\
&+\frac{C\epsilon^2\overline{c}^\frac{1}{2}}{\underline{c}^\frac{1}{2}}\int_0^{s_3(t_*)}X_B^{(\xi)}(\xi_0(t_*),0;\sigma)^{-4}\int_0^\sigma v_F\left(X_B(\xi_0(t_*),0;\tau);t_*\right)d\tau d\sigma\\
\leq&\frac{C\epsilon^2}{\underline{c}^\frac{1}{2}}\int_0^{s_3(t_*)}X_B^{(\xi)}(\xi_0(t_*),0;\sigma)^{-4}\left|\psi(\xi_0(t_*))\right|d\sigma\\
&+\frac{C\epsilon^2\overline{c}^\frac{1}{2}}{\underline{c}^\frac{1}{2}}\int_0^{s_3(t_*)}\left(\int_\tau^{s_3(t_*)}X_B^{(\xi)}(\xi_0(t_*),0;\sigma)^{-4}d\sigma\right) v_F\left(X_B(\xi_0(t_*),0;\tau);t_*\right)d\tau.
\end{aligned}$$ 
By the rough bound (\ref{3.10}) and Lemma \ref{lem 6.3}, we have
$$X_B^{(\xi)}(\xi_0(t_*),0;\sigma)\geq X_B^{(\xi)}(\xi_0(t_*),0;s_3(t_*))+\underline{c}(s_3(t_*)-\sigma),$$
$$\begin{aligned}
\int_\tau^{s_3(t_*)}X_B^{(\xi)}(\xi_0(t_*),0;\sigma)^{-4}d\sigma
\leq&\frac{1}{3\underline{c}}X_B^{(\xi)}(\xi_0(t_*),0;s_3(t_*))^{-3}\\
\leq&\frac{1}{3\underline{c}}\left(1+\frac{\underline{c}}{\overline{c}+\underline{c}}(\xi_0(t_*)-1)\right)^{-3}.
\end{aligned}$$
Therefore
\begin{equation}\begin{aligned}
&\int_0^{s_3(t_*)}\left|2r^{-2}u^2\frac{\psi}{c^\frac{1}{2}}\left(X_B(\xi_0(t_*),0;\sigma)\right)\right|d\sigma\\
\leq&\frac{C\epsilon^2}{\underline{c}^\frac{3}{2}}\left(1+\frac{\underline{c}}{\overline{c}+\underline{c}}(\xi_0(t_*)-1)\right)^{-3}\left|\psi(\xi_0(t_*))\right|\\
&+\frac{C\epsilon^2\overline{c}^\frac{1}{2}}{\underline{c}^\frac{3}{2}}\left(1+\frac{\underline{c}}{\overline{c}+\underline{c}}(\xi_0(t_*)-1)\right)^{-3}\int_0^{s_3(t_*)}v_F\left(X_B(\xi_0(t_*),0;\tau);t_*\right)d\tau\\
\leq&\frac{C\epsilon^2}{\underline{c}^\frac{3}{2}}\left(1+\frac{\underline{c}}{\overline{c}+\underline{c}}(\xi_0(t_*)-1)\right)^{-3}\left|\psi(\xi_0(t_*))\right|\\
&+\frac{C\epsilon^2\overline{c}^\frac{1}{2}}{\underline{c}^\frac{5}{2}}e^{C\kappa}\left(1+\frac{\underline{c}}{\overline{c}+\underline{c}}(\xi_0(t_*)-1)\right)^{-3}\int_1^{\xi_0(t_*)}v_F(\eta,0;t_*)d\eta.
\end{aligned}\label{6.14}\end{equation}
For the second part, we first use $(\partial_t+cr^2\partial_x)\psi=w_B$, (\ref{3.6}) and Lemma \ref{lem 6.3}:
$$\begin{aligned}
&\int_{s_3(t_*)}^{t_*}\left|2r^{-2}u^2\frac{\psi}{c^\frac{1}{2}}\left(X_B(\xi_0(t_*),0;\sigma)\right)\right|d\sigma\\
\leq&\frac{C\epsilon^2}{\underline{c}^\frac{1}{2}}\int_{s_3(t_*)}^{t_*}X_B^{(\xi)}(\xi_0(t_*),0;\sigma)^{-4}\left|\psi\left(X_B(\xi_0(t_*),0;\sigma)\right)\right|d\sigma\\
\leq&\frac{C\epsilon^2\overline{c}}{\underline{c}^\frac{3}{2}}e^{C\kappa}\int_0^{t_*}X_B^{(\xi)}(\xi_0(t_*),0;\sigma_4(t,t_*))^{-4}\left|\psi\left(X_F(1,t;\sigma_4(t,t_*)-t)\right)\right|dt\\
\leq&\frac{C\epsilon^2\overline{c}}{\underline{c}^\frac{3}{2}}e^{C\kappa}\int_0^{t_*}X_B^{(\xi)}(\xi_0(t_*),0;\sigma_4(t,t_*))^{-4}\left|\psi(1,t)+\int_0^{\sigma_4(t,t_*)-t}w_B\left(X_F(1,t;\tau)\right)d\tau\right|dt\\
\leq&\frac{C\epsilon^2\overline{c}}{\underline{c}^\frac{3}{2}}e^{C\kappa}\int_0^{t_*}\left(1+\frac{\underline{c}}{2}(t_*-t)\right)^{-4}\left|\psi(1,t)\right|dt\\
&+\frac{C\epsilon^2\overline{c}^\frac{3}{2}}{\underline{c}^\frac{3}{2}}e^{C\kappa}\int_0^{t_*}X_B^{(\xi)}(\xi_0(t_*),0;\sigma_4(t,t_*))^{-4}\int_0^{\sigma_4(t,t_*)-t}v_B\left(X_F(1,t;\tau)\right)d\tau dt.
\end{aligned}$$
We further apply Lemma \ref{lem 6.1} and recall $\sigma_4(t,t_*)-t=\sigma_1(\xi_0(t_*),t)$:
$$\begin{aligned}
\int_0^{\sigma_4(t,t_*)-t}v_B\left(X_F(1,t;\tau)\right)d\tau dt
\leq&\frac{e^{C\kappa}}{2\underline{c}}\int_{\xi_0(t)}^{\xi_0(t_*)}v_B\left(X_B(\eta,0;t+\sigma_1(\eta,t))\right)d\eta\\
=&\frac{e^{C\kappa}}{2\underline{c}}\int_{\xi_0(t)}^{\xi_0(t_*)}v_B(\eta,0)d\eta,
\end{aligned}$$
$$\begin{aligned}
&\int_0^{t_*}X_B^{(\xi)}(\xi_0(t_*),0;\sigma_4(t,t_*))^{-4}\int_0^{\sigma_4(t,t_*)-t}v_B\left(X_F(1,t;\tau)\right)d\tau dt\\
\leq&\frac{e^{C\kappa}}{2\underline{c}}\int_0^{t_*}X_B^{(\xi)}(\xi_0(t_*),0;\sigma_4(t,t_*))^{-4}\int_{\xi_0(t)}^{\xi_0(t_*)}v_B(\eta,0)d\eta dt\\
=&\frac{e^{C\kappa}}{2\underline{c}}\int_1^{\xi_0(t_*)}\left(\int_0^{\xi_0^{-1}(\eta)} X_B^{(\xi)}(\xi_0(t_*),0;\sigma_4(t,t_*))^{-4}dt\right)v_B(\eta,0)d\eta,
\end{aligned}$$
where $\xi_0^{-1}$ denotes the inverse function of $\xi_0$. Use Lemma \ref{lem 6.3} to estimate the integral in the bracket:
$$\begin{aligned}
\int_0^{\xi_0^{-1}(\eta)}X_B^{(\xi)}(\xi_0(t_*),0;\sigma_4(t,t_*))^{-4}dt
\leq&\int_0^{\xi_0^{-1}(\eta)}\left(1+\frac{\underline{c}}{2}(t_*-t)\right)^{-4}dt\\
\leq&\frac{2}{3\underline{c}}\left(1+\frac{\underline{c}}{2}\left(t_*-\xi_0^{-1}(\eta)\right)\right)^{-3}.
\end{aligned}$$
Combining the above inequalities yields
\begin{equation}\begin{aligned}
&\int_{s_3(t_*)}^{t_*}\left|2r^{-2}u^2\frac{\psi}{c^\frac{1}{2}}\left(X_B(\xi_0(t_*),0;\sigma)\right)\right|d\sigma\\
\leq&\frac{C\epsilon^2\overline{c}}{\underline{c}^\frac{3}{2}}e^{C\kappa}\int_0^{t_*}\left(1+\frac{\underline{c}}{2}(t_*-t)\right)^{-4}\left|\psi(1,t)\right|dt\\
&+\frac{C\epsilon^2\overline{c}^\frac{3}{2}}{\underline{c}^\frac{7}{2}}e^{C\kappa}\int_1^{\xi_0(t_*)}\left(1+\frac{\underline{c}}{2}\left(t_*-\xi_0^{-1}(\eta)\right)\right)^{-3}v_B(\eta,0)d\eta.
\end{aligned}\label{6.15}\end{equation}
\subsection{Collection of (\ref{6.5}-\ref{6.15}) and the estimate of backward pressure wave on the boundary}
We now collect all the bounds obtained in this section. First we note in addition that from the bound (\ref{3.6})(\ref{3.10}) and the smallness $\epsilon\leq\epsilon_0\lesssim 1$ it follows
$$
\exp\left\{\int_0^s\left|\frac{2u}{r}\left(X_F(1,t;\sigma)\right)\right|d\sigma\right\}\leq\exp\left\{C\epsilon\int_0^s\left(1+\underline{c}\sigma\right)^{-2}d\sigma\right\}\leq e^{C\epsilon}\leq 1+C\epsilon,
$$
$$
\exp\left\{\int_0^s\left|\frac{2u}{r}\left(X_F(\xi,0;\sigma)\right)\right|d\sigma\right\}\leq\exp\left\{C\epsilon\int_0^s\left(\xi+\underline{c}\sigma\right)^{-2}d\sigma\right\}\leq e^{C\epsilon}\leq 1+C\epsilon,
$$
$$
\exp\left\{\int_0^{t_*}\left|\frac{2u}{r}\left(X_B(\xi_0(t_*),0;\sigma)\right)\right|d\sigma\right\}\leq\exp\left\{C\epsilon\int_0^{t_*}\left(1+\underline{c}\sigma\right)^{-2}d\sigma\right\}\leq e^{C\epsilon}\leq 1+C\epsilon.
$$
For simplicity, we write $\mu:=\frac{\overline{c}}{\underline{c}}e^{C\kappa}>1$. Combining (\ref{6.5})(\ref{6.8})(\ref{6.9}) gives
\begin{equation}\begin{aligned}
&v_F(1,t;t_*)\\
\leq&(1+C\epsilon)\left|\frac{w_F}{c^\frac{1}{2}}(1,t)\right|
+C\epsilon\int_{\xi_0(t)}^{\xi_0(t_*)}\left(1+\frac{1}{2\mu}(\eta-\xi_0(t))\right)^{-1}v_B(\eta,0)d\eta+C\epsilon^2|\psi(1,t)|.
\end{aligned}\label{6.16}\end{equation}
Analogously, from (\ref{6.6})(\ref{6.10})(\ref{6.11}) it holds
\begin{equation}\begin{aligned}
&v_F(\xi,0;t_*)\\
\leq&(1+C\epsilon)\left|\frac{w_F}{c^\frac{1}{2}}(\xi,0)\right|+C\epsilon\int_\xi^{\xi_0(t_*)}\left(\xi+\frac{1}{2\mu}(\eta-\xi)\right)^{-1}v_B(\eta,0)d\eta+C\epsilon^2\xi^{-3}|\psi(\xi,0)|,
\end{aligned}\label{6.17}\end{equation}
Since $1+\frac{\underline{c}}{\overline{c}+\underline{c}}(\xi_0(t_*)-1)\geq\frac{\underline{c}}{\overline{c}+\underline{c}}\xi_0(t_*)$, from (\ref{6.7}), (\ref{6.12}-\ref{6.15}) it follows
\begin{equation}\begin{aligned}
&v_B(\xi_0(t_*),0)\\
\leq&(1+C\epsilon)\left|\frac{w_B}{c^\frac{1}{2}}(\xi_0(t_*),0)\right|\\
&+C\epsilon^2\xi_0(t_*)^{-3}|\psi(\xi_0(t_*))|+C\epsilon^2\int_0^{t_*}\left(1+\frac{\underline{c}}{2}(t_*-t)\right)^{-4}|\psi(1,t)|dt\\
&+C\epsilon\xi_0(t_*)^{-1}\int_1^{\xi_0(t_*)}v_F(\xi,0;t_*)d\xi+C\epsilon\int_0^{t_*}\left(1+\frac{\underline{c}}{2}(t_*-t)\right)^{-1}v_F(1,t;t_*)dt\\
&+C\epsilon^2\int_1^{\xi_0(t_*)}\left(1+\frac{\underline{c}}{2}\left(t_*-\xi_0^{-1}(\eta)\right)\right)^{-3}v_B(\eta,0)d\eta.
\end{aligned}\label{6.18}\end{equation}
We further call upon (\ref{6.4}) to obtain
\begin{equation}\begin{aligned}
&\left|\frac{w_B}{c^\frac{1}{2}}\left(1,t_*\right)-\frac{w_B}{c^\frac{1}{2}}\left(\xi_0(t_*),0\right)\right|\\
\leq&C\epsilon|w_B(\xi_0(t_*),0)|+C\epsilon^2\xi_0(t_*)^{-3}|\psi(\xi_0(t_*))|+C\epsilon^2\int_0^{t_*}\left(1+\frac{\underline{c}}{2}(t_*-t)\right)^{-4}|\psi(1,t)|dt\\
&+C\epsilon\xi_0(t_*)^{-1}\int_1^{\xi_0(t_*)}v_F(\xi,0;t_*)d\xi+C\epsilon\int_0^{t_*}\left(1+\frac{\underline{c}}{2}(t_*-t)\right)^{-1}v_F(1,t;t_*)dt\\
&+C\epsilon^2\int_1^{\xi_0(t_*)}\left(1+\frac{\underline{c}}{2}\left(t_*-\xi_0^{-1}(\eta)\right)\right)^{-3}v_B(\eta,0)d\eta.
\end{aligned}\label{6.19}\end{equation}
To obtain an explicit control of the last terms in (\ref{6.18}) and (\ref{6.19}), we have the following control of $\xi_0^{-1}$:
\begin{figure}[H]
\center{\includegraphics[width=10cm]  {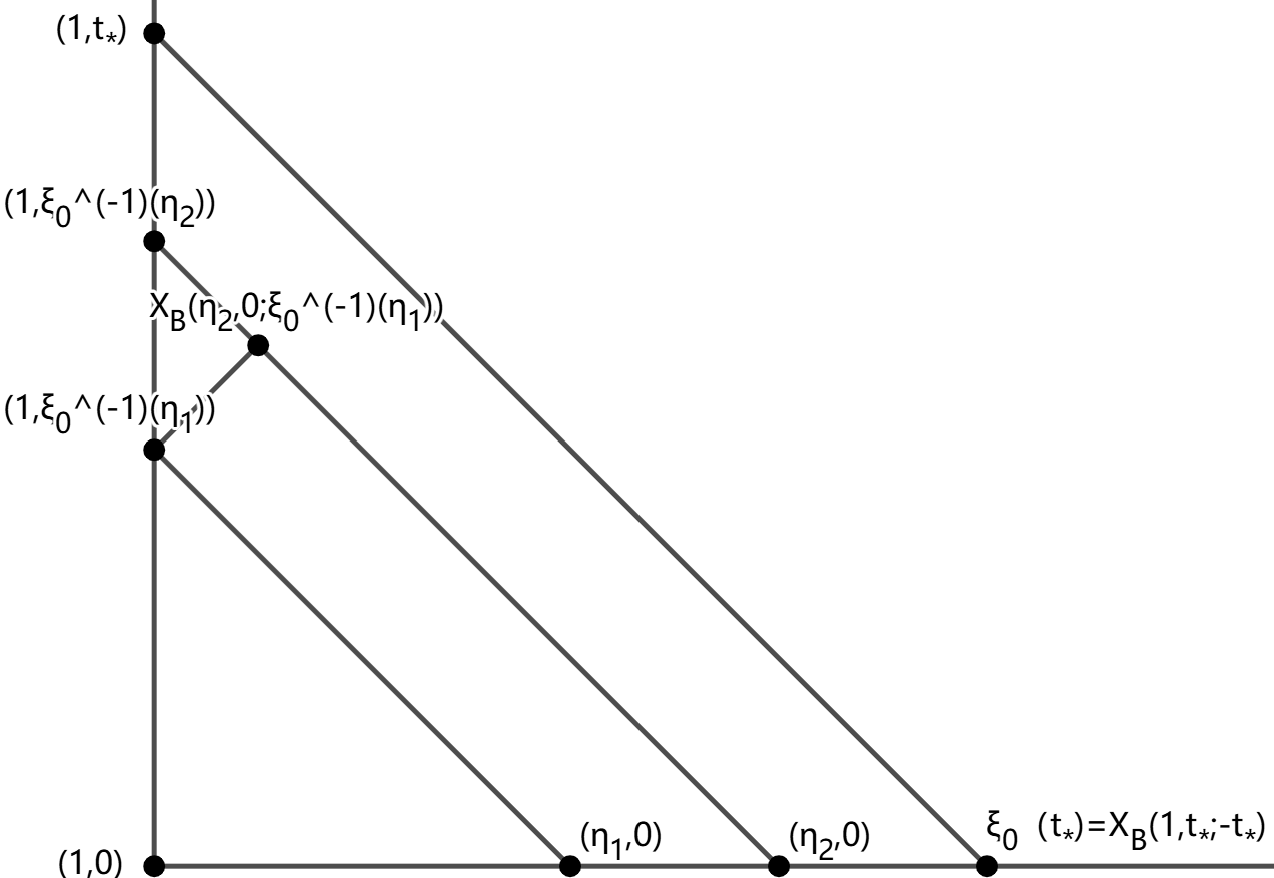}}  
\caption{Position of the point $X_B(\eta_2,0;\xi_0^{-1}(\eta_1))$ }\label{6}   
\end{figure}
\noindent 
\begin{lem}
For each $\eta_1,\eta_2\in[1,\xi_0(T)]$ with $\eta_1<\eta_2$, the following inequality holds:
$$\frac{e^{C\kappa}}{\underline{c}}(\eta_2-\eta_1)\geq\xi_0^{-1}(\eta_2)-\xi_0^{-1}(\eta_1)\geq\frac{e^{-C\kappa}}{\overline{c}}(\eta_2-\eta_1).$$
In particular, it follows that
$$\frac{e^{C\kappa}}{\underline{c}}\geq\frac{d\xi_0^{-1}}{d\eta}\geq\frac{e^{-C\kappa}}{\overline{c}},$$
$$t_*-\xi_0^{-1}(\eta)\geq\frac{e^{-C\kappa}}{\overline{c}}(\xi_0(t_*)-\eta).$$
\label{lem 6.4}\end{lem}
\begin{proof}
By simulating the proof of Lemma \ref{lem 6.1} we have
$$e^{C\kappa}(\eta_2-\eta_1)\geq X_B^{(\xi)}(\eta_2,0;\xi_0^{-1}(\eta_1))-X_B^{(\xi)}(\eta_1,0;\xi_0^{-1}(\eta_1))\geq e^{-C\kappa}(\eta_2-\eta_1).$$
Then use the rough bound (\ref{3.10}) (figure \ref{6}):
$$\begin{aligned}
\xi_0^{-1}(\eta_2)-\xi_0^{-1}(\eta_1)\geq\frac{1}{\overline{c}}\left(X_B^{(\xi)}(\eta_2,0;\xi_0^{-1}(\eta_1))-X_B^{(\xi)}(\eta_1,0;\xi_0^{-1}(\eta_1)\right)\geq \frac{e^{-C\kappa}}{\overline{c}}(\eta_2-\eta_1),
\end{aligned}$$
$$\begin{aligned}
\xi_0^{-1}(\eta_2)-\xi_0^{-1}(\eta_1)\leq\frac{1}{\underline{c}}\left(X_B^{(\xi)}(\eta_2,0;\xi_0^{-1}(\eta_1))-X_B^{(\xi)}(\eta_1,0;\xi_0^{-1}(\eta_1)\right)\leq\frac{e^{C\kappa}}{\underline{c}}(\eta_2-\eta_1).
\end{aligned}$$
\end{proof}
From the above lemma, we see that $$\left(1+\frac{\underline{c}}{2}\left(t_*-\xi_0^{-1}(\eta)\right)\right)^{-3}\geq\left(1+\frac{1}{2\mu}(\xi_0(t_*)-\eta)\right)^{-3}.$$ 
Then replacing $v_F(\eta,0;t_*)$ and $v_F(1,t;t_*)$ in (\ref{6.18}) by the bounds (\ref{6.16})(\ref{6.17}) gives
\begin{equation}\begin{aligned}
&v_B(\xi_0(t_*),0)\\
\leq&(1+C\epsilon)\left|\frac{w_B}{c^\frac{1}{2}}(\xi_0(t_*),0)\right|\\
&+C\epsilon^2\xi_0(t_*)^{-3}|\psi(\xi_0(t_*))|+C\epsilon^2\int_0^{t_*}\left(1+\frac{\underline{c}}{2}(t_*-t)\right)^{-4}|\psi(1,t)|dt\\
&+C\epsilon^3\xi_0(t_*)^{-1}\int_1^{\xi_0(t_*)}\eta^{-3}|\psi(\eta,0)|d\eta+C\epsilon^3\int_0^{t_*}\left(1+\frac{\underline{c}}{2}(t_*-t)\right)^{-1}|\psi(1,t)|dt\\
&+C\epsilon\xi_0(t_*)^{-1}\int_1^{\xi_0(t_*)}|w_F(\eta,0)|d\eta+C\epsilon\int_0^{t_*}\left(1+\frac{\underline{c}}{2}(t_*-t)\right)^{-1}|w_F(1,t)|dt\\
&+C\epsilon^2\int_1^{\xi_0(t_*)}\left(1+\frac{1}{2\mu}(\xi_0(t_*)-\eta)\right)^{-3}v_B(\eta,0)d\eta\\
&+C\epsilon^2\xi_0(t_*)^{-1}\int_1^{\xi_0(t_*)}\int_\xi^{\xi_0(t_*)}\left(\xi+\frac{1}{2\mu}(\eta-\xi)\right)^{-1}v_B(\eta,0)d\eta d\xi\\
&+C\epsilon^2\int_0^{t_*}\int_{\xi_0(t)}^{\xi_0(t_*)}\left(1+\frac{\underline{c}}{2}(t_*-t)\right)^{-1}\left(1+\frac{1}{2\mu}(\eta-\xi_0(t))\right)^{-1}v_B(\eta,0)d\eta dt.
\end{aligned}\label{6.20}\end{equation}
Since $w_F=\left(\partial_t-cr^2\xi^{-2}\partial_\xi\right)\psi=2\partial_t\psi-w_B$, it holds
\begin{equation}\begin{aligned}
&\int_0^{t_*}\left(1+\frac{\underline{c}}{2}(t_*-t)\right)^{-1}|w_F(1,t)|dt\\
\leq&2\int_0^{t_*}\left(1+\frac{\underline{c}}{2}(t_*-t)\right)^{-1}|\partial_t\psi(1,t)|dt+\overline{c}^\frac{1}{2}\int_0^{t_*}\left(1+\frac{\underline{c}}{2}(t_*-t)\right)^{-1}v_B(\xi_0(t),0)dt\\
\leq&2\int_0^{t_*}\left(1+\frac{\underline{c}}{2}(t_*-t)\right)^{-1}|\partial_t\psi(1,t)|dt\\
&+\frac{e^{C\kappa}\overline{c}^\frac{1}{2}}{\underline{c}}\int_1^{\xi_0(t_*)}\left(1+\frac{1}{2\mu}(\xi_0(t_*)-\eta)\right)^{-1}v_B(\eta,0)d\eta.
\end{aligned}\label{6.21}\end{equation}
The last inequality follows from Lemma \ref{lem 6.4}.
To simplify the last two terms in (\ref{6.20}) we compute
\begin{equation}\begin{aligned}
&\int_1^{\xi_0(t_*)}\int_\xi^{\xi_0(t_*)}\left(\xi+\frac{1}{2\mu}(\eta-\xi)\right)^{-1}v_B(\eta,0)d\eta d\xi\\
=&\int_1^{\xi_0(t_*)}\left(\int_1^\eta\left(\xi+\frac{1}{2\mu}(\eta-\xi)\right)^{-1}d\xi\right)v_B(\eta,0)d\eta\\
=&\int_1^{\xi_0(t_*)}\frac{1}{1-\frac{1}{2\mu}}\log\left(\frac{\eta}{1+\frac{1}{2\mu}\left(\eta-1\right)}\right)v_B(\eta,0)d\eta\\
\leq&C\int_1^{\xi_0(t_*)}v_B(\eta,0)d\eta,
\end{aligned}\label{6.22}\end{equation}
and use Lemma \ref{lem 6.4} to obtain
$$\begin{aligned}
&\int_0^{t_*}\int_{\xi_0(t)}^{\xi_0(t_*)}\left(1+\frac{\underline{c}}{2}(t_*-t)\right)^{-1}\left(1+\frac{1}{2\mu}(\eta-\xi_0(t))\right)^{-1}v_B(\eta,0)d\eta dt\\
=&\int_1^{\xi_0(t_*)}\left(\int_0^{\xi_0^{-1}(\eta)}\left(1+\frac{\underline{c}}{2}(t_*-t)\right)^{-1}\left(1+\frac{1}{2\mu}(\eta-\xi_0(t))\right)^{-1}dt\right)v_B(\eta,0)d\eta\\
\leq&\frac{e^{C\kappa}}{\underline{c}}\int_1^{\xi_0(t_*)}\left(\int_1^\eta\left(1+\frac{1}{2\mu}\left(\xi_0(t_*)-\xi\right)\right)^{-1}\left(1+\frac{1}{2\mu}\left(\eta-\xi\right)\right)^{-1}d\xi\right)v_B(\eta,0)d\eta\\
=&\frac{e^{C\kappa}}{\underline{c}}\int_1^{\xi_0(t_*)}\frac{4\mu^2}{\xi_0(t_*)-\eta}\left(\log\frac{2\mu+\eta-1}{2\mu}-\log\frac{2\mu+\xi_0(t_*)-1}{2\mu+\xi_0(t_*)-\eta}\right)v_B(\eta,0)d\eta\\
=&\frac{e^{C\kappa}}{\underline{c}}\int_1^{\xi_0(t_*)}\frac{4\mu^2}{\xi_0(t_*)-\eta}\log\left(1+\frac{(\eta-1)(\xi_0(t_*)-\eta)}{4\mu^2+2\mu(\xi_0(t_*)-1)}\right)v_B(\eta,0)d\eta.
\end{aligned}$$
If $\eta\in\left(1,\xi_0(t_*)-1\right)$, using the bound (\ref{3.10})(\ref{3.11}) the following holds:
$$\frac{1}{\xi_0(t_*)-\eta}\leq\frac{1+2\mu}{\xi_0(t_*)-\eta+2\mu},$$
$$\log\left(1+\frac{(\eta-1)(\xi_0(t_*)-\eta)}{2\mu(2\mu+\xi_0(t_*)-1)}\right)\leq\log\xi_0(t_*)\leq\log(1+\overline{c}t_*)\leq\frac{\kappa}{\epsilon},$$
and thus
$$\frac{4\mu^2}{\xi_0(t_*)-\eta}\log\left(1+\frac{(\eta-1)(\xi_0(t_*)-\eta)}{4\mu^2+2\mu(\xi_0(t_*)-1)}\right)\leq\frac{\kappa}{\epsilon}\frac{4\mu^2(1+2\mu)}{\xi_0(t_*)-\eta+2\mu}.$$
If $\eta\in\left[\xi_0(t_*)-1,\xi_0(t_*)\right)$, we have
$$\frac{4\mu^2}{\xi_0(t_*)-\eta}\log\left(1+\frac{(\eta-1)(\xi_0(t_*)-\eta)}{4\mu^2+2\mu(\xi_0(t_*)-1)}\right)\leq\frac{2\mu(\eta-1)}{2\mu+\xi_0(t_*)-1}\leq\frac{\kappa}{\epsilon}\frac{2\mu(2\mu+1)}{\xi_0(t_*)-\eta+2\mu}.$$
Hence we obtain in each case
$$\frac{4\mu^2}{\xi_0(t_*)-\eta}\log\left(1+\frac{(\eta-1)(\xi_0(t_*)-\eta)}{4\mu^2+2\mu(\xi_0(t_*)-1)}\right)\lesssim\frac{\kappa}{\epsilon}\frac{1}{\xi_0(t_*)-\eta+2\mu},$$
and thus 
\begin{equation}\begin{aligned}
&\int_0^{t_*}\int_{\xi_0(t)}^{\xi_0(t_*)}\left(1+\frac{\underline{c}}{2}(t_*-t)\right)^{-1}\left(1+\frac{1}{2\mu}(\eta-\xi_0(t)))\right)^{-1}v_B(\eta,0)d\eta dt\\
&\leq\frac{Ce^{C\kappa}}{\underline{c}}\frac{\kappa}{\epsilon}\int_1^{\xi_0(t_*)}\frac{1}{\xi_0(t_*)-\eta+2\mu}v_B(\eta,0)d\eta.
\end{aligned}\label{6.23}\end{equation}
From (\ref{6.21})(\ref{6.22})(\ref{6.23}) we see that all the terms on the right-hand side of (\ref{6.20}) involving $v_B$ can be controlled by 
$$C\epsilon(\kappa+1)\int_1^{\xi_0(t_*)}\left(2\mu+\xi_0(t_*)-\eta\right)^{-1}v_B(\eta,0)d\eta.$$
By recalling that
$$\mathcal{V}(\xi):=\left|w_B(\xi,0)\right|+\xi^{-1}\int_1^\xi|w_F(\eta,0)|d\eta+\epsilon\xi^{-3}|\psi(\xi,0)|+\epsilon^2\xi^{-1}\int_1^\xi\eta^{-3}|\psi(\eta,0)|d\eta,$$
$$\begin{aligned}
\Psi(t_*):=&\int_0^{t_*}\left(1+\frac{\underline{c}}{2}(t_*-t)\right)^{-1}|\partial_t\psi(1,t)|dt+\epsilon\int_0^{t_*}\left(1+\frac{\underline{c}}{2}(t_*-t)\right)^{-4}|\psi(1,t)|dt\\&+\epsilon^2\int_0^{t_*}\left(1+\frac{\underline{c}}{2}(t_*-t)\right)^{-1}|\psi(1,t)|dt,
\end{aligned}$$
and classifying the terms on the right-hand side of (\ref{6.20}), in view of $\kappa\leq\kappa_0\lesssim1$, we obtain 
\begin{equation}\begin{aligned}
v_B(\xi_0(t_*),0)\leq&\left|\frac{w_B}{c^\frac{1}{2}}(\xi_0(t_*),0)\right|+C\epsilon\mathcal{V}(\xi_0(t_*))+C\epsilon\Psi(t_*)\\
&+C\epsilon\int_1^{\xi_0(t_*)}\left(2\mu+\xi_0(t_*)-\eta\right)^{-1}v_B(\eta,0)d\eta.
\end{aligned}\label{6.24}\end{equation}
From (\ref{6.19}) and in the same manner, there is also
\begin{equation}\begin{aligned}
&\left|\frac{w_B}{c^\frac{1}{2}}\left(1,t_*\right)-\frac{w_B}{c^\frac{1}{2}}\left(\xi_0(t_*),0\right)\right|\\
\leq& C\epsilon\mathcal{V}(\xi_0(t_*))+C\epsilon\Psi(t_*)+C\epsilon\int_1^{\xi_0(t_*)}\left(2\mu+\xi_0(t_*)-\eta\right)^{-1}v_B(\eta,0)d\eta.
\end{aligned}\label{6.25}\end{equation}
Multiplying (\ref{6.24}) by $\xi_0(t_*)$ yields
\begin{equation}\begin{aligned}
&\xi_0(t_*)\left(v_B(\xi_0(t_*),0)-\left|\frac{w_B}{c^\frac{1}{2}}(\xi_0(t_*),0)\right|\right)\\
\lesssim&\epsilon\xi_0(t_*)\left(\mathcal{V}(\xi_0(t_*))+\Psi(t_*)\right)\\
&+\epsilon\xi_0(t_*)\int_1^{\xi_0(t_*)}\left(2\mu+\xi_0(t_*)-\eta\right)^{-1}\left(v_B(\eta,0)-\left|\frac{w_B}{c^\frac{1}{2}}(\eta,0)\right|\right)d\eta\\
&+\epsilon\xi_0(t_*)\int_1^{\xi_0(t_*)}\left(2\mu+\xi_0(t_*)-\eta\right)^{-1}\left|\frac{w_B}{c^\frac{1}{2}}(\eta,0)\right|d\eta.
\end{aligned}\label{6.26}\end{equation}
By recalling that (\ref{3.10})(\ref{3.11}), we obtain 
\begin{equation}\log\xi_0(t_*)\leq\log\xi_0(T)\leq\log(1+\overline{c}T)\leq\frac{\kappa}{\epsilon},
\label{6.27}\end{equation} 
and thus 
$$\begin{aligned}
&\epsilon\xi_0(t_*)\int_1^{\xi_0(t_*)}\left(2\mu+\xi_0(t_*)-\eta\right)^{-1}\eta^{-1}d\eta\\
=&\frac{\epsilon\xi_0(t_*)}{\xi_0(t_*)+2\mu}\left(\log\xi_0(t_*)+\log\frac{2\mu-1+\xi_0(t_*)}{2\mu}\right)\\
\lesssim&\kappa.
\end{aligned}$$
Therefore, taking maximum in (\ref{6.26}) over $t_*\in[0,T]$ yields
$$\begin{aligned}
&\max_{t_*\in[0,T]}\left\{\xi_0(t_*)\left(v_B(\xi_0(t_*),0)-\left|\frac{w_B}{c^\frac{1}{2}}(\xi_0(t_*),0)\right|\right)\right\}\\
\lesssim&\epsilon\max_{t_*\in[0,T]}\xi_0(t_*)\mathcal{V}(\xi_0(t_*))+\epsilon\max_{t_*\in[0,T]}\xi_0(t_*)\Psi(t_*)+\kappa\max_{t_*\in[0,T]}\left\{\xi_0(t_*)\left|\frac{w_B}{c^\frac{1}{2}}(\xi_0(t_*),0)\right|\right\}\\
&+\kappa\max_{t_*\in[0,T]}\left\{\xi_0(t_*)\left(v_B(\xi_0(t_*),0)-\left|\frac{w_B}{c^\frac{1}{2}}(\xi_0(t_*),0)\right|\right)\right\}.
\end{aligned}$$
We remark here that since $\xi_0$ is a bijection from $[0,T]$ to $[1,\xi_0(T)]$ it holds
$$\max_{t_*\in[0,T]}\xi_0(t_*)\mathcal{V}(\xi_0(t_*))=\max_{\eta\in[1,\xi(T)]}\eta\mathcal{V}(\eta),$$
$$\max_{t_*\in[0,T]}\left\{\xi_0(t_*)\left|\frac{w_B}{c^\frac{1}{2}}(\xi_0(t_*),0)\right|\right\}=\max_{\eta\in[1,\xi(T)]}\left\{\eta\left|\frac{w_B}{c^\frac{1}{2}}(\eta,0)\right|\right\}.$$
By choosing $\kappa_0$ small, we arrive at
\begin{equation}\begin{aligned}
&\max_{t_*\in[0,T]}\left\{\xi_0(t_*)\left(v_B(\xi_0(t_*),0)-\left|\frac{w_B}{c^\frac{1}{2}}(\xi_0(t_*),0)\right|\right)\right\}\\
\lesssim&\epsilon\max_{t_*\in[0,T]}\xi_0(t_*)\mathcal{V}(\xi_0(t_*))+\epsilon\max_{t_*\in[0,T]}\xi_0(t_*)\Psi(t_*)+\kappa\max_{t_*\in[0,T]}\left\{\xi_0(t_*)\left|\frac{w_B}{c^\frac{1}{2}}(\xi_0(t_*),0)\right|\right\},
\end{aligned}\label{6.28}\end{equation}
and it follows
\begin{equation}\begin{aligned}
&\max_{t_*\in[0,T]}\xi_0(t_*)v_B(\xi_0(t_*),0)\\
\lesssim&\epsilon\max_{t_*\in[0,T]}\xi_0(t_*)\mathcal{V}(\xi_0(t_*))+\epsilon\max_{t_*\in[0,T]}\xi_0(t_*)\Psi(t_*)+\max_{t_*\in[0,T]}\left\{\xi_0(t_*)\left|\frac{w_B}{c^\frac{1}{2}}(\xi_0(t_*),0)\right|\right\}.
\end{aligned}\label{6.29}\end{equation}
Furthermore, multiplying (\ref{6.25}) by $\xi_0(t_*)$, taking maximum over $t_*\in[0,T]$ and applying the bound (\ref{6.29}) to the $v_B$ term gives
\begin{equation}\begin{aligned}
&\max_{t_*\in[0,T]}\left\{\xi_0(t_*)\left(\left|\frac{w_B}{c^\frac{1}{2}}\left(1,t_*\right)-\frac{w_B}{c^\frac{1}{2}}\left(\xi_0(t_*),0\right)\right|\right)\right\}\\
\lesssim&\epsilon\max_{t_*\in[0,T]}\xi_0(t_*)\mathcal{V}(\xi_0(t_*))+\epsilon\max_{t_*\in[0,T]}\xi_0(t_*)\Psi(t_*)+\kappa\max_{t_*\in[0,T]}\left\{\xi_0(t_*)\left|\frac{w_B}{c^\frac{1}{2}}(\xi_0(t_*),0)\right|\right\},
\end{aligned}\label{6.30}\end{equation}
which proves Proposition \ref{prop 6.0}.
\subsection{The decay estimate of $\psi(1,t)$.}
Recall from (\ref{5.4})(\ref{5.5}) that 
$W(s)=\left[aw_B|_{x=0},aw_B|_{x=0}\right]^{T}(s)$ and that $a:=\left.\left(\frac{\tilde{f}^\prime(R)}{c}-\frac{\rho\partial_tu}{c}\right)\right|_{x=0}$ with $\left|a(t)-\frac{\tilde{f}^\prime(1)}{c_0}\right|\lesssim\epsilon$. We write
$$\begin{aligned}
aw_B|_{x=0}(s)=&\frac{\tilde{f}^\prime(1)}{c_0}w_B(\xi_0(s),0)+\frac{\tilde{f}^\prime(1)}{c_0}\left(w_B(1,s)-w_B(\xi_0(s),0)\right)+\left(a(s)-\frac{\tilde{f}^\prime(1)}{c_0}\right)w_B(1,s).
\end{aligned}$$
The last two terms on the right-hand side are regarded as perturbation, and from (\ref{6.29})(\ref{6.30}) they satisfy
\begin{equation}\begin{aligned}
&\left|\int_0^t\exp\left(\Lambda_i(t-s)\right)\left(w_B(1,s)-w_B(\xi_0(s),0)\right)ds\right|\\
\lesssim&\left(\int_0^t\exp\left(\text{Re} \Lambda_i(t-s)\right)\xi_0(s)^{-1}ds\right)\left(\kappa\max_{s\in[0,T]}\left\{\xi_0(s)\left|\frac{w_B}{c^\frac{1}{2}}(\xi_0(s),0)\right|\right\}\right.\\
&\left.+\epsilon\max_{s\in[0,T]}\xi_0(s)\mathcal{V}(\xi_0(s))+\epsilon\max_{s\in[0,T]}\xi_0(s)\Psi(s)\right)\\
\lesssim&\frac{1}{1+\underline{c}t}\left(\kappa\max_{s\in[0,T]}\left\{\xi_0(s)\left|\frac{w_B}{c^\frac{1}{2}}(\xi_0(s),0)\right|\right\}+\epsilon\max_{s\in[0,T]}\xi_0(s)\mathcal{V}(\xi_0(s))\right.\\
&\left.+\epsilon\max_{s\in[0,T]}\xi_0(s)\Psi(s)\right),
\end{aligned}\label{6.31}\end{equation}
\begin{equation}\begin{aligned}
&\left|\int_0^t\exp\left(\Lambda_i(t-s)\right)\left(a(s)-\frac{\tilde{f}^\prime(1)}{c_0}\right)w_B(1,s)ds\right|\\
\lesssim&\epsilon\left(\int_0^t\exp\left(\text{Re}\Lambda_i(t-s)\right)\xi_0(s)^{-1}ds\right)\left(\max_{s\in[0,T]}\left\{\xi_0(s)\left|\frac{w_B}{c^\frac{1}{2}}(\xi_0(s),0)\right|\right\}\right.\\
&\left.+\epsilon\max_{s\in[0,T]}\xi_0(s)\mathcal{V}(\xi_0(s))+\epsilon\max_{s\in[0,T]}\xi_0(s)\Psi(s)\right)\\
\lesssim&\frac{1}{1+\underline{c}t}\left(\epsilon\max_{s\in[0,T]}\left\{\xi_0(s)\left|\frac{w_B}{c^\frac{1}{2}}(\xi_0(s),0)\right|\right\}+\epsilon^2\max_{s\in[0,T]}\xi_0(s)\mathcal{V}(\xi_0(s))\right.\\
&\left.+\epsilon^2\max_{s\in[0,T]}\xi_0(s)\Psi(s)\right),
\end{aligned}\label{6.32}\end{equation}
where $i=1,2$ and we used (\ref{3.10}) and $\text{Re}\Lambda_i<0$ to deduce the inequality
$$\int_0^t\exp\left(\Lambda_i(t-s)\right)\xi_0(s)^{-1}ds\leq\int_0^t\exp\left(\text{Re}\Lambda_i(t-s)\right)\left(1+\underline{c}s\right)^{-1}ds\lesssim\frac{1}{1+\underline{c}t}.$$
Using the above inequality again, we control the second term on the right-hand side of (\ref{5.6}) by
\begin{equation}\begin{aligned}
&\left|\int_0^t\exp\left[\begin{matrix}
\Lambda_1(t-s) & \\ & \Lambda_2(t-s)
\end{matrix}\right]\left(\Delta Y\right)(s)ds\right|\\
\lesssim&\epsilon\left(\int_0^t\exp\left[\max_{i=1,2}\left\{\text{Re}\Lambda_i\right\}(t-s)\right]\left(1+\underline{c}s\right)^{-1}(s)ds\right)\left(\max_{s\in[0,t]}(1+\underline{c}s)\left(|\partial_t\psi(1,s)|+|\psi(1,s)|\right)\right)\\
\lesssim&\frac{\epsilon}{1+\underline{c}t}\left(\max_{s\in[0,t]}(1+\underline{c}s)\left(|\partial_t\psi(1,s)|+|\psi(1,s)|\right)\right)
\end{aligned}\label{6.33}\end{equation}
To eliminate the $\Psi(s)$ terms in (\ref{6.31})(\ref{6.32}) by $\psi(1,t)$ and $\partial_t\psi(1,t)$, we recall the notation (\ref{calR}) and compute
\begin{equation}\begin{aligned}
&\xi_0(t)\Psi(t)\\
\leq&\xi_0(t)\int_0^t\left(1+\frac{\underline{c}}{2}(t-s)\right)^{-1}|\mathcal{R}(s)|ds\\
&+\xi_0(t)\left(\int_0^t\left(1+\frac{\underline{c}}{2}(t-s)\right)^{-1}(1+\underline{c}s)^{-1}ds\right)\max_{s\in[0,t]}(1+\underline{c}s)|\partial_t\psi(1,s)-\mathcal{R}(s)|\\
&+\epsilon\xi_0(t)\left(\int_0^t\left(1+\frac{\underline{c}}{2}(t-s)\right)^{-4}(1+\underline{c}s)^{-1}ds\right)\max_{s\in[0,t]}(1+\underline{c}s)|\psi(1,s)|\\
&+\epsilon^2\xi_0(t)\left(\int_0^t\left(1+\frac{\underline{c}}{2}(t-s)\right)^{-1}(1+\underline{c}s)^{-1}ds\right)\max_{s\in[0,t]}(1+\underline{c}s)|\psi(1,s)|\\
\lesssim&\xi_0(t)\int_0^t\left(1+\frac{\underline{c}}{2}(t-s)\right)^{-1}|\mathcal{R}(s)|ds+\frac{\kappa}{\epsilon}\max_{s\in[0,t]}(1+\underline{c}s)|\partial_t\psi(1,s)-\mathcal{R}(s)|\\
&+\epsilon\max_{s\in[0,t]}(1+\underline{c}s)|\psi(1,s)|+\epsilon\kappa\max_{s\in[0,t]}(1+\underline{c}s)|\psi(1,s)|.
\end{aligned}\label{6.34}\end{equation}
Now applying (\ref{6.31}-\ref{6.34}) to (\ref{5.6}) yields
$$\begin{aligned}
&\left|Y(t)-Y_0(t)\right|\\
\lesssim&\frac{1}{1+\underline{c}t}\left\{\kappa\max_{s\in[0,T]}\left\{\xi_0(s)\left|\frac{w_B}{c^\frac{1}{2}}(\xi_0(s),0)\right|\right\}+\epsilon\max_{s\in[0,T]}\xi_0(s)\mathcal{V}(\xi_0(s))\right.\\
&\left.+\epsilon\xi_0(t)\int_0^t\left(1+\frac{\underline{c}}{2}(t-s)\right)^{-1}|\mathcal{R}(s)|ds+\kappa\max_{s\in[0,t]}(1+\underline{c}s)|\partial_t\psi(s)-\mathcal{R}(s)|\right.\\
&\left.+\epsilon^2\max_{s\in[0,t]}(1+\underline{c}s)|\psi(s)|\right\},
\end{aligned}$$
where $Y_0(t)$ denotes the principle part of $Y(t)$, namely
$$Y_0(t):=\exp\left[\begin{matrix}
\Lambda_1t & \\ & \Lambda_2t  
\end{matrix}\right]Y(0)-\frac{\tilde{f}^\prime(1)}{c_0}\int_0^t\exp\left[\begin{matrix}
\Lambda_1(t-s) & \\ & \Lambda_2(t-s)
\end{matrix}\right]\left[
\begin{matrix}
w_B(\xi_0(s),0) \\ w_B(\xi_0(s),0)
\end{matrix}\right]ds.$$
This automatically yields $$\mathcal{R}(t)=\left[\begin{matrix}
\frac{\Lambda_1}{\Lambda_1-\Lambda_2} & -\frac{\Lambda_2}{\Lambda_1-\Lambda_2}
\end{matrix}\right]Y_0(t),\quad \partial_t\psi(t)-\mathcal{R}(t)=\left[\begin{matrix}
\frac{\Lambda_1}{\Lambda_1-\Lambda_2} & -\frac{\Lambda_2}{\Lambda_1-\Lambda_2}
\end{matrix}\right](Y(t)-Y_0(t)).$$
Then the above inequality can be rewritten as
\begin{equation}\begin{aligned}
&\left|Y(t)-Y_0(t)\right|\\
\lesssim&\frac{1}{1+\underline{c}t}\left\{\kappa\max_{s\in[0,T]}\left\{\xi_0(s)\left|\frac{w_B}{c^\frac{1}{2}}(\xi_0(s),0)\right|\right\}+\epsilon\max_{s\in[0,T]}\xi_0(s)\mathcal{V}(\xi_0(s))\right.\\
&\left.+\epsilon\xi_0(t)\int_0^t\left(1+\frac{\underline{c}}{2}(t-s)\right)^{-1}|Y_0(s)|ds+\epsilon^2\max_{s\in[0,t]}(1+\underline{c}s)|Y_0(s)|\right.\\
&\left.+(\kappa+\epsilon^2)\max_{s\in[0,t]}(1+\underline{c}s)|Y(s)-Y_0(s)|\right\}.
\end{aligned}\label{6.35}\end{equation} 
Multiplying (\ref{6.35}) by $1+\underline{c}t$ and taking maximum over $[0,T]$ yields
$$\begin{aligned}
&\max_{t\in[0,T]}(1+\underline{c}t)|Y(t)-Y_0(t)|\\
\lesssim&\kappa\max_{s\in[0,T]}\left\{\xi_0(s)\left|\frac{w_B}{c^\frac{1}{2}}(\xi_0(s),0)\right|\right\}+\epsilon\max_{s\in[0,T]}\xi_0(s)\mathcal{V}(\xi_0(s))+\epsilon^2\max_{s\in[0,T]}(1+\underline{c}s)|Y_0(s)|\\
&+\epsilon\max_{t\in[0,T]}\left\{\xi_0(t)\int_0^t\left(1+\frac{\underline{c}}{2}(t-s)\right)^{-1}|Y_0(s)|ds\right\}\\
&+(\kappa+\epsilon^2)\max_{t\in[0,T]}(1+\underline{c}t)|Y(t)-Y_0(t)|.
\end{aligned}$$
For small enough $\kappa_0$ and $\epsilon_0$, this shows
$$\begin{aligned}
&\max_{t\in[0,T]}(1+\underline{c}t)|Y(t)-Y_0(t)|\\
\lesssim&\kappa\max_{s\in[0,T]}\left\{\xi_0(s)\left|\frac{w_B}{c^\frac{1}{2}}(\xi_0(s),0)\right|\right\}+\epsilon\max_{s\in[0,T]}\xi_0(s)\mathcal{V}(\xi_0(s))\\
&+\epsilon\max_{t\in[0,T]}\left\{\xi_0(t)\int_0^t\left(1+\frac{\underline{c}}{2}(t-s)\right)^{-1}|Y_0(s)|ds\right\}+\epsilon^2\max_{s\in[0,T]}(1+\underline{c}s)|Y_0(s)|.
\end{aligned}$$
For the last two terms, we have
\begin{equation}\begin{aligned}
|Y_0(t)|\lesssim&\exp\left(\max_{i=1,2}\left\{\text{Re}\Lambda_i\right\}t\right)|Y(0)|\\&+\int_0^t\exp\left(\max_{i=1,2}\left\{\text{Re}\Lambda_i\right\}(t-s)\right)\xi_0(s)^{-1}ds\max_{s\in[0,T]}\left\{\xi_0(s)\left|\frac{w_B}{c^\frac{1}{2}}(\xi_0(s),0)\right|\right\}\\
\lesssim&\exp\left(\max_{i=1,2}\left\{\text{Re}\Lambda_i\right\}t\right)|Y(0)|+\frac{1}{1+\underline{c}t}\max_{s\in[0,T]}\left\{\xi_0(s)\left|\frac{w_B}{c^\frac{1}{2}}(\xi_0(s),0)\right|\right\},
\end{aligned}\label{6.38}\end{equation}
and thus
\begin{equation}\begin{aligned}
&\epsilon\max_{t\in[0,T]}\left\{\xi_0(t)\int_0^t\left(1+\frac{\underline{c}}{2}(t-s)\right)^{-1}|Y_0(s)|ds\right\}+\epsilon^2\max_{s\in[0,T]}(1+\underline{c}s)|Y_0(s)|\\
\lesssim&\epsilon|Y(0)|+\epsilon^2\max_{s\in[0,T]}\left\{\xi_0(s)\left|\frac{w_B}{c^\frac{1}{2}}(\xi_0(s),0)\right|\right\}\\
&+\epsilon\max_{t\in[0,T]}\left\{\xi_0(t)\int_0^t\left(1+\frac{\underline{c}}{2}(t-s)\right)^{-1}\left(1+\underline{c}s\right)^{-1}ds\right\}\max_{s\in[0,T]}\left\{\xi_0(s)\left|\frac{w_B}{c^\frac{1}{2}}(\xi_0(s),0)\right|\right\}\\
\lesssim&\epsilon|Y(0)|+(\kappa+\epsilon^2)\max_{s\in[0,T]}\left\{\xi_0(s)\left|\frac{w_B}{c^\frac{1}{2}}(\xi_0(s),0)\right|\right\}.
\end{aligned}\label{6.39}\end{equation}
Here we used that
$$\begin{aligned}
\xi_0(t)\int_0^t\left(1+\frac{\underline{c}}{2}(t-s)\right)^{-1}\left(1+\underline{c}s\right)^{-1}ds
=\xi_0(t)\frac{2}{3\underline{c}+\underline{c}^2t}\left(\log(1+\underline{c}t)+\log\left(1+\frac{\underline{c}}{2}t\right)\right)
\lesssim\frac{\kappa}{\epsilon}.
\end{aligned}$$
Therefore, by noting that $|w_B(\eta)|\leq\mathcal{V}(\eta)$, 
\begin{equation}\begin{aligned}
&\max_{t\in[0,T]}(1+\underline{c}t)|Y(t)-Y_0(t)|\\
\lesssim&(\kappa+\epsilon^2)\max_{s\in[0,T]}\left\{\xi_0(s)\left|\frac{w_B}{c^\frac{1}{2}}(\xi_0(s),0)\right|\right\}+\epsilon\max_{s\in[0,T]}\xi_0(s)\mathcal{V}(\xi_0(s))+\epsilon|Y(0)|\\
\lesssim&\kappa\max_{s\in[0,T]}\left\{\xi_0(s)\left|\frac{w_B}{c^\frac{1}{2}}(\xi_0(s),0)\right|\right\}+\epsilon\max_{s\in[0,T]}\xi_0(s)\mathcal{V}(\xi_0(s))+\epsilon|Y(0)|.
\end{aligned}\label{6.40}\end{equation} 
By substituting the last three terms on the right-hand side of (\ref{6.35}) by the bounds in (\ref{6.39})(\ref{6.40}), we obtain
\begin{equation}\begin{aligned}
&\left|Y(t)-Y_0(t)\right|\\
\lesssim&\frac{1}{1+\underline{c}t}\left\{\kappa\max_{s\in[0,T]}\left\{\xi_0(s)\left|\frac{w_B}{c^\frac{1}{2}}(\xi_0(s),0)\right|\right\}+\epsilon\max_{s\in[0,T]}\xi_0(s)\mathcal{V}(\xi_0(s))+\epsilon|Y(0)|\right\},
\end{aligned}\label{6.41}\end{equation}
which proves Proposition \ref{prop 6.01}.
\section{Close the bootstrap}\label{sec7}
With the help of the bounds obtained in the previous sections, we shall complete the bootstrap argument in this section.
\subsection{Proof of almost global existence}
To prove the bound (\ref{3.12}), recall the energy identity (\ref{kee}), and we obtain by dropping the iteration subscripts that
\begin{equation}\begin{aligned}
&\frac{1}{2}\frac{d}{dt}\left[\int_0^\infty((\partial_t^ju)^2+c^2(\partial_t^jq)^2)dx+(c^2r^2)|_{x=0}f'(R)^{-1}|\partial_t^jf(R)|^2\right]=\sum_{l=1}^5nl_l^j(t),\end{aligned}\label{7.1}\end{equation}
where we denote
$$nl_1^j(t):=\int_0^\infty\left(\frac{1}{2}\partial_tc^2(\partial_t^jq)^2-r^2\partial_xc^2\partial_t^ju\partial_t^jq\right)dx,$$
$$nl_2^j(t):=\int_0^\infty[\partial_t^j,c^2r^2\partial_x]q
\partial_t^judx,\quad nl_3^j(t):=\int_0^\infty c^2[\partial_t^j,\partial_xr^2]u_k\partial_t^jqdx,$$
$$nl_4^j(t):=\frac{1}{2}\partial_t\left[(c^2r^2)|_{x=0}f'(R)^{-1}\right]|\partial_t^jf(R)|^2,$$
$$nl_5^j(t):=-(c^2r^2)|_{x=0}[\partial_t^j,f'(R)^{-1}]\partial_tf(R)\partial_t^jf(R).$$
From (\ref{4.8}-\ref{4.12}) and (\ref{3.7}-\ref{3.9}), it follows 
\begin{equation}
\sum_{j=0}^2\left(|\partial_t^ju|^2+|\partial_t^jq|^2\right)\simeq\sum_{j=0}^2\left(|\partial_t^{j+1}\varphi|^2+|cr^2\partial_x\partial_t^j\varphi|^2\right),
\label{7.2}\end{equation}
and thus
$$\begin{aligned}
&\sum_{j=0}^2\left\{\int_0^T\int_0^\infty\xi^{-1}\left((\partial_t^{1+j}\varphi)^2+(cr^2\partial_x\partial_t^j\varphi)^2\right)dxdt+\int_0^T\int_0^\infty\xi^{-3}\left(\partial_t^j\varphi\right)^2dxdt\right\}\\
\gtrsim&\sum_{j=0}^2\int_0^T\int_0^\infty\xi^{-1}\left((\partial_t^ju)^2+(\partial_t^j q)^2\right)dxdt,
\end{aligned}$$
with which we further deduce from (\ref{4.1}) that
\begin{equation}\begin{aligned}
\sum_{j=0}^2\int_0^T\int_0^\infty\xi^{-1}\left((\partial_t^ju)^2+(\partial_t^j q)^2\right)dxdt
\lesssim\max\left\{\log_2 T,1\right\}\sum_{j=0}^2\max_{0\leq t\leq T}e^j(t)+\int_0^T\varphi^2|_{x=0}dt
\end{aligned}\label{7.3}\end{equation}
Using the above inequality with (\ref{3.7})(\ref{3.11}), we estimate 
\begin{equation}\begin{aligned}
\int_0^T |nl_1^j(t)|dt\lesssim&\left(\|\xi\partial_tc^2\|_{L^\infty}+\|\xi r^2\partial_xc^2\|_{L^\infty}\right)\int_0^T\int_0^\infty\xi^{-1}\left((\partial_t^ju)^2+(\partial_t^jq)^2\right)dxdt\\
\lesssim&\epsilon\left(\max\left\{\log_2T,1\right\}\sum_{j=0}^2\max_{0\leq t\leq T}e^j(t)+\int_0^T\varphi^2|_{x=0}dt\right)\\
\lesssim&\kappa\sum_{j=0}^2\max_{0\leq t\leq T}e^j(t)+\epsilon\int_0^T\varphi^2|_{x=0}dt.
\end{aligned}\label{7.4}
\end{equation}
Next, we compute
$$[\partial_t,r^2c^2\partial_x]q=\partial_t(r^2c^2)\partial_xq=\frac{\partial_t(c^2r^2)}{c^2r^2}\partial_tu,$$
$$\begin{aligned}
&[\partial_t^2,r^2c^2\partial_x]q=[\partial_t^2,r^2c^2]\partial_xq=[\partial_t^2,r^2c^2](r^{-2}c^{-2}\partial_tu)\\
=&2\frac{\partial_t(r^2c^2)}{r^2c^2}\partial_t^2u+\frac{\partial_t^2(r^2c^2)}{r^2c^2}\partial_tu+2\partial_t(r^2c^2)\partial_t(r^{-2}c^{-2})\partial_tu,
\end{aligned}$$
$$\begin{aligned}
\frac{\partial_t^2(r^2c^2)}{r^2c^2}=&\frac{\partial_t^2c^2}{c^2}+2\frac{\partial_tr^2\partial_tc^2}{r^2c^2}+\frac{\partial_t^2r^2}{r^2}\\
=&-(\gamma+1)\rho\partial_t^2q+(\gamma+1)(\gamma+2)\rho^2(\partial_tq)^2+2\frac{\partial_tu}{r}+2\frac{u^2}{r^2}+8\frac{u\partial_tc}{rc}.
\end{aligned}$$
Hence by (\ref{3.6})(\ref{3.7}), 
$$\begin{aligned}
\left|[\partial_t^2,r^2c^2\partial_x]q\right|\lesssim\epsilon\xi^{-1}\left(|\partial_t^2u|+|\partial_t^2q|+|\partial_tu|\right).
\end{aligned}$$
From (\ref{7.3})(\ref{3.11}), it follows
\begin{equation}\begin{aligned}
\int_0^T|nl_2^j(t)|dt\lesssim&\epsilon\sum_{i=0}^j\int_0^T\int_0^\infty\xi^{-1}\left((\partial_t^ju)^2+(\partial_t^jq)^2\right)dxdt\\
\lesssim&\kappa\sum_{j=0}^2\max_{0\leq t\leq T}e^j(t)+\epsilon\int_0^T\varphi^2|_{x=0}dt.
\end{aligned}\label{7.5}\end{equation}
Similarly, to control $nl_3^j$, we compute
$$[\partial_t,\partial_xr^2]u=\partial_x[\partial_t,r^2]u=\partial_x(2ru^2)=4\frac{u}{r}\partial_tq-6\frac{u^2}{\rho r^2},$$
$$\begin{aligned}
\relax[\partial_t^2,\partial_x r^2]u=\partial_x[\partial_t^2,r^2]u
=\partial_x[2u^3+4ru\partial_tu]=\partial_x\left[4\frac{u}{r}\partial_t(r^2u)-6u^3\right],
\end{aligned}$$
$$\begin{aligned}
\partial_x\left[\frac{u}{r}\partial_t(r^2u)\right]=&\frac{u}{r}\partial_t^2q+4\partial_x\left(\frac{u}{r}\right)\partial_t(r^2u)\\
=&\frac{u}{r}\partial_t^2q+r^{-1}\left(\partial_tq-3\rho^{-1}r^{-1}u\right)\left(\partial_tu+2r^{-1}u^2\right),
\end{aligned}$$
$$\partial_xu^3=3\frac{u^2}{r^2}\left(\partial_tq-2\rho^{-1}r^{-1}u\right).$$
By (\ref{7.3})(\ref{3.11}), we then obtain
\begin{equation}\begin{aligned}
\int_0^T|nl_3^j(t)|dt\lesssim&\epsilon\sum_{i=0}^2\int_0^T\int_0^\infty\xi^{-1}\left((\partial_t^ju)^2+(\partial_t^jq)^2\right)dxdt\\
\lesssim&\kappa\sum_{j=0}^2\max_{0\leq t\leq T}e^j(t)+\epsilon\int_0^T\varphi^2|_{x=0}dt.
\end{aligned}\label{7.6}\end{equation}
Recall that $f(R)=q|_{x=0}$ and $\rho^{\gamma-1}-1=-\left(\frac{Ca\gamma}{2(\gamma-1)}\right)^{-1}\left(\partial_x\varphi-\frac{1}{2}u^2\right)$, from which we write
$$q=\rho^{-1}-1=-\frac{\rho^{-1}-1}{\rho^{\gamma-1}-1}\left(\frac{Ca\gamma}{2(\gamma-1)}\right)^{-1}\left(\partial_x\varphi-\frac{1}{2}u^2\right).$$
In view of (\ref{3.6})(\ref{3.7}), this shows
$$\begin{aligned}
|f(R)|&\lesssim\left|\left.\left(\partial_t\varphi-\frac{1}{2}u^2\right)\right|_{x=0}\right|\leq\left|\partial_t\varphi|_{x=0}\right|+C\epsilon\left|\left(cr^2\partial_x\varphi\right)|_{x=0}\right|\\
\leq&\left|\left.\partial_t\varphi\right|_{x=0}\right|+C\epsilon\left|\left.\left(cr^2\varphi+\frac{c}{\rho r}\varphi\right)\right|_{x=0}\right|+C\epsilon|\varphi|_{x=0}|.
\end{aligned}$$
For derivatives of $f(R)$, we have
$$\partial_tq=-(\gamma-1)^{-1}\rho^{-\gamma}\partial_t\left(\rho^{\gamma-1}-1\right)=\left(\frac{Ca\gamma}{2}\right)^{-1}\rho^{-\gamma}\left(\partial_t^2\varphi-u\partial_tu\right),$$ 
$$\begin{aligned}
|\partial_tf(R)|\lesssim&|\partial_t^2\varphi|_{x=0}|+C\epsilon\left|\partial_tu|_{x=0}\right|\\
=&|\partial_t^2\varphi|_{x=0}|+C\epsilon\left|\partial_t\left.\left(\rho r^2\partial_x\varphi\right)\right|_{x=0}\right|\\
\leq& |\partial_t^2\varphi|_{x=0}|+C\epsilon\left|\rho r^2\partial_x\partial_t\varphi|_{x=0}\right|+C\epsilon\left|\left.\frac{\partial_t(\rho r^2)}{\rho r^2}\right|_{x=0}\right|\cdot\left|\left(\rho r^2\partial_x\varphi\right)|_{x=0}\right|\\
\leq&|\partial_t^2\varphi|_{x=0}|+C\epsilon\left|\left.\left(cr^2\partial_x\partial_t\varphi+\frac{c}{\rho r}\partial_t\varphi\right)\right|_{x=0}\right|+C\epsilon\left|\partial_t\varphi|_{x=0}\right|\\
&+C\epsilon\left|\left.\left(cr^2\partial_x\varphi+\frac{c}{\rho r}\varphi\right)\right|_{x=0}\right|+C\epsilon|\varphi_{x=0}|.
\end{aligned}$$
Similarly, by separating the top derivatives and using (\ref{3.6})(\ref{3.7}), we continue and obtain
$$\begin{aligned}
|\partial_t^2f(R)|
=&|\partial_t^2q|_{x=0}|\\
=&\left(\frac{Ca\gamma}{2}\right)^{-1}\left|\left.\partial_t\left(\rho^{-\gamma}\left(\partial_t^2\varphi-u\partial_tu\right)\right)\right|_{x=0}\right|\\
\lesssim&\left|\left.\left(\partial_t^3\varphi-(\partial_tu)^2-u\partial_t^2u\right)\right|_{x=0}\right|+\left|\left.\partial_t\rho^{-\gamma}\right|_{x=0}\right|\cdot\left|\left.\left(\partial_t^2\varphi-u\partial_tu\right)\right|_{x=0}\right|\\
\lesssim&\left|\left.\left(\partial_t^3\varphi-(\partial_tu)^2-u\partial_t^2u\right)\right|_{x=0}\right|\\
&+\epsilon\left(\sum_{j=0}^2\left|\partial_t^j\varphi|_{x=0}\right|+\sum_{j=0}^1\left|\left.\left(cr^2\partial_x\partial_t^j\varphi+\frac{c}{\rho r}\partial_t^j\varphi\right)\right|_{x=0}\right|\right)\\
\lesssim&|\partial_t^3\varphi|_{x=0}|+\epsilon|\partial_t^2u|_{x=0}|+\epsilon\left(\sum_{j=0}^2\left|\partial_t^j\varphi|_{x=0}\right|+\sum_{j=0}^1\left|\left.\left(cr^2\partial_x\partial_t^j\varphi+\frac{c}{\rho r}\partial_t^j\varphi\right)\right|_{x=0}\right|\right),
\end{aligned}$$
$$\begin{aligned}
|\partial_t^2u|_{x=0}|
=&\left|\left.\partial_t^2\left(\rho r^2\partial_x\varphi\right)\right|_{x=0}\right|\\
=&\left|\left.\left(\rho r^2\partial_x\partial_t^2\varphi+2\partial_t(\rho r^2)\partial_x\partial_t\varphi+\partial_t^2(\rho r^2)\partial_x\varphi\right)\right|_{x=0}\right|\\
\leq&\left|\left.\left(\rho r^2\partial_x\partial_t^2\varphi\right)\right|_{x=0}\right|+C\epsilon\left|\left.\left(\rho r^2\partial_x\partial_t\varphi\right)\right|_{x=0}\right|\\
&+\left|\left.\left(\rho^{-1}\partial_t^2\rho+4\frac{u}{\rho r}\partial_t\rho+\frac{2u^2}{r^2}+\frac{2\partial_tu}{r}\right)\right|_{x=0}\right|\cdot\left|\left.\left(\rho r^2\partial_x\varphi\right)\right|_{x=0}\right|\\
\lesssim&\left|\left.\left(\rho r^2\partial_x\partial_t^2\varphi\right)\right|_{x=0}\right|\\
&+\epsilon\left(\left|\left.\left(\rho r^2\partial_x\partial_t\varphi\right)\right|_{x=0}\right|+\left|\left.\left(\rho r^2\partial_x\varphi\right)\right|_{x=0}\right|+|\partial_t^2\rho|_{x=0}|+|\partial_t\rho|_{x=0}|\right)\\
\lesssim&\left|\left.\left(\rho r^2\partial_x\partial_t^2\varphi\right)\right|_{x=0}\right|\\
&+\epsilon\left(\left|\left.\left(\rho r^2\partial_x\partial_t\varphi\right)\right|_{x=0}\right|+\left|\left.\left(\rho r^2\partial_x\varphi\right)\right|_{x=0}\right|+|\partial_t^2q|_{x=0}|+|\partial_tq|_{x=0}|\right)\\
\lesssim&\sum_{j=0}^2\left|\left.\left(cr^2\partial_x\partial_t^j\varphi+\frac{c}{\rho r}\partial_t^j\varphi\right)\right|_{x=0}\right|+\sum_{j=0}^2\left|\partial_t^j\varphi|_{x=0}\right|+\epsilon|\partial_t^2q|_{x=0}|.
\end{aligned}$$
Combining the above two inequalities yields
$$|\partial_t^2f(R)|\lesssim|\partial_t^3\varphi|_{x=0}|+\epsilon\sum_{j=0}^2\left|\left.\left(cr^2\partial_x\partial_t^j\varphi+\frac{c}{\rho r}\partial_t^j\varphi\right)\right|_{x=0}\right|+\epsilon\sum_{j=0}^2\left|\partial_t^j\varphi|_{x=0}\right|+\epsilon^2|\partial_t^2f(R)|.$$
For small enough $\epsilon$, this gives
$$|\partial_t^2f(R)|\lesssim|\partial_t^3\varphi|_{x=0}|+\epsilon\sum_{j=0}^2\left|\left.\left(cr^2\partial_x\partial_t^j\varphi+\frac{c}{\rho r}\partial_t^j\varphi\right)\right|_{x=0}\right|+\epsilon\sum_{j=0}^2\left|\partial_t^j\varphi|_{x=0}\right|,$$
and thus for $j=0,1,2$
\begin{equation}
|\partial_t^jf(R)|\lesssim|\partial_t^{1+j}\varphi|_{x=0}|+\epsilon\sum_{i=0}^j\left|\left.\left(cr^2\partial_x\partial_t^j\varphi+\frac{c}{\rho r}\partial_t^j\varphi\right)\right|_{x=0}\right|+\epsilon\sum_{i=0}^j\left|\partial_t^j\varphi|_{x=0}\right|
\label{7.7}\end{equation}
We return to the estimate of $nl_4^j$. Since $c^2=c_0^2(1+q)^{-\gamma-1}$ and $q|_{x=0}=f(R)$, there exists a function $f_1$ which is smooth near $R=1$ such that $(cr^2)|_{x=0}f'(R)^{-1}=f_1(R)$, and thus
$$\left|\partial_t\left[(cr^2)|_{x=0}f'(R)^{-1}\right]\right|=\left|\partial_tf_1(R)\right|=\left|f_1^\prime(R)\right||u|_{x=0}|\lesssim\epsilon.$$
It then follows from (\ref{7.7})(\ref{4.1})(\ref{7.2})(\ref{3.11}) that
\begin{equation}\begin{aligned}
\int_0^T|nl_4^j(t)|dt\lesssim&\epsilon\left(\max\left\{\log
_2T,1\right\}\sum_{j=0}^2\max_{0\leq t\leq T}e^j(t)+\int_0^T\varphi^2|_{x=0}dt\right)\\
\lesssim&\kappa\sum_{j=0}^2\max_{0\leq t\leq T}e^j(t)+\epsilon\int_0^T\varphi^2|_{x=0}dt.
\end{aligned}\label{7.8}\end{equation}
The $nl_5^j$ terms can also be treated by direct computation.
$$\left|[\partial_t,f^\prime(R)^{-1}]\partial_tf(R)\cdot\partial_tf(R)\right|=\left|\partial_t(f^\prime(R)^{-1})\right|\cdot|\partial_tf(R)|^2\lesssim\epsilon|\partial_tf(R)|^2,$$
$$[\partial_t^2,f^\prime(R)^{-1}]\partial_tf(R)\cdot\partial_t^2f(R)=2\partial_t(f^\prime(R)^{-1})|\partial_t^2f(R)|^2+\partial_t^2(f^\prime(R))^{-1}\partial_tf(R)\partial_t^2f(R),$$
By Leibniz rule,
$$\begin{aligned}
\partial_t^2(f^\prime(R)^{-1})=&\partial_t^2\left(\frac{1}{f^{\prime}}\circ f^{-1}\circ f(R)\right)=\partial_t\left[\left(\frac{1}{f^\prime}\circ f^{-1}\right)^\prime\circ f(R)\cdot\partial_t f(R)\right]\\
=&\left(\frac{1}{f^\prime}\circ f^{-1}\right)^\prime\circ f(R)\cdot\partial_t^2 f(R)+\left(\frac{1}{f^\prime}\circ f^{-1}\right)^{\prime\prime}\circ f(R)\cdot(\partial_tf(R))^2.
\end{aligned}$$
Since $|\partial_t R|=|u|_{x=0}|\lesssim\epsilon$, it follows
$$\left|[\partial_t^2,f^\prime(R)^{-1}]\partial_tf(R)\cdot\partial_t^2f(R)\right|\lesssim\epsilon\left((\partial_t^2 f(R))^2+(\partial_t f(R))^2\right),$$
and by (\ref{7.7})(\ref{4.1})(\ref{3.11}),
\begin{equation}\begin{aligned}
\int_0^T|nl_5^j(t)|dt
\lesssim&\epsilon\left(\left|\partial_t^{1+j}\varphi|_{x=0}\right|+\epsilon\sum_{i=0}^j\left|\left.\left(cr^2\partial_x\partial_t^j\varphi+\frac{c}{\rho r}\partial_t^j\varphi\right)\right|_{x=0}\right|+\epsilon\sum_{i=0}^j\left|\partial_t^j\varphi|_{x=0}\right|\right)\\
\lesssim&\epsilon\left(\max\left\{\log
_2T,1\right\}\sum_{j=0}^2\max_{0\leq t\leq T}e^j(t)+\int_0^T\varphi^2|_{x=0}dt\right)\\
\lesssim&\kappa\sum_{j=0}^2\max_{0\leq t\leq T}e^j(t)+\epsilon\int_0^T\varphi^2|_{x=0}dt.
\end{aligned}\label{7.9}\end{equation}
Integrating (\ref{7.1}) over $[0,T]$ and using (\ref{7.4})(\ref{7.5})(\ref{7.6})(\ref{7.8})(\ref{7.9}) yields
$$e^j(t)-e^j(0)\lesssim\kappa\sum_{j=0}^2\max_{0\leq t\leq T}e^j(t)+\epsilon\int_0^T\varphi^2|_{x=0}dt,$$
and it follows by taking maximum
$$\sum_{j=0}^2\max_{0\leq t\leq T}e^j(t)\leq\sum_{j=0}^2e^j(0)+C\kappa\sum_{j=0}^2\max_{0\leq t\leq T}e^j(t)+C\epsilon\int_0^T\varphi^2|_{x=0}dt.$$
Hence for small enough $\kappa_0$ and $\kappa\leq\kappa_0$ such that $C\kappa_0\leq\frac{1}{2}$ , we obtain
\begin{equation}
\sum_{j=0}^2\max_{0\leq t\leq T}e^j(t)\leq2\sum_{j=0}^2e^j(0)+C\epsilon\int_0^T\varphi^2|_{x=0}dt.
\label{7.10}\end{equation}
To prove (\ref{3.12}), it remains to bound $\int_0^T\varphi^2|_{x=0}dt$. To this end, we call upon Proposition \ref{prop 6.01} and use the following lemma.
\begin{lem}
Under the condition (\ref{1.18}) and for $\xi\geq1$, we have the bounds
$$|Y(0)|\lesssim\epsilon,\quad \int_1^\infty|w_B(\eta,0)|^2d\eta\lesssim\epsilon^2,$$
$$\max_{\eta\in[1,\xi_0(T)]}\left\{\eta\left|w_B(\eta,0)\right|\right\}\lesssim\epsilon^\frac{1}{2}\tilde{\epsilon}^\frac{1}{2},\quad \max_{\xi\in[1,\xi_0(T)]}\int_1^\xi|w_F(\eta,0)|d\eta\lesssim\epsilon^\frac{1}{2}\left(\epsilon^\frac{1}{2}+\tilde{\epsilon}^\frac{1}{2}\right),$$
$$\xi^{-2}|\psi(\xi,0)|\lesssim\epsilon\xi^{-1},\quad \int_1^\xi\eta^{-3}|\psi(\eta,0)|d\eta\lesssim\epsilon.$$
In particular,
$$\epsilon\max_{\eta\in[1,\xi_0(T)]}\eta\mathcal{V}(\eta)\lesssim\epsilon^\frac{3}{2}\tilde{\epsilon}^\frac{1}{2}+\epsilon^2.$$
\label{lem 7.1}\end{lem}
\begin{proof}
Recall that $\rho r^2\partial_x\varphi=u$. By the bootstrap bound (\ref{3.2})(\ref{3.8}) at $t=0$ and a H\"{o}lder inequality, it holds
$$|\varphi(x,0)|\leq\int_{x}^\infty\left|(1+q_{in})\frac{u_{in}}{r_{in}^2}(y)\right|dy\lesssim\|u_{in}\|_{L^2}\lesssim\epsilon.$$
From (\ref{rho-phi}), we also have
$$|\partial_t\varphi(x,0)|\lesssim\|u_{in}\|_{L^\infty}^2+\|q_{in}\|_{L^\infty}\lesssim\epsilon.$$
From $\psi=r\varphi$, $\partial_t\psi=r\partial_t\varphi+u\varphi$, and $r_{in}|_{\xi=1}=R_{in}\simeq 1$, it follows that
$$|Y(0)|\lesssim|\partial_t\psi|_{\left\{\xi=1,t=0\right\}}|+|\psi|_{\left\{\xi=1,t=0\right\}}|
\lesssim|\partial_t\varphi|_{\left\{x=0,t=0\right\}}|+|\varphi|_{\left\{x=0,t=0\right\}}|\lesssim\epsilon,$$
$$\xi^{-2}|\psi(\xi,0)|=\xi^{-2}r_{in}(x)|\varphi(x,0)|\lesssim\epsilon\xi^{-1},$$
$$\int_1^\xi\eta^{-3}|\psi(\eta,0)|d\eta\lesssim\epsilon\int_1^\xi\eta^{-2}d\eta\lesssim\epsilon.$$
Using (\ref{3.2})(\ref{3.8}) again, we compute
$$\begin{aligned}
\xi|w_B(\xi,0)|=&\xi\left|(\partial_t+cr^2\partial_x)(r\varphi)(x,0)\right|\\
\leq&\xi r_{in}(x)\left(|\partial_t\varphi(x,0)|+|cr^2\partial_x\varphi(x,0)|\right)+\xi\left|\left(u_{in}(x)+\frac{c(x,0)}{\rho_{in}(x)}\right)\varphi(x,0)\right|\\
\lesssim&\xi^2(|q_{in}(x)|+|u_{in}(x)|)+\xi|\varphi(x,0)|.
\end{aligned}$$
Similarly,
$$\xi|w_F(\xi,0)|\lesssim\xi^2(|q_{in}(x)|+|u_{in}(x)|)+\xi|\varphi(x,0)|.$$
Then by condition (\ref{1.18}),
$$\begin{aligned}
\xi^2|u_{in}(x)|\leq\left(2\int_0^\infty\xi^2|u_{in}(\xi)|\left|\partial_x(\xi^2 u_{in})\right|dx\right)^{\frac{1}{2}}
\lesssim\|\tilde{L}_0^1u_{in}\|_{L^2}^\frac{1}{2}\|\xi^2u_{in}\|_{L^2}^\frac{1}{2}
\lesssim\epsilon^\frac{1}{2}\tilde{\epsilon}^\frac{1}{2},
\end{aligned}$$
$$\begin{aligned}
\xi^2|q_{in}(x)|\leq\left(2\int_0^\infty\xi^2|q_{in}(\xi)|\left|\partial_x(\xi^2q_{in})\right|dx\right)^{\frac{1}{2}}
\lesssim\left(\|L_0^1q_{in}\|_{L^2}+\|q_{in}\|_{L^2}\right)^\frac{1}{2}\|\xi^2q_{in}\|_{L^2}^\frac{1}{2}
\lesssim\epsilon^\frac{1}{2}\tilde{\epsilon}^\frac{1}{2},
\end{aligned}$$
$$\xi|\varphi(x,0)|\leq\xi\int_x^\infty(1+q_{in})\frac{|u_{in}|}{r_{in}^2}(y)dy\lesssim\xi\epsilon^\frac{1}{2}\tilde{\epsilon}^\frac{1}{2}\int_x^\infty(1+3y)^{-\frac{4}{3}}dy\lesssim\epsilon^\frac{1}{2}\tilde{\epsilon}^\frac{1}{2}.$$
Hence we obtain $\xi|w_B(\xi,0)|\lesssim\epsilon^\frac{1}{2}\tilde{\epsilon}^\frac{1}{2}$. Next, for the bound on $\int_1^\infty|w_B(\eta,0)|^2d\eta$, we use a change of variable and Hardy inequality to obtain
$$\begin{aligned}
\int_1^\infty|w_B(\eta,0)|^2d\eta\lesssim&\int_1^\infty\left(|\eta q_{in}(x(\eta))|^2+|\eta u_{in}(x(\eta))|^2+|\varphi(x(\eta),0)|^2\right)d\eta\\
\lesssim&\int_0^\infty\left(|q_{in}(x)|^2+|u_{in}(x)|^2+|\xi^{-1}\varphi(x,0)|^2\right)dx\\
\lesssim&\|u_{in}\|_{L^2}^2+\|q_{in}\|_{L^2}^2\\
\lesssim&\epsilon^2.
\end{aligned}$$
It remains to control $\max\limits_{\xi\in[1,\xi_0(T)]}\int_1^\xi|w_F(\eta,0)|d\eta$, for which we use
$$\begin{aligned}
&\max_{\xi\in[1,\xi_0(T)]}\int_1^\xi|w_F(\eta,0)|d\eta\lesssim\left(\int_1^\infty|\eta w_F(\eta,0)|^2d\eta\right)^\frac{1}{2}\\
\lesssim&\left(\int_1^\infty|\eta^2q_{in}(x(\eta))|^2d\eta\right)^\frac{1}{2}+\left(\int_1^\infty|\eta^2u_{in}(x(\eta))|^2d\eta\right)^\frac{1}{2}+\left(\int_1^\infty|\eta\varphi(x(\eta),0)|^2d\eta\right)^\frac{1}{2}.
\end{aligned}$$
By a change of variable and H\"{o}lder, 
$$\left(\int_1^\infty|\eta^2q_{in}(x(\eta))|^2d\eta\right)^\frac{1}{2}=\|\xi q_{in}\|_{L^2}\leq\|q_{in}\|_{L^2}^\frac{1}{2}\|\xi^2q_{in}\|_{L^2}^\frac{1}{2}\leq\epsilon^\frac{1}{2}\tilde{\epsilon}^\frac{1}{2},$$
and similarly,
$$\left(\int_1^\infty|\eta^2u_{in}(x(\eta))|^2d\eta\right)^\frac{1}{2}\leq\epsilon^\frac{1}{2}\tilde{\epsilon}^\frac{1}{2}.$$
For the last term, integration by parts gives
$$\begin{aligned}
\int_1^\infty|\eta\varphi(x(\eta),0)|^2d\eta
\lesssim&|\varphi(0,0)|^2+\int_1^\infty|\eta\varphi(x(\eta),0)||\eta^2\partial_\eta\varphi(x(\eta),0)|d\eta\\
\lesssim&|\varphi(0,0)|^2+\left(\int_1^\infty|\eta\varphi(x(\eta),0)|^2d\eta\right)^\frac{1}{2}\left(\int_1^\infty|\eta^2u_{in}(x(\eta),0)|^2d\eta\right)^\frac{1}{2},
\end{aligned}$$
where we used that $\rho r^2\simeq\xi^2$ by (\ref{3.2})(\ref{3.8}). Hence by Cauchy-Schwartz 
$$\int_1^\infty|\eta\varphi(x(\eta),0)|^2d\eta\lesssim\epsilon^2+\epsilon\tilde{\epsilon}.$$
Collecting the above bounds yields
$$\max_{\xi\in[1,\xi_0(T)]}\int_1^\xi|w_F(\eta,0)|d\eta\lesssim\left(\int_1^\infty|\eta w_F(\eta,0)|^2d\eta\right)^\frac{1}{2}\lesssim\epsilon^\frac{1}{2}\left(\epsilon^\frac{1}{2}+\tilde{\epsilon}^\frac{1}{2}\right).$$
\end{proof}
Now we combine Proposition \ref{prop 6.01} with Lemma \ref{7.1} in view of $\epsilon\leq\epsilon_0$, $\tilde{\epsilon}\leq\tilde{\epsilon}_0$, $\kappa\leq\kappa_0$ to deduce
$$\begin{aligned}
&\int_0^T\varphi^2|_{x=0}dt\\
\leq&C\int_0^T|Y_0(t)|^2dt+C\left(\epsilon|Y(0)|+\kappa\max_{\eta\in[1,\xi_0(T)]}\left\{\eta\left|\frac{w_B}{c^\frac{1}{2}}(\eta,0)\right|\right\}+\epsilon\max_{\eta\in[1,\xi_0(T)]}\eta\mathcal{V}(\eta)\right)^2\\
\leq&C\int_0^T|Y_0(t)|^2dt+C\left(\epsilon^2+\kappa\epsilon^\frac{1}{2}\tilde{\epsilon}^\frac{1}{2}+\epsilon^\frac{3}{2}\tilde{\epsilon}^\frac{1}{2}\right)^2\\
\leq&C\int_0^T|Y_0(t)|^2dt+C\epsilon\left(\epsilon_0^\frac{3}{2}+\kappa_0\tilde{\epsilon}_0^\frac{1}{2}+\epsilon_0\tilde{\epsilon}_0^\frac{1}{2}\right)^2.
\end{aligned}$$
Meanwhile, by Young inequality, Lemma {\ref{lem 6.4} and Lemma {\ref{lem 7.1} again we obtain
$$\begin{aligned}
&\left\|\int_0^t\exp\left(\max_{i=1,2}\left\{\text{Re}\Lambda_i\right\}(t-s)\right)|w_B(\xi_0(s),0)|ds\right\|_{L^2[0,T]}^2\\
\lesssim&\int_0^T|w_B(\xi_0(t),0)|^2dt\lesssim\int_0^{\xi_0(T)}|w_B(\eta,0)|^2d\eta\lesssim\epsilon^2,
\end{aligned}$$
and thus $\int_0^T|Y_0(t)|^2dt\lesssim\epsilon^2$, which further gives
$$C\epsilon\int_0^T\varphi^2|_{x=0}dt\leq C\epsilon^2\left(\epsilon_0+\left(\epsilon_0^\frac{3}{2}+\kappa_0\tilde{\epsilon}_0^\frac{1}{2}+\epsilon_0\tilde{\epsilon}_0^\frac{1}{2}\right)^2\right).$$
Hence for sufficiently small $\epsilon_0,\tilde{\epsilon}_0, \kappa_0$ such that $C\left(\epsilon_0+\left(\epsilon_0^\frac{3}{2}+\kappa_0\tilde{\epsilon}_0^\frac{1}{2}+\epsilon_0\tilde{\epsilon}_0^\frac{1}{2}\right)^2\right)\leq\tilde{B}$, we obtain 
$$C\epsilon\int_0^T\varphi^2|_{x=0}dt\leq\tilde{B}\epsilon^2,$$
which together with (\ref{7.10}) yields (\ref{3.12}). Therefore, the solution can be continued up to $[0,T_\epsilon]$ with $1+\overline{c}T_\epsilon\geq\exp\left(\frac{\kappa}{\epsilon}\right)$ for all $\kappa\leq\kappa_0$. In particular, choosing $\kappa=\kappa_0$ gives $T\geq T_0$.
\subsection{Estimate of bubble radius}
Now assume that $1+\overline{c}T=\exp\left(\frac{\kappa}{\epsilon}\right)$ for $\kappa\leq\kappa_0$ and $t\in[0,T]$.
To retrieve an estimate of $R(t)$, recall the relation (\ref{rho-phi}). For simplicity, denote \begin{equation}
g(q):=-\frac{Ca\gamma}{2(\gamma-1)}\left((1+q)^{1-\gamma}-1\right)
\label{7.11}\end{equation} 
so that
$$g(q)=\partial_t\varphi-\frac{1}{2}u^2=r^{-1}\partial_t\psi-\frac{u}{r^2}\psi-\frac{1}{2}u^2.$$
From the boundary condition (\ref{dbce}), it follows that 
\begin{equation}
F(R)=\partial_t\psi|_{x=0}-\left.\left(\frac{u}{r}\psi+\frac{1}{2}ru^2\right)\right|_{x=0},
\label{7.12}\end{equation}
where
\begin{equation}
F(R)=R\cdot g\circ f(R).
\label{7.13}\end{equation}
Write $$\partial_t\psi|_{x=0}=\frac{\Lambda_1}{\Lambda_1-\Lambda_2}\left.\left(\partial_t\psi-\Lambda_2\psi\right)\right|_{x=0}-\frac{\Lambda_2}{\Lambda_1-\Lambda_2}\left.\left(\partial_t\psi-\Lambda_1\psi\right)\right|_{x=0},$$ 
and from (\ref{6.41}) we obtain
\begin{equation}\begin{aligned}
&\left|\partial_t\psi|_{x=0}-\mathcal{R}(t)\right|\\
\lesssim&\frac{1}{1+\underline{c}t}\left\{\kappa\max_{s\in[0,T]}\left\{\xi_0(s)\left|\frac{w_B}{c^\frac{1}{2}}(\xi_0(s),0)\right|\right\}+\epsilon\max_{s\in[0,T]}\xi_0(s)\mathcal{V}(\xi_0(s))+\epsilon|Y(0)|\right\}\\
\lesssim&\frac{1}{1+\underline{c}t}\left(\kappa\epsilon^\frac{1}{2}\tilde{\epsilon}^\frac{1}{2}+\epsilon^2+\epsilon^\frac{3}{2}\tilde{\epsilon}^\frac{1}{2}\right),
\end{aligned}\label{7.14}\end{equation}
where we used Lemma \ref{lem 7.1} and recall that
$$\begin{aligned}
\mathcal{R}(t):=&\frac{\Lambda_1}{\Lambda_1-\Lambda_2}e^{\Lambda_1t}\left.\left(\partial_t\psi-\Lambda_2\psi\right)\right|_{\left\{x=0,t=0\right\}}-\frac{\Lambda_2}{\Lambda_1-\Lambda_2}e^{\Lambda_2t}\left.\left(\partial_t\psi-\Lambda_1\psi\right)\right|_{\left\{x=0,t=0\right\}}\\
&-\frac{\tilde{f}^\prime(1)}{c_0}\int_0^t\left(\frac{\Lambda_1}{\Lambda_1-\Lambda_2}e^{\Lambda_1(t-s)}-\frac{\Lambda_2}{\Lambda_1-\Lambda_2}e^{\Lambda_2(t-s)}\right)w_B(\xi_0(s),0)ds.
\end{aligned}$$
To control the nonlinearity $\left.\left(\frac{u}{r}\psi+\frac{1}{2}ru^2\right)\right|_{x=0}$, we use
$$u=\rho r^2\partial_x\varphi=\rho r\partial_x\psi-r^{-2}\psi=\frac{\rho}{cr}\left(w_B-\partial_t\psi-\frac{c}{\rho r}\psi\right)$$
to obtain
$$\frac{u}{r}\psi+\frac{1}{2}ru^2=\frac{\rho^2}{2c^2r}\left(w_B-\partial_t\psi-\frac{c}{\rho r}\psi\right)\left(w_B-\partial_t\psi+\frac{c}{\rho r}\psi\right).$$
In view of (\ref{3.2})(\ref{3.10})(\ref{6.38})(\ref{6.41}) and Lemma \ref{lem 7.1}, it holds
\begin{equation}\begin{aligned}
&\left|\partial_t\psi|_{x=0}\right|+\left|\left.\left(\frac{c}{\rho r}\psi\right)\right|_{x=0}\right|\\
\lesssim&|Y(t)-Y_0(t)|+|Y_0(t)|\\
\lesssim&\frac{1}{1+\underline{c}t}\left\{\kappa\max_{s\in[0,T]}\left\{\xi_0(s)\left|\frac{w_B}{c^\frac{1}{2}}(\xi_0(s),0)\right|\right\}+\epsilon\max_{s\in[0,T]}\xi_0(s)\mathcal{V}(\xi_0(s))+\epsilon|Y(0)|\right\}\\
&+|Y_0(t)|\\
\lesssim&\frac{1}{1+\underline{c}t}\left(\kappa\epsilon^\frac{1}{2}\tilde{\epsilon}^\frac{1}{2}+\epsilon^\frac{3}{2}\tilde{\epsilon}^\frac{1}{2}+\epsilon^2\right)+\frac{1}{1+\underline{c}t}\left(\epsilon^\frac{1}{2}\tilde{\epsilon}^\frac{1}{2}+\epsilon\right)\\
\lesssim&\frac{\epsilon^\frac{1}{2}}{1+\underline{c}t}\left(\epsilon^\frac{1}{2}+\tilde{\epsilon}^\frac{1}{2}\right).
\end{aligned}\label{7.15}\end{equation}
Meanwhile, by (\ref{6.29})(\ref{6.30})
$$\begin{aligned}
|w_B(1,t)|
\lesssim&\frac{1}{1+\underline{c}t}\left(\epsilon\max_{\eta\in[1,\xi_0(T)]}\eta\mathcal{V}(\eta)+\epsilon\max_{t_*\in[0,T]}\xi_0(t_*)\Psi(t_*)+\max_{\eta\in[1,\xi_0(T)]}\left\{\eta\left|\frac{w_B}{c^\frac{1}{2}}(\eta,0)\right|\right\}\right).
\end{aligned}$$
From (\ref{6.34})(\ref{6.39})(\ref{6.40}), we obtain 
\begin{equation}\begin{aligned}
&\epsilon\max_{t_*\in[0,T]}\xi_0(t_*)\Psi(t_*)\\
\lesssim&\epsilon|Y(0)|+(\kappa+\epsilon^2)\max_{s\in[0,T]}\left\{\xi_0(s)\left|\frac{w_B}{c^\frac{1}{2}}(\xi_0(s),0)\right|\right\}+(\kappa+\epsilon^2)\epsilon\max_{s\in[0,T]}\xi_0(s)\mathcal{V}(\xi_0(s)).
\end{aligned}\label{7.16}\end{equation}
In view of $\epsilon,\kappa\lesssim 1$, it follows
\begin{equation}\begin{aligned}
|w_B(1,t)|\lesssim&\frac{1}{1+\underline{c}t}\left(\epsilon\max_{\eta\in[1,\xi_0(T)]}\eta\mathcal{V}(\eta)+\epsilon|Y(0)|+\max_{\eta\in[1,\xi_0(T)]}\left\{\eta\left|\frac{w_B}{c^\frac{1}{2}}(\eta,0)\right|\right\}\right)\\
\lesssim&\frac{1}{1+\underline{c}t}\left(\epsilon^\frac{1}{2}\tilde{\epsilon}^\frac{1}{2}+\epsilon^2\right).
\end{aligned}\label{7.17}\end{equation}
Combining (\ref{7.15})(\ref{7.17}) yields
$$
\left|\left.\left(\frac{u}{r}\psi+\frac{1}{2}ru^2\right)\right|_{x=0}\right|\lesssim\frac{\epsilon}{(1+\underline{c}t)^2}(\epsilon+\tilde{\epsilon}).
$$
Therefore, (\ref{7.12})(\ref{7.14}) give
\begin{equation}
\left|F(R)-\mathcal{R}(t)\right|\lesssim\frac{\kappa}{1+\underline{c}t}\epsilon^\frac{1}{2}\tilde{\epsilon}^\frac{1}{2}+\frac{\epsilon^\frac{3}{2}}{1+\underline{c}t}\left(\epsilon^\frac{1}{2}+\tilde{\epsilon}^\frac{1}{2}\right)+\frac{\epsilon}{(1+\underline{c}t)^2}(\epsilon+\tilde{\epsilon}).
\label{7.18}\end{equation}
Now we set $T=t$ so that $\kappa=\epsilon\log(1+\overline{c}t)$. Then (\ref{7.18}) becomes
$$
\left|F(R)-\mathcal{R}(t)\right|\lesssim\frac{\log(1+\overline{c}t)}{1+\underline{c}t}\epsilon^\frac{3}{2}\tilde{\epsilon}^\frac{1}{2}+\frac{\epsilon^\frac{3}{2}}{1+\underline{c}t}\left(\epsilon^\frac{1}{2}+\tilde{\epsilon}^\frac{1}{2}\right)+\frac{\epsilon}{(1+\underline{c}t)^2}(\epsilon+\tilde{\epsilon}),
$$
which is exactly the desired estimate of Theorem \ref{thm 1.3}.
If we assume additionally $w_B|_{t=0}=0$, then (\ref{7.14})(\ref{7.15})(\ref{7.17}) can be improved to
$$\left|\partial_t\psi|_{x=0}-\mathcal{R}(t)\right|\lesssim\frac{1}{1+\underline{c}t}\left(\epsilon^2+\epsilon^\frac{3}{2}\tilde{\epsilon}^\frac{1}{2}\right),$$
$$\begin{aligned}
\left|\partial_t\psi|_{x=0}\right|+\left|\left.\left(\frac{c}{\rho r}\psi\right)\right|_{x=0}\right|\lesssim\frac{1}{1+\underline{c}t}\left(\epsilon^\frac{3}{2}\tilde{\epsilon}^\frac{1}{2}+\epsilon^2\right)+\exp\left(\max_{i=1,2}\left\{\text{Re}\Lambda_i\right\}t\right)\epsilon,
\end{aligned}$$
$$|w_B(1,t)|\lesssim\frac{1}{1+\underline{c}t}\left(\epsilon^\frac{3}{2}\tilde{\epsilon}^\frac{1}{2}+\epsilon^2\right).$$
Hence 
$$\left|\left.\left(\frac{u}{r}\psi+\frac{1}{2}ru^2\right)\right|_{x=0}\right|\lesssim\left(\frac{1}{1+\underline{c}t}\left(\epsilon^2+\epsilon^\frac{3}{2}\tilde{\epsilon}^\frac{1}{2}\right)+\exp\left(\max_{i=1,2}\left\{\text{Re}\Lambda_i\right\}t\right)\epsilon\right)^2,$$
$$\left|F(R)-\mathcal{R}(t)\right|\lesssim\frac{\epsilon^\frac{3}{2}}{1+\underline{c}t}\left(\epsilon^\frac{1}{2}+\tilde{\epsilon}^\frac{1}{2}\right)+\exp\left(2\max_{i=1,2}\left\{\text{Re}\Lambda_i\right\}t\right)\epsilon^2\lesssim\frac{\epsilon^\frac{3}{2}}{1+\underline{c}t}\left(\epsilon^\frac{1}{2}+\tilde{\epsilon}^\frac{1}{2}\right).$$
\subsection{Proof of Corollary \ref{cor 1.4}}
Now we assume additionally $u_{in}$, $q_{in}$ are compactly supported in $[0,x_b)$. From the relations $u=\rho r^2\partial_x\varphi$, $\psi=r\varphi$ and (\ref{rho-phi}), it follows $w_B(\cdot,0)=\left.\left(\partial_t+cr^2\xi^{-2}\partial_\xi\right)\psi\right|_{t=0}$ is compactly supported in $\xi\in[1,\xi_b)$ with $\xi_b=(1+x_b)^\frac{1}{3}$. Suppose $(x,t)\in D_b$ with $D_b$ given by (\ref{1.19}). Define $T$ such that $T=t+\frac{\xi-1}{\underline{c}}$ if $t\geq\frac{1}{\underline{c}}$ and $T=\frac{\xi}{\underline{c}}$ if $t<\frac{1}{\underline{c}}$ so that $T\leq T_0$ and $(\xi,t)$ lies in the backward acoustic cone determined by the point $(1,T)$ on the time axis. Let $\kappa\leq\kappa_0$ be such that $1+\overline{c}T=\exp\left(\frac{\kappa}{\epsilon}\right)$, so we always have $\kappa\geq\epsilon$.
Note that $\xi_0$ is a bijection from $[0,T]$ to $[1,\xi_0(T)]$. Hence taking (\ref{7.16}) back to (\ref{6.28}) yields that for $\eta\in[1,\xi_0(T)]$
\begin{equation}\begin{aligned}
&v_B(\eta,0)-\left|\frac{w_B}{c^\frac{1}{2}}\left(\eta,0\right)\right|\\
\lesssim&\eta^{-1}\left\{\epsilon\max_{s\in[0,T]}\xi_0(s)\mathcal{V}(\xi_0(s))+\kappa\max_{s\in[0,T]}\left\{\xi_0(s)\left|\frac{w_B}{c^\frac{1}{2}}(\xi_0(s),0)\right|\right\}+\epsilon|Y(0)|\right\}.
\end{aligned}\label{7.19}\end{equation} 
For any $t^\prime\leq T$ with $\xi_0(t^\prime)\geq\xi_b$, using (\ref{7.19}) in (\ref{6.16}) gives that 
$$\begin{aligned}
&v_F(1,t^\prime;T)-\left|\frac{w_F}{c^\frac{1}{2}}(1,t^\prime)\right|\\
\lesssim&\epsilon\left|\frac{w_F}{c^\frac{1}{2}}(1,t^\prime)\right|+\epsilon^2\left|\psi(1,t^\prime)\right|\\
&+\epsilon\left(\max_{\eta\in[\xi_0(t^\prime),\xi_0(T)]}\eta v_B(\eta,0)\right)\int_{\xi_0(t^\prime)}^{\xi_0(T)}\left(1+\frac{1}{2\mu}\left(\eta-\xi_0(t^\prime)\right)\right)^{-1}\eta^{-1}d\eta\\
\lesssim&\epsilon\left|\frac{w_F}{c^\frac{1}{2}}(1,t^\prime)\right|+\epsilon^2\left|\psi(1,t^\prime)\right|+\epsilon\int_{\xi_0(t^\prime)}^{\xi_0(T)}\left(1+\frac{1}{2\mu}\left(\eta-\xi_0(t)\right)\right)^{-1}\eta^{-1}d\eta\\
&\cdot\left\{\epsilon\max_{s\in[0,T]}\xi_0(s)\mathcal{V}(\xi_0(s))+\kappa\max_{s\in[0,T]}\left\{\xi_0(s)\left|\frac{w_B}{c^\frac{1}{2}}(\xi_0(s),0)\right|\right\}+\epsilon|Y(0)|\right\}.
\end{aligned}$$
We compute for $\xi_0(t^\prime)>2\mu$  to get
$$\begin{aligned}
&\int_{\xi_0(t^\prime)}^{\xi_0(T)}\left(1+\frac{1}{2\mu}\left(\eta-\xi_0(t^\prime)\right)\right)^{-1}\eta^{-1}d\eta\\
=&\frac{2\mu}{\xi_0(t^\prime)-2\mu}\int_{\xi_0(t^\prime)}^{\xi_0(T)}\left(\frac{1}{\eta-\xi_0(t^\prime)+2\mu}-\frac{1}{\eta}\right)d\eta\\
=&\frac{2\mu}{\xi_0(t^\prime)-2\mu}\log\frac{\left(\xi_0(T)-\xi_0(t^\prime)+2\mu\right)\xi_0(t^\prime)}{2\mu\xi_0(T)}\\
=&\frac{2\mu}{\xi_0(t^\prime)-2\mu}\log\left(1+\frac{1}{2\mu\xi_0(T)}\left(\xi_0(T)-\xi_0(t^\prime)\right)\left(\xi_0(t^\prime)-2\mu\right)\right).
\end{aligned}$$
If $\xi_0(t^\prime)\in(2\mu,2\mu+1]$, 
$$\begin{aligned}
\frac{2\mu}{\xi_0(t^\prime)-2\mu}\log\left(1+\frac{1}{2\mu\xi_0(T)}\left(\xi_0(T)-\xi_0(t^\prime)\right)\left(\xi_0(t^\prime)-2\mu\right)\right)
\leq\frac{\xi_0(T)-\xi_0(t^\prime)}{\xi_0(T)}\leq 1\lesssim\frac{1}{\xi_0(t^\prime)}.
\end{aligned}$$ 
If $\xi_0(t^\prime)>2\mu+1$, then 
$$\begin{aligned}
&\frac{2\mu}{\xi_0(t^\prime)-2\mu}\log\left(1+\frac{1}{2\mu\xi_0(T)}\left(\xi_0(T)-\xi_0(t^\prime)\right)\left(\xi_0(t^\prime)-2\mu\right)\right)\\\leq&\frac{2\mu}{\xi_0(t^\prime)-2\mu}\log\left(1+\frac{(\xi_0(T)-2\mu)^2}{8\mu\xi_0(T)}\right)\\
\lesssim&\frac{\kappa}{\epsilon}\frac{1}{\xi_0(t^\prime)}.
\end{aligned}$$
For $\xi_0(t^\prime)\leq 2\mu$, the estimate is much simpler:
$$\int_{\xi_0(t^\prime)}^{\xi_0(T)}\left(1+\frac{1}{2\mu}\left(\eta-\xi_0(t^\prime)\right)\right)^{-1}\eta^{-1}d\eta\leq2\mu\int_{\xi_0(t^\prime)}^{\xi_0(T)}\eta^{-2}d\eta\leq 2\mu\xi_0(t^\prime)^{-1}.$$
Hence in each case we have
$$\int_{\xi_0(t^\prime)}^{\xi_0(T)}\left(1+\frac{1}{2\mu}\left(\eta-\xi_0(t^\prime)\right)\right)^{-1}\eta^{-1}d\eta\lesssim\frac{\kappa}{\epsilon}\xi_0(t^\prime)^{-1},$$
and it follows
\begin{equation}\begin{aligned}
&v_F(1,t^\prime;T)-\left|\frac{w_F}{c^\frac{1}{2}}(1,t^\prime)\right|\\
\lesssim&\epsilon\left|\frac{w_F}{c^\frac{1}{2}}(1,t^\prime)\right|+\epsilon^2|\psi(1,t^\prime)|\\
&+\kappa\xi_0(t^\prime)^{-1}\left\{\epsilon\max_{s\in[0,T]}\xi_0(s)\mathcal{V}(\xi_0(s))+\kappa\max_{s\in[0,T]}\left\{\xi_0(s)\left|\frac{w_B}{c^\frac{1}{2}}(\xi_0(s),0)\right|\right\}+\epsilon|Y(0)|\right\}.
\end{aligned}\label{7.20}\end{equation}
Since $w_F=-w_B+2\partial_t\psi$, it holds $|w_F(1,t^\prime)|\leq|w_B(1,t^\prime)|+2|\partial_t\psi(1,t^\prime)|$, and thus
\begin{equation}
\left|\frac{w_F}{c^\frac{1}{2}}(1,t^\prime)\right|\lesssim v_B(1,t^\prime)+|Y(t^\prime)|=v_B(\xi_0(t^\prime),0)+|Y(t^\prime)|.
\label{7.21}\end{equation}
Denote $\Lambda:=\max\left\{\text{Re}\Lambda_1,\text{Re}\Lambda_2\right\}$. Since $w_B$ is compactly supported in $[1,\xi_b)$, we have $w_B(\xi_0(s),0)=0$ for $s\geq T_b$. It then follows from (\ref{6.41}) that
\begin{equation}\begin{aligned}
|Y(t^\prime)|\leq&e^{\Lambda t^\prime}\left(|Y(0)|+\frac{|\tilde{f}^\prime(1)|}{c_0}\int_0^{T_b}e^{-\Lambda s}|w_B(\xi_0(s),0)|ds\right)\\
&+\frac{C}{1+\underline{c}t^\prime}\left\{\epsilon\max_{s\in[0,T]}\xi_0(s)\mathcal{V}(\xi_0(s))+\kappa\max_{s\in[0,T]}\left\{\xi_0(s)\left|\frac{w_B}{c^\frac{1}{2}}(\xi_0(s),0)\right|\right\}+\epsilon|Y(0)|\right\}.
\end{aligned}\label{7.22}\end{equation}
Since $(x,t)\in D_b$, $1+\underline{c}t\geq\xi+\xi_b-1\geq\xi$, and thus there exists a unique $T_\xi(t)\in[t-\frac{\xi-1}{\underline{c}},t]$ such that $X_F(1,T_\xi(t);t-T_\xi(t))=(\xi,t)$ by continuity, cf. (figure \ref{7}).
Since $u=\rho r^2\partial_x\varphi$ and $\varphi=r^{-1}\psi$, it holds $u=\rho r\partial_x\psi-r^{-2}\psi$. In view of $r\simeq\xi$ and $\underline{c}\leq c\leq\overline{c}$, it follows 
\begin{equation}\begin{aligned}
|u(x,t)|\lesssim&\xi^{-1}\left(|w_B(\xi,t)|+|w_F(\xi,t)|\right)+\xi^{-2}|\psi(\xi,t)|\\
\lesssim&\xi^{-1}\left(v_B(\xi,t)+v_F(\xi,t;T)\right)+\xi^{-2}|\psi(\xi,t)|\\
=&\xi^{-1}\left(v_B(X_B(\xi,t;-t))+v_F(1,T_\xi(t);T)\right)+\xi^{-2}|\psi(\xi,t)|.
\end{aligned}\label{7.23}\end{equation} 
To control $|\psi(\xi,t)|$, we recall from Subsection \ref{subsec6.1} that $X_F(1,t;\sigma_1(\eta,t))=X_B(\eta,0;t+\sigma_1(\eta,t))$ and use Lemma \ref{lem 6.1} to get
$$\begin{aligned}
|\psi(\xi,t)|=&|\psi(X_F(1,T_\xi(t);t-T_\xi(t))|\\
\leq&|\psi(1,T_\xi(t))|+\int_0^{t-T_\xi(t)}|w_B\left(X_F(1,T_\xi(t);\tau)\right)|d\tau\\
\lesssim&|Y(T_\xi(t))|+\int_0^{t-T_\xi(t)}v_B\left(X_F(1,T_\xi(t);\tau)\right)d\tau\\
\lesssim&|Y(T_\xi(t))|+\int_{\xi_0(T_\xi(t))}^{X_B^{(\xi)}(\xi,t;-t)}v_B(\eta,0)d\eta.
\end{aligned}$$
\noindent Since $T_\xi(t)\geq t-\frac{\xi-1}{\underline{c}}$ and $(x,t)\in D_b$, it holds $$\xi_0(T_\xi(t))\geq 1+\underline{c}T_\xi(t)\geq\underline{c}t+2-\xi\geq\xi_b.$$
Hence it follows from (\ref{7.19}) and $w_B(\eta,0)=0$, $\eta\geq\xi_b$ that
$$\begin{aligned}
|\psi(\xi,t)|\lesssim&|Y(T_\xi(t))|+\int_{\xi_0(T_\xi(t))}^{X_B^{(\xi)}(\xi,t;-t)}\eta^{-1}d\eta\\
&\cdot\left\{\epsilon\max_{s\in[0,T]}\xi_0(s)\mathcal{V}(\xi_0(s))+\kappa\max_{s\in[0,T]}\left\{\xi_0(s)\left|\frac{w_B}{c^\frac{1}{2}}(\xi_0(s),0)\right|\right\}+\epsilon|Y(0)|\right\}.
\end{aligned}$$
We claim that for all $s\in[0,T_\xi(t)]$, 
$$X_B^{(\xi)}(\xi,t;T_\xi(t)-t-s)-X_B^{(\xi)}(1,T_\xi(t);-s)\leq e^{C\kappa}\left(X_B^{(\xi)}(\xi,t;T_\xi(t)-t)-1\right).$$
\begin{figure}[H]
\center{\includegraphics[width=10cm]  {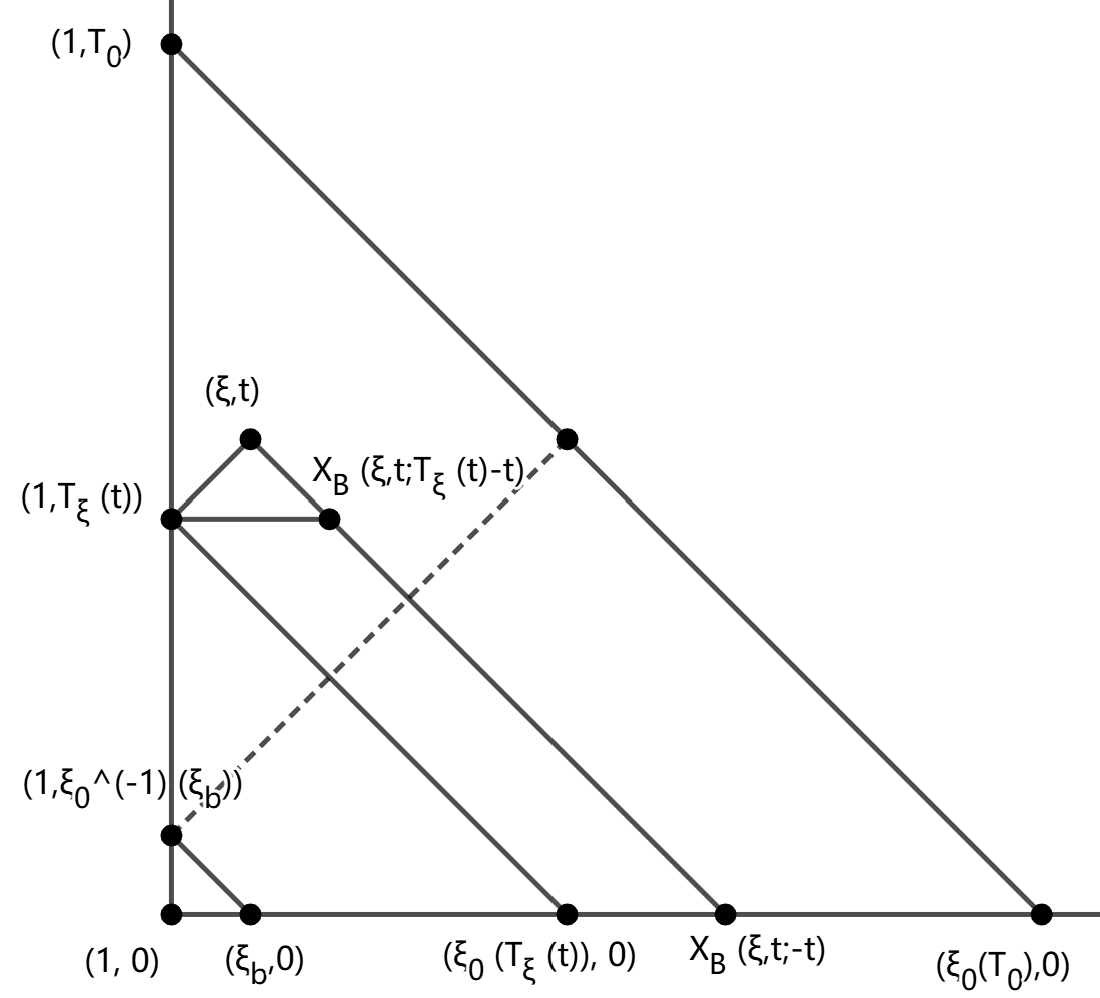}}  
\caption{Position of the point $X_B(\xi,t;T_\xi(t)-t)$}\label{7}  
\end{figure}
\noindent
In particular, taking $s=T_\xi(t)$ yields
$$\begin{aligned}
&X_B^{(\xi)}(\xi,t;-t)-\xi_0(T_\xi(t))\leq e^{C\kappa}\left(X_B^{(\xi)}(\xi,t;T_\xi(t)-t)-1\right)\\
\leq&e^{C\kappa}\left(\xi-1+\overline{c}(t-T_\xi(t)\right)\leq e^{C\kappa}\frac{\overline{c}+\underline{c}}{\underline{c}}(\xi-1).
\end{aligned}$$ 
Note that $X_B^{(\xi)}(\xi,t;T_\xi(t)-t-\tau)=X_B^{(\xi)}\left(X_B(\xi,t;T_\xi(t)-t);-\tau\right)$. Using (\ref{3.6})(\ref{3.7}) and a mean value formula, we obtain
$$\begin{aligned}
&\left|\frac{d}{d\tau}\left(X_B^{(\xi)}(\xi,t;T_\xi(t)-t-\tau)-X_B^{(\xi)}(1,T_\xi(t);-\tau)\right)\right|\\
\leq&C\epsilon X_B^{(\xi)}(1,T_\xi(t);-\tau)^{-1}\left(X_B^{(\xi)}(\xi,t;T_\xi(t)-t-\tau)-X_B^{(\xi)}(1,T_\xi(t);-\tau)\right)\\
\leq&C\epsilon(1+\underline{c}\tau)^{-1}\left(X_B^{(\xi)}(\xi,t;T_\xi(t)-t-\tau)-X_B^{(\xi)}(1,T_\xi(t);-\tau)\right).
\end{aligned}$$
Gronwall's inequality then yields 
$$\begin{aligned}
X_B^{(\xi)}(\xi,t;T_\xi(t)-t-s)-X_B^{(\xi)}(1,T_\xi(t);-s)
\leq&\exp\left\{\int_0^s C\epsilon(1+\underline{c}\tau)^{-1}d\tau\right\}\left(X_B^{(\xi)}(\xi,t;T_\xi(t)-t)-1\right)\\
\leq&e^{C\kappa}\left(X_B^{(\xi)}(\xi,t;T_\xi(t)-t)-1\right),
\end{aligned}$$
which proves the claim. Therefore
$$\begin{aligned}
\int_{\xi_0(T_\xi(t))}^{X_B^{(\xi)}(\xi,t;-t)}\eta^{-1}d\eta
\leq\log\left(1+\xi_0(T_\xi(t))^{-1}e^{C\kappa}\frac{\overline{c}+\underline{c}}{\underline{c}}(\xi-1)\right)\lesssim\xi_0(T_\xi(t))^{-1}(\xi-1),
\end{aligned}$$
and 
\begin{equation}\begin{aligned}
|\psi(\xi,t)|\lesssim&|Y(T_\xi(t))|+\xi_0(T_\xi(t))^{-1}(\xi-1)\\
&\cdot\left\{\epsilon\max_{s\in[0,T]}\xi_0(s)\mathcal{V}(\xi_0(s))+\kappa\max_{s\in[0,T]}\left\{\xi_0(s)\left|\frac{w_B}{c^\frac{1}{2}}(\xi_0(s),0)\right|\right\}+\epsilon|Y(0)|\right\}.
\end{aligned}\label{7.24}\end{equation}
Now we take $\eta=X_B(\xi,t;-t)$ in (\ref{7.19}), $t^\prime=T_\xi(t)$ in (\ref{7.20}), and collect the resulted bounds as well as (\ref{7.24}). It then follows from (\ref{7.23}) that 
$$\begin{aligned}
|u(x,t)|\lesssim&\xi^{-1}\left|\frac{w_F}{c^\frac{1}{2}}(1,T_\xi(t))\right|+\xi^{-2}|Y(T_\xi(t))|+\epsilon^2\xi^{-1}|\psi(1,T_\xi(t))|\\
&+\left(\xi^{-1}X_B^{(\xi)}(\xi,t;-t)^{-1}+\xi^{-2}(\xi-1)\xi_0(T_\xi(t))^{-1}+\kappa\xi^{-1}\xi_0(T_\xi(t))^{-1}\right)\\
&\cdot\left\{\epsilon\max_{s\in[0,T]}\xi_0(s)\mathcal{V}(\xi_0(s))+\kappa\max_{s\in[0,T]}\left\{\xi_0(s)\left|\frac{w_B}{c^\frac{1}{2}}(\xi_0(s),0)\right|\right\}+\epsilon|Y(0)|\right\}.
\end{aligned}$$
Next, we proceed with the help of (\ref{7.21}) and (\ref{7.19}):
\begin{equation}\begin{aligned}
|u(x,t)|
\lesssim&\xi^{-1}|Y(T_\xi(t)|+\left(\xi^{-1}X_B^{(\xi)}(\xi,t;-t)^{-1}+\xi^{-2}(\xi-1)\xi_0(T_\xi(t))^{-1}+\xi^{-1}\xi_0(T_\xi(t))^{-1}\right)\\
&\cdot\left\{\epsilon\max_{s\in[0,T]}\xi_0(s)\mathcal{V}(\xi_0(s))+\kappa\max_{s\in[0,T]}\left\{\xi_0(s)\left|\frac{w_B}{c^\frac{1}{2}}(\xi_0(s),0)\right|\right\}+\epsilon|Y(0)|\right\}\\
\lesssim&\xi^{-1}|Y(T_\xi(t)|+\xi^{-1}\xi_0(T_\xi(t))^{-1}\\
&\cdot\left\{\epsilon\max_{s\in[0,T]}\xi_0(s)\mathcal{V}(\xi_0(s))+\kappa\max_{s\in[0,T]}\left\{\xi_0(s)\left|\frac{w_B}{c^\frac{1}{2}}(\xi_0(s),0)\right|\right\}+\epsilon|Y(0)|\right\},
\end{aligned}\label{7.25}\end{equation}
where we used that $\xi^{-2}(\xi-1)\leq\xi^{-1}$ and $\xi_0(T_\xi(t))\leq X_B^{(\xi)}(\xi,t;-t)$.
Since $u_{in}$, $q_{in}$ are compactly supported in $[1,x_b)$, it holds $\tilde{\epsilon}\lesssim\xi_b^2\epsilon$. From Lemma \ref{lem 7.1} it follows
$$\epsilon\max_{s\in[0,T]}\xi_0(s)\mathcal{V}(\xi_0(s))+\epsilon\max_{s\in[0,T]}\left\{\xi_0(s)\left|\frac{w_B}{c^\frac{1}{2}}(\xi_0(s),0)\right|\right\}+\epsilon|Y(0)|\lesssim\kappa\epsilon\xi_b+\epsilon^2\xi_b\lesssim\kappa\epsilon\xi_b.$$
In view of (\ref{7.22}) and the facts that $T_\xi(t)\geq t-\frac{\xi-1}{\underline{c}}$, $\xi_0(T_\xi(t))\geq1+\underline{c}T_\xi(t)\geq2+\underline{c}t-\xi$, (\ref{7.25}) gives 
$$\begin{aligned}
|u(x,t)|\lesssim\xi^{-1}e^{\Lambda(t-\frac{\xi-1}{\underline{c}})}\left(|Y(0)|+\frac{|\tilde{f}^\prime(1)|}{c_0}\int_0^{T_b}e^{-\Lambda s}|w_B(\xi_0(s),0)|ds\right)+\frac{\kappa\epsilon\xi_b}{\xi(2+\underline{c}t-\xi)}.
\end{aligned}$$
Hence for $t\geq\frac{1}{\underline{c}}$ we obtain the desired bound of $u$
\begin{equation}\begin{aligned}
|u(x,t)|\lesssim&\xi^{-1}e^{\Lambda(t-\frac{\xi-1}{\underline{c}})}\left(|Y(0)|+\frac{|\tilde{f}^\prime(1)|}{c_0}\int_0^{T_b}e^{-\Lambda s}|w_B(\xi_0(s),0)|ds\right)\\
&+\xi^{-1}(2+\underline{c}t-\xi)^{-1}\log\left(1+\overline{c}\left(t+\frac{\xi-1}{\underline{c}}\right)\right)\epsilon^2\xi_b.
\end{aligned}\label{7.26}\end{equation}
If assume additionally $w_B|_{t=0}=0$, then 
$$\epsilon\max_{s\in[0,T]}\xi_0(s)\mathcal{V}(\xi_0(s))+\epsilon\max_{s\in[0,T]}\left\{\xi_0(s)\left|\frac{w_B}{c^\frac{1}{2}}(\xi_0(s),0)\right|\right\}+\epsilon|Y(0)|\lesssim\epsilon^2\xi_b,$$
and thus
\begin{equation}\begin{aligned}
|u(x,t)|\lesssim&\xi^{-1}e^{\Lambda(t-\frac{\xi-1}{\underline{c}})}\left(|Y(0)|+\frac{|\tilde{f}^\prime(1)|}{c_0}\int_0^{T_b}e^{-\Lambda s}|w_B(\xi_0(s),0)|ds\right)\\
&+\xi^{-1}(2+\underline{c}t-\xi)^{-1}\epsilon^2\xi_b.
\end{aligned}\label{7.27}\end{equation}
In view of $\rho\simeq1$, $|u|\lesssim\epsilon$, relation (\ref{rho-phi}) and that $\partial_t\varphi=r^{-1}\partial_t\psi-\frac{u}{r^2}\psi$, we can bound $q$ by
$$\begin{aligned}
|q(x,t)|\lesssim&\xi^{-1}|\partial_t\psi(\xi,t)|+|u(x,t)|\left(|u(x,t)|+\xi^{-2}|\psi(\xi,t)|\right)\\
\lesssim&\xi^{-1}\left(|w_B(\xi,t)|+|w_F(\xi,t)|\right)+|u(x,t)|\left(|u(x,t)|+\xi^{-2}|\psi(\xi,t)|\right)\\
\lesssim&\xi^{-1}\left(v_B(X_B(\xi,t;-t))+v_F(1,\xi_0(T_\xi(t);T)\right)+\xi^{-2}|\psi(\xi,t)|.
\end{aligned}$$
Hence $|q(x,t)|$ can be controlled by the same bound as in (\ref{7.26}) or (\ref{7.27}).
\section*{Declarations}
\noindent \textbf{Funding}: This work was supported by National Natural Science Foundation of China Grant No. 12271497.\vspace{1mm}\\
No potential conflict of interest was reported by the author.
\section*{Data availability statements}
Data availability is not applicable to this article since the article has no associated data.
\begin{appendices}
\section{Detailed calculation of identity (\ref{4.3}) } In this appendix we write down the explicit calculation of (\ref{4.3}), for which we consider the following model question.\\
Assume that $\varphi$ satisfies for $x>0$, $t\in[0,T]$ that
\begin{equation}
\partial_t^2\varphi-c^2\partial_x(r^4\partial_x\varphi)=F,
\label{ap1}\end{equation}
and the modified momentum density are defined as
$$P_0=A\partial_t\varphi\cdot cr^2\partial_x\varphi+\frac{1}{2}B\varphi\partial_t\varphi,$$
$$P_1=\frac{1}{2}A\left[\left(\partial_t\varphi\right)^2+\left(cr^2\partial_x\varphi\right)^2\right]+\frac{1}{2}B\varphi\cdot cr^2\partial_x\varphi-\frac{1}{4}cr^2\partial_x B\varphi^2,$$
where $A$ and $B$ are undetermined weight functions.
We compute:
$$\begin{aligned}
\partial_x\left(cr^2P_1\right)=&\frac{1}{2}\partial_x\left[Acr^2\left(\left(\partial_t\varphi\right)^2+\left(cr^2\partial_x\varphi\right)^2\right)\right]+\frac{1}{2}B\left(cr^2\partial_x\varphi\right)^2+\frac{1}{2}B\varphi\partial_x\left(c^2r^4\partial_x\varphi\right)\\
&+\frac{1}{2}cr^2\partial_xB\cdot\varphi\cdot cr^2\partial_x\varphi-\frac{1}{4}\partial_x\left(c^2r^4\partial_xB\right)\varphi^2-\frac{1}{2}cr^2\partial_xB\cdot\varphi\cdot cr^2\partial_x\varphi\\
=&\frac{1}{2}\partial_x\left(cr^2A\right)\left(\left(\partial_t\varphi\right)^2+\left(cr^2\partial_x\varphi\right)^2\right)+\frac{1}{2}A cr^2\partial_x\left(\left(\partial_t\varphi\right)^2+\left(cr^2\partial_x\varphi\right)^2\right)\\
&+\frac{1}{2}B\left(cr^2\partial_x\varphi\right)^2+\frac{1}{2}B\varphi\partial_x\left(c^2r^4\partial_x\varphi\right)-\frac{1}{4}\partial_x\left(c^2r^4\partial_xB\right)\varphi^2,
\end{aligned}$$
$$\begin{aligned}
-\partial_tP_0=&-\partial_t\left(A\partial_t\varphi\cdot cr^2\partial_x\varphi\right)-\frac{1}{2}(\partial_tB)\varphi\partial_t\varphi-\frac{1}{2}B\left(\partial_t\varphi\right)^2-\frac{1}{2}B\varphi\partial_t^2\varphi\\
=&-(\partial_tA)\partial_t\varphi\cdot cr^2\partial_x\varphi-A\partial_t\left(\partial_t\varphi\cdot cr^2\partial_x\varphi\right)-\frac{1}{2}B\left(\partial_t\varphi\right)^2\\
&-\frac{1}{2}B\varphi\partial_t^2\varphi-\frac{1}{2}(\partial_tB)\varphi\partial_t\varphi.
\end{aligned}$$
Note that by \ref{ap1} we have 
$$\begin{aligned}
&\frac{1}{2}cr^2\partial_x\left(\left(\partial_t\varphi\right)^2+\left(cr^2\partial_x\varphi\right)^2\right)-\partial_t\left(\partial_t\varphi\cdot cr^2\partial_x\varphi\right)\\
=&\partial_t\varphi\left(cr^2\partial_x\partial_t\varphi-\partial_t\left(cr^2\partial_x\varphi\right)\right)+cr^2\partial_x\varphi\left(cr^2\partial_x\left(cr^2\partial_x\varphi\right)-\partial_t^2\varphi\right)\\
=&-\frac{\partial_t(cr^2)}{cr^2}\partial_t\varphi\cdot cr^2\partial_x\varphi+cr^2\partial_x\varphi\left(cr^2\partial_x\left(\frac{c}{r^2}\right)r^4\partial_x\varphi-F\right)\\
=&-c\partial_xr^2\left(cr^2\partial_x\varphi\right)^2
+r^2\partial_xc\left(cr^2\partial_x\varphi\right)^2-\frac{\partial_t(cr^2)}{cr^2}\partial_t\varphi\cdot cr^2\partial_x\varphi-cr^2\partial_x\cdot F,
\end{aligned}$$
and 
$$\begin{aligned}
\partial_x\left(c^2r^4\partial_x\varphi\right)-\partial_t^2\varphi=2\left(r^2\partial_xc\right)cr^2\partial_x\varphi-F.
\end{aligned}$$
Hence
$$\begin{aligned}
&\partial_x\left(cr^2P_1\right)-\partial_tP_0\\
=&\left(\frac{\partial_x\left(cr^2A\right)}{2}-\frac{B}{2}\right)\left(\partial_t\varphi\right)^2+\left(\frac{\partial_x\left(cr^2A\right)}{2}-Ac\partial_xr^2+\frac{B}{2}\right)\left(cr^2\partial_x\varphi\right)^2\\
&-\frac{1}{4}\partial_x\left(c^2r^4\partial_xB\right)\varphi^2+r^2\partial_xc\cdot A\left(cr^2\partial_x\varphi\right)^2-A\frac{\partial_t(cr^2)}{cr^2}\partial_t\varphi\cdot cr^2\partial_x\varphi\\
&+Br^2\partial_xc\cdot\varphi\cdot cr^2\partial_x\varphi-\left(Acr^2\partial_x\varphi+\frac{1}{2}B\varphi\right)F-\partial_tA\cdot\partial_t\varphi\cdot cr^2\partial_x\varphi-\frac{\partial_tB}{2}\varphi\partial_t\varphi.
\end{aligned}$$
Therefore we obtain (\ref{4.3}) by substituting $\varphi$ by $\partial_t^j\varphi$, $F$ by $[\partial_t^j,c^2\partial_xr^4\partial_x]\varphi$, $A$ by $M_k$ and $B$ by $N_k$ in view of (\ref{4.2}).
\end{appendices}
\bibliography{Bubble}
\end{document}